\documentclass[english,ruled]{article}
\usepackage[T1]{fontenc}
\usepackage[latin9]{inputenc}
\usepackage{verbatim}
\usepackage{subcaption}
\usepackage{algorithm2e}
\usepackage{amsmath}
\usepackage{amsthm}
\usepackage{amssymb}
\usepackage{graphicx}
\usepackage{xcolor}
\usepackage{multicol}
\usepackage{multirow}
\usepackage{geometry}
\usepackage{booktabs}
\usepackage{enumitem}
\usepackage{setspace}
\makeatletter
\usepackage[toc,page,header]{appendix}
\usepackage{minitoc}
\usepackage{ifthen}
\newboolean{doublecolumn}
\setboolean{doublecolumn}{false}
\newboolean{arxiv}
\setboolean{arxiv}{true}
\usepackage{pgfplots}
\usepackage{pdflscape}
\usepackage{hyperref}
\usepackage{cleveref}
\usepackage{authblk}
\pgfplotsset{compat=1.15}

\newcommand{\assign}{:=}

\newcommand{\cdummy}{\cdot}

\newcommand{\tmop}[1]{\ensuremath{\operatorname{#1}}}

\newcommand{\minf}[1]{\underset{#1}{\text{minimize}}}
\newcommand{\hdm}{{\texttt{HDM}}}
\newcommand{\hdmagd}{{\texttt{HDM-AGD}}}
\newcommand{\hdmhb}{{\texttt{HDM-HB}}}
\newcommand{\hdmbest}{{\texttt{HDM-Best}}}

\newcommand{\bfgs}{{\texttt{BFGS}}}
\newcommand{\lbfgs}{{\texttt{L-BFGS}}}

\newcommand{\adagrad}{{\texttt{AdaGrad}}}
\newcommand{\adam}{{\texttt{Adam}}}

\newcommand{\agd}{{\texttt{AGD}}}
\newcommand{\mathd}{\mathrm{d}}
\newcommand{\nin}{\not\in}

\global\long\def\vertiii#1{\left\vert \kern-0.25ex  \left\vert \kern-0.25ex  \left\vert #1\right\vert \kern-0.25ex  \right\vert \kern-0.25ex  \right\vert }%

\global\long\def\diam{\mathrm{diam}}%

\global\long\def\argmin{\operatornamewithlimits{arg\,min}}%

\global\long\def\diag{\mathrm{diag}}%

\global\long\def\and{\mathrm{and}}%

\global\long\def\dist{\mathrm{dist}}%

\global\long\def\Rbb{\mathbb{R}}%

\global\long\def\Bcal{\mathcal{B}}%

\global\long\def\Dcal{\mathcal{D}}%

\global\long\def\Lcal{\mathcal{L}}%

\global\long\def\Ocal{\mathcal{O}}%

\global\long\def\Pcal{\mathcal{P}}%

\global\long\def\Scal{\mathcal{S}}%

\global\long\def\Xcal{\mathcal{X}}%

\theoremstyle{plain}
\newtheorem{lem}{\protect\lemmaname}[section]
\theoremstyle{remark}
\newtheorem{rem}{\protect\remarkname}
\theoremstyle{plain}
\newtheorem{thm}{\protect\theoremname}[section]
\theoremstyle{plain}

\providecommand{\corollaryname}{Corollary}
\theoremstyle{plain}

\theoremstyle{plain}

\theoremstyle{plain}
\newtheorem{definition}{\protect\definitionname}[section]

\providecommand{\lemmaname}{Lemma}
\providecommand{\remarkname}{Remark}
\providecommand{\theoremname}{Theorem}
\providecommand{\examplename}{Example}
\providecommand{\propositionname}{Proposition}
\providecommand{\definitionname}{Definition}

\crefdefaultlabelformat{#2\textbf{#1}#3} 
\crefname{section}{\textbf{section}}{\textbf{sections}}
\Crefname{section}{\textbf{Section}}{\textbf{Sections}}
\crefname{thm}{\textbf{theorem}}{\textbf{theorems}}
\Crefname{thm}{\textbf{Theorem}}{\textbf{Theorems}}
\crefname{lem}{\textbf{lemma}}{\textbf{lemmas}}
\Crefname{lem}{\textbf{Lemma}}{\textbf{Lemmas}}
\crefname{prop}{\textbf{proposition}}{\textbf{propositions}}
\Crefname{prop}{\textbf{Proposition}}{\textbf{Propositions}}
\crefname{algorithm}{\textbf{algorithm}}{\textbf{algorithms}}
\Crefname{algorithm}{\textbf{Algorithm}}{\textbf{Algorithms}}
\crefname{coro}{\textbf{Corollary}}{\textbf{corollaries}}
\Crefname{coro}{\textbf{Corollary}}{\textbf{corollaries}}
\crefname{definition}{\textbf{Definition}}{\textbf{definitions}}
\Crefname{definition}{\textbf{Definition}}{\textbf{definitions}}
\crefname{table}{\textbf{Table}}{\textbf{tables}}
\Crefname{table}{\textbf{Table}}{\textbf{tables}}
\crefname{figure}{\textbf{Figure}}{\textbf{figures}}
\Crefname{figure}{\textbf{Figure}}{\textbf{figures}}

\newcommand{\YC}[1]{ }
\renewcommand{\YC}[1]{\textcolor{blue}{[YC: #1]}}

\begin{document}

\title{Provable and Practical Online Learning Rate Adaptation with Hypergradient Descent}

\author[1]{Ya-Chi Chu\thanks{ycchu97@stanford.edu}}
\author[2]{Wenzhi Gao\thanks{gwz@stanford.edu, equal contribution}}
\author[2,3]{Yinyu Ye\thanks{yyye@stanford.edu}}
\author[2,3]{Madeleine Udell\thanks{udell@stanford.edu}}
\affil[1]{Department of Mathematics, Stanford University}
\affil[2]{ICME, Stanford University}
\affil[3]{Department of Management Science and Engineering, Stanford University}

\maketitle

 \begin{abstract}
 This paper investigates the convergence properties of the hypergradient descent method ({\hdm}), a 25-year-old heuristic originally proposed for adaptive stepsize selection in stochastic first-order methods \cite{almeida1999parameter, gunes2018online}.  
 We provide the first rigorous convergence analysis of {\hdm} using the online learning framework of \cite{gao2024gradient} and apply this analysis to develop new state-of-the-art adaptive gradient methods with empirical and theoretical support. 
 Notably, {\hdm} automatically identifies the optimal stepsize for the local optimization landscape and achieves local superlinear convergence. Our analysis explains the instability of {\hdm} reported in the literature and proposes efficient strategies to address it. We also develop two {\hdm} variants with heavy-ball and Nesterov momentum. Experiments on deterministic convex problems show {\hdm} with heavy-ball momentum (\hdmhb) exhibits robust performance and significantly outperforms other adaptive first-order methods. Moreover, {\hdmhb} often matches the performance of \texttt{L-BFGS}, an efficient and practical quasi-Newton method, using less memory and cheaper iterations.
\end{abstract}
\section{Introduction}\label{sec:intro}

We consider the smooth convex optimization problem
\begin{eqnarray*}
  \minf{x \in \mathbb{R}^n} & f (x), & 
\end{eqnarray*}
where $f : \mathbb{R}^n \rightarrow \mathbb{R}$ is convex and $L$-smooth with $f
(x^{\star}) \assign \min_x f (x) > - \infty$. Theoretically, gradient descent
\begin{equation*}
	x^{k+1} = x^k - \alpha_k \nabla f(x^k)
\end{equation*}
 with constant stepsize $\alpha_k \equiv 1/L$ is guaranteed to converge. However, the choice of stepsize $\alpha_k$ strongly affects the performance of gradient descent in practice \cite{defazio2024road}, and various stepsize selection strategies have been proposed to improve the practical convergence of gradient descent. Examples include line-search \cite{armijo1966minimization}, Polyak stepsize \cite{polyak1987introduction}, stepsize scheduling  \cite{li2021second, wang2023convergence}, hypergradient descent \cite{almeida1999parameter,rubio2017convergence,gunes2018online} and  the well-known adaptive stepsizes \cite{orabona2016coin,duchi2011adaptive,kingma2014adam}.

Our paper focuses on the \emph{hypergradient descent method} (\hdm), which was initially proposed in \cite{almeida1999parameter} as a heuristic for stochastic optimization. It was later tested on modern machine learning problems and exhibited promising performance \cite{gunes2018online}. In {\hdm}, the stepsize $\alpha_k$ is adjusted by another gradient descent update:
\begin{align*}
\alpha_{k+1} &= \alpha_k - \tilde{\eta}_k \tfrac{\mathd}{\mathd \alpha}[f(x^k - \alpha \nabla f(x^k))]\big|_{\alpha = \alpha_k} = \alpha_k - \eta_k \tfrac{-\langle \nabla f(x^{k+1}), \nabla f(x^k) \rangle}{\| \nabla f(x^k) \|^2},
\end{align*}
where the hypergradient stepsize $\tilde{\eta}_k$ is often set to be $\tilde{\eta}_k = \tfrac{\eta_k}{\| \nabla f(x^k) \|^2}$ for $\eta_k > 0$ to be invariant of the scaling of $f$.
Gao et al. \cite{gao2024gradient} generalized {\hdm}
to support learning a \emph{preconditioned} gradient descent update with 
preconditioner (matrix stepsize) $P_k \in \Rbb^{n\times n}$
through the iteration
\begin{align}
    x^{k+1} ={} & x^k - P_k \nabla f(x^k), \label{eqn:pgd} \\
    P_{k+1} ={} & \Pi_{\Pcal} \big[P_k - \eta_k \tfrac{-\nabla f(x^{k+1}) \nabla f(x^k)^{\top}}{\| \nabla f(x^k) \|^2}\big], \label{eqn:hdm-Pk-update}
\end{align}
where \eqref{eqn:hdm-Pk-update} follows from $\nabla_P[ f(x^k - P \nabla f(x^k))] \big|_{P = P_k} = - \nabla f(x^{k+1}) \nabla f(x^k)^{\top}$ and $\Pi_{\Pcal}[\cdot]$ is orthogonal projection on to a compact set of candidate preconditioners $\Pcal$. Note that $P$ does not need to be positive definite \cite{gao2024gradient}, hence the projection is easy to compute in practice. We call the update \eqref{eqn:pgd}-\eqref{eqn:hdm-Pk-update} \emph{vanilla} {\hdm} throughout the paper. In practice, $P_k$ is often set to be diagonal and \eqref{eqn:hdm-Pk-update} simplifies to $P_{k+1} = P_k - \eta_k \tfrac{-\diag(\nabla f(x^{k+1}) \circ \nabla f(x^k))}{\| \nabla f(x^k) \|^2}$, where $\circ$ is entry-wise product.\\

\begin{figure}[t]
  \centering
  \begin{subfigure}{0.35\textwidth}
\includegraphics[height=0.2\textheight]{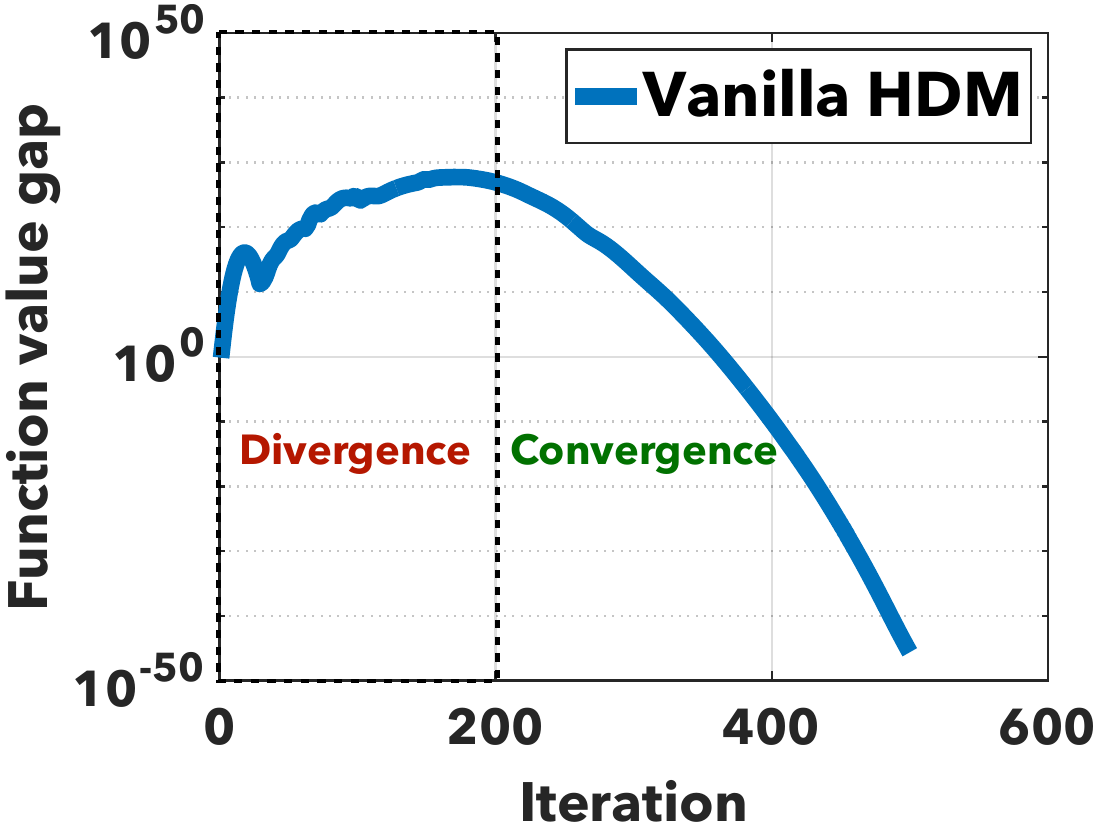}
    \caption{Two-phase behavior}
    \label{fig:demo:a}
  \end{subfigure}
  \begin{subfigure}{0.35\textwidth}~~
\includegraphics[height=0.2\textheight]{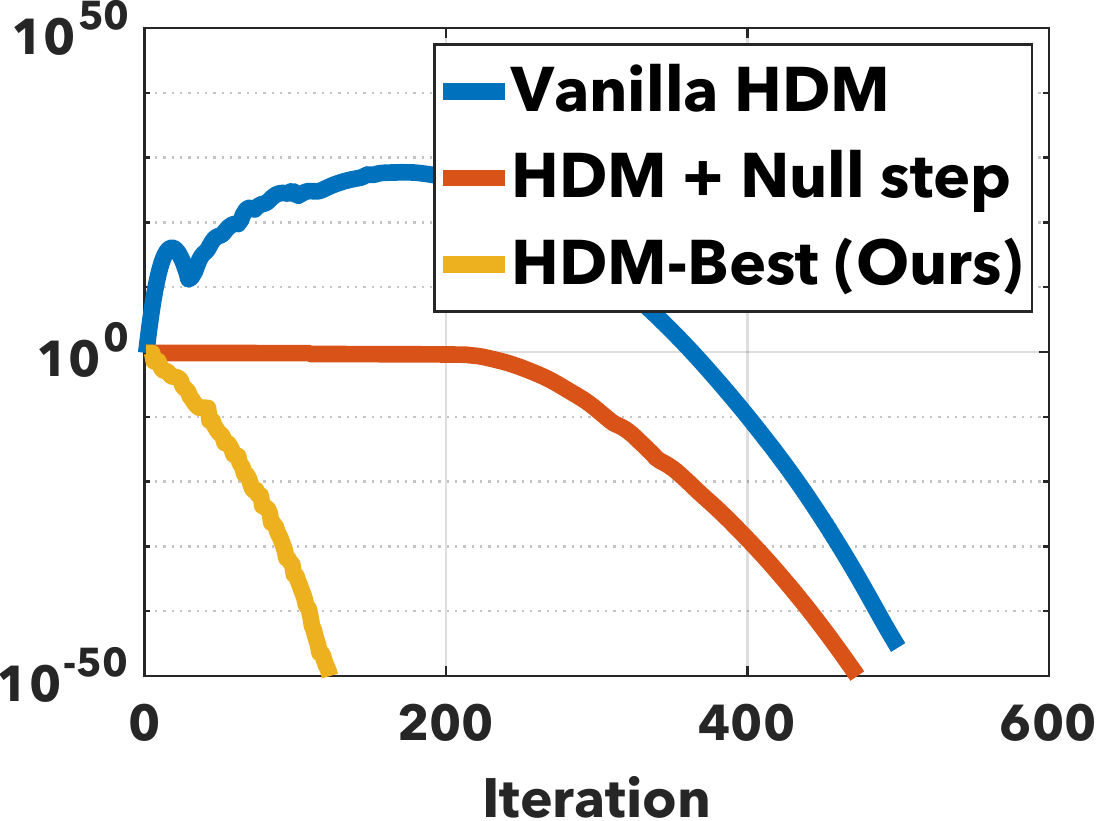}
    \caption{Addressing instability}
    \label{fig:demo:b}
  \end{subfigure}
\caption{The behavior of different {\hdm} variants on the toy quadratic optimization problem. \Cref{fig:demo:a}: two-phase convergence behavior of vanilla {\hdm}. \Cref{fig:demo:b}: effect of null step and our best variant {\hdmbest}. \label{fig-demos}}
  \label{fig:demo:instability}
\end{figure}

Vanilla {\hdm} is widely used in practice, but it can be unstable if the hypergradient stepsize $\eta_k$ is not carefully tuned \cite{kunstner2024searching,chandra2022gradient,rubio2017convergence}. 
\Cref{fig:demo:a} shows $f(x^k)$ can spike as high as $10^{30}$ in the early iterations of vanilla {\hdm}, which would lead experienced users to abandon the algorithm. 
Surprisingly, our analysis reveals that this behavior of {\hdm} is not true divergence; instead, it can be understood as the warm-up phase of an online learning procedure, and is followed by fast convergence (\Cref{fig:demo:a}).
Moreover, we show in both theory and practice that the explosion of $f(x^k)$ can be circumvented by taking a \emph{null step}, which skips the update whenever the new iterate fails to decrease the objective value, i.e., $f(x^k - P_k \nabla f(x^k)) \geq f(x^k)$.
The null steps flatten the objective value curve in the warm-up phase of {\hdm} but cannot shorten the warm-up (\Cref{fig:demo:b}).\\

Our analysis exploits the online learning framework in \cite{gao2024gradient}, in which the authors observe that the $P$-update \eqref{eqn:hdm-Pk-update} in vanilla {\hdm} can be viewed as online gradient descent with respect to the online surrogate loss
\begin{equation} \label{eqn:hypergrad-feedback}
h_x(P) \assign \tfrac{f ( x - P \nabla f (x) ) - f (x) }{\| \nabla f (x) \|^2}.
\end{equation}
The function $h_x(P)$, called \emph{hypergradient feedback} in this paper, is a function of preconditioner $P$ and is well-defined for all non-stationary $x$. 
To see that \eqref{eqn:hdm-Pk-update} aligns with the online gradient descent update, notice $\nabla h_{x^k}(P_k) = -\tfrac{\nabla f(x^{k+1}) \nabla f(x^k)^{\top}}{\| \nabla f(x^k) \|^2}$ so the update \eqref{eqn:hdm-Pk-update} sets $P_{k+1} = P_k - \eta_k \nabla h_{x^k} (P_k)$.
Using insights from our analysis, we develop a variant of {\hdm} 
based on {\adagrad} that improves convergence of {\hdm} (\Cref{fig:demo:b}) both in theory and practice. 

\begin{algorithm}[h]
  {\textbf{input} initial point $x^1, P_1 \in \Pcal$}\\
  \For{k =\rm{ 1, 2,...}}{
  $x^{k + 1} = \displaystyle \argmin_{x \in \{x^k, x^k - P_k \nabla f(x^k)\}} f(x) $\\
  $P_{k+1} = \Pi_{\mathcal{P}} [P_k - \eta_k \nabla h_{x^k} (P_k)]$\\
  }
  \caption{Hypergradient Descent Method (\hdm) \label{alg:hdm}}
\end{algorithm}

Vanilla {\hdm} + null steps (\Cref{alg:hdm}) was first considered by \cite{gao2024gradient} and guaranteed to converge globally.
However, their analysis is not sufficient to explain the practical behavior of {\hdm} and provides no advice for how to design a practically efficient {\hdm}.
In this paper, we dive deeper into the convergence behavior of {\hdm} (\Cref{alg:hdm}), establishing sharper global convergence guarantees and conducting a local convergence analysis. Our findings offer new insights into (vanilla) {\hdm} and serve as a foundation to design more efficient and practical variants of {\hdm}.
The contributions of this paper include:

\begin{itemize}[leftmargin=10pt]
    \item We provide the first rigorous convergence analysis for {\hdm}, including both global and local convergence guarantees (\Cref{sec:hdm}) that show {\hdm} can adapt to the local optimization landscape. Our analysis provides several new insights into how {\hdm} adapts to optimization landscapes (\Cref{sec:global-conv}), why vanilla {\hdm} is unstable in practice (\Cref{sec:instability}), and the connection between {\hdm} and quasi-Newton methods (\Cref{sec:local-conv}). 
    
    \item We develop and analyze two improved variants of {\hdm}: {\hdm} + heavy-ball momentum ({\hdmhb} in \Cref{sec:heavyball}), which has the same convergence rate as {\hdm} but is faster than {\hdm} in practice; and {\hdm} + Nesterov momentum ({\hdmagd} in \Cref{sec:nesterov}), which is faster in theory and intermediate between {\hdm} and {\hdmhb} in practice. 
    
    \item We develop a practically efficient variant {\hdmbest}, which updates $x^k$ by preconditioned gradient descent with heavy-ball momentum and jointly updates $P_k$ and momentum parameter by {\adagrad}.
    Our {\hdmbest} outperforms adaptive first-order methods and performs on par with \texttt{L-BFGS} (with memory size 5 or 10) using \emph{less memory} (memory size 1) (\Cref{sec:exp}).
\end{itemize}

\subsection{Related Literature}

\paragraph{Adaptive First-order Methods.}
Notable adaptive first-order methods include {\adagrad} \cite{duchi2011adaptive,mcmahan2010adaptive}, \texttt{Adam} \cite{kingma2014adam,zhang2024adam}, and parameter-free stepsizes \cite{orabona2016coin,defazio2024road}.
Most of these techniques originate in the online learning community, and they typically achieve both strong empirical convergence and online learning regret guarantees.

\paragraph{Hypergradient Descent.}
Hypergradient descent dates back to \cite{almeida1999parameter}, which was first proposed as a heuristic to accelerate stochastic gradient descent. 
Similar concepts were also explored in \cite{sutton1992adapting,schraudolph1999local,jacobs1988increased,mahmood2012tuning}, while those works employed slightly different algorithmic updates. 
Later, \cite{gunes2018online} rediscovered the {\hdm} and named it ``hypergradient descent''; \cite{gunes2018online} also extended {\hdm} to other first-order methods with extensive experimental validation of its practical efficacy.  Recent studies \cite{jie2022adaptive,chandra2022gradient,ozkara2024mada} further empirically enhanced {\hdm} for broader applicability, reporting promising numerical results.\\

Despite these empirical successes, a rigorous theoretical understanding of {\hdm} has emerged only recently. \cite{rubio2017convergence} showed that {\hdm} converges on convex quadratic functions and established several analytic properties. 
Subsequently, \cite{kunstner2024searching} demonstrated that when using a diagonal preconditioner, hypergradient can be employed to generate cutting planes in the preconditioner space, achieving an $\Ocal(\sqrt{n}\kappa^\star \log (1/\varepsilon))$ complexity result on smooth strongly convex functions. Here, $\kappa^\star$ is the condition number associated with the optimal diagonal preconditioner. 
More recently, \cite{gao2024gradient} showed that {\hdm} can be viewed as online gradient descent applied to some surrogate loss function and that {\hdm} has strong trajectory-based convergence guarantees.

\subsection{Notations}
We denote Euclidean norm by $\| \cdot \|$ and  Euclidean inner product by $\langle \cdot, \cdot \rangle$.
The upper and lower case letters $A, a$ respectively denote matrices and scalars. 
Denote the Frobenius norm by $\| A \|_F
\assign \sqrt{\sum_{i j} a_{i j}^2}$. Define $[\cdot]_+ := \max\{\cdot, 0\}$. We use
$\Pi_{\mathcal{C}} [\cdot]$ to denote the orthogonal projection onto a
closed convex set $\mathcal{C}$ and use $\tmop{dist}
(x, \mathcal{C}) \assign \| x - \Pi_{\mathcal{C}} [x] \|$ to denote the distance between a point $x$ and a closed convex set $\mathcal{C}$. 
Denote the optimal set of $f (x)$ by $\mathcal{X}^{\star} = \{ x : f (x) = f (x^{\star}) \}$; and the $\alpha$-sublevel set of $f$ by $\Lcal_\alpha := \{x: f(x) \leq \alpha\}$. 
For consistency of notation, a \textit{stepsize} $P$ in this paper always refers to a matrix applied in the gradient update.
Define $\mathcal{S} \assign \{ P = \alpha I : \alpha \in
\mathbb{R} \}$ and $\mathcal{D} \assign \{ P = \tmop{diag} (d) : d \in \mathbb{R}^n \}$. The condition number of an $L$-smooth and $\mu$-strongly convex function is $\kappa := L/\mu$.
\section{Background: {\hdm} and Online Learning} \label{sec:hdm-ol}

This section establishes the connection between {\hdm} and online learning through the framework in \cite{gao2024gradient}. 
We refer to the following assumptions in the paper.

\begin{enumerate}[leftmargin=30pt,label=\textbf{A\arabic*:},ref=\rm{\textbf{A\arabic*}},start=1]
  \item $f(x)$ is $L$-smooth and convex.  \label{A1}
  \item $f(x)$ is $\mu$-strongly convex with $\mu > 0$. \label{Ascvx}
  \item Closed convex set $\Pcal$ satisfies $0 \in \Pcal, L^{-1} I\in \Pcal$ and $\diam (\mathcal{P}) \leq D < \infty$. \label{A2}
\end{enumerate}

\subsection{Descent Lemma and Hypergradient Feedback}

Hypergradient feedback \eqref{eqn:hypergrad-feedback} is motivated by  descent lemma:
\[ f ( x - \tfrac{1}{L} \nabla f (x) ) - f (x) \leq - \tfrac{1}{2
   L} \| \nabla f (x) \|^2 . \]
The descent lemma states that, under the constant stepsize $P_k \equiv \tfrac{1}{L} I$,
the function value progress of a gradient step is proportional to $\| \nabla f
(x) \|^2$ with ratio $- 1/(2L)$. 
When an (effective) preconditioner $P_k$ is used, the \emph{effective} smoothness constant decreases, and thus the ratio $h_x (P) = \tfrac{f (x - P \nabla f (x)) - f (x)}{\| \nabla f (x) \|^2}$ is expected to become smaller than $- 1 / (2 L)$, yielding a faster convergence. 
Hence, the ratio $h_x (P)$ is a suitable feedback to measure the quality of a preconditioner. {\hdm} uses this
feedback to learn a good preconditioner using online gradient descent. The hypergradient feedback $h_x (P)$ has the following properties.

\begin{lem}[{Extension of Proposition 6.1 in \cite{gao2024gradient}}]\label{lem:hx-properties}
  For any $x \nin \mathcal{X}^{\star}$.
  \begin{itemize}[leftmargin=10pt]
    \item Under \ref{A1}, $h_x (P)$ is convex and $L$-smooth and $h_x(\frac{1}{L}I)\leq -\frac{1}{2L}$. Moreover, if \ref{Ascvx} holds and $\mathcal{P} \subseteq \mathcal{S}$,
    then $h_x (P) = h_x (\alpha)$ is $\mu$-strongly convex. 
    
    \item Under \ref{A1} and \ref{A2}, $h_x (P)$ is  $(L D + 1)$-Lipschitz. Moreover, if \ref{Ascvx} holds and $\mathcal{P} \subseteq \mathcal{D}$,
    then $h_x (P) = h_x (d)$ is $\frac{\mu}{(1 + L D)^2}$-exponential concave \cite{hazan2007logarithmic}.
  \end{itemize}
\end{lem}

\subsection{Online Learning Guarantees}

Using the convexity and Lipschitz continuity of $h_x(P)$, standard analysis in online learning literature  \cite{orabona2019modern,hazan2016introduction} guarantees sublinear regret for online gradient descent. 

\begin{lem}[{Sublinear regret \cite{gao2024gradient}}] \label{lem:regret-sublinear}
  Under \ref{A1} and \ref{A2}, online gradient descent
\begin{equation} \label{eqn:olalg-ogd}
	P_{k + 1} = \Pi_{\mathcal{P}} [P_k - \eta_k \nabla h_{x^k} (P_k)]
\end{equation}
with stepsize $\eta_k \equiv \tfrac{D}{ (L D + 1) \sqrt{K}}$ generates 
  $\{P_k\}$ such that
\begin{align}
\textstyle \sum_{k = 1}^K h_{x^k} (P_k) - \displaystyle \min_{P \in \mathcal{P}}  \textstyle \sum_{k = 1}^K h_{x^k} (P)
\leq \rho_{K} \assign D (L D + 1) \sqrt{K} \label{eqn:ogd-constant}. 
\end{align}
\end{lem}

If strong convexity \ref{Ascvx} is further assumed and $P_k \in \Scal$, a different choice of hypergradient stepsize $\eta_k$ in \eqref{eqn:olalg-ogd} improves the regret to $\log K$.

\begin{lem}[Logarithmic regret] \label{lem:regret-log}
Instate \ref{A1} to \ref{A2} and suppose $\mathcal{P} \subseteq
  \mathcal{S}$. Then online gradient descent \eqref{eqn:olalg-ogd}	with  $\eta_k = 1/(k \mu)$ generates a sequence of $\{
  P_k \}$ such that $\textstyle \sum_{k = 1}^K h_{x^k} (P_k) -  \min_{P \in \mathcal{P}}  \textstyle \sum_{k
  = 1}^K h_{x^k} (P) = \Ocal(\log K). $
\end{lem}

\begin{rem}
Given exponential-concavity of $h_x$ established in \Cref{lem:hx-properties}, it is possible to apply online learning algorithms such as online Newton method \cite{hazan2007logarithmic}.
\end{rem}

\subsection{Hypergradient Reduction and {\hdm}}
One major contribution of \cite{gao2024gradient} is an online-to-offline reduction that relates the minimization of cumulative hypergradient feedback $\textstyle \sum_{k=1}^K h_{x^k}(P_k)$ to the function value gap.
We provide a sharper version of this reduction.

\begin{lem}[{Sharper version of Lemma 6.1 in \cite{gao2024gradient}}] \label{lem:hypergrad-to-online}
Under \ref{A1}, the iterates generated by \Cref{alg:hdm} satisfy
\begin{equation*}
f(x^{K+1}) - f(x^\star)
\leq \min \Big\{ \tfrac{\Delta^2}{K \max \{ \frac{1}{K} \sum_{k = 1}^K - h_{x^k} (P_k), 0 \}}, f (x^1) - f (x^{\star}) \Big\},
\end{equation*}
where $\Delta = \max_{x \in \mathcal{L}_{f (x^1)}} \min_{x^{\star} \in \mathcal{X}^{\star}} \| x - x^{\star} \|$.
Further, under  \ref{A1} and \ref{Ascvx},
\begin{equation*}
f(x^{K+1}) - f(x^\star)
\leq (f (x^1) - f (x^{\star})) \big( 1 - 2 \mu \max \big\{ \tfrac{1}{K} \textstyle\sum_{k = 1}^K - h_{x^k} (P_k), 0 \big\} \big)^K.
\end{equation*}
\end{lem}

According to \Cref{lem:hypergrad-to-online}, the negative average feedback $\tfrac{1}{K} \textstyle\sum_{k = 1}^K - h_{x^k} (P_k)$ determines the rate for sublinear/linear convergence of \Cref{alg:hdm}: larger $\tfrac{1}{K} \textstyle\sum_{k = 1}^K - h_{x^k} (P_k)$ implies faster convergence. Given the objective $\tfrac{1}{K} \textstyle\sum_{k = 1}^K - h_{x^k} (P_k)$, {\hdm} applies online gradient descent to generate a sequence of preconditioners $\{P_k\}$ that guarantee the following lower bound:
\begin{equation} \label{eqn:avg-guarantee}
  \tfrac{1}{K} \textstyle\sum_{k = 1}^K - h_{x^k} (P_k) \geq \displaystyle \max_{P \in \mathcal{P}}  \tfrac{1}{K} \textstyle \sum_{k
  = 1}^K - h_{x^k} (P) + o(1),
\end{equation}
which follows from the sublinear regret $\rho_K = o(K)$ in \Cref{lem:regret-sublinear} and \Cref{lem:regret-log}, implying $\tfrac{\rho_K}{K} = o(1)$.
\section{The Convergence Behavior of {\hdm} \label{sec:hdm}}

This section presents our main convergence results on {\hdm} and consequent insights. All the analyses are based on the online learning framework established in \Cref{sec:hdm-ol}. Unless specified, we assume the online gradient descent in {\hdm} (\Cref{alg:hdm}) uses the constant stepsize $\eta_k \equiv \eta > 0$ throughout this section.

\subsection{{\hdm} Adapts to the Local  Landscape} \label{sec:global-conv}

Our first convergence result follows by combining \Cref{lem:hypergrad-to-online} and \Cref{lem:regret-sublinear}:
\begin{thm}[Static adaptivity]\label{thm:adaptivity}
Under \ref{A1} and \begin{rm}{(\ref{A1} + \ref{Ascvx})}\end{rm} respectively, \Cref{alg:hdm} satisfies
\begin{align}
f (x^{K + 1}) - f (x^{\star}) & \leq{} \min \{ \tfrac{\Delta^2}{K \max\{ \gamma_K^{\star} - \frac{\rho_K}{K}, 0 \} }, f (x^1) - f (x^{\star}) \} \tag{\ref{A1}} \\
    f (x^{K + 1}) - f (x^{\star}) & \leq{} [f (x^1) - f (x^{\star})] ( 1 - 2 \mu \max\{ \gamma_K^{\star} - \tfrac{\rho_K}{K}, 0 \}
    )^K, \tag{\ref{A1} + \ref{Ascvx}}
\end{align}
where $\Delta$ is the same as defined in \Cref{lem:hypergrad-to-online}, $\rho_K$ is defined in \eqref{eqn:ogd-constant}, and $\gamma_K^{\star} \assign - \min_{P \in \mathcal{P}}  \tfrac{1}{K} \textstyle\sum_{k = 1}^K h_{x^k}(P)$.
\end{thm}

\Cref{thm:adaptivity} has two implications: 
1) Since $\gamma_K^\star \geq - \frac{1}{K}\sum_{k = 1}^K  h_{x^k} (\frac{1}{L}I) \geq \frac{1}{2L}$ (by descent lemma) and $\tfrac{\rho_K}{K} = o(1)$, both upper bounds in \Cref{thm:adaptivity} converge to $0$ when $K$ goes to infinity, guaranteeing global convergence of {\hdm}.
2) More importantly, $\gamma_K^\star$ reflects the possibly improved convergence rate of {\hdm} through the adaptive $P$-update, which depends on the local optimization landscape.
To see this, when $K$ is large and $\tfrac{\rho_K}{K}$ is negligible, the convergence of {\hdm} is competitive with preconditioned gradient descent \eqref{eqn:pgd} with any \emph{static} preconditioner. In particular, the optimal $P_k \equiv P_K^{\star} \assign \argmin_{P \in \mathcal{P}}  \tfrac{1}{K} \textstyle\sum_{k = 1}^K h_{x^k}(P)$ achieves the rate $\tfrac{\Delta^2}{K \gamma_K^\star}$.
Note that $\gamma_K^\star$ (or $P_K^\star$) depends only on the past trajectory $\{x^k\}_{k \leq K}$; and thus if the algorithm visits a local region with a smaller smoothness constant than the global constant $L$, one can expect $\gamma_K^\star \gg \frac{1}{2L}$. Adaptivity leads to faster convergence than standard gradient descent.
In summary, {\hdm} \emph{adapts to the local optimization landscape}.\\

We borrow a standard dynamic regret argument in online convex optimization literature \cite{hazan2016introduction} to provide an even stronger notion of adaptivity of {\hdm}:
\begin{thm}[Dynamic adaptivity]\label{thm:dynamic-adaptivity}
Under \ref{A1} and \begin{rm}{(\ref{A1} + \ref{Ascvx})}\end{rm} respectively, \Cref{alg:hdm} satisfies
\begin{align}
f (x^{K + 1}) - f (x^{\star}) & \leq{} \min \big\{ \tfrac{\Delta^2}{K \max\{ \delta_K^{\star} - \frac{\rho_K}{K}, 0 \} }, f (x^1) - f (x^{\star}) \big\}; \tag{\ref{A1}} \\
f (x^{K + 1}) - f (x^{\star}) & \leq{} (f (x^1) - f (x^{\star})) ( 1 - 2 \mu \max\{ \delta_K^{\star} - \tfrac{\rho_K}{K}, 0 \} )^K, \tag{\ref{A1} + \ref{Ascvx}}
\end{align}
where $\Delta$ is the same as defined in \Cref{lem:hypergrad-to-online},
\begin{align}\label{eqn:dynamic-regret}
\delta_K^{\star} \assign -\min_{\{\hat{P}_k \in \Pcal\}}  \big[ \tfrac{1}{K} \textstyle\sum_{k = 1}^K h_{x^k} (\hat{P}_k) + \mathsf{PL}(\{\hat{P}_k\}) \big], 
\end{align}
and $\mathsf{PL}(\{\hat{P}_k\}):=\tfrac{(L D + 1)}{2 \sqrt{K}} \textstyle\sum_{k = 1}^K \| \hat{P}_{k + 1} - \hat{P}_k \|_F$.
\end{thm}

\Cref{thm:adaptivity} and \Cref{thm:dynamic-adaptivity} differ in the constants $\gamma_K^\star$ and $\delta_K^\star$, as the minimum in \eqref{eqn:dynamic-regret} searches over different optimal preconditioners for different $h_{x^k}$. \Cref{thm:dynamic-adaptivity} shows that, even if the sequence $\{ x^k \}$ traverses different regions of the landscape, 
{\hdm} automatically chooses $\hat{P}_k$
to adapt to the local region, at the price of an additional regret term $\mathsf{PL}(\{\hat{P}_k\})$. 
Adaptivity of {\hdm} undergirds its good empirical performance.

\subsection{Online Regret and Instability} \label{sec:instability}

Though adaptive $P$-update underpins the strong performance of {\hdm}, vanilla {\hdm} is observed unstable in practice. 
This section identifies the source of instabilty in vanilla {\hdm} (\Cref{fig:demo:instability}) based on our analysis. We also propose two simple yet effective strategies to address the instability.

\paragraph{Divergence Behavior due to Regret.} 
Recall from \Cref{thm:adaptivity} that the optimality gap at $x^{K + 1}$ is bounded by $\tfrac{\Delta^2}{K \max\{ \gamma_K^{\star} - \frac{\rho_K}{K}, 0 \} }$.
This rate can be better than that of gradient descent when $K$ is large and $\gamma_K^{\star} \gg \tfrac{1}{2L}$, but the analysis provides no guarantee on earlier iterates $\{ x^k \}_{k \leq K}$.
In particular, the convergence rate makes sense only if $\gamma_K^{\star} > \frac{\rho_K}{K}$. That is, the progress $\sum_k h_{x^k} (P_k)$ accumulated by the online gradient descent outweighs its regret $\rho_K$. 
In other words, online gradient descent takes time to learn a good preconditioner, and the regret accumulated during this warm-up phase causes {\hdm} to behave as if it is diverging until the progress $\sum_k h_{x^k} (P_k)$ outpaces the regret $\rho_K$. 
Since $\rho_K$ grows sublinearly with the iteration count $K$, {\hdm} will eventually converge. 
However, the objective value will usually explode (and be terminated by the user) before convergence begins.
Consequently, the two-phase convergence behavior (\Cref{fig:demo:a}) is rarely observed.

\paragraph{Addressing Instability.} While our analysis guarantees {\hdm}  eventually converges, an algorithm that diverges up to $10^{30}$ before converging is not practical. We propose two simple but effective fixes based on our analysis:

\begin{itemize}[leftmargin=10pt]
\item \emph{Null step.} 
The $x$-update is skipped if the new iterate increases the objective value:
 \begin{equation} \label{eqn:null-step}
 x^{k + 1} = \displaystyle \argmin_{x \in \{x^k, x^k - P_k \nabla f(x^k)\}} f(x).
 \end{equation}
The null step ensures a monotonic decrease as {\hdm} learns a good preconditioner, although it requires an additional function value oracle call at each iteration. 
Even on iterations when $x^k$ is not updated, 
the preconditioner $P_k$ is updated using online gradient descent, so the algorithm is still making progress.
In \Cref{fig:instablity}, the null steps  flatten the objective value curve in the divergence phase.

\item \emph{Advanced Learning Algorithms.} Better online learning algorithms with lower regret shorten the divergence phase. \Cref{fig:instablity} shows a significant speedup when the online gradient descent in \Cref{alg:hdm} is replaced by \texttt{AdaGrad}.
In our experiments, {\adagrad} often improves the robustness of {\hdm} since it does not require pre-specifying algorithm parameters that depend on the total iteration count $K$
and provides convergence guarantees for the earlier iterates $\{ x^k \}_{k \leq K}$ (\Cref{sec:intermediate-iter}).
\end{itemize}
\begin{figure}
\centering
\includegraphics[scale=0.3]{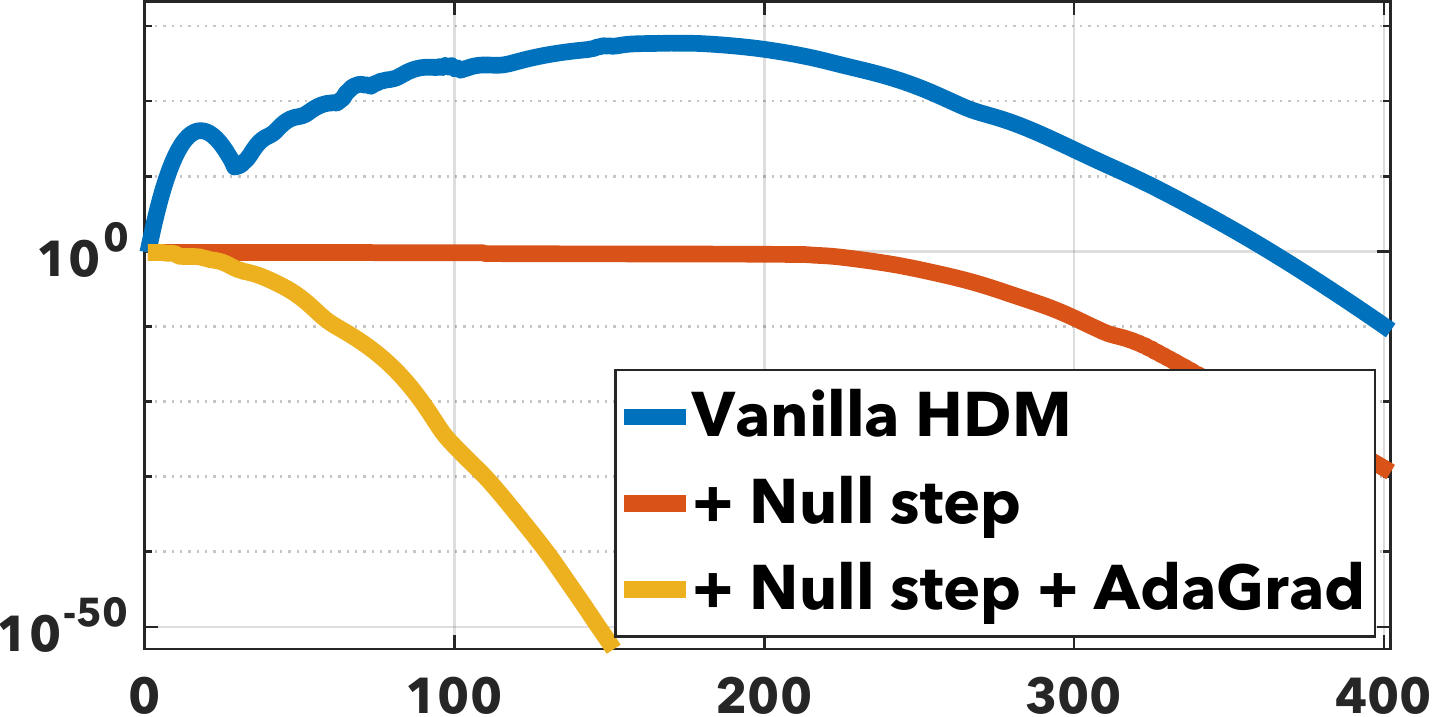}	
\caption{Addressing instability of {\hdm} \label{fig:instablity}}
\end{figure}
\subsection{Local Superlinear Convergence} \label{sec:local-conv}
\Cref{fig:demo:instability} shows {\hdm} converges faster than the (linearly convergent) first-order methods. In fact, {\hdm} exhibits local superlinear convergence on strongly convex objectives (\Cref{thm:superlin} below). Results in this subsection assume a strongly convex objective (\ref{Ascvx}) and Lipschitz Hessian (\ref{A3}):
\begin{enumerate}[leftmargin=30pt,label=\textbf{A\arabic*:},ref=\rm{\textbf{A\arabic*}},start=4]
  \item $f (x)$ has $H$-Lipschitz Hessian. \label{A3}
\end{enumerate}

\paragraph{Strongly Convex Quadratics.}  We develop intuition by considering a strongly convex quadratic. For $f (x) = \tfrac{1}{2} \langle x, A x \rangle$, we have $x^{\star} = x - [\nabla^2 f (x^{\star})]^{- 1} \nabla f (x)$ \footnote{since $\nabla f(x)$ equals to its first-order Taylor expansion at $x^{\star}$ for quadratics: $\nabla f(x) = \nabla f(x^{\star}) + \nabla^2 f(x^\star) (x - x^\star)$}.
In other words, $P^{\star} = [\nabla^2 f (x^{\star})]^{- 1}$ is a universal minimizer of $h_x(P)$ that drives any non-optimal point $x \nin \mathcal{X}^{\star}$ to the optimum $x^{\star}$ \emph{in one step}. When $[\nabla^2 f (x^{\star})]^{- 1} \in \mathcal{P}$, 
\Cref{thm:adaptivity} guarantees the performance of {\hdm} is competitive,
so we should expect the descent curve to decrease more and more sharply, giving superlinear convergence (\Cref{fig:demo:a}).

\paragraph{Local Superlinear Convergence.}
For general functions satisfying \ref{A3}, $f(x)$ locally behaves like a quadratic
 \begin{equation*}
   f (x) = f (x^{\star}) + \tfrac{1}{2} \langle x - x^{\star}, \nabla^2 f (x^{\star}) (x - x^{\star}) \rangle +\mathcal{O} (\| x - x^{\star} \|^3).
 \end{equation*}
Therefore, local superlinear convergence is expected for {\hdm} near $x^{\star}$. \Cref{thm:superlin} formalizes this intuition.

\begin{thm}[Local superlinear convergence]\label{thm:superlin}
Suppose $ [\nabla^2 f (x^{\star})]^{- 1} \in \mathcal{P}$ and assume \ref{A1} to \ref{A3}. Then \Cref{alg:hdm} has local superlinear convergence:
\begin{equation} \label{eqn:superlin}
f (x^{K + 1}) - f (x^{\star}) \leq{} (f (x^{1}) - f (x^{\star})) \big(\min\{\tfrac{H^2 \kappa^2}{4 \mu^2 K} \textstyle\sum_{k = 1}^K {\| x^k - x^{\star} \|^2} + \tfrac{2L \rho_K}{K}, 1 \} \big)^K.
\end{equation}
\end{thm}
\Cref{thm:superlin} justifies our observation of superlinear convergence in \Cref{fig:demo:a}: for strongly convex quadratics, the Hessian Lipschitz constant is zero ($H = 0$) and \eqref{eqn:superlin} guarantees the superlinear convergence at rate $\mathcal{O} ( ( \tfrac{\rho_K}{K} )^K )$. 
For general strongly convex objectives, global linear convergence (\Cref{thm:adaptivity}) implies 
$\lim_{K \rightarrow \infty} \frac{1}{K} \textstyle \sum_{k = 1}^K
\| x^k - x^{\star} \|^2 = 0$. 
So eventually, the first term in \eqref{eqn:superlin} vanishes, giving superlinear convergence.

\paragraph{{\hdm} Learns the Hessian at the Optimum.}
In fact, $\{P_k\}$ in {\hdm} will converge to $[\nabla^2 f (x^{\star})]^{-1}$ under an assumption similar to one studied in the quasi-Newton literature \cite{conn1991convergence, nocedal1999numerical}. \Cref{lem:scalmat-conv} quantifies the effect of learning the preconditioner through the distance $\|P_{k} - [\nabla^2 f (x^{\star})]^{-1}\|_F$.
\begin{lem}\label{lem:scalmat-conv}
  Under the same assumptions as \Cref{thm:superlin}, \Cref{alg:hdm} generates $\{P_k\}$ such that
\begin{align}
&\| P_{k + 1} - [\nabla^2 f (x^{\star})]^{- 1} \|_F^2 \nonumber \\
\leq{} & \| P_k - [\nabla^2 f (x^{\star})]^{- 1} \|_F^2  - \tfrac{\mu (\eta - L \eta^2)}{2} \big\| (P_k - [\nabla^2 f (x^{\star})]^{- 1}) \tfrac{\nabla f (x^k)}{\| \nabla f (x^k) \|} \big\|^2 
+ (2\eta - L \eta^2) \tfrac{H^2 \kappa}{4 \mu^3}\| x^k - x^{\star} \|^2.  \label{eqn:scalmat-conv}
\end{align}
\end{lem}
Relation \eqref{eqn:scalmat-conv} consists of three terms: the distance $\| P_k - [\nabla^2 f (x^{\star})]^{- 1} \|_F^2$; a decrement in the distance (second term); and an error term (last term) that converges to zero as $x^k \rightarrow x^{\star}$.
The decrement is determined by the magnitude of $\big\| (P_k - [\nabla^2 f (x^{\star})]^{- 1}) \tfrac{\nabla f (x^k)}{\| \nabla f (x^k) \|} \big\|^2$, which measures the difference between the operators $P_k$ and $[\nabla^2 f (x^{\star})]^{-1}$ in the (unit) gradient direction $\tfrac{\nabla f (x^k)}{\|\nabla f (x^k)\|}$.
To ensure fast convergence, it suffices for $P_k \nabla f (x^k)$ and $[\nabla^2 f (x^{\star})]^{-1} \nabla f (x^k)$ to remain sufficiently close. If the set $\big\{ \tfrac{\nabla f (x^k)}{\|\nabla f (x^k)\|} \big\}$ spans the entire space over the iterations, $P_k$ and $[\nabla^2 f (x^{\star})]^{-1}$ should align in all directions, leading to convergence of $\{P_k\}$.

\begin{thm}[Convergence of the preconditioner]\label{thm:scal-mat-conv}
Instate the same assumptions as in \Cref{lem:scalmat-conv} and let $\eta_k \equiv \eta \in (0, \tfrac{1}{2L(LD+1)^2 \kappa}]$ in online gradient descent \eqref{eqn:olalg-ogd}. Suppose the gradient directions
$\big\{ \tfrac{\nabla f (x^k)}{\| \nabla f (x^k) \|} \big\}$ are uniformly independent \footnote{The formal definition of a uniformly independent sequence is given in \Cref{app:thm-pf-scal-mat-conv}, which is adapted from quasi-Newton literature \cite{conn1991convergence, nocedal1999numerical}}. Then 
$\lim_{k \rightarrow \infty}  \| P_k - [\nabla^2 f (x^{\star})]^{- 1} \| = 0$.
\end{thm}

\begin{rem}
The convergence of parameters in {\hdm} was observed experimentally by \cite{gunes2018online} for a scalar stepsize ($\Pcal \subseteq \Scal$). Our result theoretically justifies this observation.
\end{rem}

\paragraph{{\hdm} and Quasi-Newton Methods.}
Our results identify a similarity between {\hdm} and quasi-Newton methods.
Both learn the inverse Hessian operator $g \mapsto [\nabla^2 f (x^{\star})]^{-1} g$ as the algorithm progresses, but through different properties of the operator.
The quasi-Newton methods use the secant equation 
$x - y \approx [\nabla^2 f(x^\star)]^{-1}(\nabla f(x)- \nabla f(y))$ for $x, y$ close to $x^\star$ and enforce this equation, replacing the inverse Hessian by $P_k$,
to guide learning \cite{jiang2023online, jiang2024online}.
In contrast, {\hdm} learns an optimal preconditioner for the function. 
Since the function is locally quadratic, this optimal preconditioner is the inverse Hessian.
{\hdm} uses the hypergradient feedback $h_x(P)$ to directly measure the quality of the preconditioner and can search for an optimal preconditioner in a given closed convex set $\mathcal{P}$, whereas quasi-Newton methods use the secant equation as an indirect proxy.
Both approaches require a safeguard to prevent divergence in the warm-up phase, which is achieved by line-search in quasi-Newton  and null step in {\hdm}. 
 In a word, both {\hdm} and quasi-Newton leverage complementary perspectives on $g \mapsto [\nabla^2 f (x^{\star})]^{-1} g$, so it is natural that they achieve similar convergence guarantees.

\section{{\hdm} with Momentum} \label{sec:momentum}

This section develops two variants of {\hdm}, 
with heavy-ball momentum \cite{polyak1964some} and with Nesterov momentum \cite{nesterov1983method}.

\subsection{Heavy-ball Momentum}
\label{sec:heavyball}

The heavy-ball method is a practical acceleration technique: 
\begin{equation} \label{eqn:heavyball-update}
x^{k + 1} = x^k - P_k \nabla f (x^k) + B_k (x^k - x^{k - 1}).
\end{equation}
The momentum parameter $B_k$ is typically chosen as a scalar $B_k = \beta_k I$ with $\beta_k > 0$.
{\hdm} can learn a matrix momentum
$B_k \in \mathcal{B} \subseteq \mathbb{R}^{n \times n}$
with convergence guarantees (\Cref{thm:heavyball}) 
when $\mathcal{B}$ satisfies this assumption:
\begin{enumerate}[leftmargin=30pt,label=\textbf{A\arabic*:},ref=\rm{\textbf{A\arabic*}},start=5]
  \item Closed convex set $\Bcal$ satisfies $\tfrac{1}{2} I\in \Bcal$, $\diam (\Bcal) \leq D$. \label{ABcal}
\end{enumerate}

{\hdm} can \emph{jointly} learn the pair $(P_k, B_k)$ using the modified feedback function
\begin{equation} \label{eqn:heavyball-feedback}
  h_{x, x^-} (P, B) \assign 
  \tfrac{\psi(x^{+}(P, B), x) - \psi(x, x^{-})}{\| \nabla f (x) \|^2 + \frac{\tau}{2} \| x - x^- \|^2}
  = \tfrac{[f (x^+(P, B)) + \frac{\omega}{2} \| x^+(P, B) - x \|^2] - [f (x) + \frac{\omega}{2} \| x - x^- \|^2]}{\| \nabla f (x) \|^2 + \frac{\tau}{2} \| x - x^- \|^2}, 
\end{equation}
where $\psi$ is the potential function for heavy-ball momentum defined by $\psi (x, x^-) \assign f (x) + \tfrac{\omega}{2} \| x - x^- \|^2$ \cite{danilova2020non}; \[x^{+}(P, B) \assign x - P \nabla f (x) + B (x - x^{-})\] updates $x$; and $\omega > 0$ and $ \tau > 0$ are constants. \Cref{alg:ospolyak} presents the resulting method, \hdmhb, 
which uses {\hdm}, heavy-ball momentum, and a null step to ensure decrease of the potential function $\psi$.
\Cref{fig:demo:c} compares non-adaptive heavy-ball ($P_k \equiv  \alpha I, B_k \equiv \beta I$) against {\hdmhb} with full-matrix/diagonal preconditioner and scalar momentum.
 \Cref{thm:heavyball} presents the convergence of {\hdmhb}. 

\begin{algorithm}[h]
{\textbf{input} initial point $x^0 = x^1, \eta_p, \eta_b > 0$, $P_1$, $B_1$}\\
\For{k =\rm{ 1, 2,...}}{
$\hspace{1.2pt}~~~~x^{k+1/2} = x^k - P_k \nabla f(x^k) + B_k (x^k - x^{k-1})$ \\
$\hspace{2pt}~~~~~~P_{k+1} = \Pi_{\Pcal}[P_k - \eta_p \nabla_{P} h_{x^k, x^{k-1}}(P_k, B_k)]$ \\
$\hspace{1pt}~~~~~~B_{k+1} = \Pi_{\Bcal}[B_k - \eta_b \nabla_{B} h_{x^k, x^{k-1}}(P_k, B_k)]$ \\
$(x^{k + 1}, x^k) = \displaystyle \argmin_{(x^+, x) \in \{(x^k, x^{k-1}), (x^{k+1/2}, x^k) \}} \psi(x^+, x)$
}
{\textbf{output} $x^{K+1}$}
\caption{{\hdm} with heavy-ball momentum (\hdmhb)\label{alg:ospolyak}}
\end{algorithm}

\begin{thm}[Convergence of {\hdmhb}]\label{thm:heavyball}
Under \ref{A1}, \ref{A2} and \ref{ABcal}, \Cref{alg:ospolyak} satisfies
\begin{equation*}
  f (x^{K + 1}) - f (x^{\star}) \leq \tfrac{f (x^{1}) - f (x^{\star})}{K V \max\{ \gamma_K^{\star} - \frac{\rho_K}{K}, 0 \} + 1},
\end{equation*}
where $\gamma_{K}^{\star} \assign - \min_{(P, B) \in \mathcal{P} \times \mathcal{B}} \tfrac{1}{K} \sum_{k=1}^K h_{x^k, x^{k-1}}(P, B)$ depends on the iteration trajectory $\{x^k\}_{k \leq K}$; $\rho_K = \mathcal{O}(\sqrt{K})$ is the regret with respect to feedback \eqref{eqn:heavyball-feedback}; $V \assign \min\big\{ \tfrac{f (x^{1}) - f (x^{\star})}{4 \Delta^2}, \tfrac{\tau}{4 \omega} \big\}$; $\Delta$ is defined in \Cref{lem:hypergrad-to-online}.
\end{thm}

\subsection{Nesterov Momentum} \label{sec:nesterov}
{\hdm} can also improve accelerated gradient descent {\agd}:\begin{align}
  y^k ={} & x^k + ( 1 - \tfrac{A_k}{A_{k + 1}} ) (z^k - x^k)
  \nonumber\\
  x^{k + 1} ={} & y^k - \tfrac{1}{L} \nabla f (y^k) \label{eqn:agd-descent-lemma-0}\\
  z^{k + 1} ={} & z^k + \tfrac{A_{k + 1} - A_k}{L} \nabla f (y^k), \nonumber
\end{align}
where is a pre-specified sequence. 
{\hdm} can learn a preconditioner $P_k$ that replaces $\frac 1 L$ to accelerate the gradient step \eqref{eqn:agd-descent-lemma-0} in {\agd}. We call the resulting algorithm {\hdmagd}. \Cref{alg:osnes} provides a
realization of the {\hdmagd} based on a monotone variant of {\agd} \cite{d2021acceleration}. The convergence of {\hdmagd} is established in \Cref{thm:osnes}, the proof of which is deferred to \Cref{app:proof-osnes}.
\begin{algorithm}[h]
{\textbf{input} starting point $x^1, z^1, \eta > 0$, $\theta \in [\tfrac{1}{2}, LD) $, $A_0 = 0$}\\
\For{k =\rm{ 1, 2,...}}{
$\begin{aligned}
 A_{k + 1} ={} & (A_{k + 1} - A_k)^2 \nonumber \\
  y^k ={} & x^k + ( 1 - \tfrac{A_k}{A_{k + 1}} ) (z^k - x^k)
  \nonumber\\
  x^{k + 1} ={} & \underset{x \in \{ y^k - \frac{1}{L} \nabla f (y^k), y^k
  - P_k \nabla f (y^k), x^k \}}{\argmin} f (x) \nonumber\\
  P_{k + 1} ={} & \Pi_{\mathcal{P}} [P_k - \eta \nabla h_{y^k} (P_k)]
  \nonumber\\
  v_k ={} & \max \{ \tfrac{1}{2 \max \{ -h_{y^k} (P_k), 1 / (2 L) \}},
   \tfrac{L}{2 \theta} \}  \\
  z^{k + 1} ={} & z^k + \tfrac{(A_{k + 1} - A_k)}{v_k} \nabla f (y^k)
  \nonumber
\end{aligned}$
}
{\textbf{output} $x^{K+1}$}
\caption{{\hdm} with Nesterov momentum  \label{alg:osnes}}
\end{algorithm}

\begin{thm}
\label{thm:osnes}
Assume \ref{A1} and \ref{A2}. Suppose {\agd} starts from $(x', z')$ and runs for $K$ iterations to output $\hat{x}$.
Then \Cref{alg:osnes} starting from $(x^1, z^1) = (\hat{x}, z')$ and $\theta \in [\tfrac{1}{2}, LD)$ satisfies
\begin{align}
f (x^{K + 1}) - f (x^{\star}) \leq \big[ \tfrac{1}{2 \theta} + ( 8 - \tfrac{4}{\theta} ) ( \tfrac{L D - \omega^\star_K}{L D - \theta} ) \big] \tfrac{2 L \| z' - x^{\star} \|^2}{K^2} +\mathcal{O} ( \tfrac{\rho_K}{K^3} ), \nonumber
\end{align}
where $\omega^\star_K = - \min_{P \in \mathcal{P}} \tfrac{L}{K} \sum_{k = 1}^K h_{y^k} (P)$ depends on the iteration trajectory $\{x^k\}_{k \leq K}$.
\end{thm}

The parameter $\theta$ serves as a smooth interpolation between {\hdm} and {\hdmagd}: when $\theta = 1 / 2$, \Cref{thm:osnes} recovers the convergence rate of
vanilla {\agd}; when $\theta > 1 / 2$ and $\omega_K^{\star}
\rightarrow L D$, we expect {\hdmagd} to yield faster convergence. As suggested by \Cref{fig:demo:d}, {\hdmagd} achieves faster convergence than {\agd}.
\begin{rem}
To mitigate the effect of regret, \Cref{alg:osnes} needs a warm start from vanilla {\agd}. However, experiments
  suggest that it is unnecessary in practice, and we leave an improved analysis to future work.
\end{rem}

\begin{rem}
For strongly convex problems, we can combine \Cref{thm:osnes} with a
  standard restart argument \cite{d2021acceleration,roulet2017sharpness} and achieve a similar trajectory-based linear
  convergence rate.
\end{rem}

\begin{figure}
  \centering
  \begin{subfigure}{0.4\textwidth}
    \centering
\includegraphics[height=0.18\textheight]{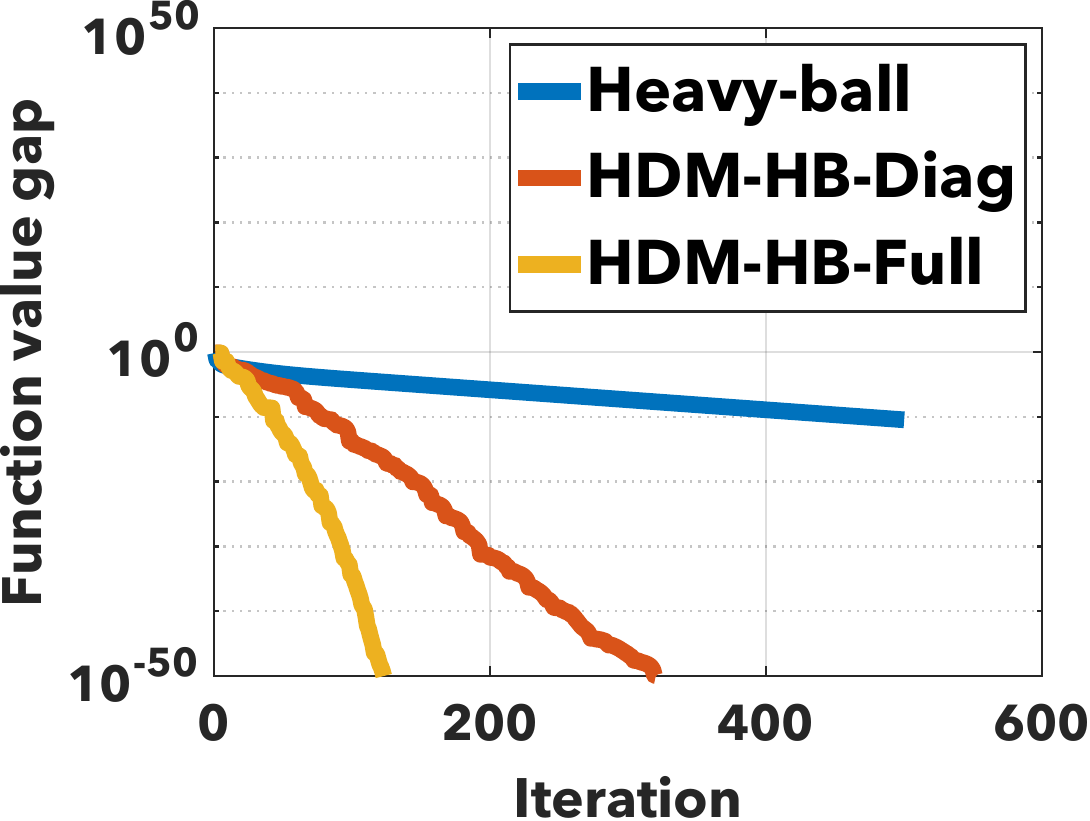}
    \caption{\hdmhb}
    \label{fig:demo:c}
  \end{subfigure}
  \begin{subfigure}{0.4\textwidth}
    \centering
\includegraphics[height=0.18\textheight]{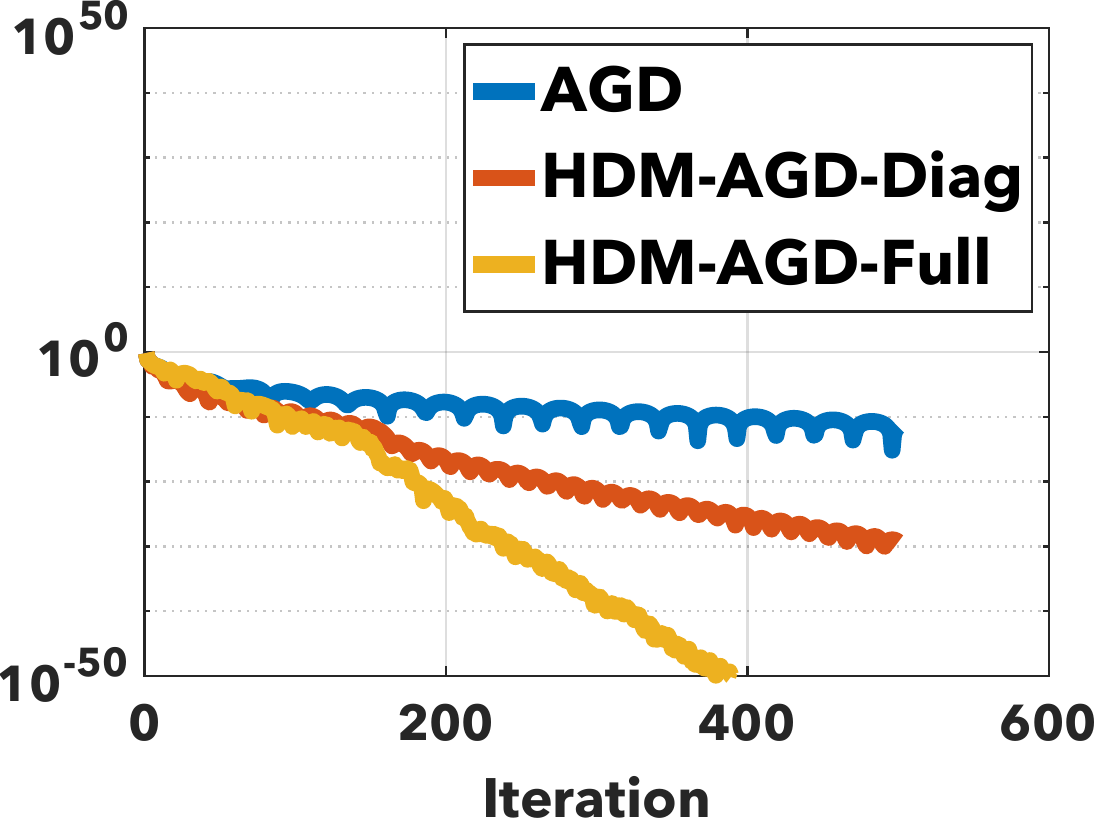}
    \caption{\hdmagd}
    \label{fig:demo:d}
  \end{subfigure}
\caption{The convergence behavior of {\hdmhb} and {\hdmagd} on a toy quadratic problem. \Cref{fig:demo:c}: {\hdmhb}. \label{fig-demos}\Cref{fig:demo:d}: {\hdm} with Nesterov momentum.} 
  \label{fig:demo:all}
\end{figure}
\section{Experiments} \label{sec:exp}

This section conducts numerical experiments to validate the empirical
performance of hypergradient descent. We compare {\hdmbest} (see \Cref{sec:hdmbest} below) with 
different adaptive optimization algorithms.

\subsection{Efficient and Practical Variant: {\hdmbest}} \label{sec:hdmbest}

This section highlights the major components of our most competitive variant {\hdmbest}. The algorithm and a more detailed explanation are available in \Cref{app:hdmprac}. The implementation is available at \url{https://github.com/udellgroup/hypergrad}. 

\paragraph{Diagonal Preconditioner and Heavy-ball Momentum.} {\hdmbest} updates $x$ by \eqref{eqn:heavyball-update} with diagonal preconditioner \cite{qu2024optimal,gao2023scalable} $\Pcal \subseteq \Dcal$ and scalar momentum $\Bcal = \{ \beta I : \beta \in \mathbb{R} \}$.
This choice balances practical efficiency and implementation complexity. Boundedness of $\Pcal$ does not greatly impact the performance, while the bound on $\Bcal$ can significantly change algorithm behavior. Two empirically robust ranges for $\Bcal$ are $[0,0.9995]$ and $[-0.9995,0.9995]$. 

\paragraph{$\adagrad$ for Online Learning. } {\hdmbest} uses {\adagrad} to shorten the warm-up phase for learning of $(P_k, \beta_k)$ (see \Cref{sec:instability}). {\adagrad} usually yields faster convergence of {\hdm} than online gradient descent at the cost of additional memory of size $n$.

\subsection{Dataset and Testing Problems}
We test {\hdmbest} on deterministic convex problems. We adopt two convex optimization tasks in machine learning: support vector machine \cite{lee2001ssvm} and logistic regression \cite{hastie2009elements}. The testing datasets are obtained from \texttt{LIBSVM} \cite{chang2011libsvm}.

\subsection{Experiment Setup}

\paragraph{Algorithm Benchmark.}
We benchmark the following algorithms.\\
\begin{itemize}[leftmargin=10pt,itemsep=2pt,topsep=0pt]
    \item \texttt{GD}. Vanilla gradient descent.
    \item \texttt{GD-HB}. Gradient descent with heavy-ball momentum. \cite{polyak1964some}
    \item \texttt{AGD-CVX}. The smooth convex version of accelerated gradient descent (Nesterov momentum). \cite{d2021acceleration}
    \item \texttt{AGD-SCVX}. The smooth strongly convex version of accelerated gradient descent. \cite{d2021acceleration}
    \item \texttt{Adam}. Adaptive momentum estimation. \cite{kingma2014adam}
    \item \texttt{AdaGrad}. Adaptive (sub)gradient method. \cite{duchi2011adaptive}
    \item \texttt{BFGS}. {\bfgs} from \texttt{scipy} \cite{nocedal1999numerical,virtanen2020scipy}.
    \item \texttt{L-BFGS-Mk}. {\lbfgs} with memory size \texttt{k} in \texttt{scipy}.
    \item Practical variant {\hdmbest} uses as memory $7$ vectors of size $n$, comparable to memory for \texttt{L-BFGS-M1}.
\end{itemize}

\paragraph{Algorithm Configuration.} See \Cref{app:hdmprac} for details.
\begin{itemize}[leftmargin=10pt]
  \item For {\hdmbest}, we search for the optimal $\eta_p$ within $\{ 0.1 / L, 1 / L, 10 / L, 100/L \}$
  and $\eta_b \in \{ 1, 3, 5, 10, 100 \}$. 
  
  \item Stepsize in \texttt{GD}, \texttt{GD-HB},
  \texttt{AGD-CVX}, and \texttt{AGD-SCVX} are all set to $1 / L$.
  
  \item The momentum parameter in \texttt{GD-HB} is chosen within the set $\{
  0.1, 0.5, 0.9, 0.99 \}$.
  
  \item The \texttt{Adam} stepsize is chosen within the set $\{ 1 / L, 10^{- 3},
  10^{- 2}, 10^{- 1}, 1, 10 \}$. $\beta_1 = 0.9, \beta_2 = 0.999$.
  
  \item  The \texttt{AdaGrad} stepsize is chosen within the set $\{ 1 / L, 10^{- 3},
  10^{- 2}, 10^{- 1}, 1, 10 \}$.
  
  \item {\bfgs}, \texttt{L-BFGS-Mk} use default parameters in
  \texttt{scipy}.
\end{itemize}

\paragraph{Testing Configurations.}

\begin{enumerate}[leftmargin=15pt,label=\textbf{\arabic*)}]
  \item {\textit{Maximum oracle access.}} We allow a maximum of 1000 gradient
  oracles for each algorithm.
  
  \item {\textit{Initial point}}. All the algorithms are initialized from the
  same starting point generated from normal distribution $\mathcal{N} (0,
  I_n)$ and normalized to have unit length.
  
  \item {\textit{Stopping criterion.}} Algorithms stop if $\| \nabla f 
  \|_\infty \leq 10^{- 4}$.
\end{enumerate}

\begin{table}[h]
\centering
\caption{Number of solved problems for each algorithm. \label{table:stats}}
\begin{tabular}{ccc}
\toprule
    Algorithm/Problem & SVM (33) $\uparrow$ & Logistic Regression (33) $\uparrow$\\
\midrule
    \texttt{GD} & 5 & 2\\
    \texttt{GD-HB} & 9 & 7\\
    \texttt{AGD-CVX} & 8 & 3\\
    \texttt{AGD-SCVX} & 7 & 6\\
    \texttt{Adam} & 26 & 11\\
    \texttt{AdaGrad} & 9 & 8\\
    \texttt{L-BFGS-M1} & 13 & 11\\
    \texttt{L-BFGS-M3} & 20 & 14\\
    \texttt{L-BFGS-M5} & 26 & 16\\
    \texttt{L-BFGS-M10} & \textcolor{orange}{31} & 18\\
    \texttt{BFGS} & \textcolor{red}{32} & \textcolor{red}{26}\\
    \texttt{HDM-Best} & \textcolor{red}{32} & \textcolor{red}{21}\\
\bottomrule
\end{tabular}
\end{table}

\begin{figure*}[!h]
\centering
\includegraphics[scale=0.2]{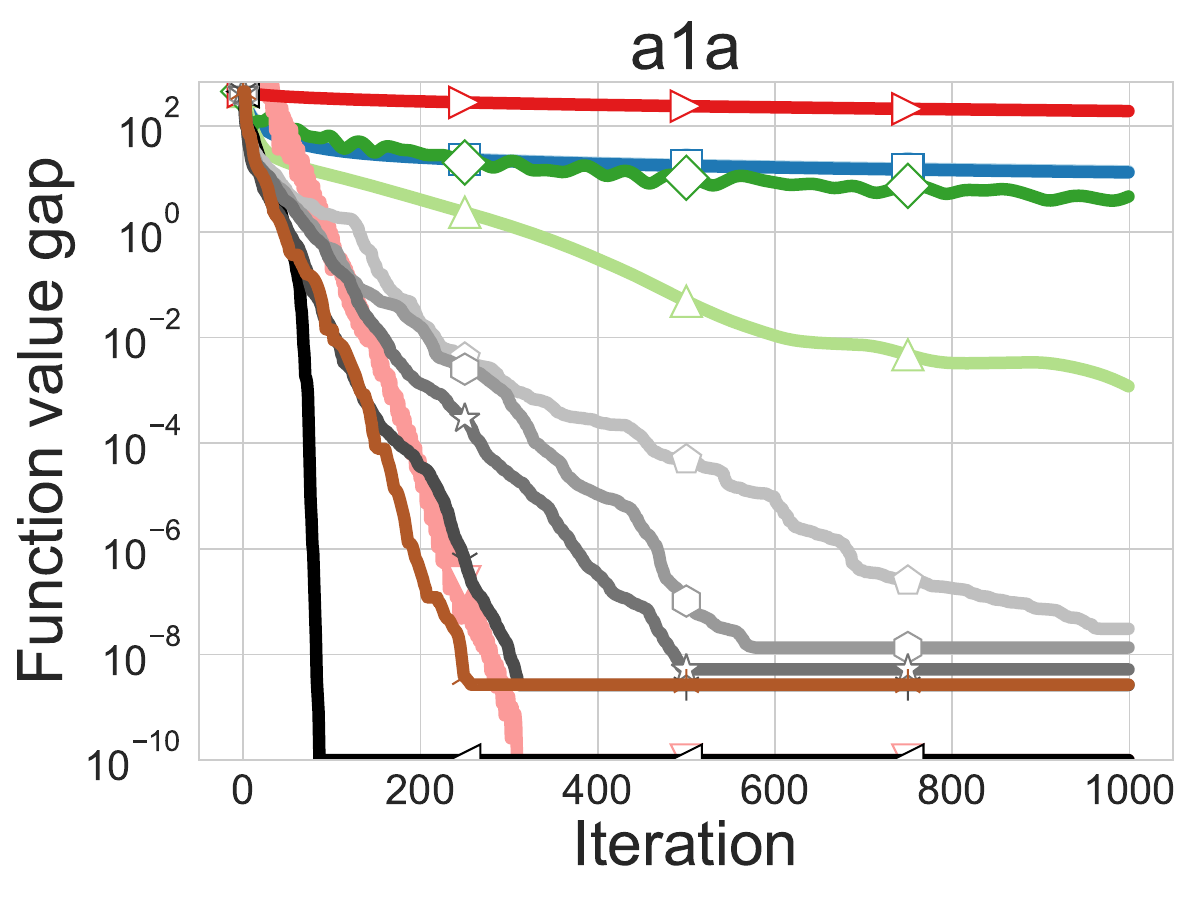}
\includegraphics[scale=0.2]{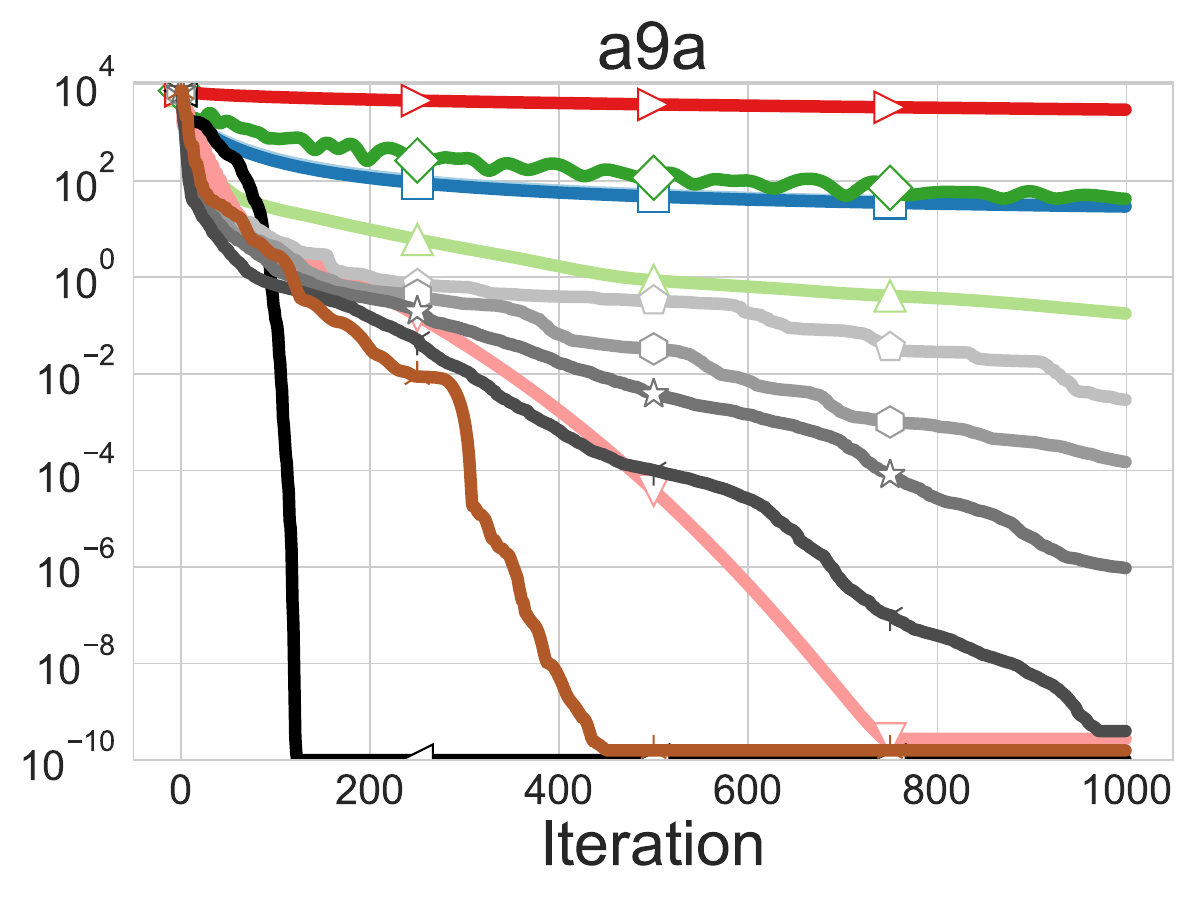}
\includegraphics[scale=0.2]{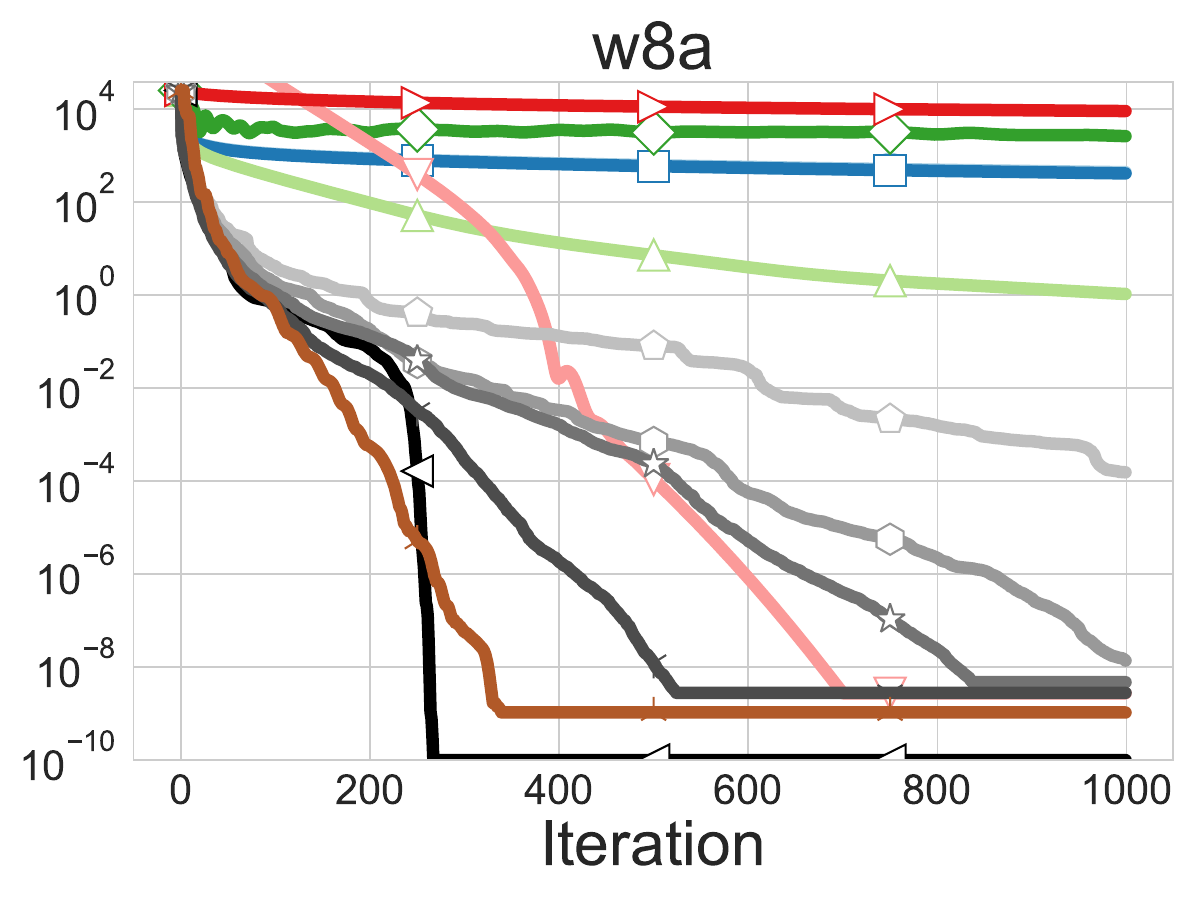}
\includegraphics[scale=0.2]{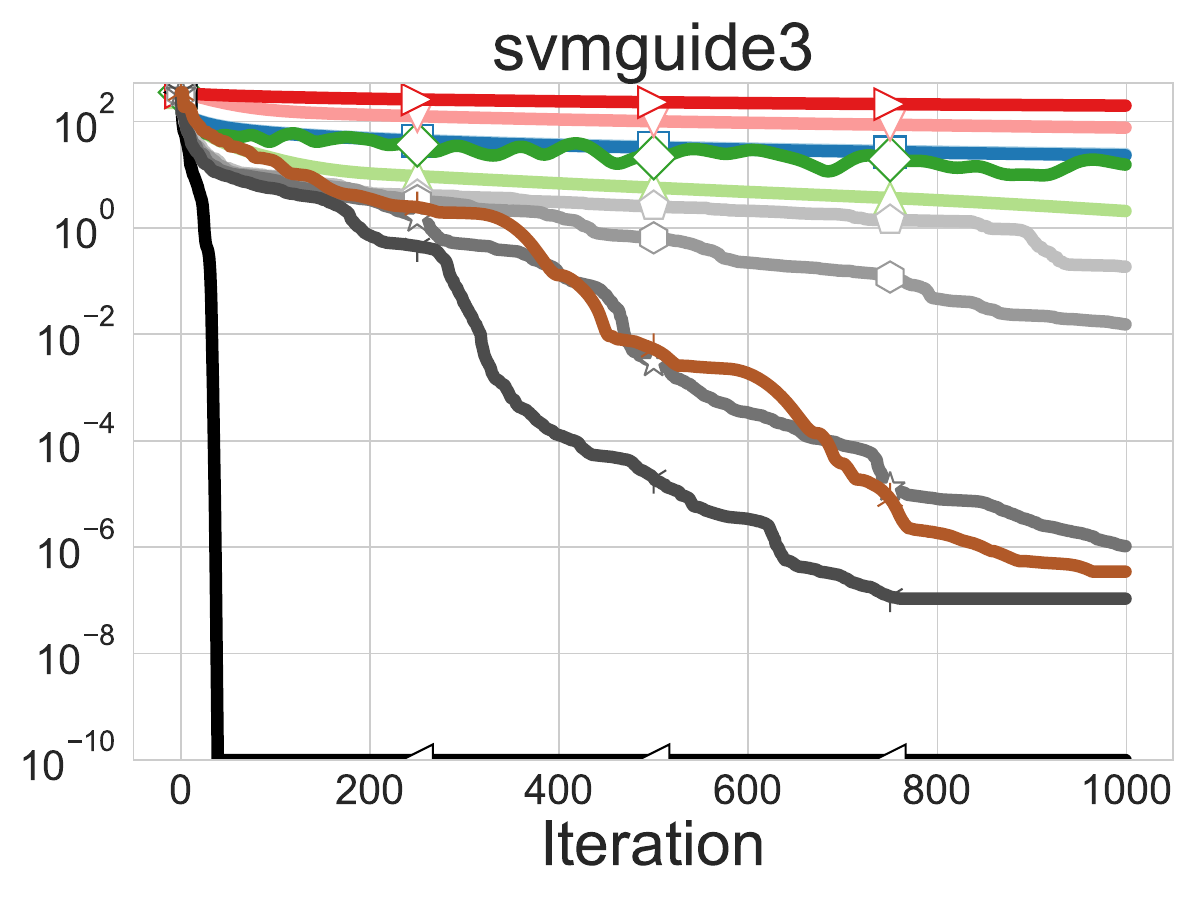}
\includegraphics[scale=0.2]{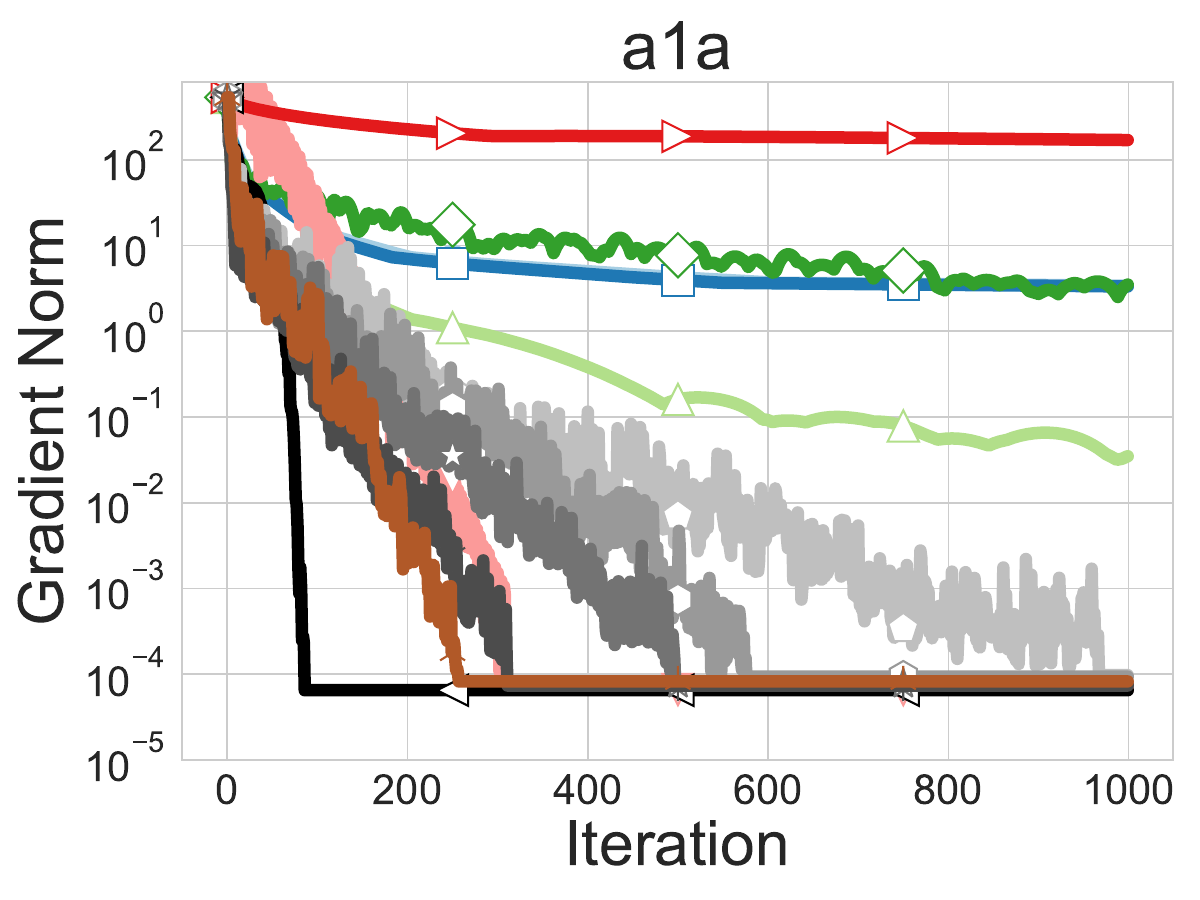}
\includegraphics[scale=0.2]{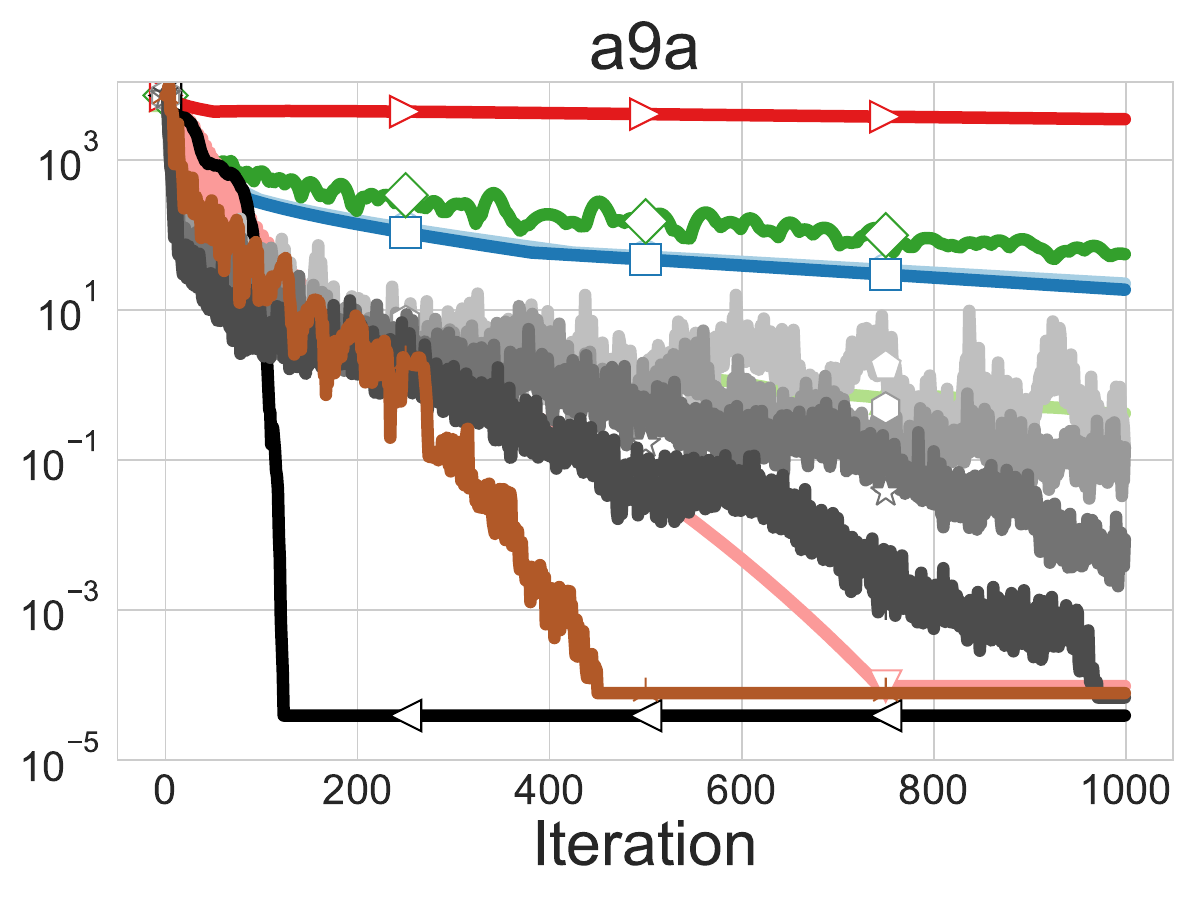}
\includegraphics[scale=0.2]{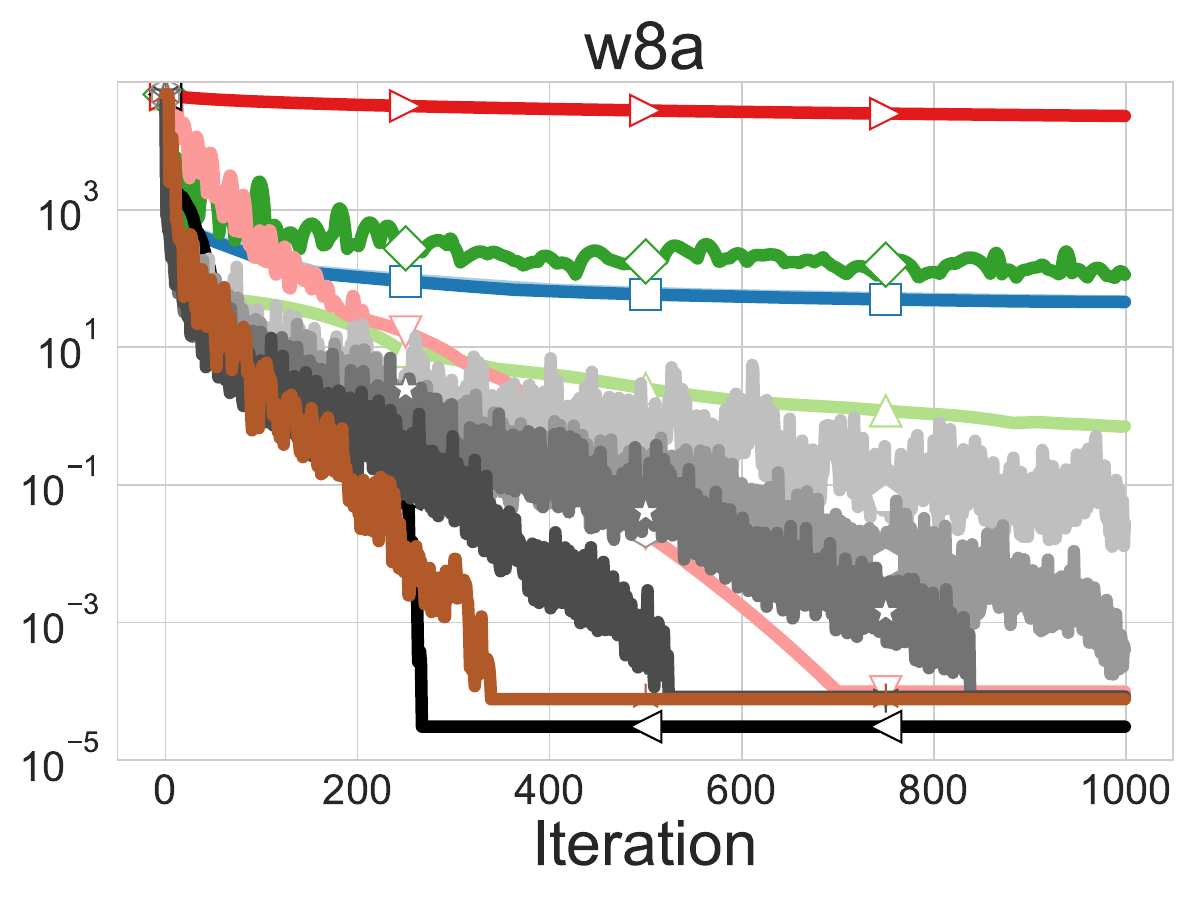}
\includegraphics[scale=0.2]{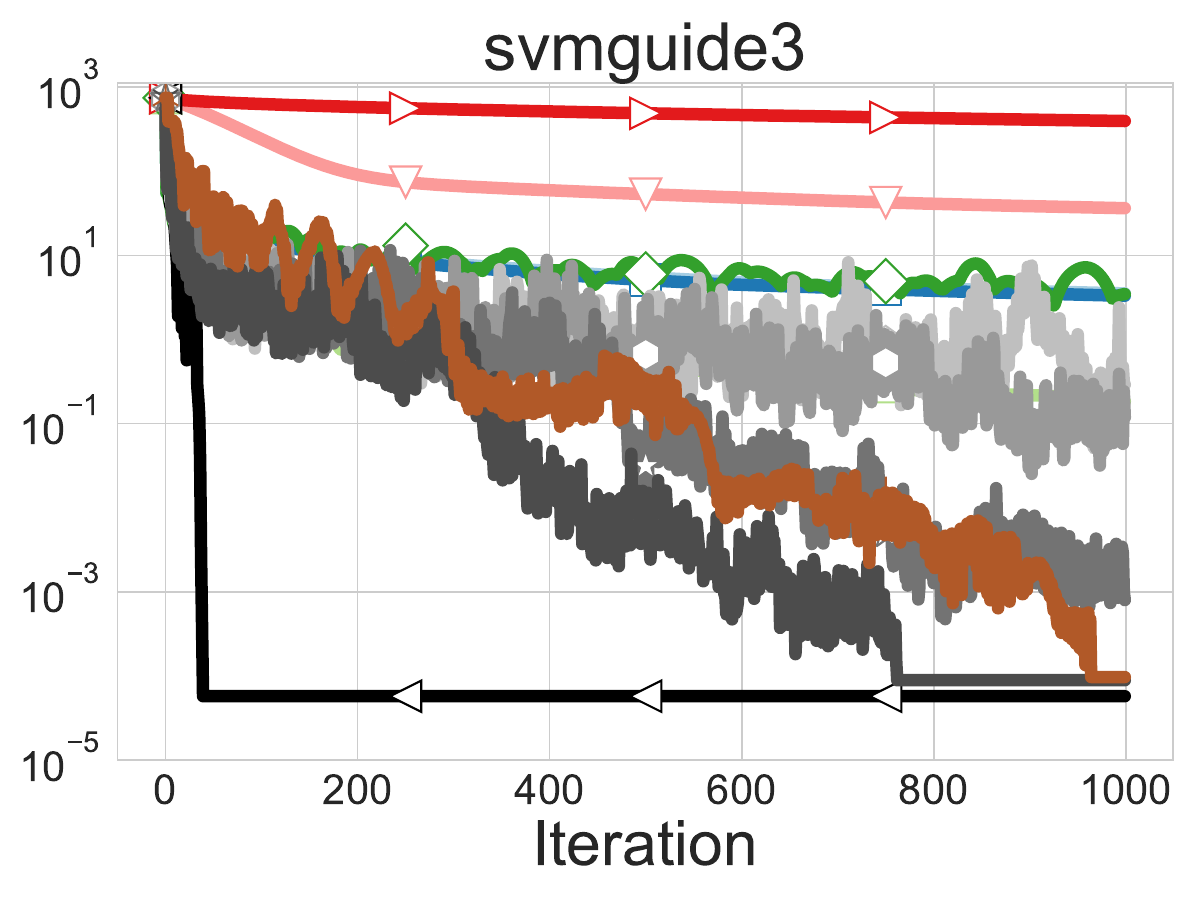}
\includegraphics[scale=0.4]{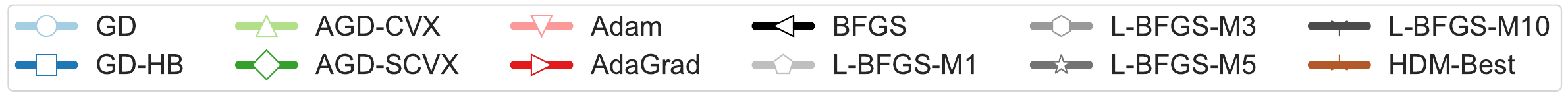}
\caption{Support vector-machine problems. First row: function value gap. Second row: gradient norm. \label{fig:svm}}
\end{figure*}

\begin{figure*}[!h]
\centering
\includegraphics[scale=0.2]{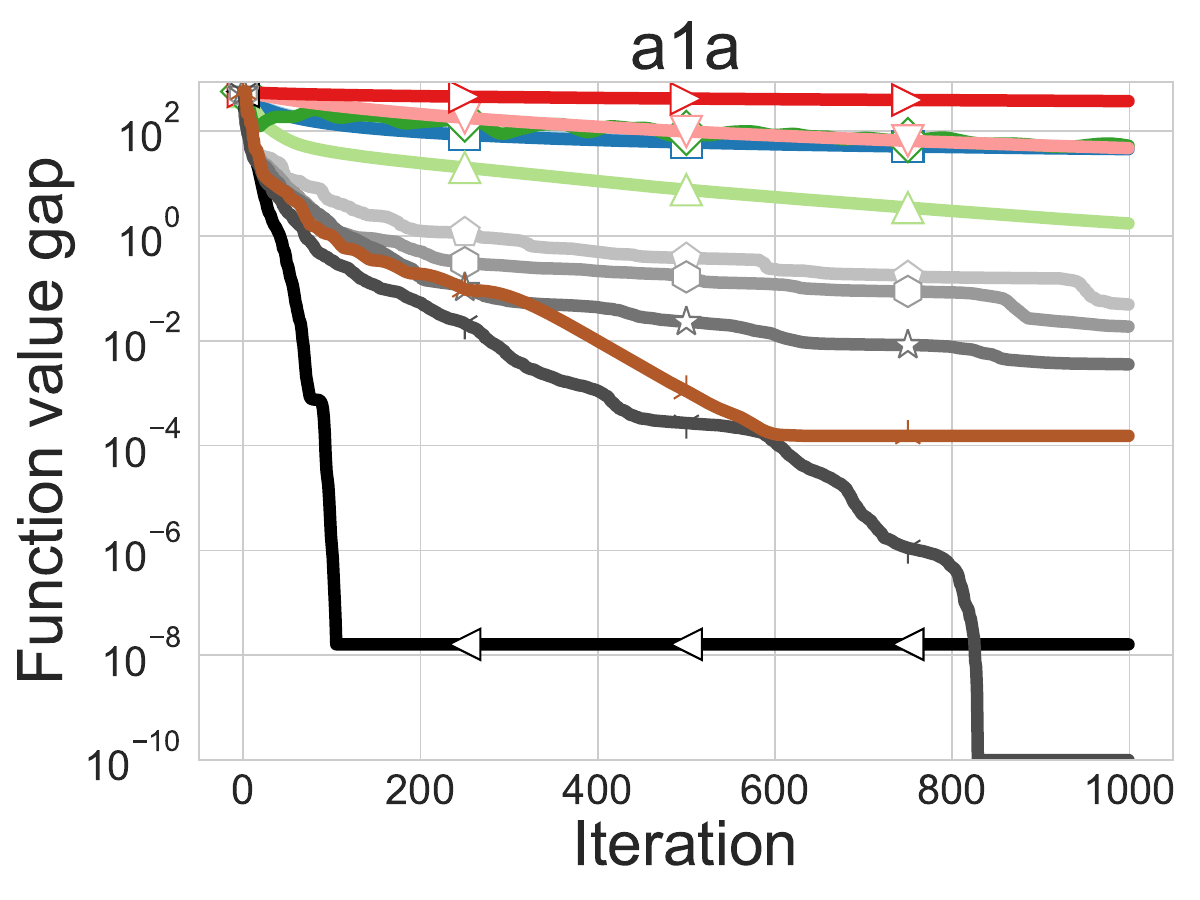}
\includegraphics[scale=0.2]{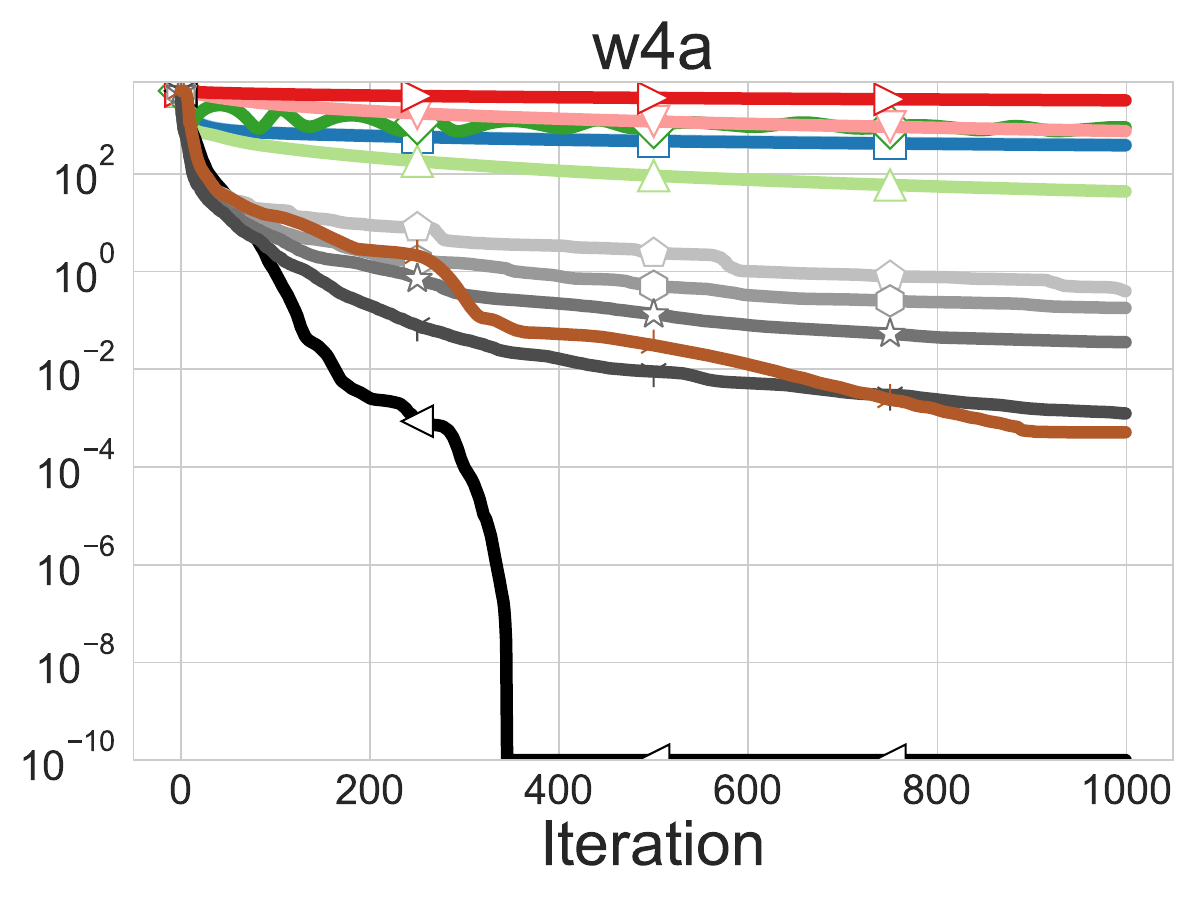}
\includegraphics[scale=0.2]{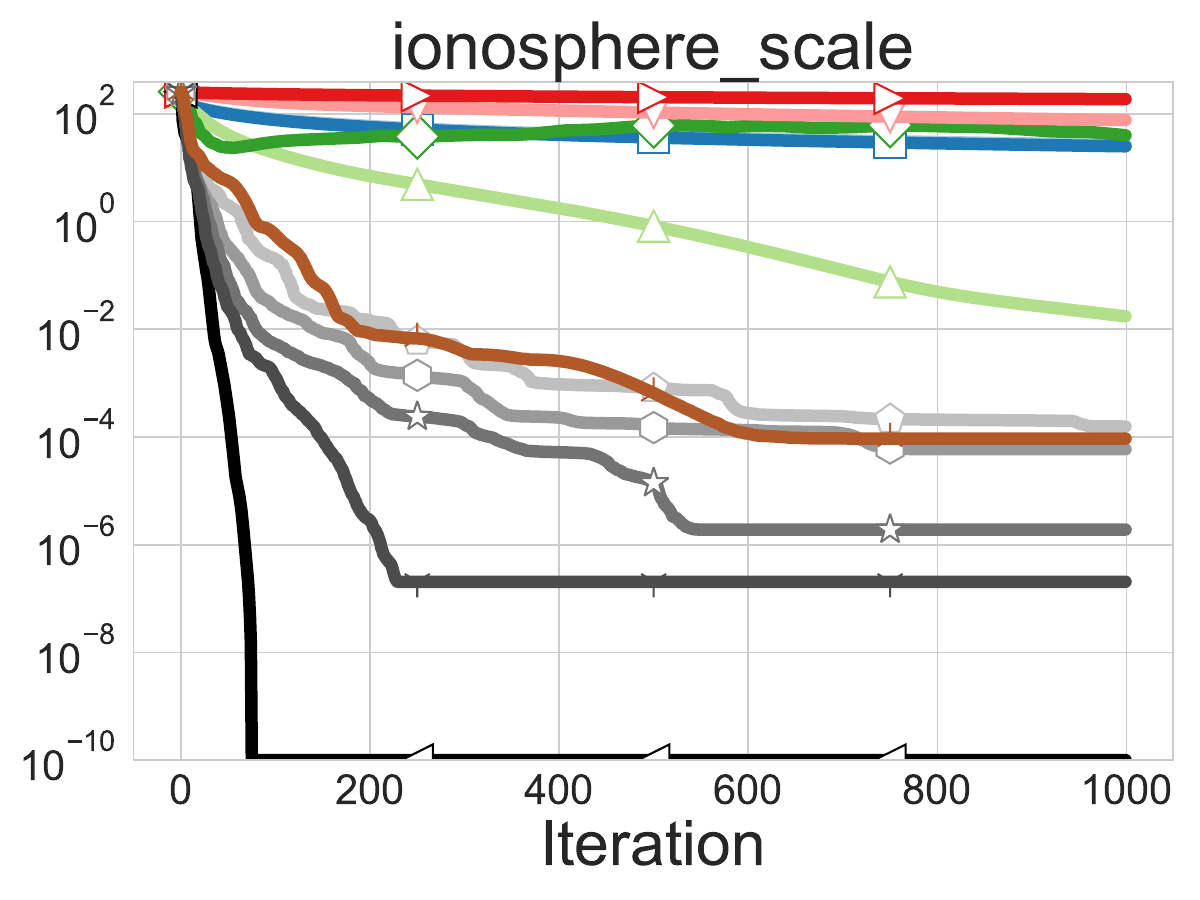}
\includegraphics[scale=0.2]{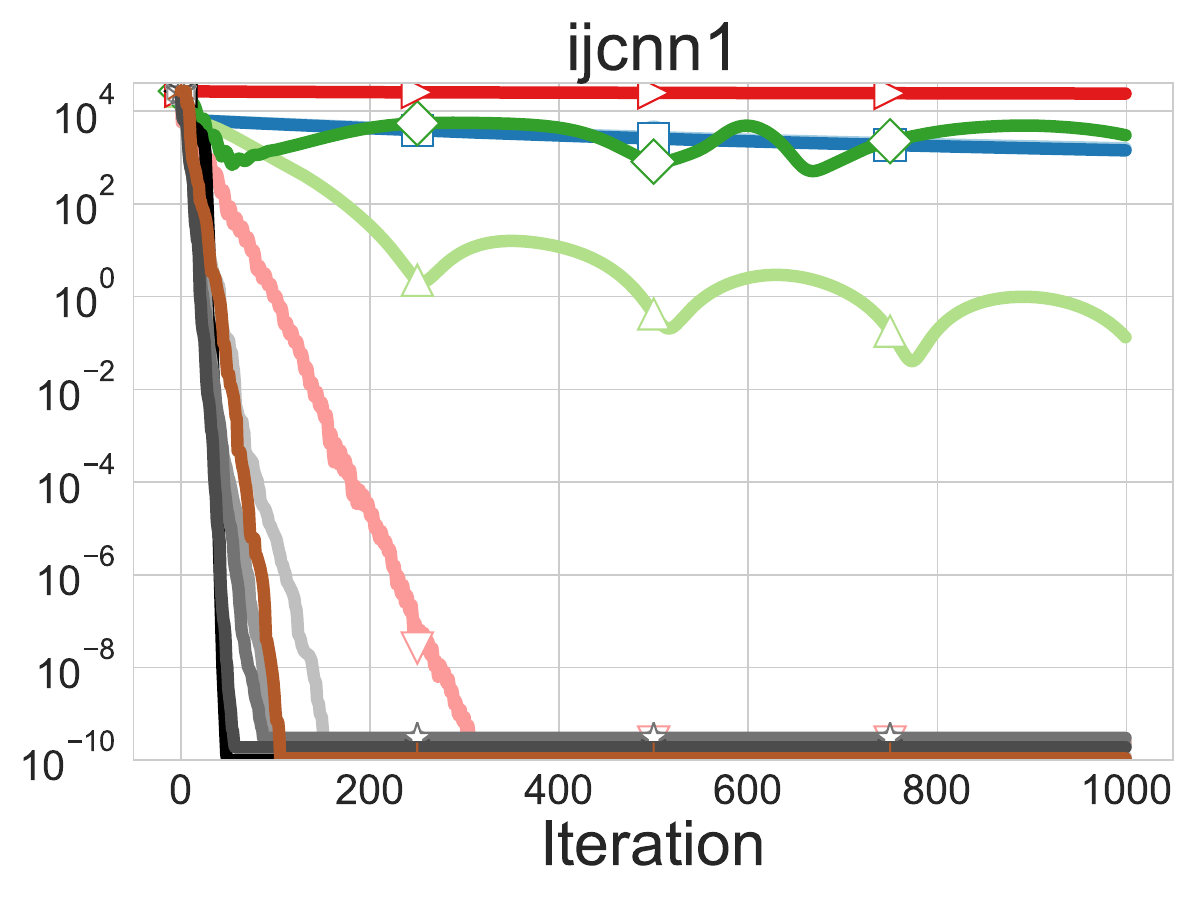}
\includegraphics[scale=0.2]{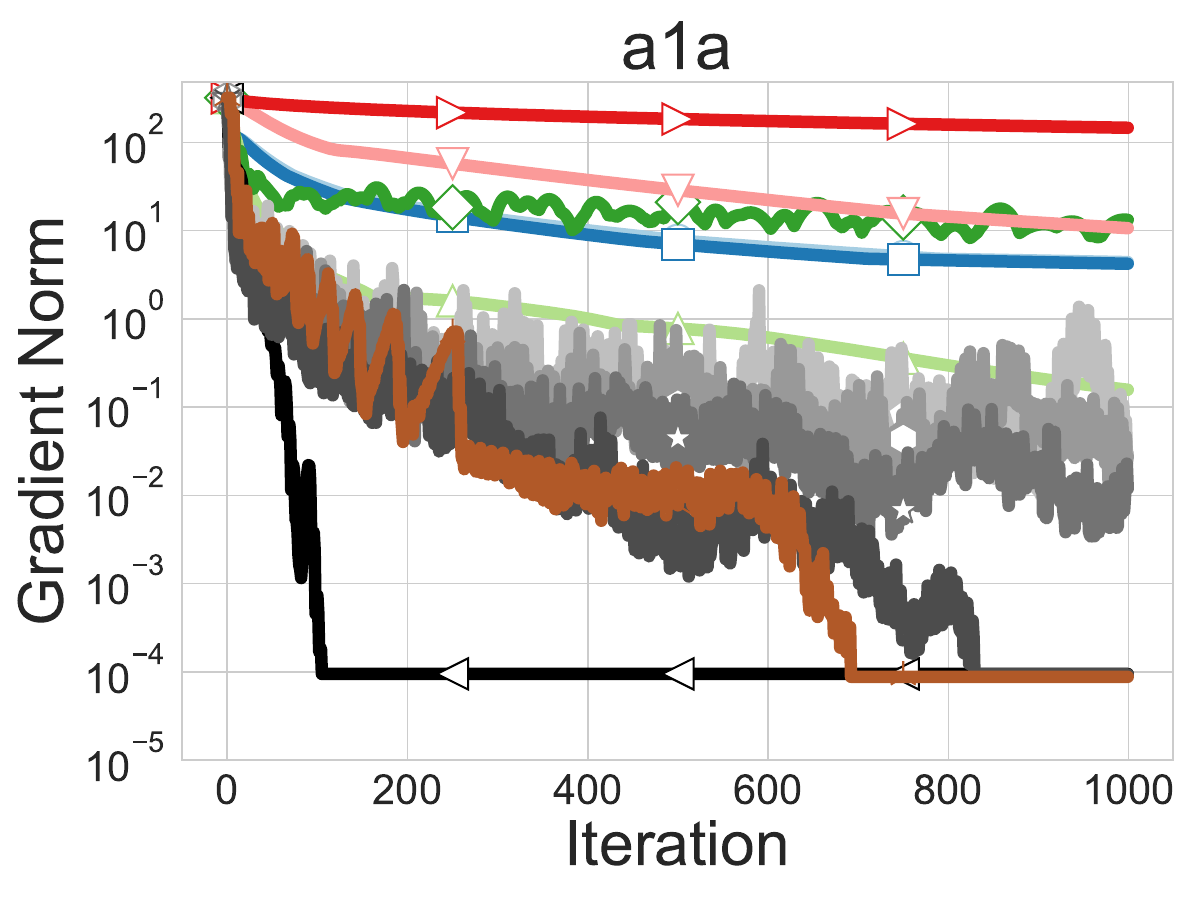}
\includegraphics[scale=0.2]{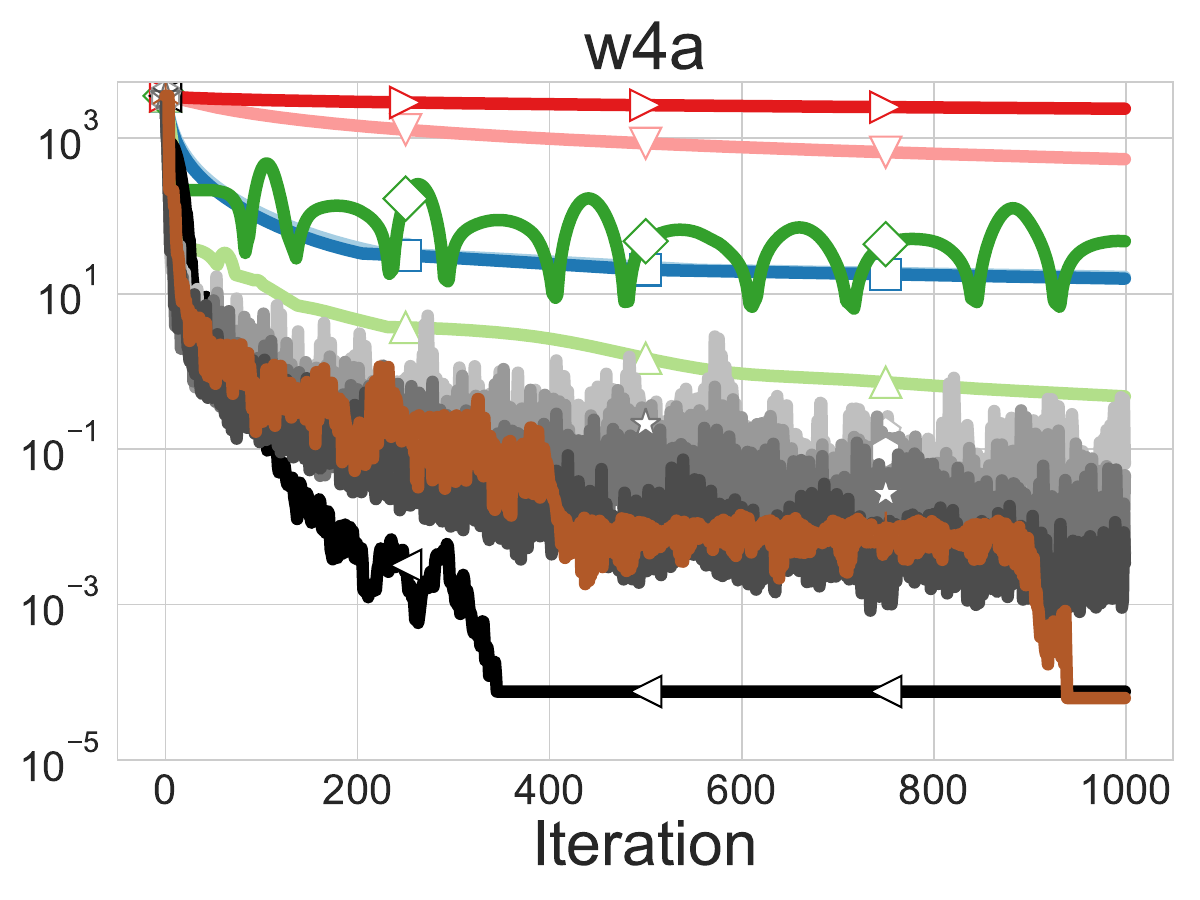}
\includegraphics[scale=0.2]{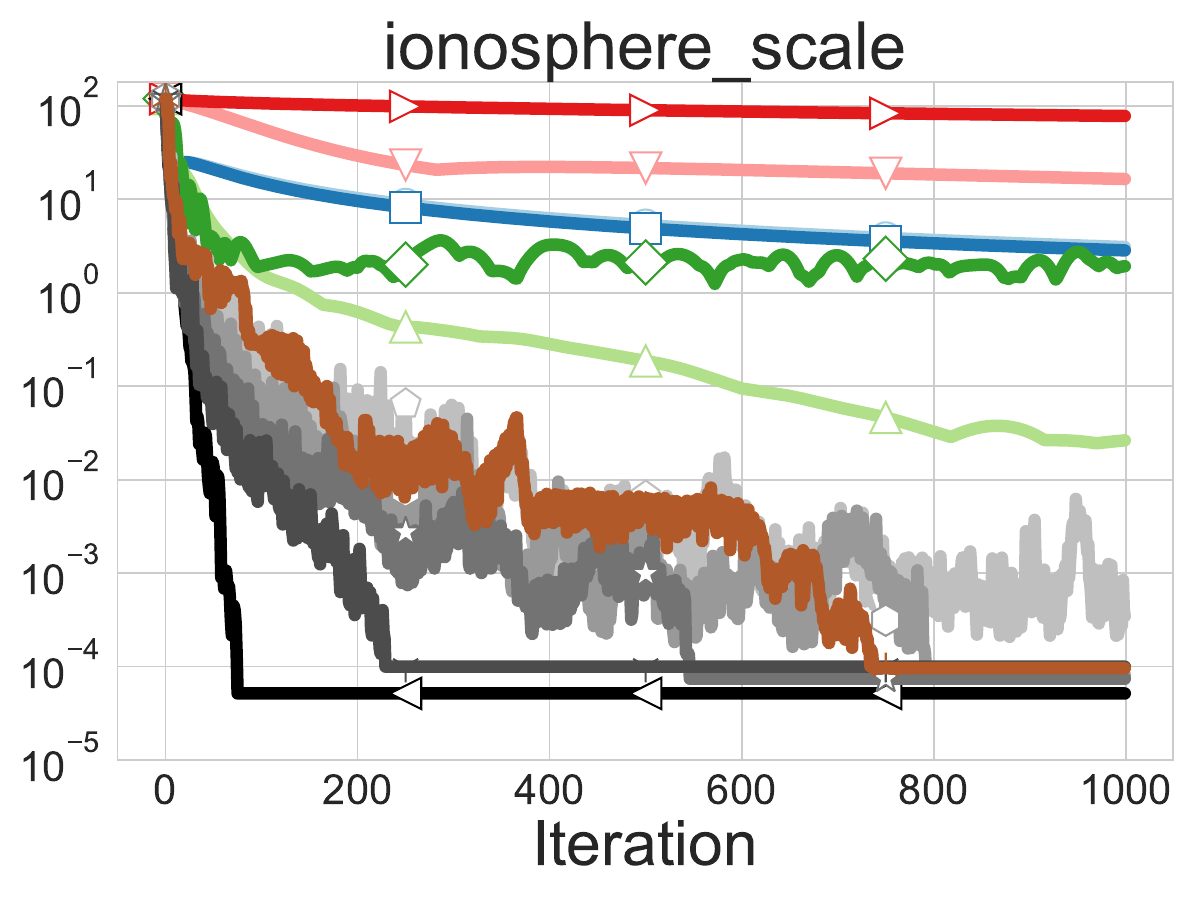}
\includegraphics[scale=0.2]{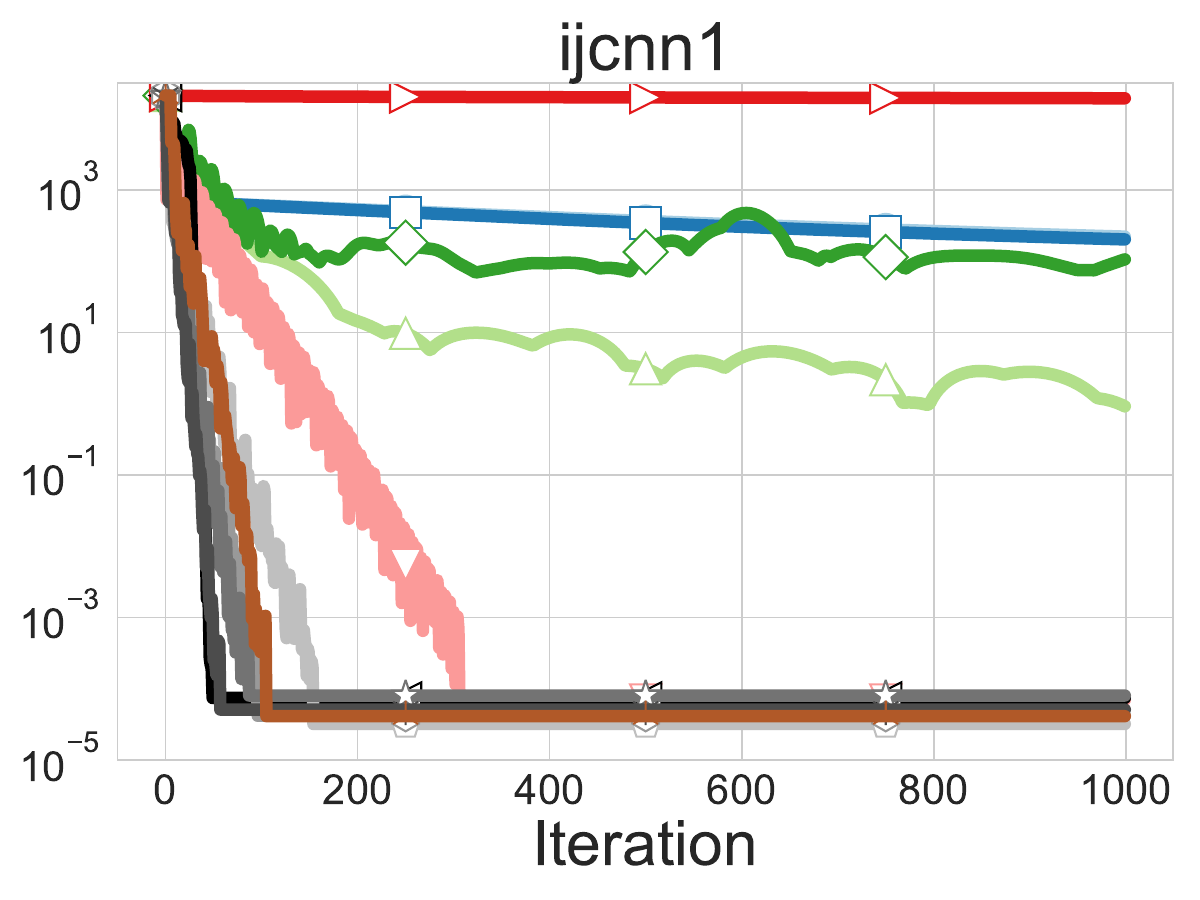}
\caption{Logistic regression problems. First row: function value gap. Second row: gradient norm.  \label{fig:logistic}}
\end{figure*}

For each algorithm, we record the number of successfully solved instances ($\|\nabla f\|_\infty \leq 10^{-4}$ within 1000 gradient oracles). \Cref{table:stats} summarizes the detailed statistics. The number of instances solved by {\hdmbest} is comparable to that of \texttt{L-BFGS-M10}.

\paragraph{Support Vector Machine.} \Cref{fig:svm} shows the function value gap and gradient norm plots on sample test instances on support vector machine problems. The optimal value for each instance is obtained by running {\bfgs} until $\|\nabla f \|_\infty \leq 10^{-4}$.  We see that the practical variant of {\hdmbest} achieves a significant speedup over other adaptive first-order methods. In particular, {\hdmbest} often matches \texttt{L-BFGS-M5} and \texttt{L-BFGS-M10}, while its memory usage is closer to \texttt{L-BFGS-M1}. Notably, {\adam} also achieves competitive performance in several instances.

\paragraph{Logistic Regression.} In logistic regression (\Cref{fig:logistic}), {\hdmbest} still compares well with \texttt{L-BFGS-M5} and is significantly faster than other adaptive first-order methods.\\

Overall, {\hdmbest} demonstrates superior performance on deterministic convex problems and is comparable with the mature \texttt{L-BFGS} family. We believe that further development of {\hdm} will fully unleash its potential for a broad range of optimization tasks.

\section{Conclusion}
This paper addresses the long-standing challenge of establishing convergence of the hypergradient descent heuristic. We provide the first rigorous theoretical foundation for hypergradient descent and introduce a novel online learning perspective that extends to other first-order methods with adaptive hyperparameter updates. Our theoretical advances support effective and scalable enhancements that allow the (first-order) {\hdm} to achieve superlinear convergence with guarantees that resemble quasi-Newton methods. 
Building on these results, we propose {\hdmbest}, an efficient variant of {\hdm} that performs competitively with the widely used \texttt{L-BFGS} method on convex problems. This empirical success positions {\hdm} as a compelling alternative for modern machine learning. Extending the theory of {\hdm} to stochastic and nonconvex optimization is a crucial next step to understanding its potential to speed up the training of large-scale models.

\newpage 
\renewcommand \thepart{}
\renewcommand \partname{}

\bibliography{ref.bib}

\begin{thebibliography}{10}

\bibitem{almeida1999parameter}
Lu{\'\i}s~B Almeida, Thibault Langlois, Jos{\'e}~D Amaral, and Alexander
  Plakhov.
\newblock Parameter adaptation in stochastic optimization.
\newblock In {\em On-line learning in neural networks}, pages 111--134. 1999.

\bibitem{armijo1966minimization}
Larry Armijo.
\newblock Minimization of functions having lipschitz continuous first partial
  derivatives.
\newblock {\em Pacific Journal of mathematics}, 16(1):1--3, 1966.

\bibitem{gunes2018online}
Atilim~Gunes Baydin, Robert Cornish, David~Martinez Rubio, Mark Schmidt, and
  Frank Wood.
\newblock Online learning rate adaptation with hypergradient descent.
\newblock In {\em International Conference on Learning Representations}, 2018.

\bibitem{chandra2022gradient}
Kartik Chandra, Audrey Xie, Jonathan Ragan-Kelley, and Erik Meijer.
\newblock Gradient descent: The ultimate optimizer.
\newblock {\em Advances in Neural Information Processing Systems},
  35:8214--8225, 2022.

\bibitem{chang2011libsvm}
Chih-Chung Chang and Chih-Jen Lin.
\newblock Libsvm: a library for support vector machines.
\newblock {\em ACM transactions on intelligent systems and technology (TIST)},
  2(3):1--27, 2011.

\bibitem{conn1991convergence}
Andrew~R Conn, Nicholas~IM Gould, and Ph~L Toint.
\newblock Convergence of quasi-newton matrices generated by the symmetric rank
  one update.
\newblock {\em Mathematical programming}, 50(1):177--195, 1991.

\bibitem{danilova2020non}
Marina Danilova, Anastasiia Kulakova, and Boris Polyak.
\newblock Non-monotone behavior of the heavy ball method.
\newblock In {\em Difference Equations and Discrete Dynamical Systems with
  Applications: 24th ICDEA, Dresden, Germany, May 21--25, 2018 24}, pages
  213--230. Springer, 2020.

\bibitem{d2021acceleration}
Alexandre d'Aspremont, Damien Scieur, Adrien Taylor, et~al.
\newblock Acceleration methods.
\newblock {\em Foundations and Trends{\textregistered} in Optimization},
  5(1-2):1--245, 2021.

\bibitem{defazio2024road}
Aaron Defazio, Harsh Mehta, Konstantin Mishchenko, Ahmed Khaled, Ashok
  Cutkosky, et~al.
\newblock The road less scheduled.
\newblock {\em arXiv preprint arXiv:2405.15682}, 2024.

\bibitem{duchi2011adaptive}
John Duchi, Elad Hazan, and Yoram Singer.
\newblock Adaptive subgradient methods for online learning and stochastic
  optimization.
\newblock {\em Journal of machine learning research}, 12(7), 2011.

\bibitem{gao2024gradient}
Wenzhi Gao, Ya-Chi Chu, Yinyu Ye, and Madeleine Udell.
\newblock Gradient methods with online scaling.
\newblock {\em arXiv preprint arXiv:2411.01803}, 2024.

\bibitem{gao2023scalable}
Wenzhi Gao, Zhaonan Qu, Madeleine Udell, and Yinyu Ye.
\newblock Scalable approximate optimal diagonal preconditioning.
\newblock {\em arXiv preprint arXiv:2312.15594}, 2023.

\bibitem{hastie2009elements}
Trevor Hastie.
\newblock The elements of statistical learning: data mining, inference, and
  prediction, 2009.

\bibitem{hazan2007logarithmic}
Elad Hazan, Amit Agarwal, and Satyen Kale.
\newblock Logarithmic regret algorithms for online convex optimization.
\newblock {\em Machine Learning}, 69(2):169--192, 2007.

\bibitem{hazan2016introduction}
Elad Hazan et~al.
\newblock Introduction to online convex optimization.
\newblock {\em Foundations and Trends{\textregistered} in Optimization},
  2(3-4):157--325, 2016.

\bibitem{jacobs1988increased}
Robert~A Jacobs.
\newblock Increased rates of convergence through learning rate adaptation.
\newblock {\em Neural networks}, 1(4):295--307, 1988.

\bibitem{jiang2023online}
Ruichen Jiang, Qiujiang Jin, and Aryan Mokhtari.
\newblock Online learning guided curvature approximation: A quasi-newton method
  with global non-asymptotic superlinear convergence.
\newblock In {\em The Thirty Sixth Annual Conference on Learning Theory}, pages
  1962--1992. PMLR, 2023.

\bibitem{jiang2024online}
Ruichen Jiang and Aryan Mokhtari.
\newblock Online learning guided quasi-newton methods with global
  non-asymptotic convergence.
\newblock {\em arXiv preprint arXiv:2410.02626}, 2024.

\bibitem{jie2022adaptive}
Renlong Jie, Junbin Gao, Andrey Vasnev, and Minh-Ngoc Tran.
\newblock Adaptive hierarchical hyper-gradient descent.
\newblock {\em International Journal of Machine Learning and Cybernetics},
  13(12):3785--3805, 2022.

\bibitem{kingma2014adam}
Diederik~P Kingma.
\newblock Adam: A method for stochastic optimization.
\newblock {\em arXiv preprint arXiv:1412.6980}, 2014.

\bibitem{kunstner2024searching}
Frederik Kunstner, Victor Sanches~Portella, Mark Schmidt, and Nicholas Harvey.
\newblock Searching for optimal per-coordinate step-sizes with multidimensional
  backtracking.
\newblock {\em Advances in Neural Information Processing Systems}, 36, 2024.

\bibitem{lee2001ssvm}
Yuh-Jye Lee and Olvi~L Mangasarian.
\newblock Ssvm: A smooth support vector machine for classification.
\newblock {\em Computational optimization and Applications}, 20:5--22, 2001.

\bibitem{li2021second}
Xiaoyu Li, Zhenxun Zhuang, and Francesco Orabona.
\newblock A second look at exponential and cosine step sizes: Simplicity,
  adaptivity, and performance.
\newblock In {\em International Conference on Machine Learning}, pages
  6553--6564. PMLR, 2021.

\bibitem{mahmood2012tuning}
Ashique~Rupam Mahmood, Richard~S Sutton, Thomas Degris, and Patrick~M Pilarski.
\newblock Tuning-free step-size adaptation.
\newblock In {\em 2012 IEEE international conference on acoustics, speech and
  signal processing (ICASSP)}, pages 2121--2124. IEEE, 2012.

\bibitem{mcmahan2010adaptive}
H~Brendan McMahan and Matthew Streeter.
\newblock Adaptive bound optimization for online convex optimization.
\newblock {\em arXiv preprint arXiv:1002.4908}, 2010.

\bibitem{nesterov1983method}
Yurii Nesterov.
\newblock A method for solving the convex programming problem with convergence
  rate o (1/k2).
\newblock In {\em Dokl akad nauk Sssr}, volume 269, page 543, 1983.

\bibitem{nocedal1999numerical}
Jorge Nocedal and Stephen~J Wright.
\newblock {\em Numerical optimization}.
\newblock Springer, 1999.

\bibitem{orabona2019modern}
Francesco Orabona.
\newblock A modern introduction to online learning.
\newblock {\em arXiv preprint arXiv:1912.13213}, 2019.

\bibitem{orabona2016coin}
Francesco Orabona and D{\'a}vid P{\'a}l.
\newblock Coin betting and parameter-free online learning.
\newblock {\em Advances in Neural Information Processing Systems}, 29, 2016.

\bibitem{ozkara2024mada}
Kaan Ozkara, Can Karakus, Parameswaran Raman, Mingyi Hong, Shoham Sabach,
  Branislav Kveton, and Volkan Cevher.
\newblock {MADA}: Meta-adaptive optimizers through hyper-gradient descent.
\newblock In {\em Forty-first International Conference on Machine Learning},
  2024.

\bibitem{polyak1964some}
Boris~T Polyak.
\newblock Some methods of speeding up the convergence of iteration methods.
\newblock {\em Ussr computational mathematics and mathematical physics},
  4(5):1--17, 1964.

\bibitem{polyak1987introduction}
Boris~T Polyak.
\newblock Introduction to optimization.
\newblock 1987.

\bibitem{qu2024optimal}
Zhaonan Qu, Wenzhi Gao, Oliver Hinder, Yinyu Ye, and Zhengyuan Zhou.
\newblock Optimal diagonal preconditioning.
\newblock {\em Operations Research}, 2024.

\bibitem{roulet2017sharpness}
Vincent Roulet and Alexandre d'Aspremont.
\newblock Sharpness, restart and acceleration.
\newblock {\em Advances in Neural Information Processing Systems}, 30, 2017.

\bibitem{rubio2017convergence}
David~Mart{\i}nez Rubio.
\newblock Convergence analysis of an adaptive method of gradient descent.
\newblock {\em University of Oxford, Oxford, M. Sc. thesis}, 2017.

\bibitem{schraudolph1999local}
Nicol~N Schraudolph.
\newblock Local gain adaptation in stochastic gradient descent.
\newblock 1999.

\bibitem{sutton1992adapting}
Richard~S Sutton.
\newblock Adapting bias by gradient descent: An incremental version of
  delta-bar-delta.
\newblock In {\em AAAI}, volume~92, pages 171--176. Citeseer, 1992.

\bibitem{virtanen2020scipy}
Pauli Virtanen, Ralf Gommers, Travis~E Oliphant, Matt Haberland, Tyler Reddy,
  David Cournapeau, Evgeni Burovski, Pearu Peterson, Warren Weckesser, Jonathan
  Bright, et~al.
\newblock Scipy 1.0: fundamental algorithms for scientific computing in python.
\newblock {\em Nature methods}, 17(3):261--272, 2020.

\bibitem{wang2023convergence}
Xiaoyu Wang and Ya-xiang Yuan.
\newblock On the convergence of stochastic gradient descent with
  bandwidth-based step size.
\newblock {\em Journal of Machine Learning Research}, 24(48):1--49, 2023.

\bibitem{zhang2024adam}
Yushun Zhang, Congliang Chen, Ziniu Li, Tian Ding, Chenwei Wu, Yinyu Ye,
  Zhi-Quan Luo, and Ruoyu Sun.
\newblock Adam-mini: Use fewer learning rates to gain more.
\newblock {\em arXiv preprint arXiv:2406.16793}, 2024.

\end{thebibliography}
\bibliographystyle{plain}

\doparttoc
\faketableofcontents
\part{}

\newpage
\appendix
\onecolumn

\addcontentsline{toc}{section}{Appendix}
\part{Appendix} 
\parttoc

\paragraph{Structure of the Appendix. } The appendix is organized as follows. In \Cref{app:hdmprac}, we introduce a practical variant of hypergradient descent and explain its implementation details. \Cref{app:additional} provides additional experiment details on the tested problems. \Cref{app:proof-hdm-ol} to \Cref{app:proof-momentum} provide proofs of the main results in the paper.

\newpage
\section{{\hdm} in Practice} \label{app:hdmprac}

This section introduces {\hdmbest}, our recommended practical hypergradient descent method. This variant is adapted from {\hdmhb}, with simplifications to reduce the implementation complexity. The algorithm is given in \Cref{alg:practical}.

\begin{algorithm}[h]
{\textbf{input} starting point $x^0 = x^1$, $\Pcal = \mathbb{S}^n_{+} \cap \Dcal, \Bcal = [0,0.9995]$, initial diagonal preconditioner $P_1 \in \mathbb{S}^n_{+} \cap \Dcal$, \\initial scalar momentum parameter $\beta_1 = 0.95$, {\adagrad} stepsize $\eta_p, \eta_b > 0$, {\adagrad}  diagonal matrix $U_1 = 0$, {\adagrad} momentum scalar $v_1 = 0$, $\tau > 0$}\\
\For{k =\rm{ 1, 2,...}}{
$\begin{aligned}
x^{k + 1 / 2} & = x^k - P_k \nabla f (x^k) + \beta_k (x^k - x^{k - 1}) \\
\nabla_P h_{x^k, x^{k - 1}} (P_k, \beta_k) & = \tfrac{\diag(\nabla f (x^{k + 1 / 2}) \circ \nabla f (x^k))}{\| \nabla f (x^k) \|^2 + \frac{\tau}{2} \| x^k - x^{k - 1} \|^2} \texttt{ \# Element-wise product} \\
\nabla_{\beta} h_{x^k, x^{k - 1}} (P_k, \beta_k) & = \tfrac{\langle \nabla f (x^{k + 1 / 2}), x^k - x^{k - 1} \rangle}{\| \nabla f (x^k) \|^2 + \frac{\tau}{2} \| x^k - x^{k - 1} \|^2} \texttt{ \# Inner product} \\
U_{k + 1} & = U_k + \nabla_P h_{x^k, x^{k - 1}} (P_k, \beta_k) \circ \nabla_P h_{x^k, x^{k - 1}} (P_k, \beta_k) \texttt{ \# Diagonal matrix} \\
v_{k + 1} & = v_k + \nabla_{\beta} h_{x^k, x^{k - 1}} (P_k, \beta_k) \cdot \nabla_{\beta} h_{x^k, x^{k - 1}} (P_k, \beta_k) \texttt{ \# Scalar matrix} \\
P_{k + 1} & = \Pi_{\Rbb^n_{+} \cap \Dcal} [P_k - \eta_p U_{k + 1}^{- 1 / 2} \nabla_P h_{x^k, x^{k - 1}} (P_k, \beta_k)] \texttt{ \# Diagonal matrix} \\
\beta_{k + 1} & = \Pi_{[0, 0.9995]} [\beta_k - \eta_b v_{k + 1}^{- 1 / 2} \nabla_{\beta} h_{x^k, x^{k - 1}} (P_k, \beta_k)] \\
x^{k + 1} & = \argmin_{x \in \{ x^k, x^{k + 1 / 2} \}} f (x).
\end{aligned}
$
}
{\textbf{output} $x^{K+1}$}
\caption{{\hdmbest} \label{alg:practical}}
\end{algorithm}

We make several remarks about \Cref{alg:practical}. 

\begin{itemize}[leftmargin=10pt]
\item \textit{Choice of online learning algorithm.} Unless $f(x)$ is quadratic, adaptive online learning algorithms such as {\adagrad} often significantly outperform online gradient descent with constant stepsize. Note that {\adagrad} introduces additional memory of size $n$ to store the diagonal online learning preconditioner $U$.

\item \textit{Sensitivity of parameters.} The two stepsize parameters in {\adagrad} are the most important algorithm parameters: $\eta_p, \eta_b$. According to the experiments, $\eta_p$ should be set proportional to $1/L$, the smoothness constant, while an aggressive choice of $\eta_b \in \{1,10,100\}$ often yields fast convergence. A local estimator of the smoothness constant $L$ can significantly enhance algorithm performance.
\item \textit{Heavy-ball feedback and null step.} In practice, it is observed that dropping the $\frac{\omega}{2}\|x^+(P, B) - x\|^2$ in the numerator of heavy-ball feedback \eqref{eqn:heavyball-feedback} often does not affect algorithm performance. Therefore, in \Cref{alg:practical} the hypergradient with respect to $\frac{\omega}{2}\|x^+(P, B) - x\|^2$ is ignored. 
On the other hand, the $\frac{\tau} {2}\|x^+(P, B) - x\|^2$  term in the denominator smoothes the update of $\beta_k$ and can strongly affect convergence. The parameter $\tau$ should be taken to be proportional to $L^2$ according to the discussions in \Cref{app:heavy-ball}. The null step is taken with respect to the function value $f(x)$ instead of the heavy-ball potential function.
\item \textit{Memory usage.} The memory usage of {\hdmbest}, measured in terms of number of vectors of length $n$ is $7n$: 1) three vectors store primal iterates $x^{-}, x, x^{+}$. 2) Two vectors store past and buffer gradients $\nabla f(x), \nabla f(x^+)$. 3) A vector stores the diagonal preconditioner $P_k$. 4) A vector stores the {\adagrad} stepsize matrix $U$.
\item \textit{Computational cost. } The major additional computation cost arises from computing hypergradient $\nabla h$, which involves one element-wise product and one inner product for vectors of size $n$. In addition, {\hdmbest} needs to maintain a diagonal matrix for {\adagrad}. The overall additional computational cost is several $\Ocal({n})$ operations.
\end{itemize}

\newpage
\section{Proof of Results in Section \ref{sec:hdm-ol}} \label{app:proof-hdm-ol}

\subsection{Auxiliary Results}

\begin{lem}[Sublinear dynamic regret \cite{hazan2016introduction}] \label{lem:auxi-dynamic}
  Given a family of convex and $\gamma$-Lipschitz losses $\{ h_k \}$, online
  gradient descent $P_{k + 1} = \Pi_{\mathcal{P}} [P_k - \eta \nabla h_k
  (P_k)]$ with constant stepsize $\eta > 0$ generates a sequence of scaling
  matrices $\{ P_k \}$ such that
\begin{equation} \label{eqn:auxi-ogd-dynamic-regret}
	\textstyle \sum_{k = 1}^K h_k (P_k) - h_k (\hat{P}_k) \leq \tfrac{D^2}{2 \eta} +
     \tfrac{\eta}{2} \gamma^2 K + \tfrac{D}{2 \eta} \mathsf{PL} (\{ \hat{P}_k
     \}).
\end{equation}
  where $\{ \hat{P}_k \}, \hat{P}_k\in \Pcal$ are arbitrarily chosen competitors and $\mathsf{PL}
  (\{ \hat{P}_k \}) \assign \sum_{k = 1}^K \| \hat{P}_k - \hat{P}_{k + 1}
  \|_F$ is the path length of the competitors. In particular, if $\hat{P}_k \equiv P$, then $\mathsf{PL} (\{ \hat{P}_k
     \})=0$ and 
\begin{equation} \label{eqn:auxi-ogd-static-regret}
	\textstyle \sum_{k = 1}^K h_k (P_k) - h_k (\hat{P}_k) \leq \tfrac{D^2}{2 \eta} +
     \tfrac{\eta}{2} \gamma^2 K.
\end{equation}
\end{lem}
\begin{proof}
The result follows from a standard dynamic regret analysis from online convex
optimization literature, and we adapt the proof for our analysis. For any $P
\in \mathcal{P}$, we deduce
\begin{align}
  \| P_{k + 1} - P \|_F^2 ={} & \| \Pi_{\mathcal{P}} [P_k - \eta \nabla h_k
  (P_k)] - P \|_F^2 \nonumber\\
  \leq{} & \| P_k - P - \eta \nabla h_k (P_k) \|_F^2 \label{eqn:dynamic-regret-1} \\
  \leq{} & \| P_k - P \|_F^2 - 2 \eta \langle \nabla h_k (P_k), P_k - P \rangle
  + \eta^2 \| \nabla h_k (P_k) \|^2_F \nonumber \\
  \leq{} & \| P_k - P \|_F^2 - 2 \eta [h_k (P_k) - h_k (P)] + \eta^2 \gamma^2, \label{eqn:dynamic-regret-2} 
\end{align}
where \eqref{eqn:dynamic-regret-1} uses non-expansiveness of orthogonal projection; \eqref{eqn:dynamic-regret-2} applies convexity and $\gamma$-Lipschitz continuity of $h_k$.
Now, let $P = \hat{P}_k$ and we re-arrange to get
\begin{align}
  h_k (P_k) - h_k (\hat{P}_k) \leq{} & \tfrac{1}{2 \eta} [\| P_k - \hat{P}_k
  \|_F^2 - \| P_{k + 1} - \hat{P}_k \|_F^2] + \tfrac{\eta}{2} \gamma^2
  \nonumber\\
  ={} & \tfrac{1}{2 \eta} [\| P_k \|_F^2 - \| P_{k + 1} \|_F^2 + 2 \langle
  \hat{P}_k, P_{k + 1} - P_k \rangle] + \tfrac{\eta}{2} \gamma^2 \nonumber\\
  \leq{} & \tfrac{D^2}{2 \eta} + \tfrac{D}{2 \eta} \| P_{k + 1} - P_k \|_F +
  \tfrac{\eta}{2} \gamma^2,\label{eqn:dynamic-regret-3}
\end{align}
where \eqref{eqn:dynamic-regret-3} uses Cauchy's inequality $\langle
  \hat{P}_k, P_{k + 1} - P_k \rangle \leq \|\hat{P}_k\|\cdot\|P_{k + 1} - P_k\| \leq D\|P_{k + 1} - P_k\|$. 
   Telescoping,
\[ \textstyle \sum_{k = 1}^K h_k (P_k) - h_k (\hat{P}_k) \leq \tfrac{D^2}{2 \eta} +
   \tfrac{\eta}{2} \gamma^2 K + \tfrac{D}{2 \eta} \sum_{k = 1}^K \| \hat{P}_k
   - \hat{P}_{k + 1} \|_F \]
and this completes the proof.
\end{proof}
\begin{lem}[Logarithmic static regret \cite{orabona2019modern}] \label{lem:auxi-log-regret}
  Given a family of $\mu$-strongly convex and $\gamma$-Lipschitz losses $\{
  h_k \}$, online gradient descent $P_{k + 1} = \Pi_{\mathcal{P}} [P_k - \eta_k \nabla h_k
  (P_k)]$ with stepsize $\eta_k = 1 / (\mu k)$
  generates a sequence of scaling matrices $\{ P_k \}$ such that $\sum_{k =
  1}^K h_k (P_k) - h_k (P) \leq \tfrac{1}{2} \gamma^2 \log K.$
\end{lem}

\begin{proof}

Using strong convexity, we have $h_k (P) \geq h_k (P_k) + \langle \nabla h_k (P_k), P - P_k \rangle +
   \tfrac{\mu}{2} \| P - P_k \|_F^2  $ and 
\begin{align}
  \| P_{k + 1} - P \|_F^2 \leq{} & \| P_k - P \|_F^2 - 2 \eta_k \langle \nabla
  h_k (P_k), P_k - P \rangle + \eta_k^2 \gamma^2 \nonumber\\
  \leq{} & \| P_k - P \|_F^2 - 2 \eta_k [h_k (P_k) - h_k (P)] + \eta_k^2
  \gamma^2 - \mu \eta_k \| P - P_k \|_F^2 \nonumber\\
  ={} & \tfrac{k - 1}{k} \| P_k - P \|_F^2 - \tfrac{2}{k \mu} [h_k (P_k) - h_k
  (P)] + \tfrac{\gamma^2}{k^2 \mu^2} \label{eqn:dynamic-regret-4}, 
\end{align}
where \eqref{eqn:dynamic-regret-4} plugs in $\eta_k = 1/(\mu k)$. Re-arranging the terms, \[h_k (P_k) - h_k (P) \leq \tfrac{\mu}{2} [(k - 1) \|
P_k - P \|_F^2 - k \| P_{k + 1} - P \|_F^2] + \tfrac{\gamma^2}{2 k \mu}\] and
telescoping gives $\sum_{k = 1}^K h_k (P_k) - h_k (P) \leq
\sum_{k = 1}^K \tfrac{\gamma^2}{2 k \mu} \leq \tfrac{\gamma^2}{2 \mu} (\log K
+ 1)$, which completes the proof.
\end{proof}

\subsection{Proof of Lemma \ref{lem:hx-properties}}

Consider the first property. Convexity and smoothness follow directly from \cite{gao2024gradient}.
To verify strong convexity, note that for $h_x (\alpha) = \tfrac{f (x - \alpha
\nabla f (x)) - f (x^{\star})}{\| \nabla f (x) \|^2}$
\[ h_x'' (\alpha) = \tfrac{\mathd}{\mathd \alpha} \Big[ \tfrac{\langle \nabla
   f (x - \alpha \nabla f (x)), \nabla f (x) \rangle}{\| \nabla f (x) \|^2}
   \Big] = \big\langle \tfrac{\nabla f (x)}{\| \nabla f (x) \|}, \nabla^2 f
   (x) \tfrac{\nabla f (x)}{\| \nabla f (x) \|} \big\rangle \geq \mu \]
since $\nabla^2 f (x) \succeq \mu I$ and $x \nin \Xcal^\star$. This completes the proof of the first property.\\

Next, we consider the second property. Lipschitz continuity also follows from \cite{gao2024gradient}. To verify exp-concavity, recall that a twice-differentiable function
$h$ is $\beta$-exp-concave if $\nabla^2 h (x) \succeq \beta \nabla h (x)
\nabla h (x)^{\top}$ for some $\beta \geq 0$. By definition of $\Dcal$,
\[ \nabla h_x (P) = - \tfrac{\nabla f (x) \circ \nabla f (x - P \nabla f
   (x))}{\| \nabla f (x) \|^2} = - \tfrac{\diag (\nabla f (x)) \nabla f
   (x - P \nabla f (x))}{\| \nabla f (x) \|^2} \]
and $\nabla^2 h_x (P) = \frac{\diag (\nabla f (x)) \nabla^2 f (x - P
\nabla f (x)) \diag (\nabla f (x))}{\| \nabla f (x) \|^2}$. Using
$\nabla^2 f (x - P \nabla f (x)) \succeq \mu I$, we deduce that
\begin{align}
  & \nabla^2 h_x (P) - \beta\nabla h_x (P) \nabla h_x
  (P)^{\top} \nonumber\\
  ={} & \tfrac{\diag (\nabla f (x)) \nabla^2 f (x - P \nabla f (x))
  \diag (\nabla f (x))}{\| \nabla f (x) \|^2} - \beta \tfrac{\diag
  (\nabla f (x)) \nabla f (x - P \nabla f (x)) \nabla f (x - P \nabla f
  (x))^{\top} \diag (\nabla f (x))}{\| \nabla f (x) \|^4} \nonumber\\
  ={} & \diag \big( \tfrac{\nabla f (x)}{\| \nabla f (x) \|} \big)
  \Big[ \nabla^2 f (x - P \nabla f (x)) - \beta \tfrac{\nabla f (x - P \nabla
  f (x))}{\| \nabla f (x) \|} \tfrac{\nabla f (x - P \nabla f (x))^{\top}}{\| \nabla
  f (x) \|} \Big] \diag \big( \tfrac{\nabla f (x)}{\| \nabla f
  (x) \|} \big) \nonumber\\
  \succeq{} & \diag \big( \tfrac{\nabla f (x)}{\| \nabla f (x) \|}
  \big) \Big[ \mu I - \beta \tfrac{\nabla f (x - P \nabla f (x))}{\| \nabla
  f (x) \|} \tfrac{\nabla f (x - P \nabla f (x))^{\top}
}{\| \nabla f (x) \|}  \Big] \diag \big( \tfrac{\nabla f (x)}{\| \nabla f (x) \|} \big), \label{eqn:proof-1-1}
\end{align}
where \eqref{eqn:proof-1-1} uses $\mu$-strong convexity of $f(x)$. Now, it suffices to verify that
\begin{equation} \label{eqn:proof-1-2}
	\tfrac{\nabla f (x - P \nabla f (x))}{\| \nabla f (x) \|} \tfrac{\nabla f
   (x - P \nabla f (x))^{\top}}{\| \nabla f (x) \|} \preceq \tfrac{\mu}{\beta}  I
\end{equation}
for all $x \nin \Xcal^\star$.
Write $\tfrac{\nabla f (x - P \nabla f (x))}{\| \nabla f (x) \|} =
\tfrac{\nabla f (x)}{\| \nabla f (x) \|} + \tfrac{\nabla f (x - P \nabla f
(x)) - \nabla f (x)}{\| \nabla f (x) \|}$ and let $z \assign \nabla f (x -
P \nabla f (x)) - \nabla f (x)$, we have, by $L$-smoothness, that $\| z \| \leq L \|P \nabla f(x) \| \leq L D \| \nabla f (x) \|$
and
\begin{align}
  \Big\| \tfrac{\nabla f (x - P \nabla f (x))}{\| \nabla f (x) \|}
  \tfrac{\nabla f (x - P \nabla f (x))^{\top}}{\| \nabla f (x) \|} \Big\| ={}
  & \Big\| \Big( \tfrac{\nabla f (x)}{\| \nabla f (x) \|} + \tfrac{z}{\|
  \nabla f (x) \|} \Big) \Big( \tfrac{\nabla f (x)}{\| \nabla f (x) \|} +
  \tfrac{z}{\| \nabla f (x) \|} \Big)^{\top} \Big\| \nonumber\\
  ={} & \big\| \tfrac{\nabla f (x) \nabla f (x)^{\top}}{\| \nabla f (x) \|^2} +
  \tfrac{z \nabla f (x)^{\top}}{\| \nabla f (x) \|^2} + \tfrac{\nabla f (x)
  z^{\top}}{\| \nabla f (x) \|^2} + \tfrac{z z^{\top}}{\| \nabla f (x) \|^2}
  \big\| \nonumber\\
  \leq{} & 1 + \tfrac{2 \| z \|}{\| \nabla f (x) \|} + \tfrac{\| z \|^2}{\|
  \nabla f (x) \|^2} = ( 1 + \tfrac{\| z \|}{\| \nabla f (x) \|}
  )^2 \leq (1 + L D)^2 . \nonumber
\end{align}

Hence, for $\beta \leq \tfrac{\mu}{(1 + L D)^2}$ the relation \eqref{eqn:proof-1-2} holds. We conclude that $h_x(P) = h_x(d)$ is $\frac{\mu}{(1 + L D)^2}$-exponential concave.

\subsection{Proof of Lemma \ref{lem:regret-sublinear}}

We use Lipschitzness from \Cref{lem:hx-properties} and \eqref{eqn:auxi-ogd-static-regret} from \Cref{lem:auxi-dynamic} by taking $\gamma = 1+LD$ and $\eta = \frac{D}{(LD+1)\sqrt{K}}$.

\subsection{Proof of Lemma \ref{lem:regret-log}}
We use Lipschitzness and strong convexity from \Cref{lem:hx-properties} and invoke \Cref{lem:auxi-log-regret} by taking $\gamma = 1 + L D$.

\subsection{Proof of Lemma \ref{lem:hypergrad-to-online}}

The proof resembles \cite{gao2024gradient} and uses a tighter analysis.
Consider the optimality measure $f (x^{K + 1}) - f (x^{\star})$, and we deduce that
\begin{align}
  f (x^{K + 1}) - f (x^{\star}) ={} & \frac{1}{\frac{1}{f (x^{K + 1}) - f
  (x^{\star})}} \nonumber\\
  ={} & \frac{1}{\sum_{k = 1}^K \frac{1}{f (x^{k + 1}) - f (x^{\star})} -
  \frac{1}{f (x^k) - f (x^{\star})} + \frac{1}{f (x^1) - f (x^{\star})}}
  \nonumber\\
  ={} & \frac{1}{\sum_{k = 1}^K \frac{f (x^k) - f (x^{k + 1})}{[f (x^{k + 1}) -
  f (x^{\star})] [f (x^k) - f (x^{\star})]} + \frac{1}{f (x^1) - f
  (x^{\star})}} \nonumber\\
  ={} & \frac{1}{\sum_{k = 1}^K \frac{\max \{ - h_{x^k} (P_k), 0 \} \| \nabla f
  (x^k) \|^2}{[f (x^{k + 1}) - f (x^{\star})] [f (x^k) - f (x^{\star})]} +
  \frac{1}{f (x^1) - f (x^{\star})}} \nonumber
\end{align}

Next, using $f (x) - f (x^{\star}) \leq \| \nabla f (x) \| \cdot \| x -
x^{\star} \|$,
\[ \tfrac{\max \{ - h_{x^k} (P_k), 0 \} \| \nabla f (x^k) \|^2}{[f (x^{k + 1})
   - f (x^{\star})] [f (x^k) - f (x^{\star})]} \geq \tfrac{\max \{ - h_{x^k}
   (P_k), 0 \} \| \nabla f (x^k) \|^2}{[f (x^k) - f (x^{\star})]^2} \geq
   \tfrac{\max \{ - h_{x^k} (P_k), 0 \}}{\dist (x^k,
   \mathcal{X}^{\star})^2} \geq \tfrac{\max \{ - h_{x^k} (P_k), 0
   \}}{\Delta^2}. \]
Finally, we deduce that
\begin{align}
  f (x^{K + 1}) - f (x^{\star}) \leq{} & \tfrac{\Delta^2}{\sum_{k = 1}^K \max \{
  - h_{x^k} (P_k), 0 \} + \tfrac{\Delta^2}{f (x^1) - f (x^{\star})}}
  \nonumber\\
  \leq{} & \tfrac{\Delta^2}{\max \{ \sum_{k = 1}^K - h_{x^k} (P_k), 0
  \} + \frac{\Delta^2}{f (x^1) - f (x^{\star})}} \nonumber\\
  \leq{} & \min \Big\{ \tfrac{\Delta^2}{K \max \{ \frac{1}{K} \sum_{k =
  1}^K - h_{x^k} (P_k), 0 \}}, f (x^1) - f (x^{\star}) \Big\}
  \nonumber
\end{align}
and this completes the proof.

\subsection{Intermediate Iterate Convergence with Adaptive Online Algorithms} \label{sec:intermediate-iter}

One disadvantage of constant stepsize online gradient descent is 1) it requires the total iteration number $K$. 2) No regret guarantee for the intermediate iterates. One simple fix is let $\eta_k = \Ocal(1/\sqrt{k})$. It gives the same sublinear regret guarantee up to a constant multiplicative factor \cite{orabona2019modern}, but the regret guarantee holds for any $k$. Similar arguments hold for adaptive gradient methods \cite{duchi2011adaptive,mcmahan2010adaptive}.

\section{Proof of Results in Section \ref{sec:hdm}}  \label{app:proof-hdm}

\subsection{Proof of Theorem \ref{thm:adaptivity}}

Plugging \eqref{eqn:ogd-constant} from \Cref{lem:regret-sublinear} into \Cref{lem:hypergrad-to-online} completes the proof.

\subsection{Proof of Theorem \ref{thm:dynamic-adaptivity}}
Invoking Lipschitzness from \Cref{lem:hx-properties} and \eqref{eqn:auxi-ogd-dynamic-regret} from \Cref{lem:auxi-dynamic} with $\gamma =1+LD,\eta = \frac{D}{(LD+1)\sqrt{K}}$ gives

\begin{equation}
	\textstyle \sum_{k = 1}^K h_{x^k} (P_k) - h_{x^k} (\hat{P}_k) \leq \rho_K + \frac{LD+1}{2} \sqrt{K}\sum_{k = 1}^K \| \hat{P}_k
   - \hat{P}_{k + 1} \|_F\nonumber.
\end{equation}

Plugging the relation into \Cref{lem:hypergrad-to-online} completes the proof.

\subsection{Proof of Theorem \ref{thm:superlin}}

For \Cref{thm:superlin} and \Cref{lem:scalmat-conv} only, we will define the following modified feedback function by replacing $f(x^k)$ in the numerator by $f(x^\star)$:
\[ \hat{h}_x (P) \assign \tfrac{f (x - P \nabla f (x)) - f (x^{\star}) }{\|
   \nabla f (x) \|^2}\geq 0. \]
For a fixed $x$, $\hat{h}_x (P)$ only differs from the original hypergradient
feedback by a constant; it has the same properties as the original feedback function, and the algorithm is exactly the same since only the gradient of $\hat{h}_x$ is considered in the algorithm update.  Using the definition of $\hat{h}_x (P)$, we deduce that 
\begin{align}
  \tfrac{f (x^{K + 1}) - f (x^{\star})}{f (x^1) - f (x^{\star})} ={} & \textstyle \prod_{k = 1}^K \tfrac{f (x^{k + 1}) - f
  (x^{\star})}{f (x^k) - f (x^{\star})} \nonumber\\
  \leq{} & ( \tfrac{1}{K} \textstyle \sum_{k = 1}^K \tfrac{f (x^{k + 1}) - f
  (x^{\star})}{f (x^k) - f (x^{\star})} )^K \nonumber\\
  ={} & ( \tfrac{1}{K} \textstyle \sum_{k = 1}^K \min \{ \tfrac{\hat{h}_{x^k} (P_k) \|
  \nabla f (x^k) \|^2}{f (x^k) - f (x^{\star})}, 1 \} )^K
  \label{eqn:proof-3-1-15} \\
  \leq{} & ( \tfrac{1}{K} \textstyle \sum_{k = 1}^K \min \{ 2 L \hat{h}_{x^k} (P_k), 1 \}
  )^K \label{eqn:proof-3-1-16}\\
  \leq{} & ( \min \{ \tfrac{2 L}{K} \textstyle \sum_{k = 1}^K \hat{h}_{x^k} (P_k), 1
  \} )^K, \nonumber
\end{align}
where \eqref{eqn:proof-3-1-15} plugs in the definition of $\hat{h}_x$; \eqref{eqn:proof-3-1-16} uses $L$-smoothness and that $\hat{h}_x$ is nonnegative. Using \Cref{lem:regret-sublinear}, we get  $\textstyle \sum_{k = 1}^K \hat{h}_{x^k} (P_k) \leq \textstyle \sum_{k = 1}^K \hat{h}_{x^k} (P) + \rho_K$
for any $P \in \mathcal{P}$. Next, we consider the quantity $\hat{h}_x ([\nabla^2 f (x^{\star})]^{- 1})$ and deduce that
\begin{align}
  \hat{h}_x ([\nabla^2 f (x^{\star})]^{- 1}) ={} & \tfrac{f (x - [\nabla^2 f
  (x^{\star})]^{- 1} \nabla f (x)) - f (x^{\star})}{\| \nabla f (x) \|^2}
  \nonumber\\
  \leq{} & \tfrac{\frac{L}{2} \| x - [\nabla^2 f (x^{\star})]^{- 1} \nabla f
  (x) - x^{\star} \|^2}{\| x - x^{\star} \|^2} \tfrac{\| x - x^{\star}
  \|^2}{\| \nabla f (x) \|^2} \label{eqn:proof-3-1-17} \\
  \leq{} & \tfrac{L}{2 \mu^2} \tfrac{\| x - [\nabla^2 f (x^{\star})]^{- 1}
  \nabla f (x) - x^{\star} \|^2}{\| x - x^{\star} \|^2}, \label{eqn:proof-3-1-18}
\end{align}
where \eqref{eqn:proof-3-1-17} uses $L$-smoothness $f (x) - f (x^{\star}) \leq \tfrac{L}{2} \| x - x^{\star} \|^2$
and \eqref{eqn:proof-3-1-18} uses $\| \nabla f (x) \|^2 \geq \mu^2 \| x - x^{\star} \|^2$. Then,
\begin{align}
  x - [\nabla^2 f (x^{\star})]^{- 1} \nabla f (x) - x^{\star}={} & x - x^{\star} - [\nabla^2 f (x^{\star})]^{- 1} \nabla f (x) \nonumber\\
  ={} & [\nabla^2 f (x^{\star})]^{- 1} [\nabla^2 f (x^{\star}) (x - x^{\star}) -
  (\nabla f (x) - \nabla f (x^{\star}))]\nonumber
\end{align}
since $\nabla f(x^\star) = 0$. Plugging in $\nabla f (x) - \nabla f (x^{\star}) = \textstyle \int_0^1 \nabla^2 f (x^{\star} +
t (x - x^{\star})) (x - x^{\star}) \mathd t$, we deduce that
\begin{align}
\| \nabla^2 f (x^{\star}) (x - x^{\star}) - (\nabla f (x) - \nabla f
  (x^{\star}))\|={} & \|\nabla^2 f (x^{\star}) (x - x^{\star}) - \textstyle \int_0^1 \nabla^2 f (x^{\star}
  + t (x - x^{\star})) (x - x^{\star}) \mathd t \|\nonumber\\
  ={} & \|\textstyle \int_0^1 [\nabla^2 f (x^{\star}) - \nabla^2 f (x^{\star} + t (x -
  x^{\star}))] (x - x^{\star}) \mathd t \|\nonumber\\
  \leq{} & \textstyle \int_0^1 t H \| x - x^{\star} \|^2 \mathd t = \tfrac{H}{2} \| x -
  x^{\star} \|^2, \label{eqn:proof-3-1-19}
\end{align}
where \eqref{eqn:proof-3-1-19} uses $H$-Lipschitz continuity of $\nabla^2f(x)$
and, consequently,
\begin{align}
  & \| x - [\nabla^2 f (x^{\star})]^{- 1} \nabla f (x) - x^{\star} \|
  \nonumber\\
  ={} & \| [\nabla^2 f (x^{\star})]^{- 1} [\nabla^2 f (x^{\star}) (x -
  x^{\star}) - (\nabla f (x) - \nabla f (x^{\star}))] \| \leq \tfrac{H}{2 \mu}
  \| x - x^{\star} \|^2 \label{eqn:proof-3-1-auxi2}
\end{align}
since $\nabla^2 f(x^\star) \succeq \mu I$ due to strong convexity.
Plugging the relation back, we get
\begin{equation} \label{eqn:proof-2-1}
  \hat{h}_x ([\nabla^2 f (x^{\star})]^{- 1}) \leq \tfrac{L}{2 \mu^2}
   \tfrac{\tfrac{H^2}{4 \mu^2} \| x - x^{\star} \|^4}{\| x - x^{\star} \|^2} ={}
   \tfrac{H^2 \kappa}{8 \mu^3} \| x - x^{\star} \|^2.
\end{equation}
Since $[\nabla^2 f(x^\star)]^{-1} \in \Pcal$ by assumption, 
\[ \textstyle \sum_{k = 1}^K \hat{h}_{x^k} (P_k) \leq \textstyle \sum_{k = 1}^K \hat{h}_{x^k} ([\nabla^2 f
   (x^{\star})]^{- 1}) + \rho_K \leq \tfrac{H^2 \kappa}{8 \mu^3} \textstyle \sum_{k =
   1}^K \| x^k - x^{\star} \|^2 + \rho_K, \]
and we get
\[ f (x^{K + 1}) - f (x^{\star}) \leq [f(x^1) - f(x^\star) ]\big(\min \big\{\tfrac{H^2 \kappa^2}{4 \mu^2 K}
   \textstyle \sum_{k = 1}^K \| x^k - x^{\star} \|^2 + \tfrac{2L \rho_K}{K}, 1\big\} \big)^K, \]
which completes the proof.

\subsection{Proof of Lemma \ref{lem:scalmat-conv}}

For brevity let $P^{\star} = [\nabla^2 f (x^{\star})]^{- 1}$. We have, according to the
update of online gradient descent, that,
\begin{align}
 \| P_{k + 1} - P^{\star} \|_F^2  ={} & \| \Pi_{\mathcal{P}} [P_k - \eta \nabla \hat{h}_{x^k} (P_k) - P^{\star}]
  \|_F^2 \nonumber\\
  \leq{} & \| P_k - \eta \nabla \hat{h}_{x^k} (P_k) - P^{\star} \|_F^2 \nonumber\\
  ={} & \| P_k - P^{\star} \|_F^2 - 2 \eta \langle \nabla \hat{h}_{x^k} (P_k), P_k -
  P^{\star} \rangle + \eta^2 \| \nabla \hat{h}_{x^k} (P_k) \|_F^2 \nonumber\\
  \leq{} & \| P_k - P^{\star} \|_F^2 - 2 \eta [\hat{h}_{x^k} (P_k) - \hat{h}_{x^k}
  (P^{\star})] + 2 L \eta^2 [\hat{h}_{x^k} (P_k) - \inf_{P \in \Rbb^{n\times n}} \hat{h}_{x^k} (P)] \label{eqn:proof-3-6-1}\\
  ={} & \| P_k - P^{\star} \|_F^2 - 2 \eta [\hat{h}_{x^k} (P_k) - \hat{h}_{x^k} (P^{\star})]
  + 2 L \eta^2 [\hat{h}_{x^k} (P_k) - \hat{h}_{x^k} (P^{\star})] + 2L \eta^2 \hat{h}_{x^k}
  (P^{\star}) \label{eqn:proof-3-6-2}\\
  ={} & \| P_k - P^{\star} \|_F^2 - 2 \eta(1 - \eta L) [\hat{h}_{x^k} (P_k) -
  \hat{h}_{x^k} (P^{\star})] + 2L \eta^2 \hat{h}_{x^k} (P^{\star}), \label{eqn:proof-3-6-3}
\end{align}
where \eqref{eqn:proof-3-6-1} uses $L$-smoothness and $\inf_{P \in \Rbb^{n\times n}} \hat{h}_{x} (P)= 0$ for all $x \nin \Xcal^\star$; \eqref{eqn:proof-3-6-2} is a simple re-arrangement.

Next we lower bound $\hat{h}_{x^k} (P_k) - \hat{h}_{x^k} (P^{\star})$. Using strong
convexity,
\begin{align}
  & f (x^k - P_k \nabla f (x^k)) - f (x^k - P^{\star} \nabla f (x^k))
  \nonumber\\
  ={} & f (x^k - P_k \nabla f (x^k)) - f (x^{\star}) + f (x^{\star}) - f (x^k -
  P^{\star} \nabla f (x^k)) \nonumber\\
  \geq{} & \tfrac{\mu}{2} \| x^k - x^{\star} - P_k \nabla f (x^k) \|^2 + f
  (x^{\star}) - f (x^k - P^{\star} \nabla f (x^k)),\label{eqn:proof-3-6-4}
\end{align}
where \eqref{eqn:proof-3-6-4} uses $f(x) - f(x^\star) \geq \frac{\mu}{2}\|x-x ^\star\|^2$.
The first term can be bounded as follows:
\begin{align}
  & \| x^k - x^{\star} - P_k \nabla f (x^k) \|^2 \nonumber\\
  ={} & \| x^k - P^{\star} \nabla f (x^k) - x^{\star} + (P^{\star} - P_k) \nabla
  f (x^k) \|^2 \nonumber\\
  ={} & \| x^k - P^{\star} \nabla f (x^k) - x^{\star} \|^2 + 2 \langle x^k -
  P^{\star} \nabla f (x^k) - x^{\star}, (P^{\star} - P_k) \nabla f (x^k)
  \rangle + \| (P^{\star} - P_k) \nabla f (x^k) \|^2 \nonumber\\
  \geq{} & \tfrac{1}{2} \| (P^{\star} - P_k) \nabla f (x^k) \|^2 - \| x^k -
  P^{\star} \nabla f (x^k) - x^{\star} \|^2, \nonumber
\end{align}

where we use the inequality $2 \langle a, b \rangle \geq - \theta \| a \|^2 -
\theta^{- 1} \| b \|^2$ with $\theta = 2$. Plugging the relation back into \eqref{eqn:proof-3-6-4} and
dividing both sides by $\| \nabla f (x^k) \|^2$,
\begin{align}
\hat{h}_{x^k} (P_k) - \hat{h}_{x^k} (P^{\star})  ={} & \tfrac{f (x^k - P_k \nabla f (x^k)) - f(x^\star) + f(x^\star) - f (x^k - P^{\star} \nabla f
  (x^k))}{\| \nabla f (x^k) \|^2} \nonumber\\
  \geq{} & \tfrac{\mu}{4} \| (P^{\star} - P_k) \tfrac{\nabla f (x^k)}{\|
  \nabla f (x^k) \|} \|^2 - \tfrac{\mu}{2} \tfrac{\| x^k - P^{\star}
  \nabla f (x^k) - x^{\star} \|^2}{\| \nabla f (x^k) \|^2} + \tfrac{f
  (x^{\star}) - f (x^k - P^{\star} \nabla f (x^k))}{\| \nabla f (x^k) \|^2}
  \nonumber\\
  ={} & \tfrac{\mu}{4} \| (P^{\star} - P_k) \tfrac{\nabla f (x^k)}{\|
  \nabla f (x^k) \|} \|^2 - \tfrac{\mu}{2} \tfrac{\| x^k - P^{\star}
  \nabla f (x^k) - x^{\star} \|^2}{\| \nabla f (x^k) \|^2} - \hat{h}_{x^k}
  (P^{\star}) \label{eqn:proof-3-6-5} \\
  \geq{} & \tfrac{\mu}{4} \| (P^{\star} - P_k) \tfrac{\nabla f (x^k)}{\|
  \nabla f (x^k) \|} \|^2 - \tfrac{H^2}{8 \mu} \tfrac{\| x^k - x^{\star}
  \|^4}{\| \nabla f (x^k) \|^2} - \hat{h}_{x^k} (P^{\star}) \label{eqn:proof-3-6-6}\\
  \geq{} & \tfrac{\mu}{4} \| (P^{\star} - P_k) \tfrac{\nabla f (x^k)}{\|
  \nabla f (x^k) \|} \|^2 - \tfrac{H^2 \kappa}{8 \mu^3} \| x^k -
  x^{\star} \|^2 - \hat{h}_{x^k} (P^{\star}),\label{eqn:proof-3-6-7}
\end{align}
where \eqref{eqn:proof-3-6-5} uses the definition of $\hat{h}_{x^k}$; \eqref{eqn:proof-3-6-6} applies the relation $ \| x - P^\star \nabla f (x) - x^{\star} \| \leq \frac{H}{2\mu} \|x - x^\star\|^2$ from \eqref{eqn:proof-3-1-auxi2}; \eqref{eqn:proof-3-6-7} again uses the fact $\|\nabla f(x) \|^2 \geq \mu^2 \|x - x^\star\|^2$. Putting the relations back into \eqref{eqn:proof-3-6-3} and assuming $\eta \leq \frac{1}{2L}$,
\begin{align}
& \| P_{k + 1} - P^{\star} \|_F^2\nonumber \\
   \leq{} & \| P_k - P^{\star} \|_F^2 -
  \tfrac{\mu (\eta - L \eta^2)}{2} \| (P_k - P^{\star}) \tfrac{\nabla
  f (x^k)}{\| \nabla f (x^k) \|} \|^2 \nonumber\\
  & + \tfrac{H^2 \kappa (\eta - L \eta^2)}{4 \mu^3} \| x^k - x^{\star}
  \|^2 + 2 (\eta - L \eta^2) \hat{h}_{x^k} (P^{\star}) + 2 L \eta^2 \hat{h}_{x^k}
  (P^{\star}) \nonumber\\
  ={} & \| P_k - P^{\star} \|_F^2 - \tfrac{\mu (\eta - L \eta^2)}{2} \|
  (P_k - P^{\star}) \tfrac{\nabla f (x^k)}{\| \nabla f (x^k) \|} \|^2
  +  \tfrac{H^2 \kappa (\eta - L \eta^2)}{4 \mu^3} \| x^k - x^{\star}
  \|^2 + 2 \eta \hat{h}_{x^k} (P^{\star}) \nonumber\\
  \leq{} & \| P_k - P^{\star} \|_F^2 - \tfrac{\mu (\eta - L \eta^2)}{2}
  \| (P_k - P^{\star}) \tfrac{\nabla f (x^k)}{\| \nabla f (x^k) \|}
  \|^2 + \tfrac{H^2 \kappa (\eta - L \eta^2)}{4 \mu^3} \| x^k - x^{\star}
  \|^2 + 2 \eta \tfrac{H^2 \kappa}{8 \mu^3} \| x^k - x^{\star} \|^2 \label{eqn:proof-3-6-8} \\
  ={} & \| P_k - P^{\star} \|_F^2 - \tfrac{\mu (\eta - L \eta^2)}{2} \|
  (P_k - P^{\star}) \tfrac{\nabla f (x^k)}{\| \nabla f (x^k) \|} \|^2 +
  (2\eta - L \eta^2) \tfrac{H^2 \kappa}{4 \mu^3}\| x^k - x^{\star} \|^2, \nonumber
\end{align}
where \eqref{eqn:proof-3-6-8} uses the relation \eqref{eqn:proof-2-1} and this completes the proof.

\subsection{Proof of Theorem \ref{thm:scal-mat-conv}} \label{app:thm-pf-scal-mat-conv}

The proof of \Cref{thm:scal-mat-conv} relies on the following auxiliary results.

\begin{lem} \label{lem:scal-mat-conv-aux} Under \ref{A1} to \ref{A3}, $h_x (P) - \inf_{Q \in \mathbb{R}^{n \times n}} h_x (Q) \leq
  \tfrac{1}{2 \mu} (L D + 1)^2$.
\end{lem}

\begin{proof}
  Note that $h_x (P) = \tfrac{f (x - P \nabla f (x)) - f (x)}{\| \nabla f (x)
  \|^2} \geq \tfrac{f (x^{\star}) - f (x)}{\| \nabla f (x) \|^2}$ for all $P
  \in \mathcal{P}$, we deduce that
  \begin{align}
    h_x (P) - \inf_{Q \in \mathbb{R}^{n \times n}} h_x (Q) \leq{} & \tfrac{f (x
    - P \nabla f (x)) - f (x)}{\| \nabla f (x) \|^2} - \tfrac{f (x^{\star}) -
    f (x)}{\| \nabla f (x) \|^2} \label{eqn:proof-scalmatconv-aux-1}\\
    ={} & \tfrac{f (x - P \nabla f (x)) - f (x^{\star})}{\| \nabla f (x) \|^2}
    \nonumber\\
    \leq{} & \tfrac{1}{2 \mu} \tfrac{\| \nabla f (x - P \nabla f (x)) \|^2}{\|
    \nabla f (x) \|^2} \label{eqn:proof-scalmatconv-aux-2} \\
    \leq{} & \tfrac{1}{2 \mu} \tfrac{[\| \nabla f (x) \| + \| P \| \cdot \|
    \nabla f (x) \|]^2}{\| \nabla f (x) \|^2} \label{eqn:proof-scalmatconv-aux-3} \\
    \leq{} & \tfrac{1}{2 \mu} (L D + 1)^2,  \label{eqn:proof-scalmatconv-aux-4}
  \end{align}
  
where \eqref{eqn:proof-scalmatconv-aux-1} applies $h_x(P) \geq \tfrac{f (x^{\star}) - f (x)}{\| \nabla f (x) \|^2}$; \eqref{eqn:proof-scalmatconv-aux-2} uses $f(x) - f(x^\star) \leq \frac{1}{2\mu} \|\nabla f(x)\|^2$; \eqref{eqn:proof-scalmatconv-aux-3} uses $L$-smoothness and \eqref{eqn:proof-scalmatconv-aux-4} uses $\|P\| \leq D$. 
\end{proof}

Then we show that {\hdm} converges even when $\eta$ is a constant that does not depend on $K$. 

\begin{lem} \label{lem:scal-mat-conv-aux-2} Under \ref{A1} to \ref{A3}, \Cref{alg:hdm} with $\eta_k \equiv \eta \in (0, \frac{1}{2 L (L D + 1)^2 \kappa}]$ satisfies
  \begin{itemize}[leftmargin=15pt]
    \item $\lim_{k \rightarrow \infty} \| x^k - x^{\star} \| = 0$. 
    
    \item $\lim_{K \rightarrow \infty} \sum_{k = 1}^K \| x^k - x^{\star} \|^2
    < \infty$.
  \end{itemize}
\end{lem}

\begin{proof}
	Using the online gradient descent
update, we have
\begin{align}
  \| P_{k + 1} - P \|_F^2 \leq{} & \| P_k - \eta \nabla h_{x^k} (P_k) - P \|_F^2
  \nonumber\\
  ={} & \| P_k - P \|_F^2 - 2 \eta \langle \nabla h_{x^k} (P_k), P_k - P \rangle
  + \eta^2 \| \nabla h_{x^k} (P_k) \|_F^2 \nonumber\\
  \leq{} & \| P_k - P \|_F^2 - 2 \eta [h_{x^k} (P_k) - h_{x^k} (P)] + 2 L \eta^2
  [h_{x^k} (P_k) - \inf_{P \in \mathbb{R}^{n \times n}} h_{x^k} (P)]
  \label{eqn:proof-scalmatconv-aux-5}
 \\
  ={} & \| P_k - P \|_F^2 - 2 \eta h_{x^k} (P_k) + 2 \eta h_{x^k} (P) + 2 L
  \eta^2 [h_{x^k} (P_k) - \inf_{P \in \mathbb{R}^{n \times n}} h_{x^k} (P)],
  \nonumber
\end{align}

where \eqref{eqn:proof-scalmatconv-aux-5} follows from convexity $h_{x^k} (P) \geq h_{x^k}
(P_k) + \langle \nabla h_{x^k} (P_k), P - P_k \rangle$ and $L$-smoothness of
$h_x (P)$. Next, we invoke the upperbound on $h_{x^k} (P_k) - \inf_{Q \in
\mathbb{R}^{n \times n}} h_{x^k} (Q)$ from \Cref{lem:scal-mat-conv-aux}:
\[ {2 L \eta^2 [h_{x^k} (P_k) - \inf_{P \in \mathbb{R}^{n \times
   n}} h_{x^k} (P)]} \leq \tfrac{2 L}{2 \mu} (L D + 1)^2 \eta^2 = \kappa (L D
   + 1)^2 \eta^2 . \]
and deduce that
\begin{align}
  2 \eta h_{x^k} (P_k) \leq{} & 2 \eta h_{x^k} (P) + \| P_k - P \|_F^2 - \| P_{k
  + 1} - P \|_F^2 + {2 L \eta^2 [h_{x^k} (P_k) - \inf_{P \in
  \mathbb{R}^{n \times n}} h_{x^k} (P)]} \nonumber\\
  \leq{} & 2 \eta h_{x^k} (P) + \| P_k - P \|_F^2 - \| P_{k + 1} - P \|_F^2 +
  \eta^2 \kappa (L D + 1)^2 . \nonumber
\end{align}

Next, we divide both sides of the inequality by $2 \eta$ and
\[ h_{x^k} (P_k) \leq h_{x^k} (P) + \tfrac{\| P_k - P \|_F^2 - \| P_{k + 1} -
   P \|_F^2}{2 \eta} + \tfrac{\eta \kappa (L D + 1)^2}{2} . \]
Telescoping the relation and using $\tmop{diam} (\mathcal{P}) \leq D$, we get
\[\textstyle \sum_{k = 1}^K h_{x^k} (P_k) \leq \sum_{k = 1}^K h_{x^k} (P) +
   \tfrac{D^2}{2 \eta} + \tfrac{\eta \kappa (L D + 1)^2}{2} K \]
Taking $P = (1 / L) I$ and taking average, $\sum_{k = 1}^K h_{x^k} (P) \leq
{- \tfrac{1}{2 L}} K$ and
\begin{align}
\textstyle  \tfrac{1}{K} \sum_{k = 1}^K h_{x^k} (P_k) \leq{} & {-
  \tfrac{1}{2 L}} + \tfrac{D^2}{2 \eta K} + \tfrac{\eta \kappa (L D + 1)^2}{2}
    ={} - \tfrac{1}{4 L} + \tfrac{D^2}{2 \eta K} + \tfrac{\eta \kappa (L D +
  1)^2}{2} - \tfrac{1}{4 L} \nonumber
\end{align}

With $\eta \leq \frac{1}{2 L (L D + 1)^2 \kappa}$, we have $\tfrac{\eta \kappa
(L D + 1)^2}{2} - \tfrac{1}{4 L} \leq 0$ and
\[ \textstyle \tfrac{1}{K} \sum_{k = 1}^K h_{x^k} (P_k) \leq - \tfrac{1}{4 L} +
   \tfrac{D^2 L (L D + 1)^2 \kappa}{K} . \]
Using the reduction \Cref{lem:hypergrad-to-online}, we get, for any $k \geq 1$ (since $\eta$ does not depend
on the iteration number),
\[ f (x^{k + 1}) - f (x^{\star}) \leq [f (x^1) - f (x^{\star})] ( 1 - 2
   \mu \max \{ \tfrac{1}{4 L} - \tfrac{D^2 L (L D + 1)^2 \kappa}{k}, 0
   \} )^k \]
and there exists some $K_0$ such that for all $k \geq K_0$, that $[f (x^k) - f (x^{\star})] ( 1 -
\tfrac{1}{4 \kappa} )^k \leq  [f (x^1) - f (x^{\star})]$ since
\[ \lim_{k \rightarrow \infty} 1 - 2 \mu \max \{ \tfrac{1}{4 L} -
   \tfrac{2 D^2 L (L D + 1)^2 \kappa}{k}, 0 \} = 1 - \tfrac{1}{2 \kappa}
   < 1 - \tfrac{1}{4 \kappa} . \]
   This proves the first relation $\lim_{k \rightarrow \infty} \| x^k - x^{\star} \| = 0$ since $\|x - x^\star\|^2 \leq \frac{2}{\mu} [f(x) - f(x^\star)] $ and the second relation follows directly from    
\begin{align}
\textstyle \sum_{k=1}^\infty \|x^k - x^\star\|^2 ={} & \textstyle\sum_{k=1}^{K_0} \|x^k - x^\star\|^2 + \sum_{k=K_0}^{\infty} \|x^k - x^\star\|^2 \\
 ={} & \textstyle\sum_{k=1}^{K_0} \|x^k - x^\star\|^2 + \sum_{k=K_0}^{\infty} \frac{2}{\mu}[f(x^1)-f(x^\star)](1-\frac{1}{4\kappa})^{-k} < \infty.
\end{align}
\end{proof}

Now we are ready to prove \Cref{thm:scal-mat-conv}, and we start by stating the precise definition of a uniformly independent sequence.

\begin{definition}[Uniformly linearly indepdendent sequence {\cite{conn1991convergence}}]
  A sequence of unit-norm vectors $\{ g^k \}, g^k \in \mathbb{R}^n, \| g^k \|
  = 1$ is uniformly linearly independent if there exists a constant $c > 0, K_0 \geq 0$ and $m \geq n$ such that for each $k \geq K_0$, one can
  choose $n$ distinct indices
  \[ k \leq k_1 < \cdots < k_n \leq k + m \]
  with $\sigma_{\min} ([g^{k_1}, \ldots, g^{k_n}]) \geq c$.
\end{definition}

We prove by contradiction.
For brevity we denote $g^k \assign \tfrac{\nabla f (x^k)}{\| \nabla f (x^k)
\|}$ and $e_k \assign \| P_k - P^{\star} \|_F^2$. Recall that
$P^{\star} = [\nabla^2 f (x^{\star})]^{- 1}$. First, using \Cref{lem:scal-mat-conv-aux-2}, for any $\varepsilon > 0$, there exists some index $K_1$ such that for
all $k \geq K_1$ we have $\| x^k - x^{\star} \|^2 \leq \varepsilon$ and that
$\textstyle \sum_{k = 1}^{\infty} \| x^k - x^{\star} \|^2$ is bounded. 
Then we show that $\lim_{k \rightarrow \infty} \| \nabla h_{x^k} (P_k) \|_F =
0$ using \eqref{eqn:proof-3-6-3}: after re-arrangement, for any $K \geq 1$,
\begin{align}
  \textstyle \sum_{k = 1}^K \hat{h}_{x^k} (P_k) \leq{} & \tfrac{2 \eta}{2 \eta (1 - \eta
  L)} \textstyle \sum_{k = 1}^K \hat{h}_{x^k} (P^{\star}) + \tfrac{1}{2 \eta (1 - \eta
  L)} \| P_1 - P^{\star} \|_F^2 \nonumber\\
  \leq{} & \tfrac{2 \eta}{2 \eta (1 - \eta L)} \tfrac{H^2 \kappa}{8 \mu^3}
  \textstyle \sum_{k = 1}^K \| x^k - x^{\star} \|^2 + \tfrac{1}{2 \eta (1 - \eta L)} \|
  P_1 - P^{\star} \|_F^2 . \label{eqn:proof-scalconv-auxi}\\
  \leq{} & \tfrac{2 \eta}{2 \eta (1 - \eta L)} \tfrac{H^2 \kappa}{8 \mu^3}
  \textstyle \sum_{k = 1}^{\infty} \| x^k - x^{\star} \|^2 + \tfrac{1}{2 \eta (1 - \eta
  L)} \| P_1 - P^{\star} \|_F^2, \nonumber
\end{align}
where \eqref{eqn:proof-scalconv-auxi} applies \eqref{eqn:proof-2-1}. Since $\textstyle \sum_{k = 1}^{\infty} \| x^k - x^{\star} \|^2$ is bounded and
$\hat{h}_x (P)$ is nonnegative, we must have $\lim_{k \rightarrow \infty}
\hat{h}_{x^k} (P_k) = 0$. Further notice that $\| \nabla \hat{h}_{x^k} (P_k) \|_F^2 \leq 2 L
\hat{h}_{x^k} (P_k)$, it implies $\lim_{k \rightarrow \infty} \sum_{k=1}^K \| \nabla
h_{x^k} (P_k) \|_F^2 < \infty$, giving $\lim_{k \rightarrow \infty} \| \nabla
h_{x^k} (P_k) \|_F = 0$ and $\lim_{k \rightarrow \infty} P_k = \bar{P}$ also
exists. Now suppose by contradiction that $\| \bar{P} - P^{\star} \|_F = \theta
> 0$. Then there exists some $K_2 > 0$ such that for all $k \geq K_2$, $\| P_k
- \bar{P} \|_F \leq \varepsilon$. For $k \geq \max \{ K_0, K_1, K_2 \} + 1$, we invoke \Cref{lem:scalmat-conv} with $\eta \in (0, \frac{1}{2L}]$ to get
\begin{align}
  \| P_{k + 1} - P^{\star} \|_F^2 \leq{} & \| P_k - P^{\star} \|_F^2 - \alpha_1
  \| (P_k - P^{\star}) g^k \|^2 + \alpha_2 \varepsilon \nonumber\\
  ={} & \| P_k - P^{\star} \|_F^2 - \alpha_1 \| (P_k - \bar{P} + \bar{P} -
  P^{\star}) g^k \|^2 + \alpha_2 \varepsilon \nonumber\\
  \leq{} & \| P_k - P^{\star} \|_F^2 - \tfrac{\alpha_1}{2} \| (\bar{P} -
  P^{\star}) g^k \|^2 + 3 \alpha_1 \| (P_k - \bar{P}) g^k \|^2 + \alpha_2
  \varepsilon \nonumber\\
  \leq{} & \| P_k - P^{\star} \|_F^2 - \tfrac{\alpha_1}{2} \| (\bar{P} -
  P^{\star}) g^k \|^2 + 3 \alpha_1 \varepsilon^2 + \alpha_2 \varepsilon
  \label{eqn:proof-thm-scal-conv-0} \\
  ={} & \| P_k - P^{\star} \|_F^2 - \tfrac{\alpha_1}{2} \tmop{tr} (g^k
  (g^k)^{\top}, (\bar{P} - P^{\star})^{\top} (\bar{P} - P^{\star})) + 3
  \alpha_1 \varepsilon^2 + \alpha_2 \varepsilon ,\label{eqn:proof-thm-scal-conv-1}
\end{align}
where $\alpha_1 = \tfrac{\mu (\eta - L \eta^2)}{2} > 0$, $\alpha_2 = \frac{1}{4}
(2\eta - L \eta^2) H^2 \kappa \mu^{- 3}$, and \eqref{eqn:proof-thm-scal-conv-0} uses the fact that $\| P_k - P^{\star} \|_F\leq \varepsilon$.\\
Telescoping \eqref{eqn:proof-thm-scal-conv-1} for the next $m + 1$ iterations, we deduce that
\begin{align}
  e_{k + m + 1} ={} & \| P_{k + m + 1} - P^{\star} \|_F^2 \nonumber\\
  \leq{} & \| P_k - P^{\star} \|_F^2 - \tfrac{\alpha_1}{2} \textstyle \sum_{j = 0}^m
  \tmop{tr} (g^{k + j} (g^{k + j})^{\top}, (\bar{P} - P^{\star})^{\top} (\bar{P} -
  P^{\star})) + (3 \alpha_1 \varepsilon^2 + \alpha_2 \varepsilon) (m + 1)
  \nonumber\\
  ={} & e_k - \tfrac{\alpha_1}{2} \tmop{tr} ( \textstyle \sum_{j = 0}^m g^{k + j}
  (g^{k + j})^{\top}, (\bar{P} - P^{\star})^{\top} (\bar{P} - P^{\star})) +
  (3 \alpha_1 \varepsilon^2 + \alpha_2 \varepsilon) (m + 1) \nonumber
\end{align}

and using the independent sequence assumption, we can pick $k_1, \ldots, k_n$
such that
\[ \sigma_{\min} ([g^{k_1}, \ldots, g^{k_n}]) \geq c \]
and $\sum_{j = 0}^m g^{k + j} (g^{k + j})^{\top} \succeq \sum_{i = 1}^n
g^{k_i} (g^{k_i})^{\top} \succeq c^2 I$. Hence
\begin{align}
 & \textstyle \tmop{tr} (\sum_{j = 0}^m g^{k + j} (g^{k + j})^{\top}, (\bar{P} -
  P^{\star})^{\top} (\bar{P} - P^{\star})) \geq{}  c^2 \tmop{tr} ((\bar{P} - P^{\star})^{\top} (\bar{P} -
  P^{\star})) 
  ={}  c^2 \| \bar{P} - P^{\star} \|_F^2 = c^2 \theta^2 \nonumber
\end{align}
and $e_{k + m + 1} \leq e_k - \frac{\alpha_1 c^2 \theta^2}{2} + (3 \alpha_1 \varepsilon^2 + \alpha_2\varepsilon) (m + 1)$. Since $\varepsilon$ is arbitrary, we can repeat the argument till $e_{k + m +1} <0 $, which leads to contradiction unless $\theta = 0$. This completes the proof.
\section{Proof of Results in Section \ref{sec:momentum}}
 \label{app:proof-momentum}

\subsection{{\hdm} + Heavy-ball Momentum ({\hdmhb})} \label{app:heavy-ball}

\Cref{alg:ospolyak} uses the following \emph{heavy-ball feedback function} to guide the online learning for $(P_k, B_k)$:
\begin{equation*}
  h_{x, x^-} (P, B) \assign 
  \tfrac{\psi(x^{+}, x) - \psi(x, x^{-})}{\| \nabla f (x) \|^2 + \frac{\tau}{2} \| x - x^- \|^2}
  = \tfrac{[f (x^+) + \frac{\omega}{2} \| x^+ - x \|^2] - [f (x) + \frac{\omega}{2} \| x - x^- \|^2]}{\| \nabla f (x) \|^2 + \frac{\tau}{2} \| x - x^- \|^2}, 
\end{equation*}
where $\omega > 0$, $ \tau > 0$, $x^{+} = x - P \nabla f (x) + B (x - x^{-})$, and $\psi(x, x^-) = f (x) + \tfrac{\omega}{2} \| x - x^- \|^2$.
To show that online learning can be applied to $h_{x, x^-} (P, B)$ with regret guarantees, we need to verify the convexity and Lipschitz continuity of $h_{x, x^-} (P, B)$ with respect to the norm defined by
\begin{equation} \label{eqn:PB-norm}
  \|(P, B)\| \assign \sqrt{\|P\|_F^2 + \|B\|_F^2}.
\end{equation}

\begin{lem} \label{lem:heavyball-feedback-property}
Under \ref{A1}, \ref{A2}, and \ref{ABcal}, the heavy-ball feedback function $h_{x, x^{-}}(P, B)$ is jointly convex in $(P, B)$ and $c$-Lipschitz with respect to the norm defined in \eqref{eqn:PB-norm}, where $c \assign \sqrt{2} (1+\tfrac{2}{\tau}) [1 + 2(1+\tfrac{2}{\tau}) D (L+\omega)]$.
\end{lem}
\begin{proof}
Denote $x^{+}(P, B) \assign x - P \nabla f(x) + B (x - x^{-})$. 
Recall that the feedback function is
\begin{equation*}
h_{x, x^{-}}(P, B) = \tfrac{[f (x^+(P, \beta)) + \frac{\omega}{2} \| x^+(P, \beta) - x \|^2] - [f (x) + \frac{\omega}{2} \| x - x^- \|^2]}{\| \nabla f (x) \|^2 + \frac{\tau}{2} \| x - x^- \|^2}.
\end{equation*}
Since $x^{+}(P, B)$ is affine in $(P, B)$ and $f$ is convex, the term $f (x^+(P, \beta)) + \frac{\omega}{2} \| x^+(P, \beta) - x \|^2$ is jointly convex as a function of $(P, B)$. The other terms in the feedback function $h_{x, x^{-}}(P, B)$ are constants, so $h_{x, x^{-}}(P, B)$ is also jointly convex in $(P, B)$.\\

To prove the Lipschitz continuity of $h_{x, x^{-}}(P, B)$, it suffices to show that the gradients of $h_{x, x^{-}}(P, B)$ are bounded. 
The gradients of $h_{x, x^{-}}(P, B)$ with respect to $P$ and $B$ are
\begin{align*}
\nabla_P h_{x, x^{-}}(P, B) & = \tfrac{[- \nabla f(x^{+}(P, B)) + \omega P \nabla f(x) - \omega B (x - x^{-})] \nabla f(x)^{\top}}{\| \nabla f (x) \|^2 + \frac{\tau}{2} \| x - x^- \|^2} , \\
\nabla_B h_{x, x^{-}}(P, B) & = \tfrac{[\nabla f(x^{+}(P, B)) - \omega P \nabla f(x) + \omega B (x - x^{-})](x - x^{-})^{\top}}{\| \nabla f (x) \|^2 + \frac{\tau}{2} \| x - x^- \|^2}.
\end{align*}
Using the fact $\|a b^{\top}\|_F = \|a\| \cdot \|b\|$, the gradients have norms
\begin{align}
\|\nabla_P h_{x, x^{-}}(P, B)\|_F & = \tfrac{\| \nabla f(x^{+}(P, B)) - \omega P \nabla f(x) + \omega B (x - x^{-})\| \|\nabla f(x)\|}{\| \nabla f (x) \|^2 + \frac{\tau}{2} \| x - x^- \|^2}, \label{eqn:proof-4-1-1} \\
\|\nabla_B h_{x, x^{-}}(P, B)\|_F & = \tfrac{\| \nabla f(x^{+}(P, B)) - \omega P \nabla f(x) + \omega B (x - x^{-})\|\|x - x^{-}\|}{\| \nabla f (x) \|^2 + \frac{\tau}{2} \| x - x^- \|^2}. \label{eqn:proof-4-1-2}
\end{align}
Using \ref{A1}, we have the Lipschitz continuity of $\nabla f(x)$ and thus
\begin{align}
&\| \nabla f(x^{+}(P, B)) - \omega P \nabla f(x) + \omega B (x - x^{-})\| \nonumber \\
\leq{} & \|\nabla f(x^{+}(P, B)) - \nabla f(x)\| + \| (I - \omega P) \nabla f(x)\| + \omega \|B\| \|x - x^{-}\| \nonumber \\
\leq{} & L \| P \nabla f(x) - B (x - x^{-}) \| + (1 + \omega \|P\| ) \|\nabla f(x)\| + \omega \|B\| \|x - x^{-}\| \nonumber\\
\leq{} & L D (\| \nabla f(x) \| + \| x - x^{-} \|) + (1 + \omega D ) \|\nabla f(x)\| + \omega D \|x - x^{-}\| \nonumber \\
={} &(1 + LD + \omega D ) \|\nabla f(x)\| + (\omega + L) D \|x - x^{-}\|. \label{eqn:proof-4-1-3}
\end{align}
Now, we bound the norms in \eqref{eqn:proof-4-1-1}--\eqref{eqn:proof-4-1-2} by the case analysis.

\paragraph{Case 1.} If $\frac{\tau}{2} \| x - x^- \|^2 \leq \| \nabla f (x) \|^2$, then together with \eqref{eqn:proof-4-1-3}, we have
\begin{align*}
  \max \{ \|\nabla_P h_{x, x^{-}}(P, B)\|_F, \|\nabla_B h_{x, x^{-}}(P, B)\|_F \}  
  &\leq \tfrac{[(1 + LD + \omega D ) \|\nabla f(x)\| + (\omega + L) D \|x - x^{-}\|] \max\{ \sqrt{2\tau^{-1}}, 1\} \|\nabla f(x)\|}{\| \nabla f (x) \|^2} \\
  &\leq \max\{ \sqrt{2\tau^{-1}}, 1\} [(1 + LD + \omega D ) + \tfrac{\sqrt{2} D (\omega + L)}{\sqrt{\tau}}] \\
  &= \max\{ \sqrt{2\tau^{-1}}, 1\} (1 + D(L+\omega)(1+\sqrt{2\tau^{-1}})).
\end{align*}

\paragraph{Case 2.} If $\frac{\tau}{2} \| x - x^- \|^2 \geq \| \nabla f (x) \|^2$, then
\begin{align*}
  \max \{ \|\nabla_P h_{x, x^{-}}(P, B)\|_F, \|\nabla_B h_{x, x^{-}}(P, B)\|_F \}  
  &\leq \tfrac{[(1 + LD + \omega D ) \|\nabla f(x)\| + (\omega + L) D \|x - x^{-}\|] \max\{ \sqrt{\frac{\tau}{2}}, 1\} \|x - x^-\|}{\frac{\tau}{2}\|x - x^- \|^2} \\
  &\leq \max\{ \sqrt{\tau/2}, 1\}[\tfrac{ \sqrt{2} (1 + LD + \omega D )}{\sqrt{\tau}} + \tfrac{2 D (\omega + L)}{\tau}] \\
  &= \sqrt{2\tau^{-1}} \max\{ \sqrt{2\tau^{-1}}, 1\} (1 + D(L+\omega)(1+\sqrt{2\tau^{-1}})).
\end{align*}

Combining the two cases, we have
\begin{align*}
  \max \{ \|\nabla_P h_{x, x^{-}}(P, B)\|_F, \|\nabla_B h_{x, x^{-}}(P, B)\|_F \}  
  &\leq \max\{ \tfrac{2}{\tau}, 1\} (1 + D(L+\omega)(1+\sqrt{2\tau^{-1}})) \\
  &\leq (1+\tfrac{2}{\tau}) [1 + 2(1+\tfrac{2}{\tau}) D (L+\omega)].
\end{align*}
Then the gradient of $h_{x, x^{-}}(P, B)$ under the norm defined in \eqref{eqn:PB-norm} is bounded by the constant $c \assign \sqrt{2} (1+\tfrac{2}{\tau}) [1 + 2(1+\tfrac{2}{\tau}) D (L+\omega)]$.
\end{proof}

The next lemma bounds the potential at the last iterate $x^{K+1}$ from \Cref{alg:ospolyak} in terms of the sum of feedback functions $h_{x^k, x^{k-1}}(P_k, B_k)$.

\begin{lem} \label{lem:heavyball-conversion}
  The sequence $\{x^k\}$ generated from \Cref{alg:ospolyak} satisfies
  \begin{equation} \label{eqn:potential-bound}
    f(x^{K+1}) - f(x^{\star}) + \tfrac{\omega}{2} \|x^{K+1} - x^K\|^2 \leq \tfrac{f (x^1) - f (x^{\star})}{1 + \sum_{k = 1}^K
    \max \{ - h_{x^k, x^{k - 1}} (P_k, B_k), 0 \} V},
  \end{equation}
  where $V \assign \min \big\{\tfrac{f (x^1) - f (x^{\star})}{4 \Delta^2}, \tfrac{\tau}{4 \omega} \big\}$ and $\Delta \assign \max_{x \in \Lcal_{f(x^1)}} \min_{x^{\star} \in \mathcal{X}^{\star}} \| x - x^{\star} \|$.
\end{lem}

\begin{proof}
  The null step guarantees
\[ \tfrac{\psi (x^{k + 1}, x^k) - \psi (x^k, x^{k - 1})}{\| \nabla f
   (x^k) \|^2 + \frac{\tau}{2} \| x^k - x^{k - 1} \|^2} = \min \{ h_{x^k, x^{k
   - 1}} (P_k, B_k), 0 \}. \]
Using the initial condition $x^1 = x^0$, we have
\begin{align}
\psi (x^{K + 1}, x^K) - f (x^{\star})   ={} & \frac{1}{\frac{1}{\psi (x^{K + 1}, x^K) - f (x^{\star})}} \nonumber\\
  ={} & \frac{1}{\sum_{k = 1}^K \frac{1}{\psi (x^{k + 1}, x^k) - f
  (x^{\star})} - \frac{1}{\psi (x^k, x^{k - 1}) - f (x^{\star})} +
  \frac{1}{\psi (x^1, x^0) - f (x^{\star})}} \nonumber\\
  ={} & \frac{1}{\sum_{k = 1}^K \frac{\psi (x^k, x^{k - 1}) - \psi (x^{k +
  1}, x^k)}{[\psi (x^{k + 1}, x^k) - f (x^{\star})] [\psi (x^k, x^{k -
  1}) - f (x^{\star})]} + \frac{1}{\psi (x^1, x^0) - f (x^{\star})}}
  \nonumber\\
  ={} & \frac{1}{\sum_{k = 1}^K \frac{\max \{ - h_{x^k, x^{k - 1}} (P_k,
  B_k), 0 \} [ \| \nabla f (x^k) \|^2 + \frac{\tau}{2} \| x^k - x^{k
  - 1} \|^2 ]}{[\psi (x^{k + 1}, x^k) - f (x^{\star})] [\psi (x^k,
  x^{k - 1}) - f (x^{\star})]} + \frac{1}{f (x^1) - f (x^{\star})}} \label{eqn:proof-4-2-1}.
\end{align}

Then, by monotonicity, $\tfrac{\| \nabla f (x^k) \|^2 + \frac{\tau}{2} \| x^k
- x^{k - 1} \|^2}{[\psi (x^{k + 1}, x^k) - f (x^{\star})] [\psi (x^k,
x^{k - 1}) - f (x^{\star})]} \geq \tfrac{\| \nabla f (x^k) \|^2 +
\frac{\tau}{2} \| x^k - x^{k - 1} \|^2}{[\psi (x^k, x^{k - 1}) - f
(x^{\star})]^2}$.\\

Now we do case analysis to bound
\[ \tfrac{\| \nabla f (x^k) \|^2 + \frac{\tau}{2} \| x^k - x^{k - 1}
   \|^2}{[\psi (x^k, x^{k - 1}) - f (x^{\star})]^2} = \tfrac{\| \nabla f
   (x^k) \|^2 + \frac{\tau}{2} \| x^k - x^{k - 1} \|^2}{[ f (x^k) +
   \frac{\omega}{2} \| x^k - x^{k - 1} \|^2 - f (x^{\star}) ]^2} \]
\paragraph{Case 1.} If $\frac{\omega}{2} \| x^k - x^{k - 1} \|^2 \leq f (x^k) -
f (x^{\star})$, then
\begin{equation*}
\tfrac{\| \nabla f (x^k) \|^2 + \frac{\tau}{2} \| x^k - x^{k - 1} \|^2}{[ f (x^k) + \frac{\omega}{2} \| x^k - x^{k - 1} \|^2 - f(x^{\star}) ]^2} \geq \tfrac{\| \nabla f (x^k) \|^2}{4 [f (x^k) - f (x^{\star})]^2} \geq \tfrac{1}{4 \Delta^2},
\end{equation*}
where $\Delta \assign \max_{x \in \Lcal_{f(x^1)}} \min_{x^{\star} \in \mathcal{X}^{\star}} \| x - x^{\star} \|$.

\paragraph{Case 2.} If $\frac{\omega}{2} \| x^k - x^{k - 1} \|^2 \geq f (x^k) -
f (x^{\star})$, then $\frac{\tau}{2} \| x^k - x^{k - 1} \|^2 \geq
\frac{\tau}{\omega} [f (x^k) - f (x^{\star})]$ and
\[ \tfrac{\| \nabla f (x^k) \|^2 + \frac{\tau}{2} \| x^k - x^{k - 1}
   \|^2}{[ f (x^k) + \frac{\omega}{2} \| x^k - x^{k - 1} \|^2 - f
   (x^{\star}) ]^2} \geq \tfrac{\frac{\tau}{2} \| x^k - x^{k - 1}
   \|^2}{\omega^2 \| x^k - x^{k - 1} \|^4} = \tfrac{\tau}{2 \omega^2}
   \tfrac{1}{\| x^k - x^{k - 1} \|^2} \geq \tfrac{\tau}{4 \omega} \tfrac{1}{f
   (x^1) - f (x^{\star})} . \]
since $\frac{\omega}{2} \| x^k - x^{k - 1} \|^2 \leq \psi (x^k, x^{k - 1})
- f (x^{\star}) \leq \psi (x^1, x^0) - f (x^{\star}) = f (x^1) - f
(x^{\star})$.\\

In both cases, we have $\tfrac{\| \nabla f (x^k) \|^2 + \frac{\tau}{2} \| x^k
- x^{k - 1} \|^2}{[\psi (x^k, x^{k - 1}) - f (x^{\star})]^2} \geq \min
\{ \tfrac{1}{4 \Delta^2}, \frac{\tau}{4 \omega} \tfrac{1}{f (x^1) - f (x^{\star})} \} = \tfrac{V}{f(x^1) - f(x^\star)}$, where the constant $V$ is defined in the lemma.
Finally, plugging in the definition of $\psi$, \eqref{eqn:proof-4-2-1} gives
\[f(x^{K+1}) - f(x^{\star}) + \tfrac{\omega}{2} \|x^{K+1} - x^K\|^2 \leq \tfrac{f (x^1) - f (x^{\star})}{1 + \sum_{k = 1}^K
\max \{ - h_{x^k, x^{k - 1}} (P_k, B_k), 0 \} V}.\]
\end{proof}
The next lemma shows that there exist hindsight $\bar{P}, \bar{B}$ such that $h_{x, x^-} (\bar{P}, \bar{B}) \leq - \theta < 0$ for some $\theta$.

\begin{lem} \label{lem:heavyball-hindsight}
Let $\omega = 3 L$ and $\tau = 16 L^2$. Then for any $x, x^- \nin  \mathcal{X}^{\star}$, we have $h_{x, x^-} ( \tfrac{1}{4 L} I, \tfrac{1}{2} I ) \leq - \tfrac{1}{8 L}$. 
In particular, if $\tfrac{1}{4L}I \in \Pcal$, $\tfrac{1}{2}I \in \Bcal$, and $\{x^k\}_{k=1}^K \cap \mathcal{X}^{\star} = \varnothing$, then 
\begin{equation*}
  \gamma_K^{\star} \assign - \min_{(P, B) \in \mathcal{P} \times \mathcal{B}} \tfrac{1}{K} \textstyle \sum_{k=1}^K h_{x^k, x^{k-1}}(P, B) \geq \tfrac{1}{8 L}.
\end{equation*}
\end{lem}
\begin{proof}
When $P = \alpha I$ and $B = \beta I$ for some $\alpha, \beta > 0$, the classical analysis for the heavy-ball momentum \cite{danilova2020non} gives
\begin{equation*}
f (x^+) + \tfrac{1 - \alpha L}{2 \alpha} \| x^+ - x \|^2 \leq f (x) + \tfrac{\beta^2}{2 \alpha} \| x - x^- \|^2 - \tfrac{\alpha}{2} \| \nabla f (x) \|^2 .
\end{equation*}
Let $\alpha = \tfrac{1}{4 L}$ and $\beta = \tfrac{1}{2}$, we have
\begin{align}
  f (x^+) + \tfrac{3 L}{2} \| x^+ - x \|^2 
  \leq{}& f (x) + \tfrac{L}{2} \| x -
  x^- \|^2 - \tfrac{1}{8 L} \| \nabla f (x) \|^2 \nonumber\\
  ={}& f (x) + \tfrac{3 L}{2} \| x - x^- \|^2 - \tfrac{1}{8 L} \| \nabla f (x)
  \|^2 - L \| x - x^- \|^2 \nonumber\\
  ={}& f (x) + \tfrac{3 L}{2} \| x - x^- \|^2 - \tfrac{1}{8 L} [\| \nabla f (x)
  \|^2 + 8 L^2 \| x - x^- \|^2] \nonumber
\end{align}

and re-arranging the terms, we get
\[ \tfrac{f (x^+) + \frac{3 L}{2} \| x^+ - x \|^2 - [ f (x) + \frac{3
   L}{2} \| x - x^- \|^2 ]}{\| \nabla f (x) \|^2 + 8 L^2 \| x - x^-
   \|^2} \leq - \tfrac{1}{8 L} \]
and this completes the proof.
\end{proof}

\subsubsection{Proof of Theorem \ref{thm:heavyball}}

By \Cref{lem:heavyball-feedback-property}, the heavy-ball feedback is convex and Lipschitz, and thus the same proof of \Cref{lem:regret-sublinear} guarantees that online gradient descent 
\begin{equation*}
  (P_{k + 1}, B_{k+1}) = \Pi_{\mathcal{P} \times \mathcal{B}} [(P_k, B_k) - \eta \nabla h_{x^k, x^{k-1}} (P_k, B_k)]
\end{equation*}
(with $\eta_p = \eta_b = \eta)$ gives the regret bound
\begin{equation*}
  \tfrac{1}{K} \textstyle \sum_{k=1}^K - h_{x^k, x^{k-1}}(P_k, B_k) \geq 
  \gamma_K^{\star} - \tfrac{\rho_K}{K}
\end{equation*}
for some sublinear regret $\rho_K = \mathcal{O}(\sqrt{K})$ and the constant $\gamma_K^{\star}$ as defined in \Cref{lem:heavyball-hindsight}. 
Using the inequality
\begin{equation*}
  \tfrac{1}{K} \textstyle \sum_{k=1}^K \max \{ - h_{x^k, x^{k-1}}(P_k, B_k), 0 \}
  \geq \max \big \{\tfrac{1}{K} \textstyle \sum_{k=1}^K - h_{x^k, x^{k-1}}(P_k, B_k), 0 \big\} \geq \max \{ \gamma_K^{\star} - \tfrac{\rho_K}{K}, 0 \},
\end{equation*}
the desired result follows directly from \eqref{eqn:potential-bound} in \Cref{lem:heavyball-conversion}.

\subsection{{\hdm} + Nesterov Momentum ({\hdmagd})}
\label{app:nesterov}

\subsubsection{Auxiliary Results}
\begin{lem}\label{lem:osnes-auxi}
  Suppose a nonnegative sequence $\{ A_k \}$ satisfies $A_{k + 1} = (A_{k + 1}
  - A_k)^2$ and $A_0 = 0$, then $A_{k + 1} - A_k \leq k + 1$ for all $k \geq 1$.
\end{lem}

\begin{proof}
	We prove by induction. The induction hypothesis is $A_{k + 1} - A_k \leq k$.
	
\textbf{Base case}. For $k = 1$, $A_2 - A_1 = \frac{\sqrt{5} + 1}{2} < 2$
and the relation holds.

\textbf{Inductive step}. Suppose $A_{k + 1} - A_k \leq k$. Using $A_{k + 2} = A_{k + 1} + \frac{1}{2} \left( 1 + \sqrt{4 A_{k + 1} + 1}
\right)$, we deduce that
\begin{align}
  A_{k + 2} - A_{k + 1} ={} & \tfrac{1}{2} ( 1 + \sqrt{4 A_{k + 1} + 1}
  ) \nonumber\\
  ={} & \tfrac{1}{2} ( 1 + \sqrt{4 (A_{k + 1} - A_k)^2 + 1} )
  \nonumber\\
  \leq{} & \tfrac{1}{2} (1 + 2 (A_{k + 1} - A_k) + 1) \label{eqn:c-2-1-auxi-1} \\
  \leq{} & 1 + A_{k + 1} - A_k \nonumber\\
  \leq{} & k + 2, \label{eqn:c-2-1-auxi-2}
\end{align}
where \eqref{eqn:c-2-1-auxi-1} uses $\sqrt{a + b} \leq \sqrt{a}+ \sqrt{b}$ and \eqref{eqn:c-2-1-auxi-2} uses the induction hypothesis $A_{k + 1} - A_k \leq k + 1$. By the principle of mathematical induction, this completes the proof.
\end{proof}

\begin{lem}[{\cite{d2021acceleration}}]\label{lem:osnes-auxi-2}
Under the same conditions as \Cref{lem:osnes-auxi}, $A_k \geq \frac{k^2}{4}$ for all $k \geq 1$.
\end{lem}

\subsubsection{Proof of \Cref{thm:osnes}} \label{app:proof-osnes}
Using the definition $h_y (P) = \tfrac{f (y - P \nabla f (y)) - f (y)}{\| \nabla f (y) \|^2}$ and the $x$-update in \Cref{alg:osnes}:
\begin{equation*}
  x^{k+1} = \underset{x \in \{ y^k - \frac{1}{L} \nabla f (y^k), y^k
  - P_k \nabla f (y^k), x^k \}}{\argmin} f (x),
\end{equation*}
we have the following two inequalities:
\begin{align}
  f (x^{k + 1}) - f (y^k) \leq{} & \min \{ h_{y^k} (P_k), - \tfrac{1}{2 L}
  \} \| \nabla f (y^k) \|^2 \label{eqn:proof-3-1-1}\\
  f (x^{k + 1}) \leq{} & f (x^k). \label{eqn:proof-3-1-2}
\end{align}

In other words, with $v_k = \max \{ - \tfrac{1}{2 \min \{ h_{y^k} (P_k), - 1 / (2 L) \}},
   \tfrac{L}{2 \theta} \} $, we have
\begin{equation}
f (y^k) - \tfrac{1}{2 v_k} \| \nabla f (y^k) \|^2 \geq f (x^{k + 1})
  \label{eqn:proof-3-1-3}
\end{equation}
and using $z^{k+1} = z^k + \tfrac{(A_{k + 1} - A_k)}{v_k}
  \nabla f (y^k)$,  by algebraic rearrangement
\begin{align}
   & \tfrac{v_k}{2}\|z^{k+1} - x^\star\|^2 \\
  = {} & \tfrac{v_k}{2} \| z^k - x^{\star} + \tfrac{(A_{k + 1} - A_k)}{v_k}
  \nabla f (y^k) \|^2 \nonumber\\
  ={} & \tfrac{v_k}{2} \| z^k - x^{\star} \|^2 - (A_{k + 1} - A_k) \langle
  \nabla f (y^k), z^k - x^{\star} \rangle + \tfrac{1}{2 v_k} (A_{k + 1} -
  A_k)^2 \| \nabla f (y^k) \|^2 . \label{eqn:proof-3-1-aux-1}
\end{align}

Next, we apply convexity and have
\begin{align}
  f (x^{\star}) \geq{} & f (y^k) + \langle \nabla f (y^k), x^{\star} - y^k
  \rangle \label{eqn:proof-3-1-4}\\
  f (x^k) \geq{} & f (y^k) + \langle \nabla f (y^k), x^k - y^k \rangle
  \label{eqn:proof-3-1-5}
\end{align}
Taking a weighted summation between \eqref{eqn:proof-3-1-3}, \eqref{eqn:proof-3-1-4} and \eqref{eqn:proof-3-1-5}, we deduce that
\begin{align}
  0 \geq{} & (A_{k + 1} - A_k) [f (y^k) + \langle \nabla f (y^k), x^{\star} -
  y^k \rangle - f (x^{\star})]  + A_k [f (y^k) + \langle \nabla f (y^k), x^k - y^k \rangle - f (x^k)]
  \nonumber\\
  & + A_{k + 1} [ f (x^{k + 1}) - f (y^k) + \tfrac{1}{2 v_k} \| \nabla f
  (y^k) \|^2 ] \nonumber\\
  ={} & (A_{k + 1} - A_k) \langle \nabla f (y^k), x^{\star} - y^k \rangle -
  (A_{k + 1} - A_k) f (x^{\star}) \nonumber\\
  & + A_k \langle \nabla f (y^k), x^k - y^k \rangle - A_k f (x^k) + A_{k + 1}
  f (x^{k + 1}) + \tfrac{A_{k + 1}}{2 v_k} \| \nabla f (y^k) \|^2 \label{eqn:proof-3-1-6} \\
  ={} & A_{k + 1} [f (x^{k + 1}) - f (x^{\star})] - A_k [f (x^k) - f
  (x^{\star})] \nonumber\\
  & + (A_{k + 1} - A_k) \langle \nabla f (y^k), x^{\star} - y^k \rangle + A_k
  \langle \nabla f (y^k), x^k - y^k \rangle + \tfrac{A_{k + 1}}{2 v_k} \|
  \nabla f (y^k) \|^2 \label{eqn:proof-3-1-7}\\
  ={} & A_{k + 1} [f (x^{k + 1}) - f (x^{\star})] - A_k [f (x^k) - f
  (x^{\star})] \nonumber\\
  & A_{k + 1} \langle \nabla f (y^k), x^{\star} - y^k \rangle + A_k \langle
  \nabla f (y^k), x^k - x^{\star} \rangle + \tfrac{A_{k + 1}}{2 v_k} \| \nabla
  f (y^k) \|^2, \nonumber\\
  ={} & A_{k + 1} [f (x^{k + 1}) - f (x^{\star})] - A_k [f (x^k) - f
  (x^{\star})] \nonumber\\ 
  &- (A_{k + 1} - A_k) \langle \nabla f (y^k), z^k - x^{\star} \rangle + \tfrac{A_{k + 1}}{2 v_k} \| \nabla
  f (y^k) \|^2, \label{eqn:proof-3-1-8}
\end{align}

where \eqref{eqn:proof-3-1-6} to \eqref{eqn:proof-3-1-7} simply re-arrange the terms and \eqref{eqn:proof-3-1-8} uses the identity $y^k = x^k + ( 1 - \tfrac{A_k}{A_{k + 1}} ) (z^k - x^k)$:
\begin{align}
  & A_{k + 1} \langle \nabla f (y^k), x^{\star} - y^k \rangle + A_k \langle
  \nabla f (y^k), x^k - x^{\star} \rangle 
  =  - (A_{k + 1} - A_k) \langle \nabla f (y^k), z^k - x^{\star} \rangle .
  \nonumber
\end{align}

Putting the relations together, we arrive at
\begin{align}
  & A_{k + 1} [f (x^{k + 1}) - f (x^{\star})]
  \leq  A_k [f (x^k) - f (x^{\star})] + (A_{k + 1} - A_k) \langle \nabla f
  (y^k), z^k - x^{\star} \rangle - \tfrac{A_{k + 1}}{2 v_k} \| \nabla f (y^k)
  \|^2 . \nonumber
\end{align}
and adding \eqref{eqn:proof-3-1-aux-1} gives
\begin{align}
  & A_{k + 1} [f (x^{k + 1}) - f (x^{\star})] + \tfrac{v_k}{2} \| z^{k + 1}
  - x^{\star} \|^2 \nonumber\\
  \leq{} & A_k [f (x^k) - f (x^{\star})] + \tfrac{v_k}{2} \| z^k - x^{\star}
  \|^2 + \tfrac{A_{k + 1} - (A_{k + 1} - A_k)^2}{2 v_k} \| \nabla f (y^k)
  \|^2 \nonumber\\
  ={} & A_k [f (x^k) - f (x^{\star})] + \tfrac{v_k}{2} \| z^k - x^{\star} \|^2, \label{eqn:proof-3-1-10}
\end{align}

where \eqref{eqn:proof-3-1-10} uses the relation $A_{k + 1} - (A_{k + 1} - A_k)^2 = 0 $.\\

We are now ready to analyze the acceleration effect of online hypergradient. Recall that we can guarantee
\begin{align}
  \tfrac{1}{K}\textstyle \sum_{k = 1}^K h_{y^k} (P_k) \leq & - \gamma^{\star}_K +
  \tfrac{\rho_K}{K}, \label{eqn:proof-3-1-11}
\end{align}
where $\gamma^{\star}_K := -\min_{P \in \Pcal} \sum_{k = 1}^K h_{y^k} (P)$ is expected to be larger than $1 / (2 L)$ to improve
performance. Recall that $\gamma^{\star}_K \assign
\tfrac{\omega^\star_K}{L}, \omega^\star_K \geq 0$ and note that $\omega^\star_K$ depends on the
iteration trajectory. Moreover, we have, by convexity of $f (x)$,
\[ \tfrac{f (x - P \nabla f (x)) - f (x)}{\| \nabla f (x) \|^2} \geq \tfrac{f
   (x) - \langle \nabla f (x), P \nabla f (x) \rangle - f (x)}{\| \nabla f (x)
   \|^2} = - \tfrac{\langle \nabla f (x), P \nabla f (x) \rangle}{\| \nabla f
   (x) \|^2} \geq - D \]
and $\gamma^{\star}_K \leq D$ implies $\omega^\star_K \leq L D$. Define $\mathcal{I}
\assign \{ k : h_{y^k} (P_k) \leq - \tfrac{\theta}{L} \}$ for
$\theta \in [\tfrac{1}{2}, L D)$. Then, according to \eqref{eqn:proof-3-1-11}, 
\[ \textstyle - \tfrac{\omega^\star_K}{L} + \tfrac{\rho_K}{K} \geq \tfrac{1}{K} \sum_{k = 1}^K
   h_{y^k} (P_k) = \tfrac{1}{K} [ \sum_{k \in \mathcal{I}} h_{y^k} (P_k)
   + \sum_{k \in \bar{\mathcal{I}}} h_{y^k} (P_k) ] \geq \tfrac{1}{K}
   \sum_{k \in \mathcal{I}} h_{y^k} (P_k) - \tfrac{\theta}{L} \tfrac{K - |
   \mathcal{I} |}{K} . \]
Using $h_{y^k} (P_k) \geq - D$, we get
\[ \textstyle- \tfrac{D}{K} | \mathcal{I} | \leq \tfrac{1}{K} \sum_{k \in \mathcal{I}}
   h_{y^k} (P_k) \leq - \tfrac{\omega^\star_K}{L} + \tfrac{\theta}{L} \tfrac{K - |
   \mathcal{I} |}{K} + \tfrac{\rho_K}{K} . \]
Re-arranging the terms,
\[ ( - \tfrac{D}{K} + \tfrac{\theta}{K L} ) | \mathcal{I} | \leq
   \tfrac{\theta - \omega^\star_K}{L} + \tfrac{\rho_K}{K} . \]
Using $D > \tfrac{\theta}{L}$, we get
\[ | \mathcal{I} | \geq \tfrac{\tfrac{\theta - \omega^\star_K}{L} +
   \tfrac{\rho_K}{K}}{- \tfrac{D}{K} + \tfrac{\theta}{K L}} = \tfrac{(\theta -
   \omega^\star_K) K + L \rho_K}{\theta - L D} = \tfrac{(\omega^\star_K - \theta) K}{L D -
   \theta} - \tfrac{L}{L D - \theta} \rho_K . \]
We have, if $k \in \mathcal{I}$, that using the fact that \eqref{eqn:proof-3-1-10} holds for $v_k = \max \{ - \tfrac{1}{2 \min \{ h_{y^k} (P_k), - 1 / (2 L) \}},
   \tfrac{L}{2 \theta} \}  = \tfrac{L}{2 \theta}$,
\begin{align}
  A_{k + 1} [f (x^{k + 1}) - f (x^{\star})] + \tfrac{L}{4 \theta} \| z^{k + 1}
  - x^{\star} \|^2 \leq{} & A_k [f (x^k) - f (x^{\star})] + \tfrac{L}{4 \theta}
  \| z^k - x^{\star} \|^2. \label{eqn:proof-3-1-12}
\end{align}
On the other hand, if $k \nin \mathcal{I}$, $v_k \leq L$ and 
\begin{align}
  A_{k + 1} [f (x^{k + 1}) - f (x^{\star})] + \tfrac{v_k}{2} \| z^{k + 1} -
  x^{\star} \|^2 \leq{} & A_k [f (x^k) - f (x^{\star})] + \tfrac{v_k}{2} \| z^k -
  x^{\star} \|^2 \label{eqn:proof-3-1-13}
\end{align}

and  $f (x^{k + 1}) \leq f (x^k)$ implies
\begin{align}
  A_{k + 1} [f (x^{k + 1}) - f (x^{\star})] \leq{} & A_{k + 1} [f (x^k) - f
  (x^{\star})] \nonumber\\
  \leq{} & A_k [f (x^k) - f (x^{\star})] + (k + 1) [f (x^k) - f (x^{\star})] \label{eqn:proof-3-1-14},
\end{align}
where \eqref{eqn:proof-3-1-14} uses the condition that $A_{k + 1} - A_k \leq k + 1$ from \Cref{lem:osnes-auxi}.\\

Taking a weighted summation of \eqref{eqn:proof-3-1-13} and \eqref{eqn:proof-3-1-14}, combining \eqref{eqn:proof-3-1-12},
\begin{align}
  & A_{k + 1} [f (x^{k + 1}) - f (x^{\star})] + \tfrac{L}{4 \theta} \| z^{k +
  1} - x^{\star} \|^2 \nonumber\\
  \leq{} & A_k [f (x^k) - f (x^{\star})] + \tfrac{L}{4 \theta} \| z^k -
  x^{\star} \|^2 + ( 1 - \tfrac{L}{2 \theta v_k} ) (k + 1) [f (x^k) - f
  (x^{\star})] \cdummy \mathbb{I} \{ k \in \bar{\mathcal{I}} \} \nonumber\\
  \leq {} & A_k [f (x^k) - f (x^{\star})] + \tfrac{L}{4 \theta} \| z^k -
  x^{\star} \|^2 + ( 1 - \tfrac{1}{2 \theta} ) (k + 1) [f (x^k) - f
  (x^{\star})] \cdummy \mathbb{I} \{ k \in \bar{\mathcal{I}} \} . \nonumber
\end{align}
Telescoping the relation from 1 to $K$,
\[ \textstyle A_{K + 1} [f (x^{K + 1}) - f (x^{\star})] \leq \tfrac{L}{4 \theta} \| z^1
   - x^{\star} \|^2 + \sum_{k \in \bar{\mathcal{I}}} ( 1 - \tfrac{1}{2
   \theta} ) (k + 1) [f (x^k) - f (x^{\star})] . \]
Using $A_{K + 1} \geq \tfrac{K^2}{4}$ from \Cref{lem:osnes-auxi-2} and that $| \bar{\mathcal{I}} | = K - |
\mathcal{I} | \leq K - \tfrac{(\omega^\star_K - \theta) K}{L D - \theta} + \tfrac{L
\rho_K}{L D - \theta}$,
\begin{align}
  f (x^{K + 1}) - f (x^{\star}) \leq{} & \tfrac{L}{4 \theta A_{K + 1}} \| z^1 -
  x^{\star} \|^2 + \tfrac{1}{A_{K + 1}} ( 1 - \tfrac{1}{2 \theta} )
  \textstyle\sum_{k \in \bar{\mathcal{I}}} (k + 1) [f (x^k) - f (x^{\star})] \nonumber\\
  \leq{} & \tfrac{L}{\theta K^2} \| z^1 - x^{\star} \|^2 + \tfrac{4}{K^2}
  ( 1 - \tfrac{1}{2 \theta} ) \textstyle\sum_{k \in \bar{\mathcal{I}}} (k + 1) [f
  (x^k) - f (x^{\star})] \nonumber\\
  \leq{} & \tfrac{L}{\theta K^2} \| z^1 - x^{\star} \|^2 + \tfrac{4}{K} (
  1 - \tfrac{1}{2 \theta} ) [f (x^1) - f (x^{\star})] \cdummy |
  \bar{\mathcal{I}} | \nonumber\\
  \leq{} & \tfrac{L}{\theta K^2} \| z^1 - x^{\star} \|^2 + 4 ( 1 -
  \tfrac{1}{2 \theta} ) [f (x^1) - f (x^{\star})] ( 1 -
  \tfrac{\omega^\star_K - \theta}{L D - \theta} + \tfrac{L}{L D - \theta}
  \tfrac{\rho_K}{K} ) . \nonumber
\end{align}

Suppose we run accelerated gradient descent from $z'$ for $K$ iterations and
obtain $x^1$. \\

Plugging in $f (x^1) - f (x^{\star}) \leq \tfrac{2 L}{K^2} \| z' - x^{\star}
\|^2$ we get, using $z^1 = z'$, that
\begin{align}
  f (x^{K + 1}) - f (x^{\star}) \leq{} & \tfrac{L}{\theta K^2} \| z' - x^{\star}
  \|^2 + \tfrac{L}{\theta K^2} \| z' - x^{\star} \|^2 8 (2 \theta - 1)
  ( 1 - \tfrac{\omega^\star_K - \theta}{L D - \theta} ) +\mathcal{O} (
  \tfrac{\rho_K}{K^3} ) \nonumber\\
  ={} & \tfrac{L}{\theta K^2} \| z' - x^{\star} \|^2 + \tfrac{L}{K^2} \| z' -
  x^{\star} \|^2  ( 16 - \tfrac{8}{\theta} ) \tfrac{L D -
  \omega^\star_K}{L D - \theta} +\mathcal{O} ( \tfrac{\rho_K}{K^3} )
  \nonumber\\
  \leq{} & \tfrac{L}{\theta K^2} \| z' - x^{\star} \|^2 + \tfrac{L}{K^2} \| z' -
  x^{\star} \|^2 (16 - \tfrac{8}{\theta} ) \tfrac{L D -
  \omega^\star_K}{L D - \theta} +\mathcal{O} ( \tfrac{\rho_K}{K^3} )
  \nonumber\\
  \leq{} & [ \tfrac{1}{2 \theta} +  ( 8 - \tfrac{4}{\theta} )
  ( \tfrac{L D - \omega^\star_K}{L D - \theta} ) ] \tfrac{2 L \| z' -
  x^{\star} \|^2}{K^2} +\mathcal{O} ( \tfrac{\rho_K}{K^3} ) .
  \nonumber
\end{align}
This completes the proof.

\section{Additional Experiments} \label{app:additional}

%
%
%
%

\subsection{Additional Experiments on Support Vector Machine Problems}
See \Cref{fig:svm-add-1} and \Cref{fig:svm-add-2}.

\begin{figure}[!h]
\centering
\includegraphics[scale=0.2]{figs/a1a_objval_svm.pdf}
\includegraphics[scale=0.2]{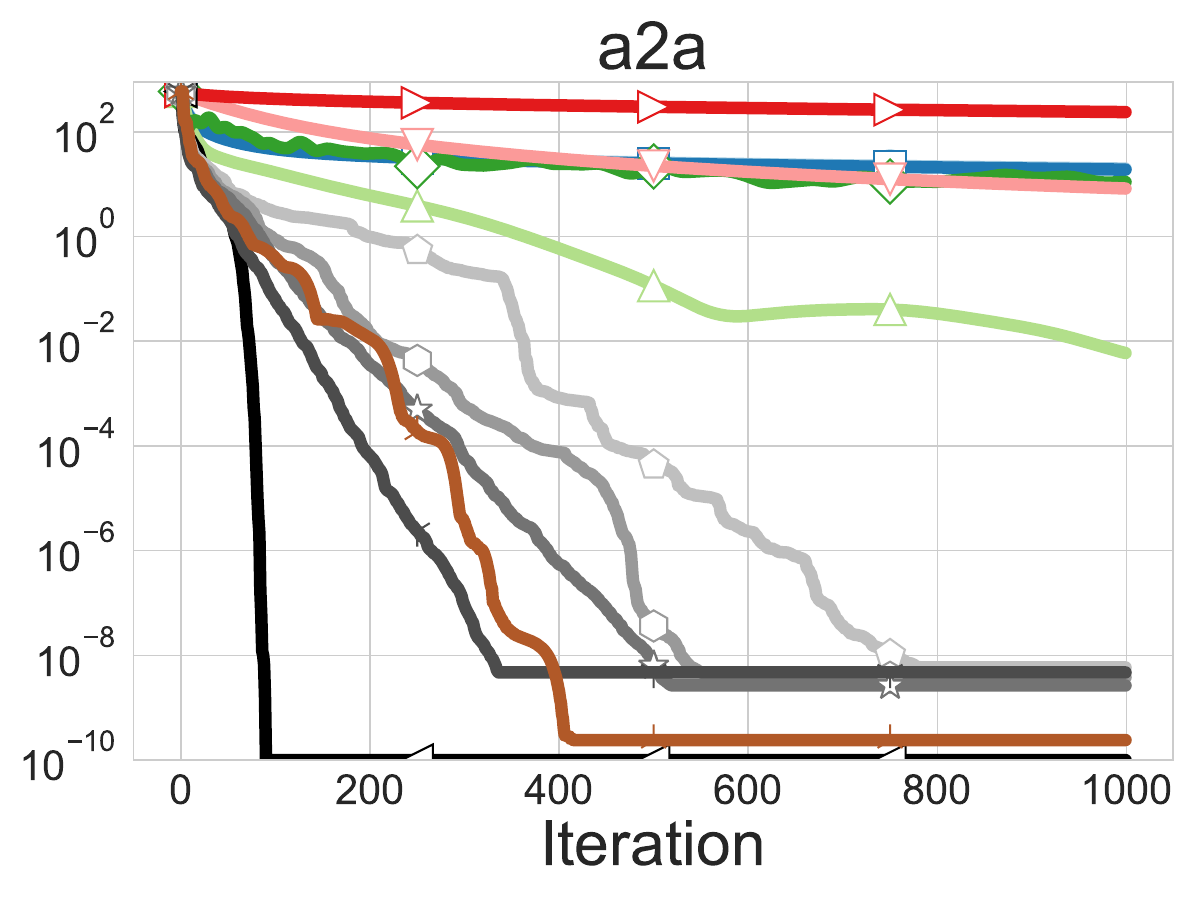}
\includegraphics[scale=0.2]{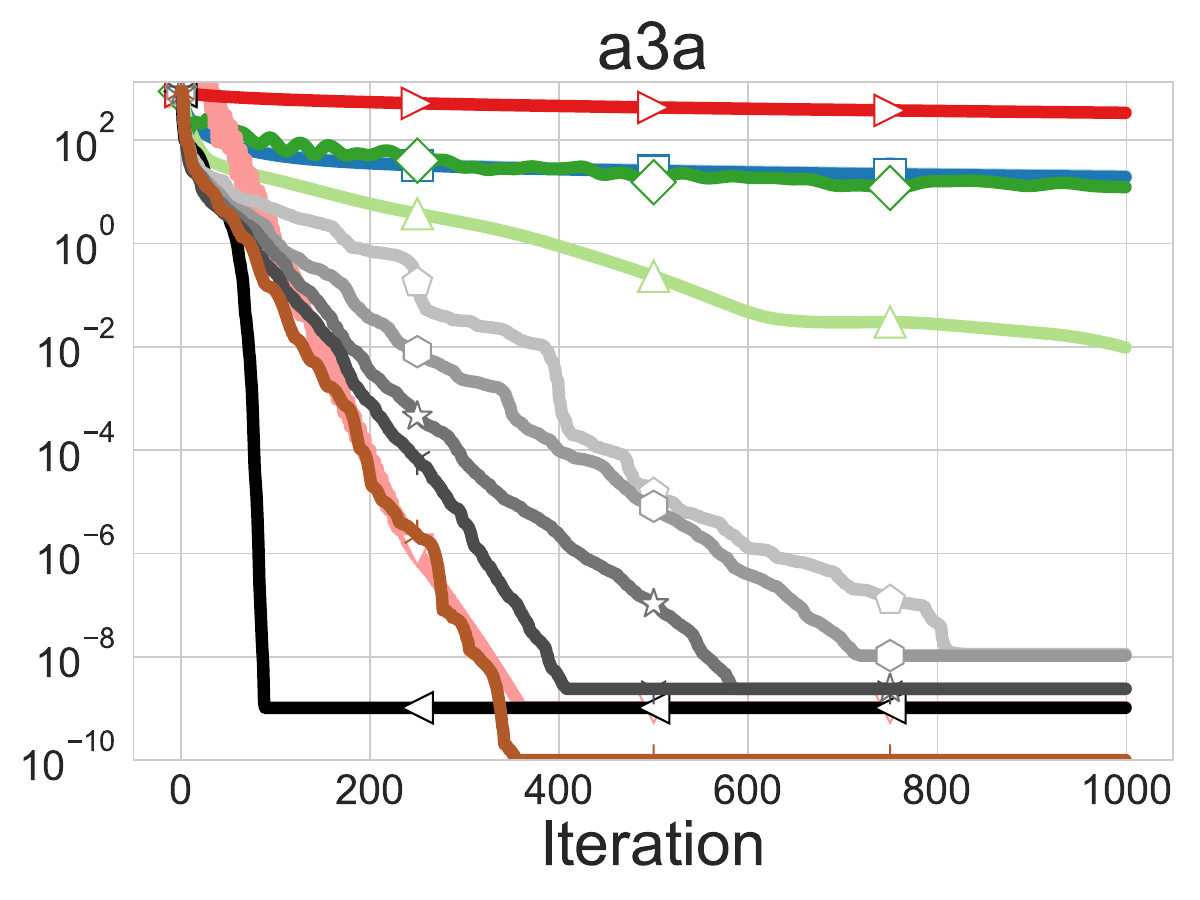}
\includegraphics[scale=0.2]{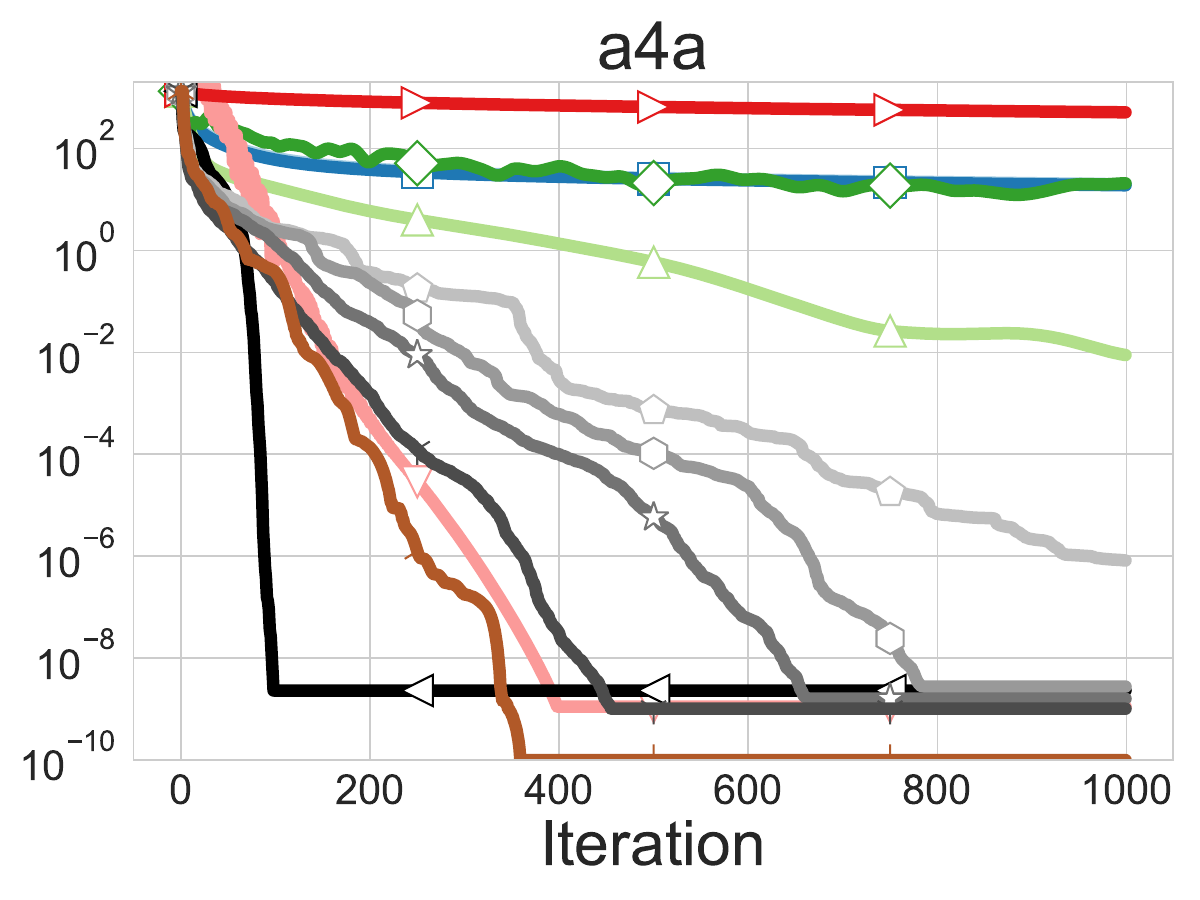}
\\
\includegraphics[scale=0.2]{figs/a1a_gnorm_svm.pdf}
\includegraphics[scale=0.2]{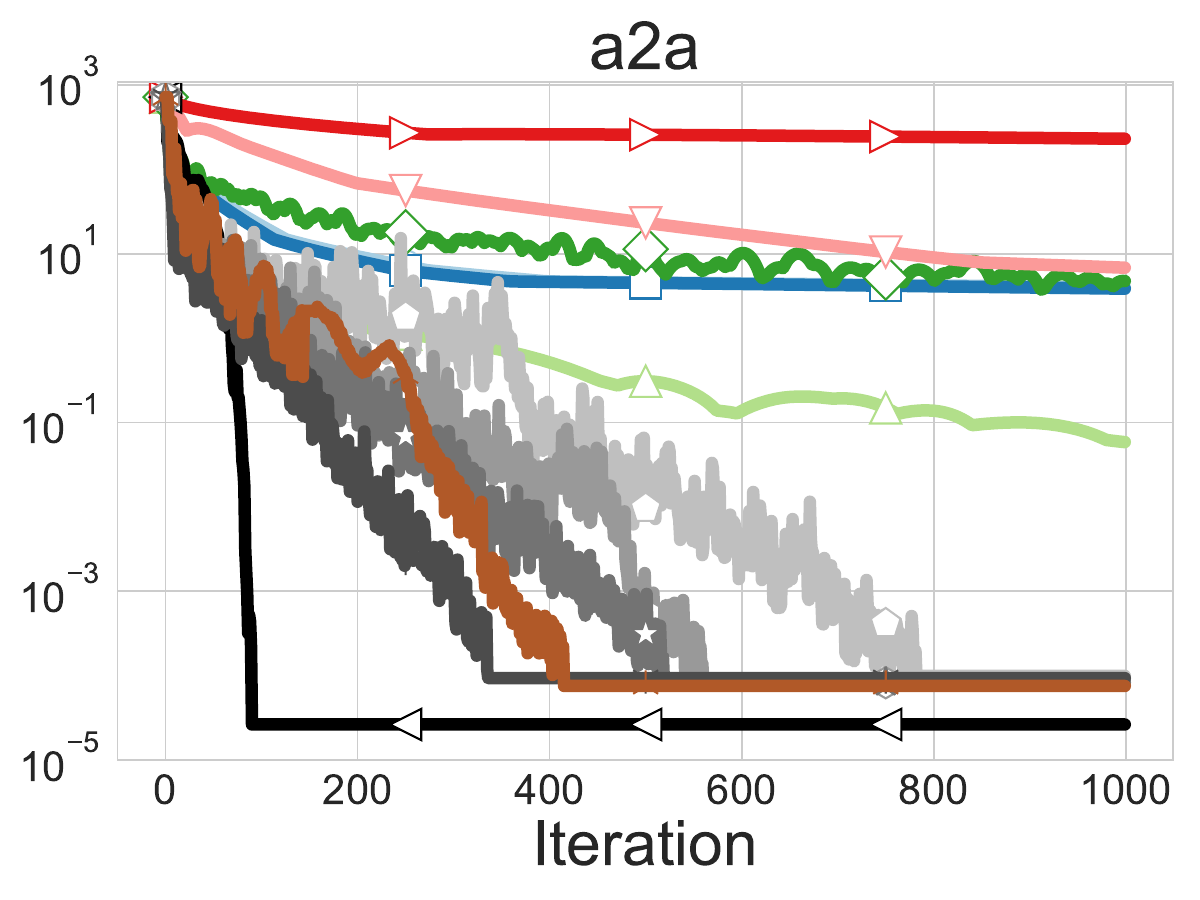}
\includegraphics[scale=0.2]{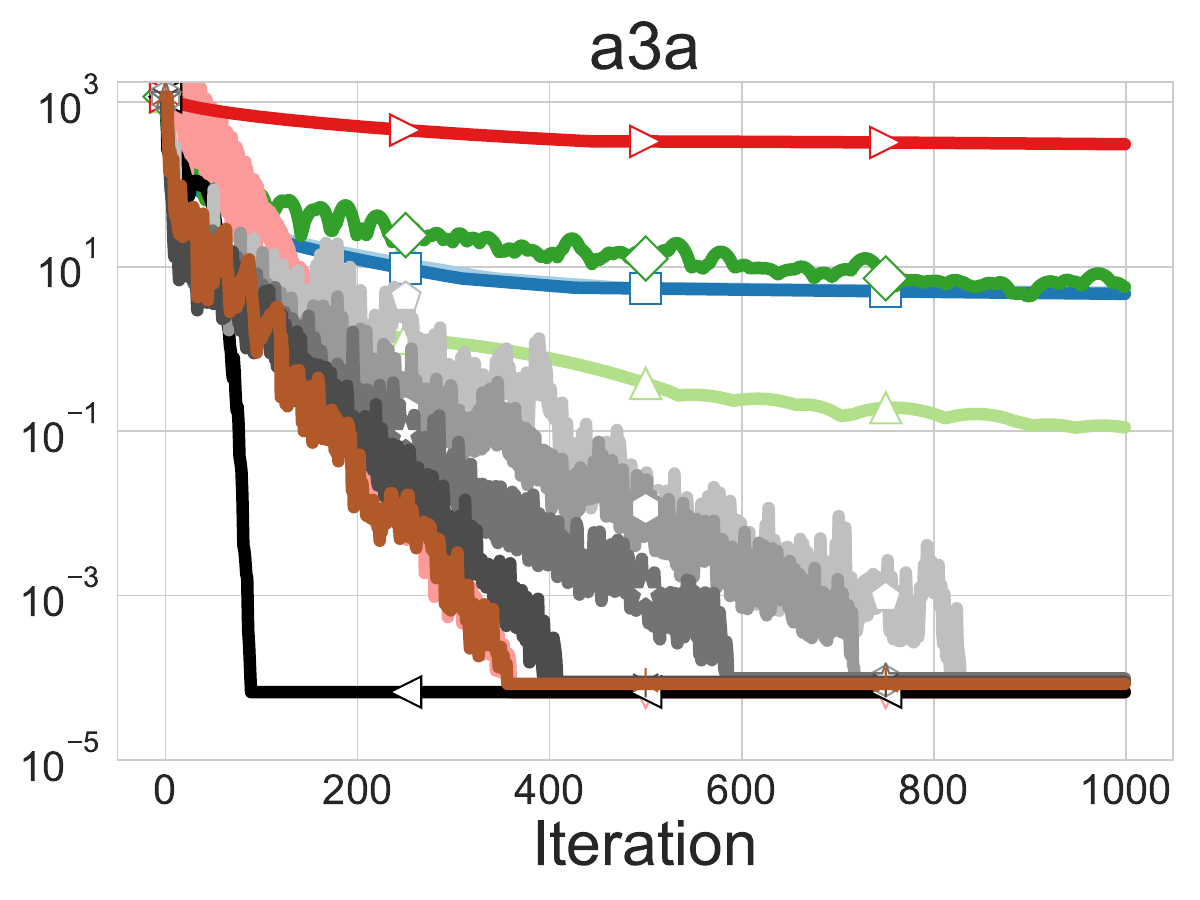}
\includegraphics[scale=0.2]{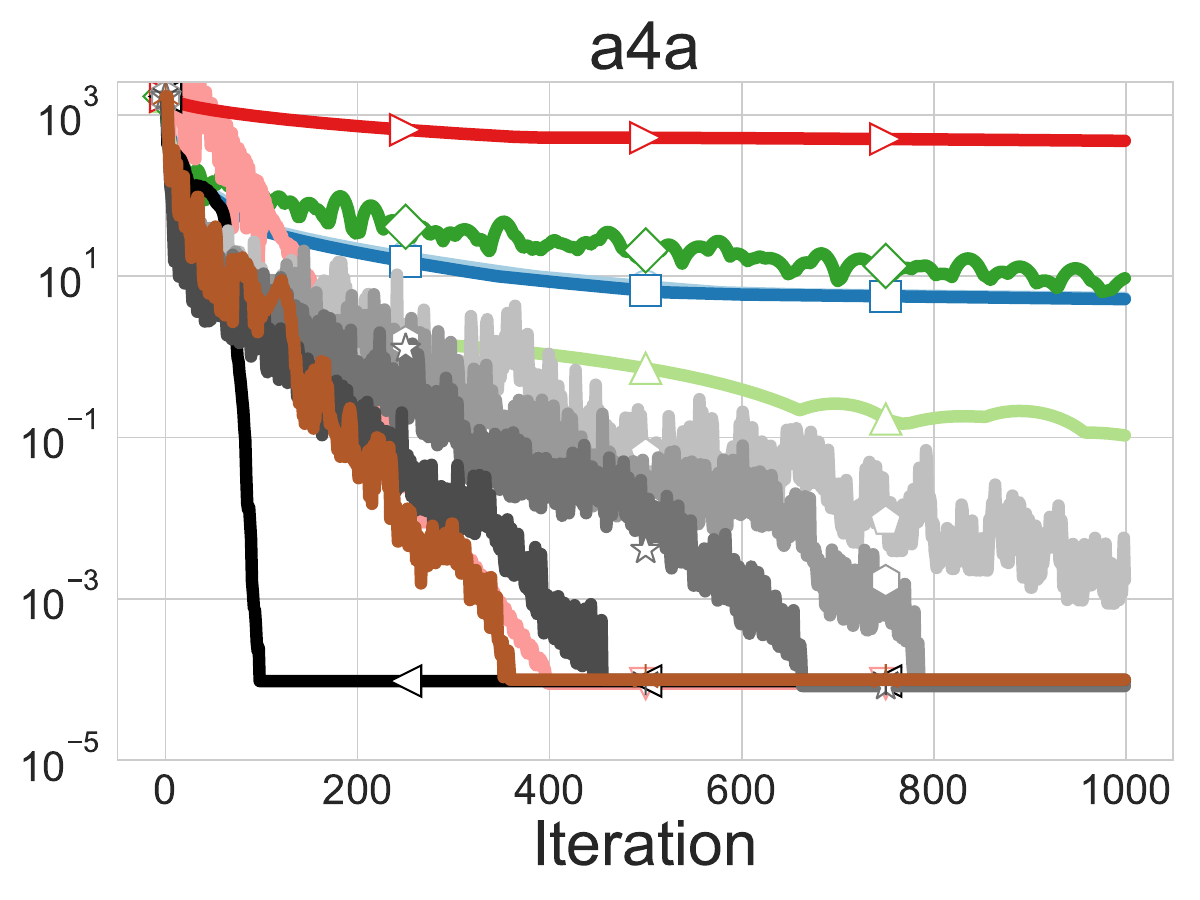}
\\
\includegraphics[scale=0.2]{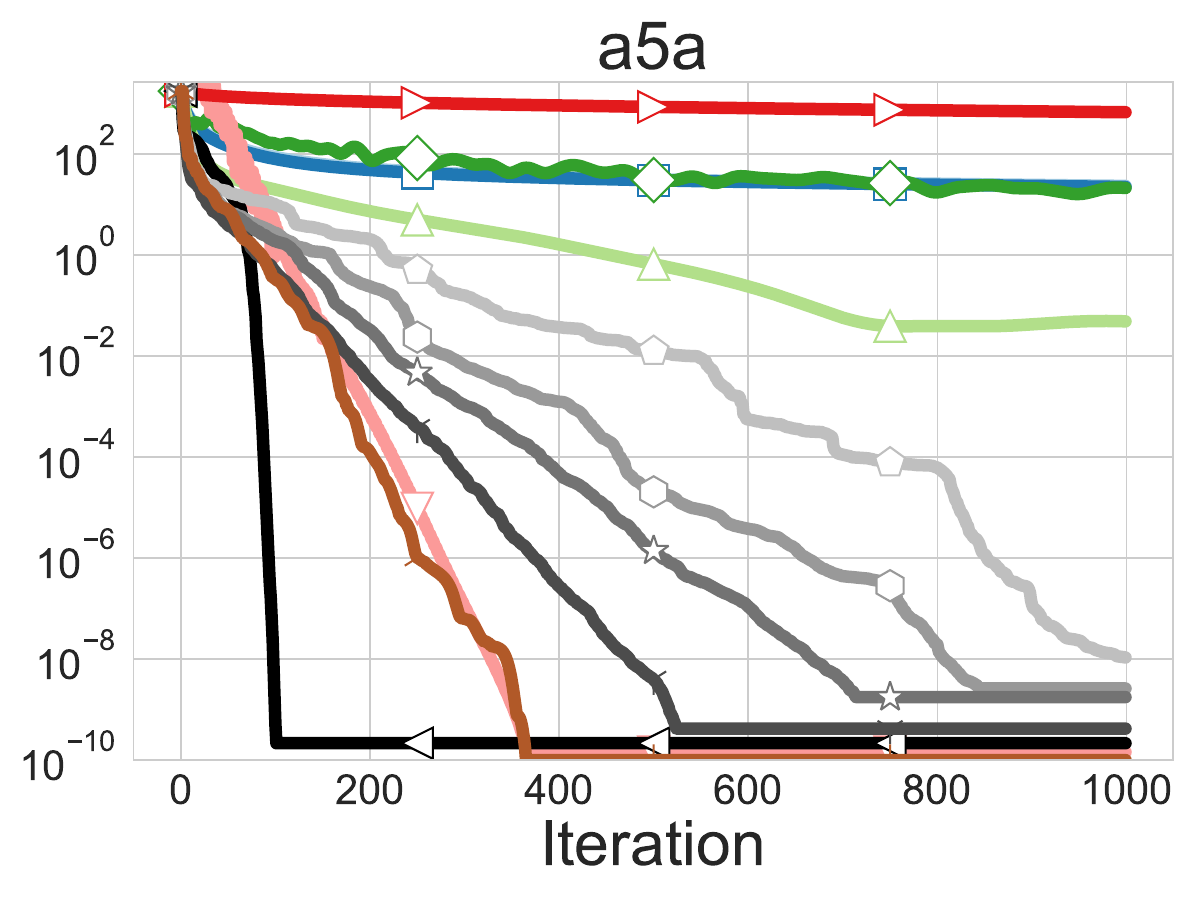}
\includegraphics[scale=0.2]{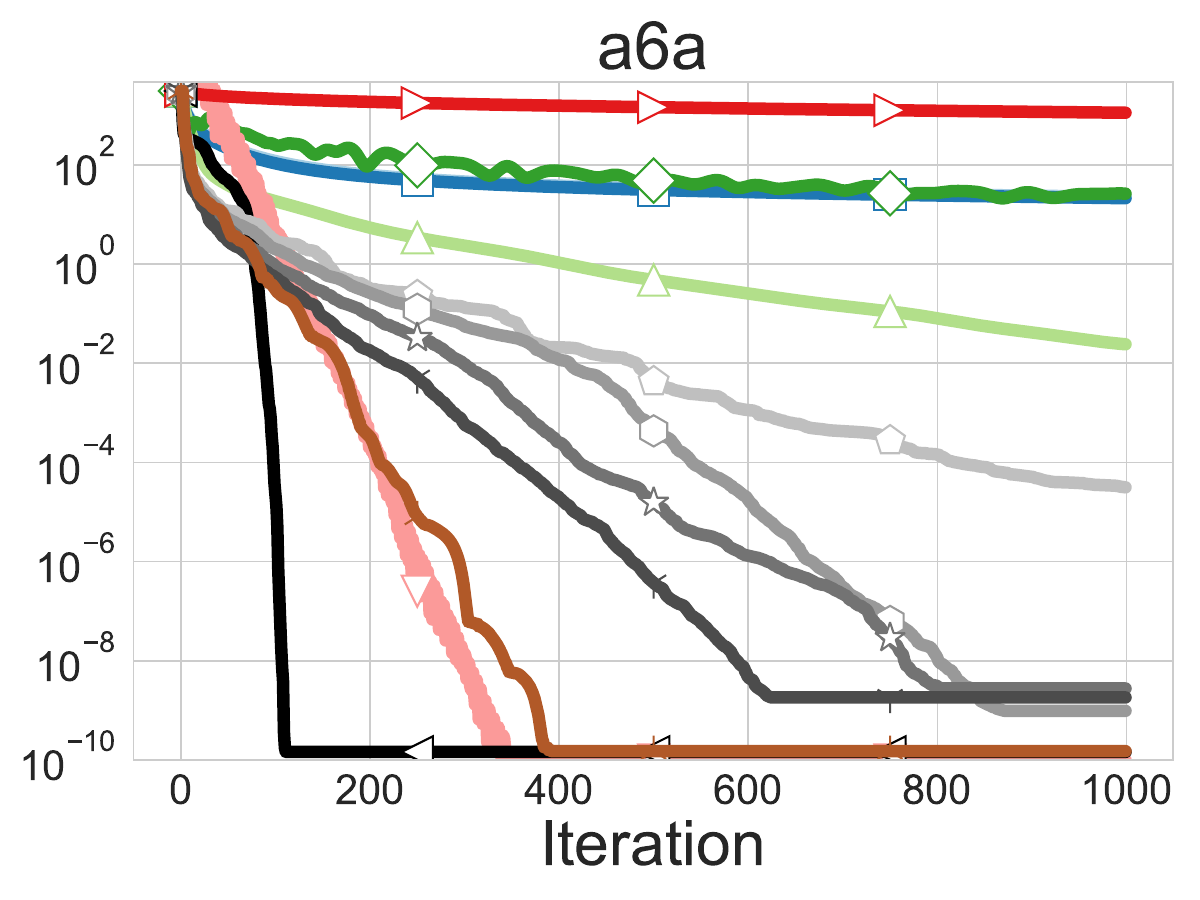}
\includegraphics[scale=0.2]{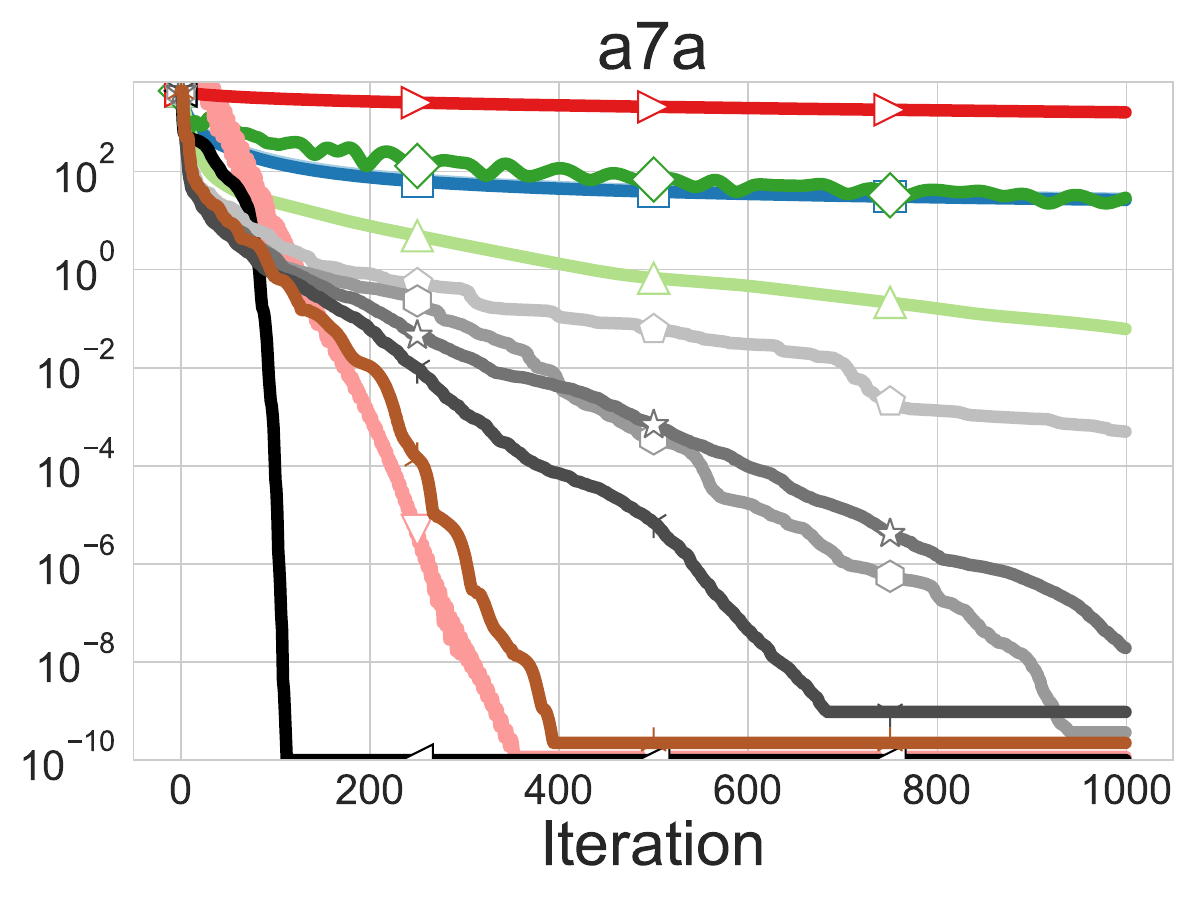}
\includegraphics[scale=0.2]{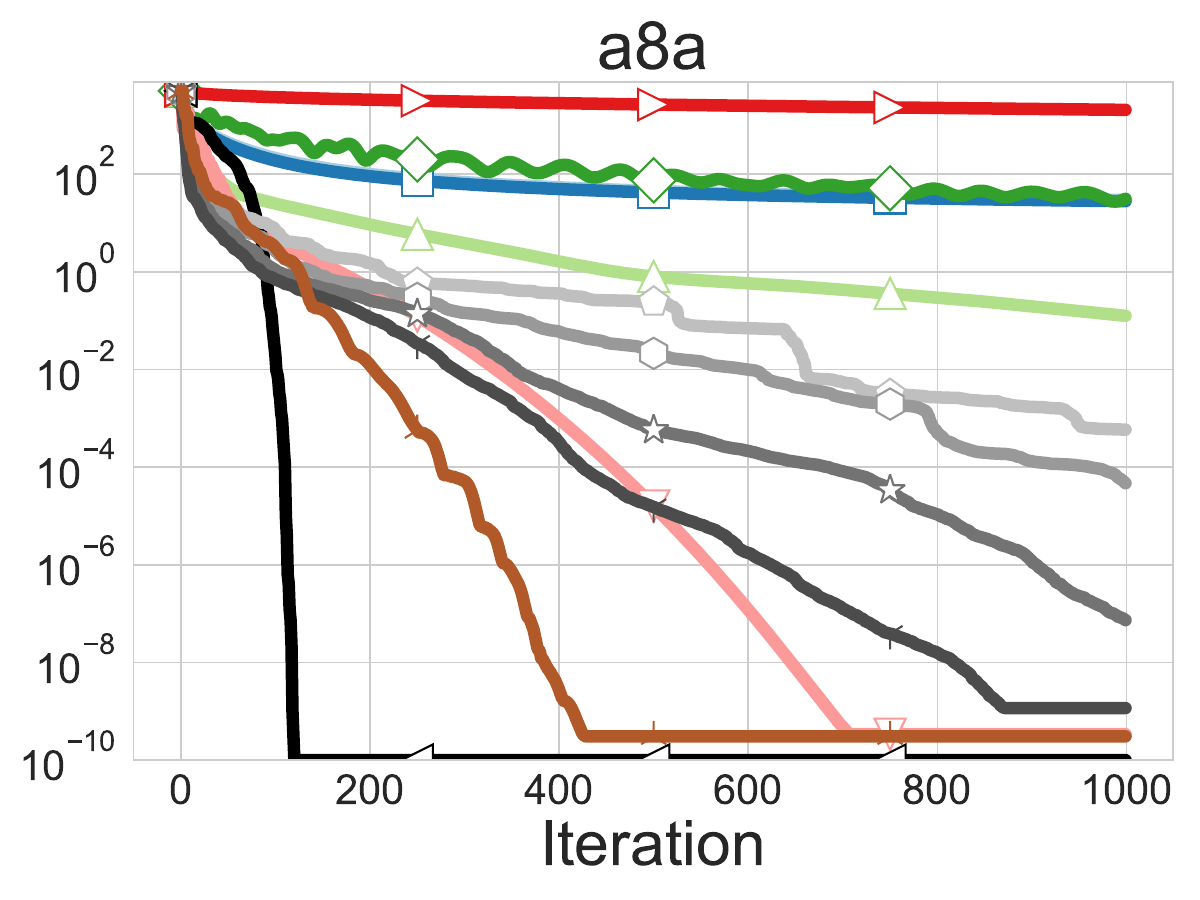}
\\
\includegraphics[scale=0.2]{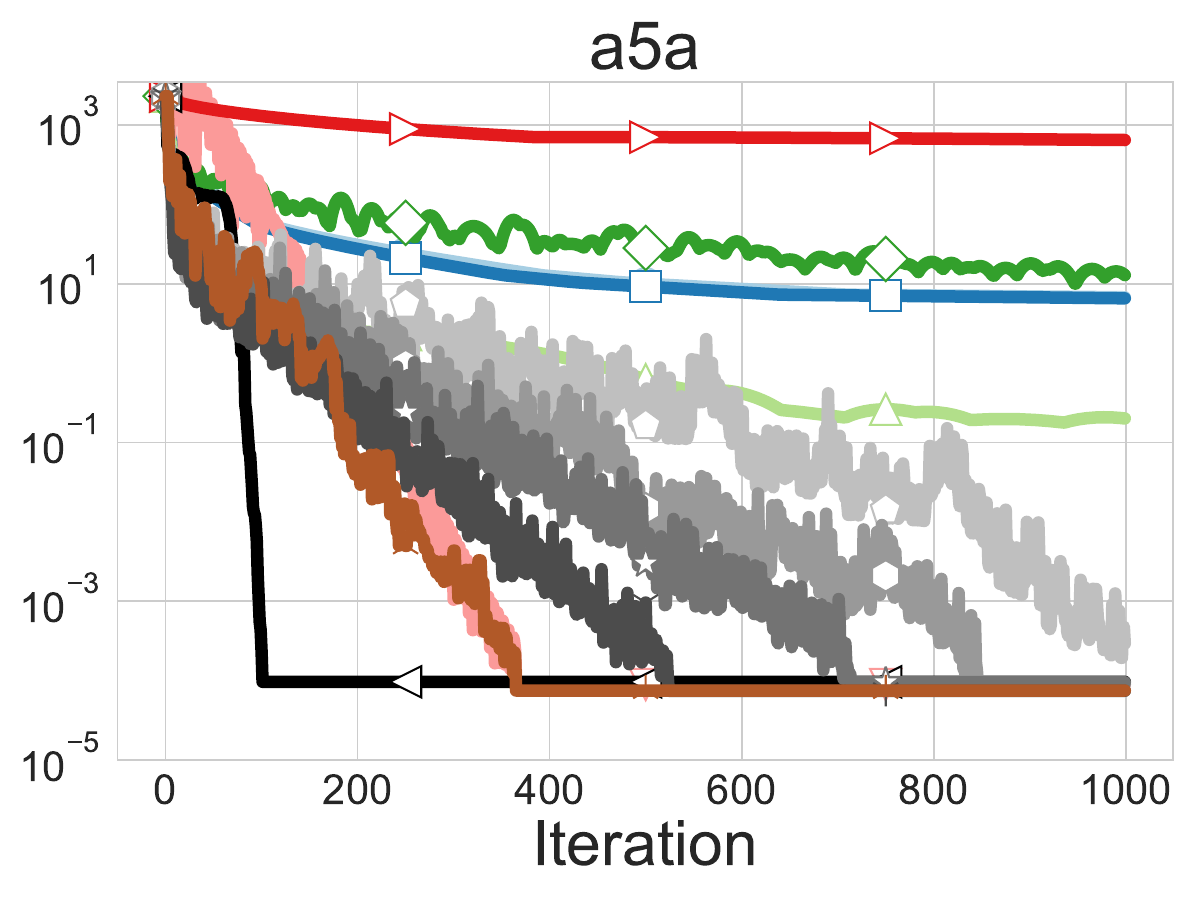}
\includegraphics[scale=0.2]{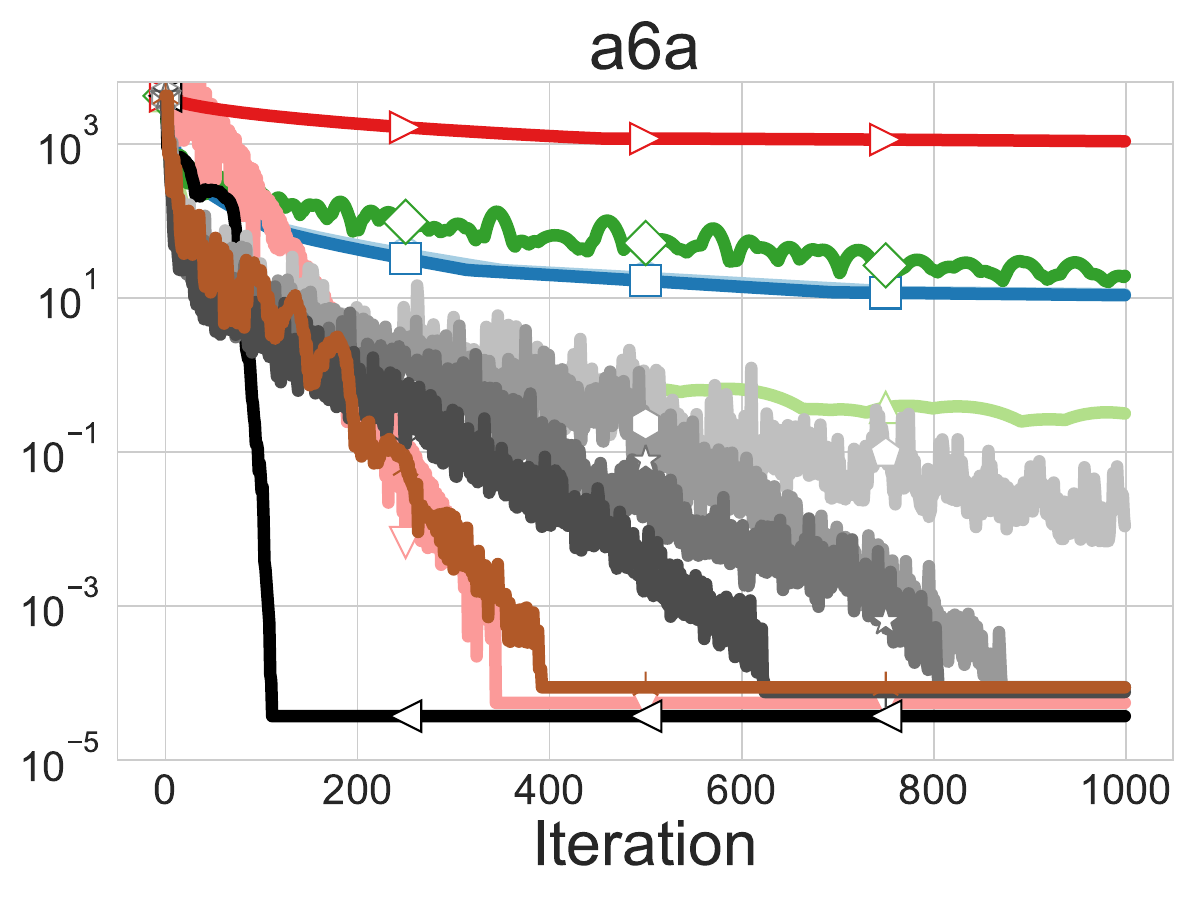}
\includegraphics[scale=0.2]{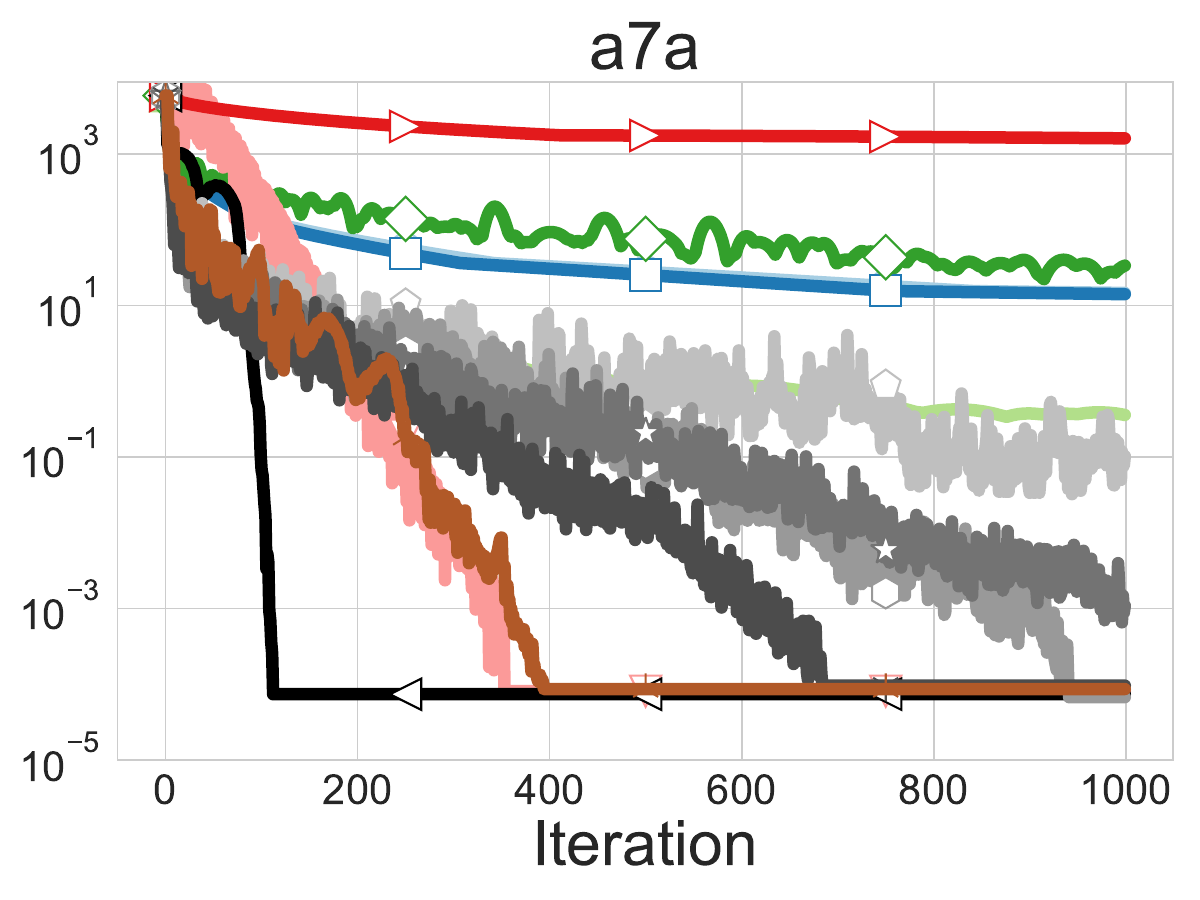}
\includegraphics[scale=0.2]{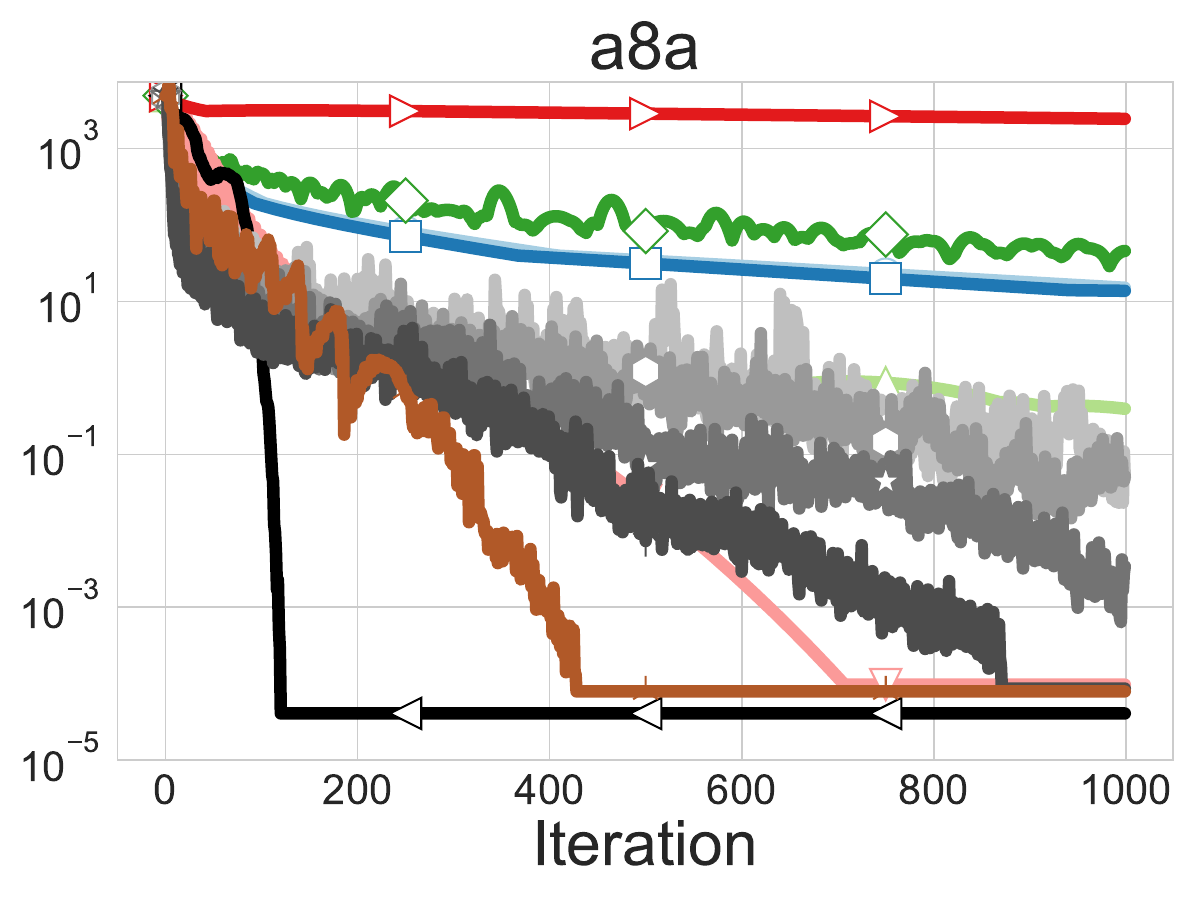}
\\
\includegraphics[scale=0.38]{figs/legend.pdf}
\caption{More experiments on support vector-machine problem}
\label{fig:svm-add-1}
\end{figure}

\begin{figure}

\centering
\includegraphics[scale=0.2]{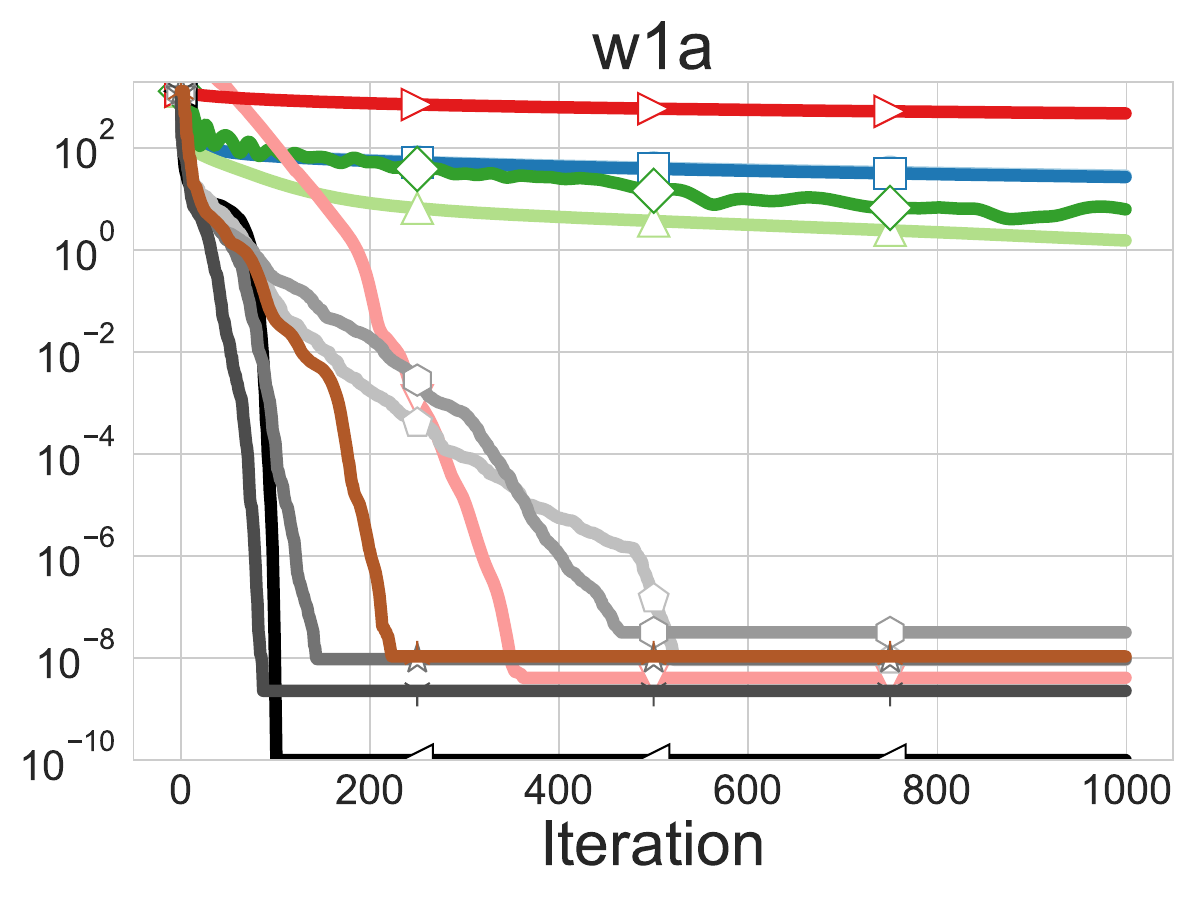}
\includegraphics[scale=0.2]{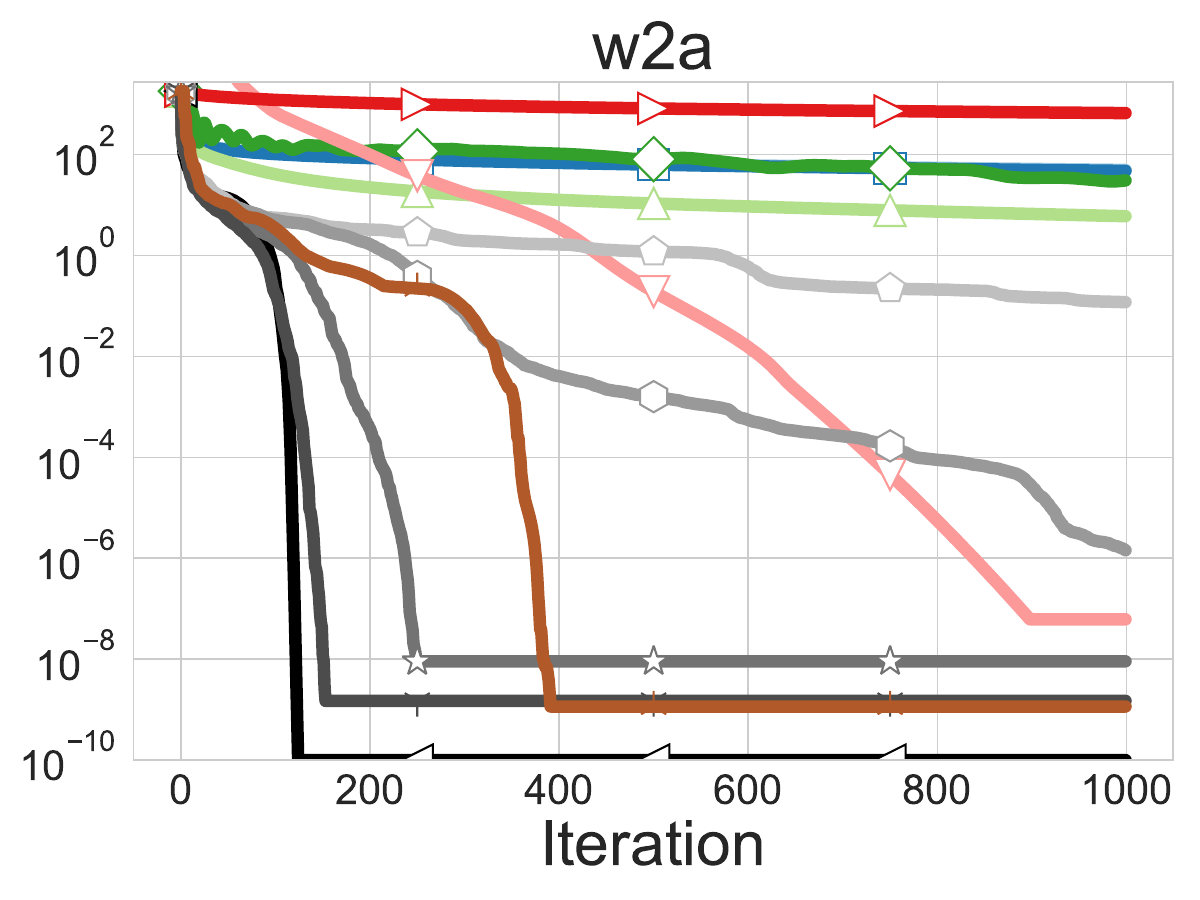}
\includegraphics[scale=0.2]{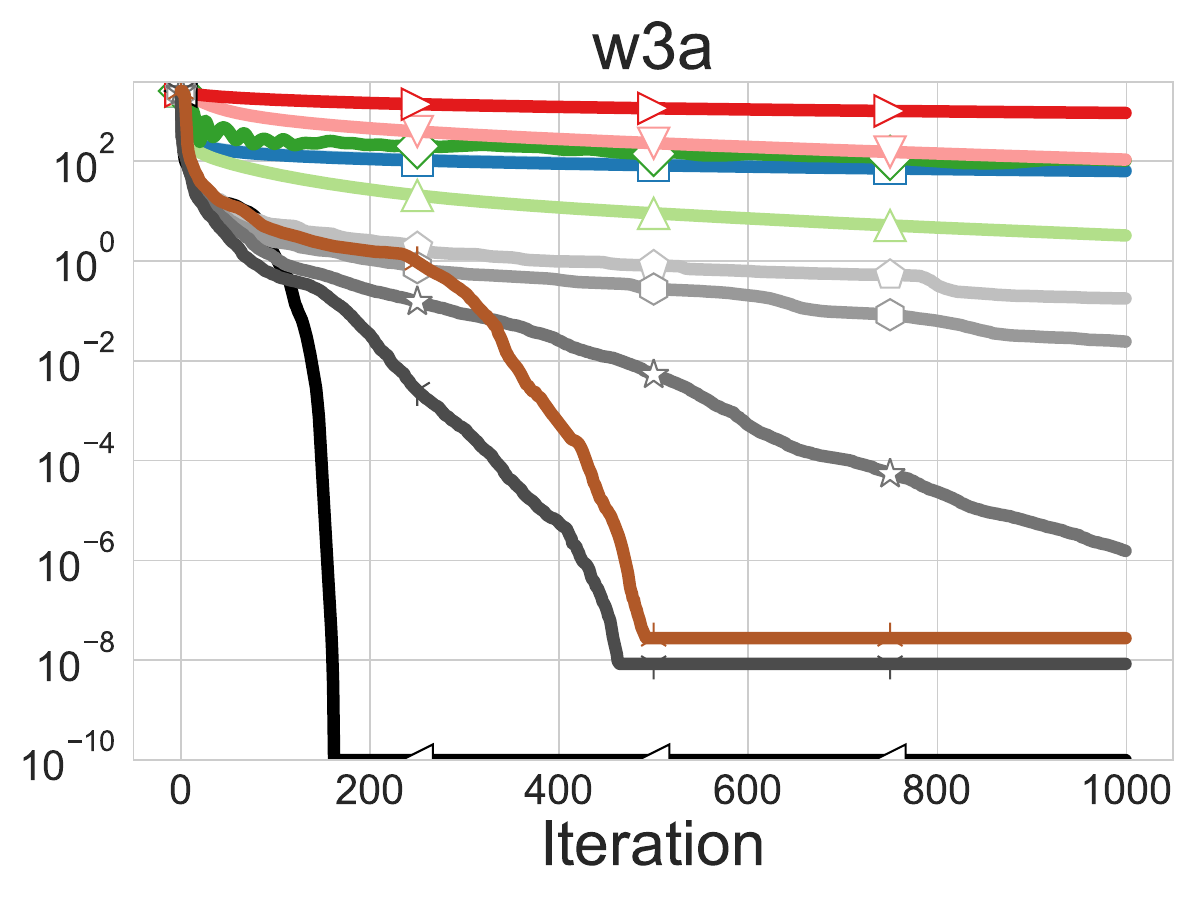}
\includegraphics[scale=0.2]{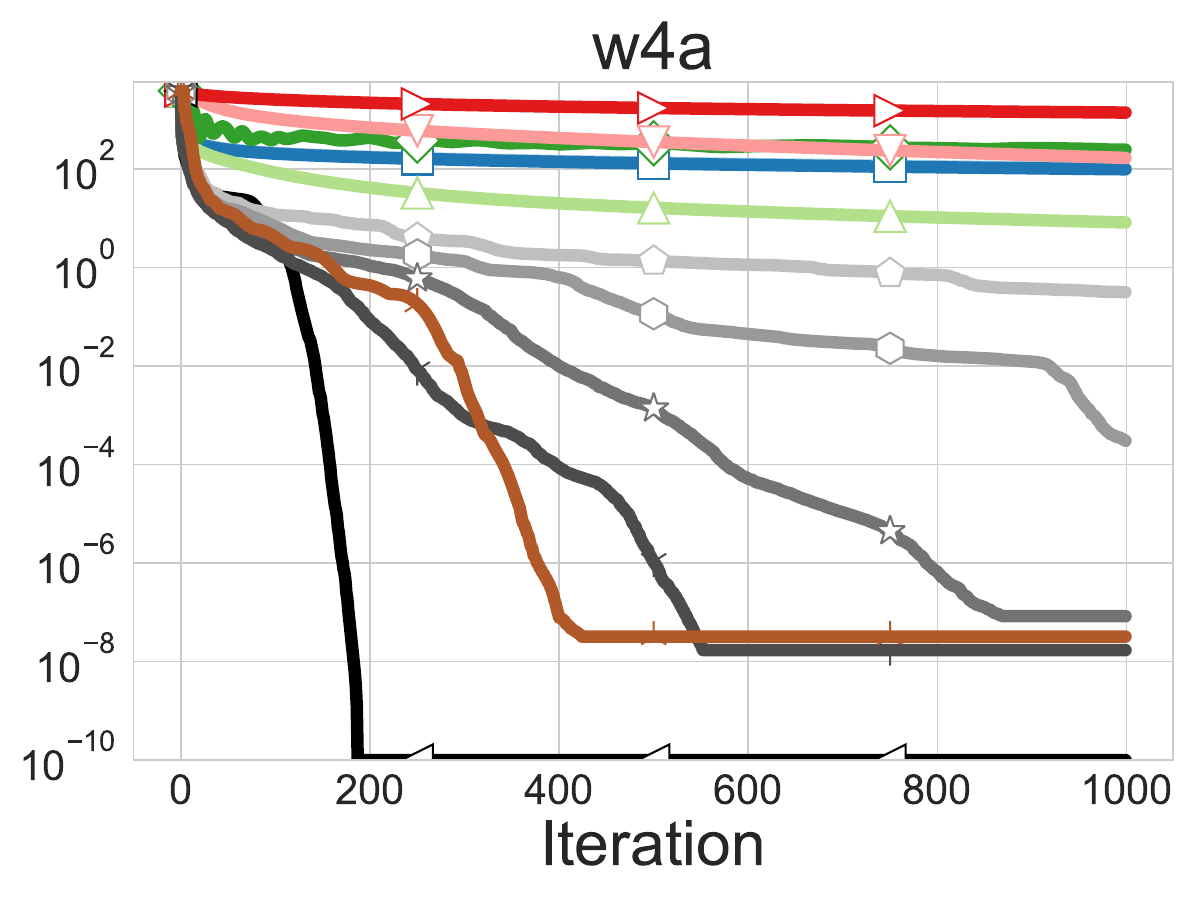}
\\
\includegraphics[scale=0.2]{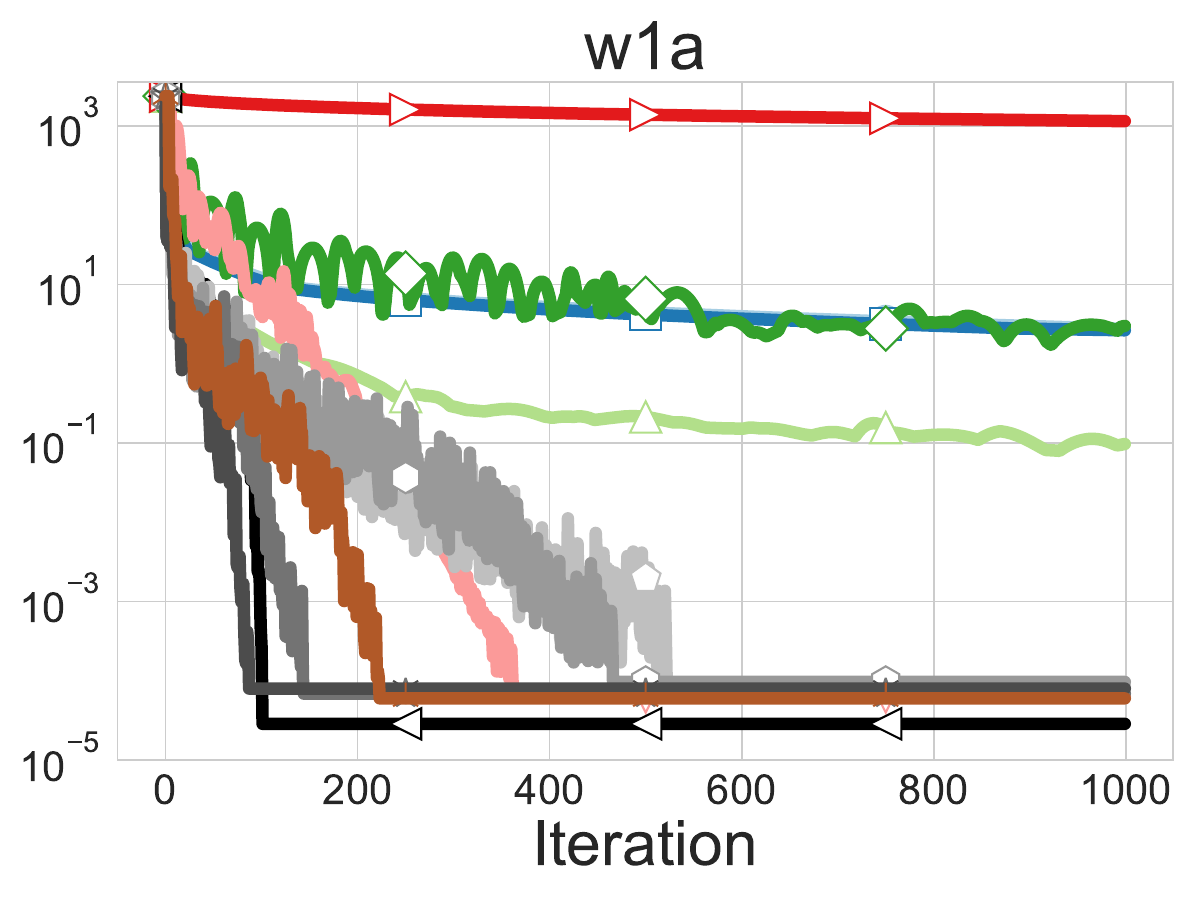}
\includegraphics[scale=0.2]{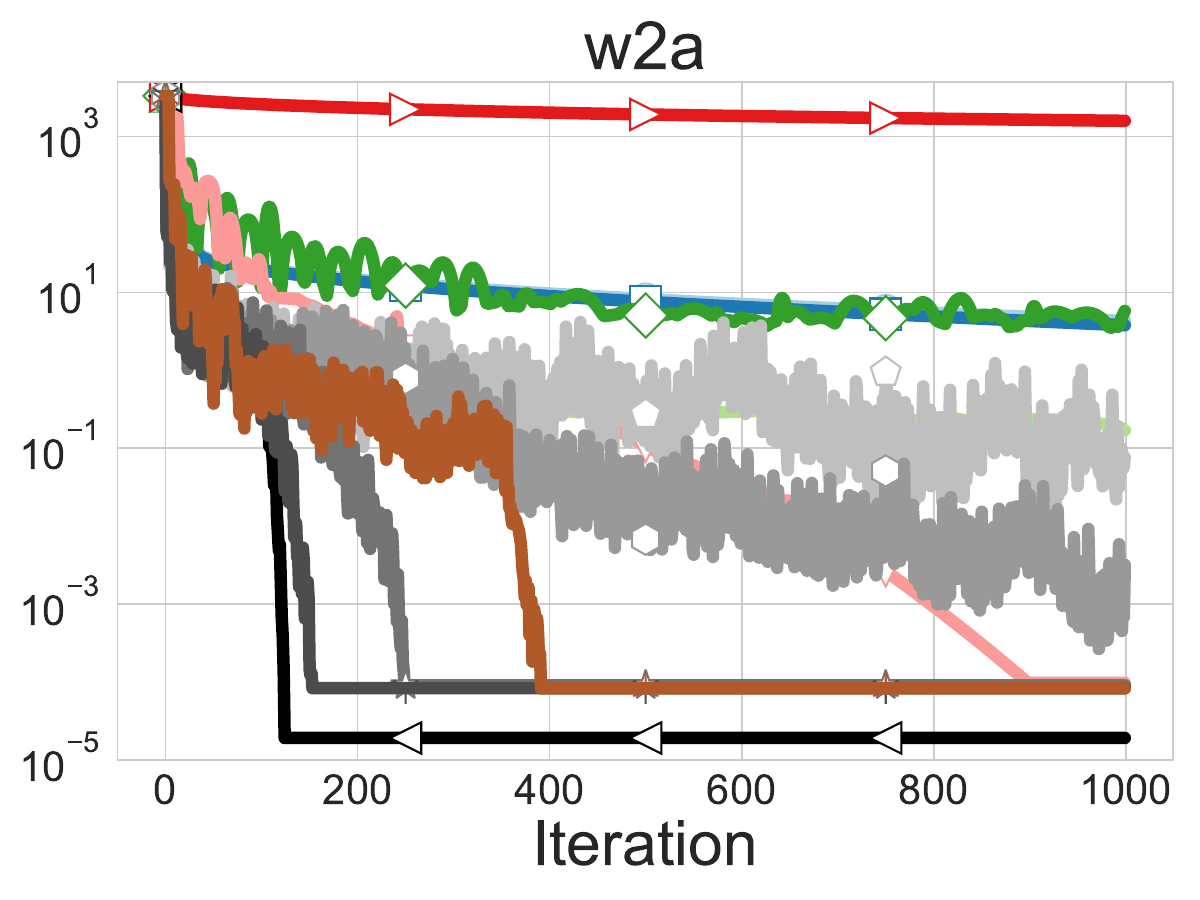}
\includegraphics[scale=0.2]{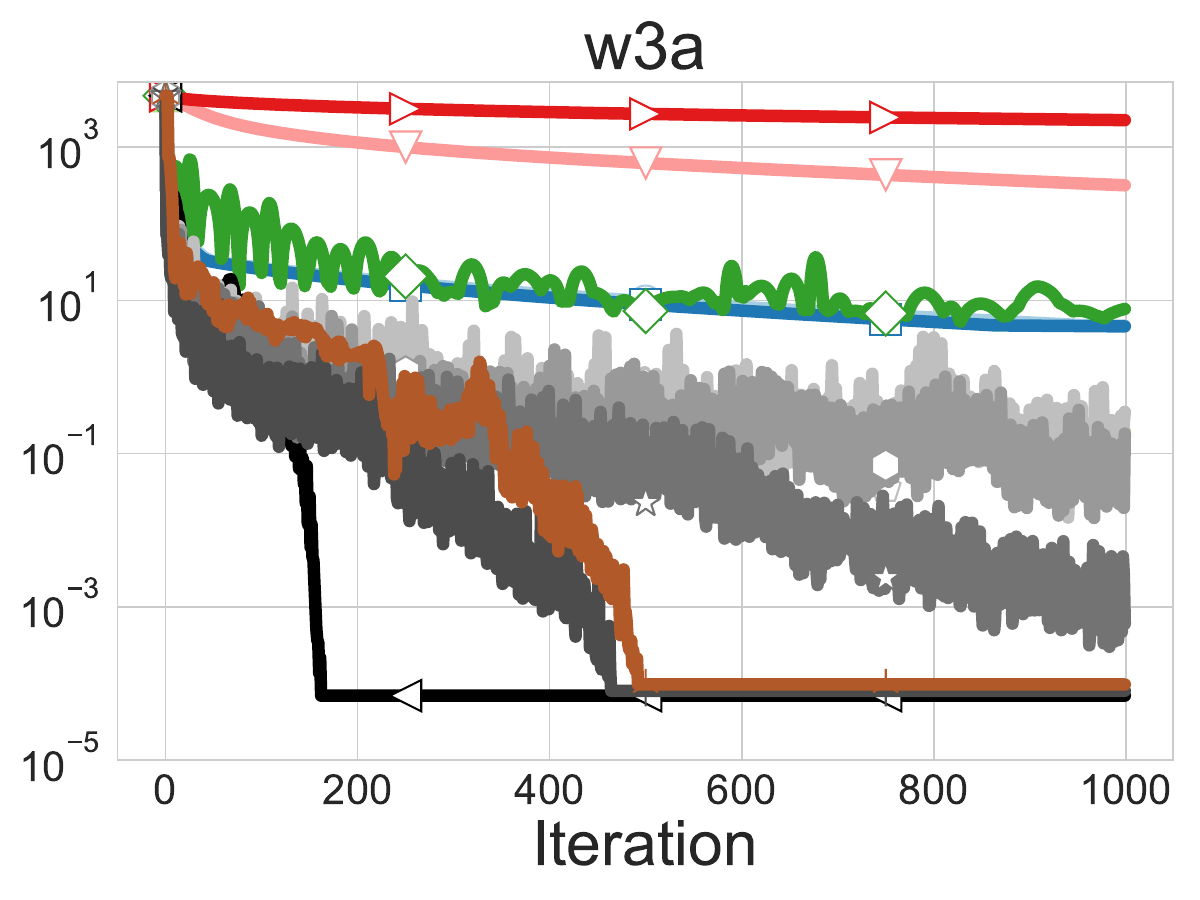}
\includegraphics[scale=0.2]{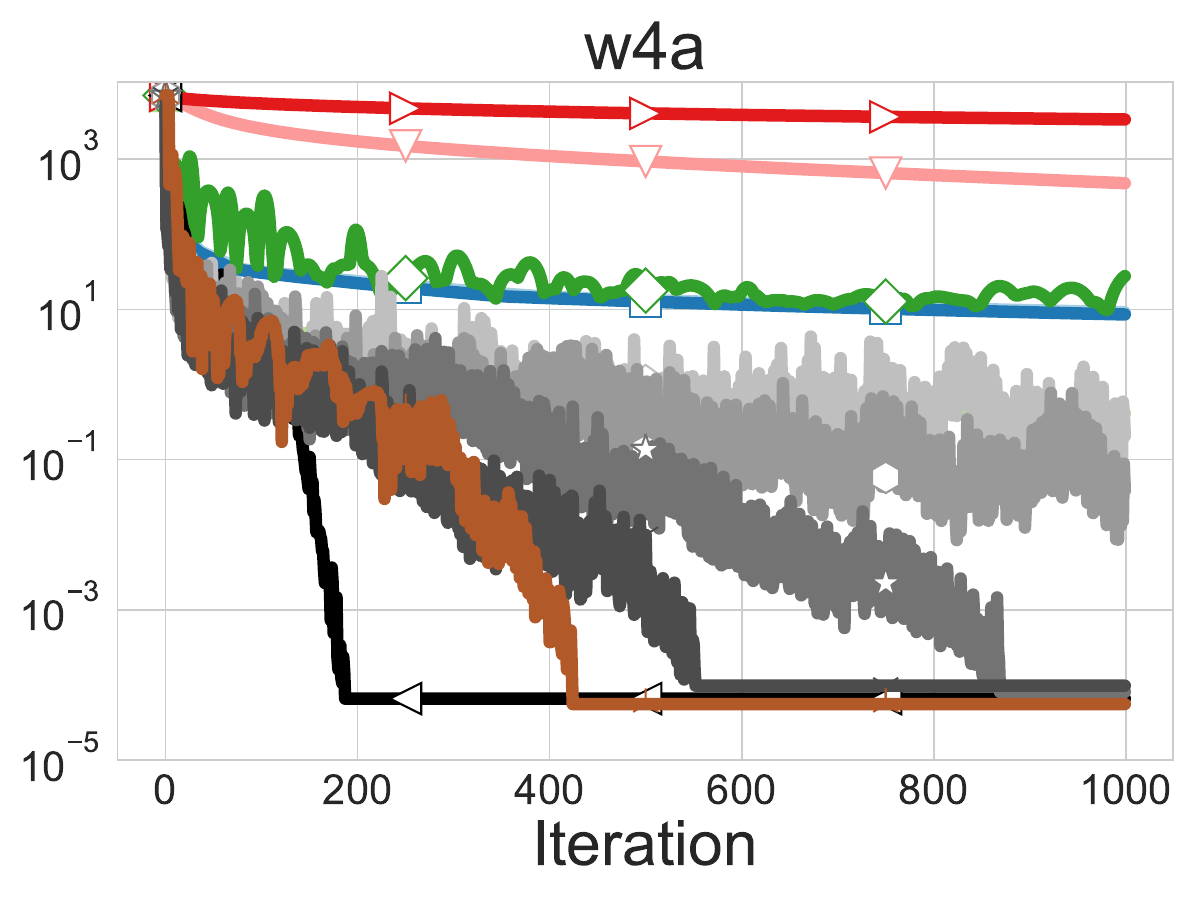}
\\
\includegraphics[scale=0.2]{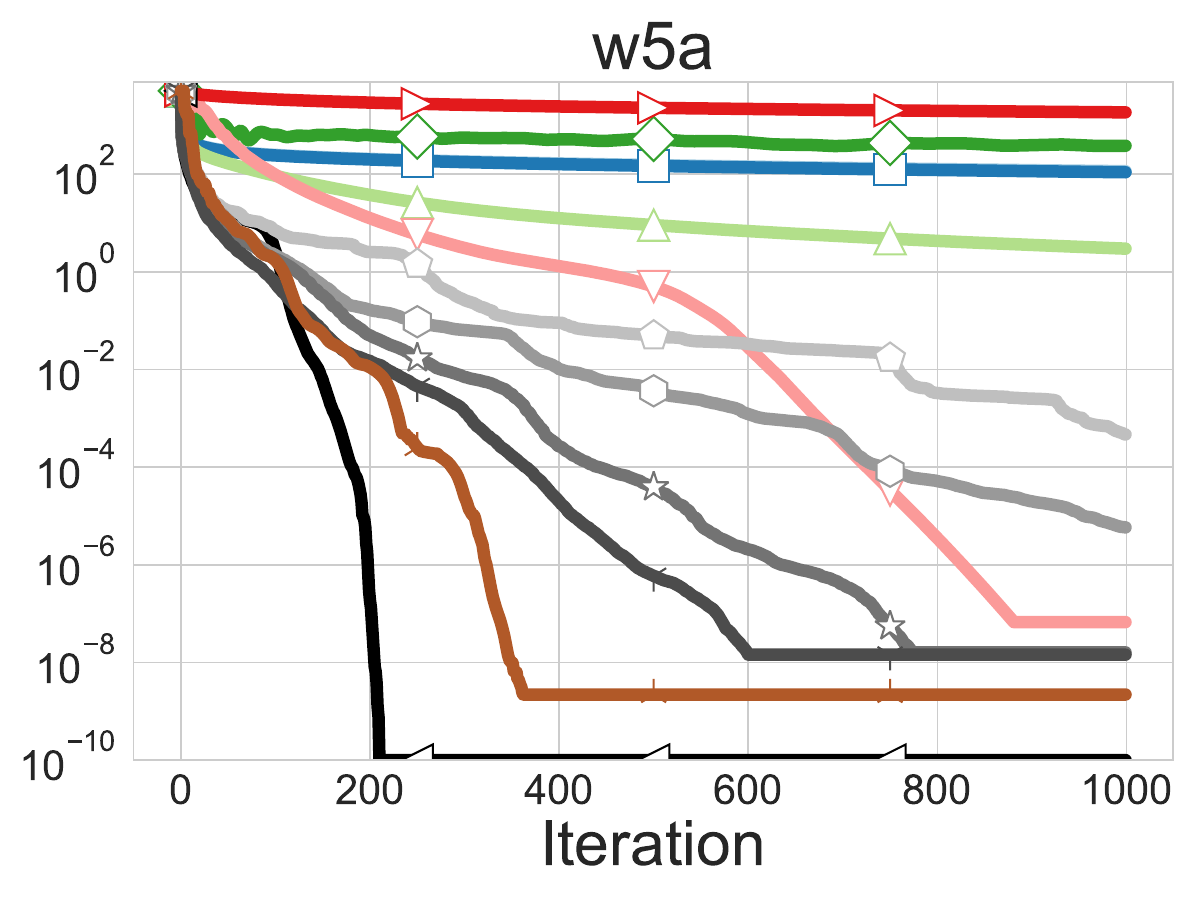}
\includegraphics[scale=0.2]{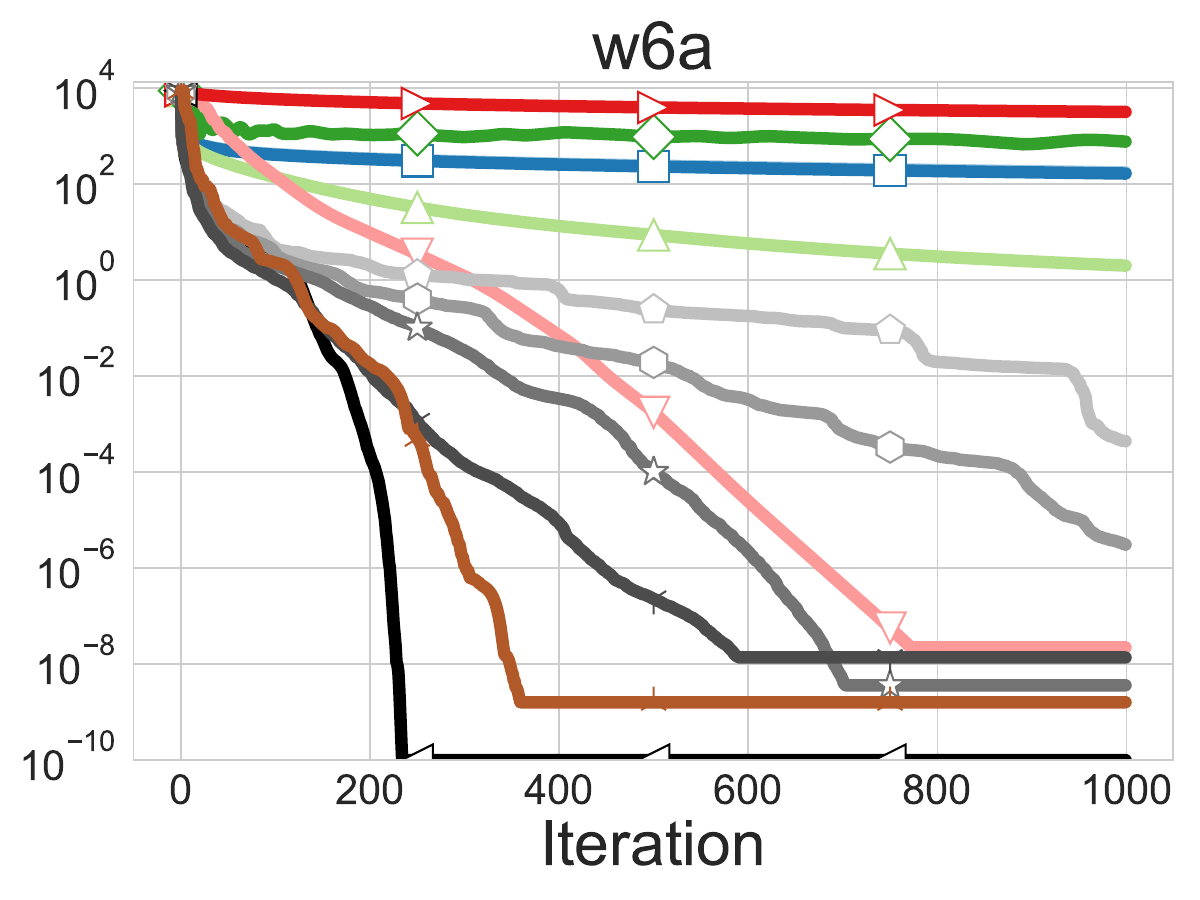}
\includegraphics[scale=0.2]{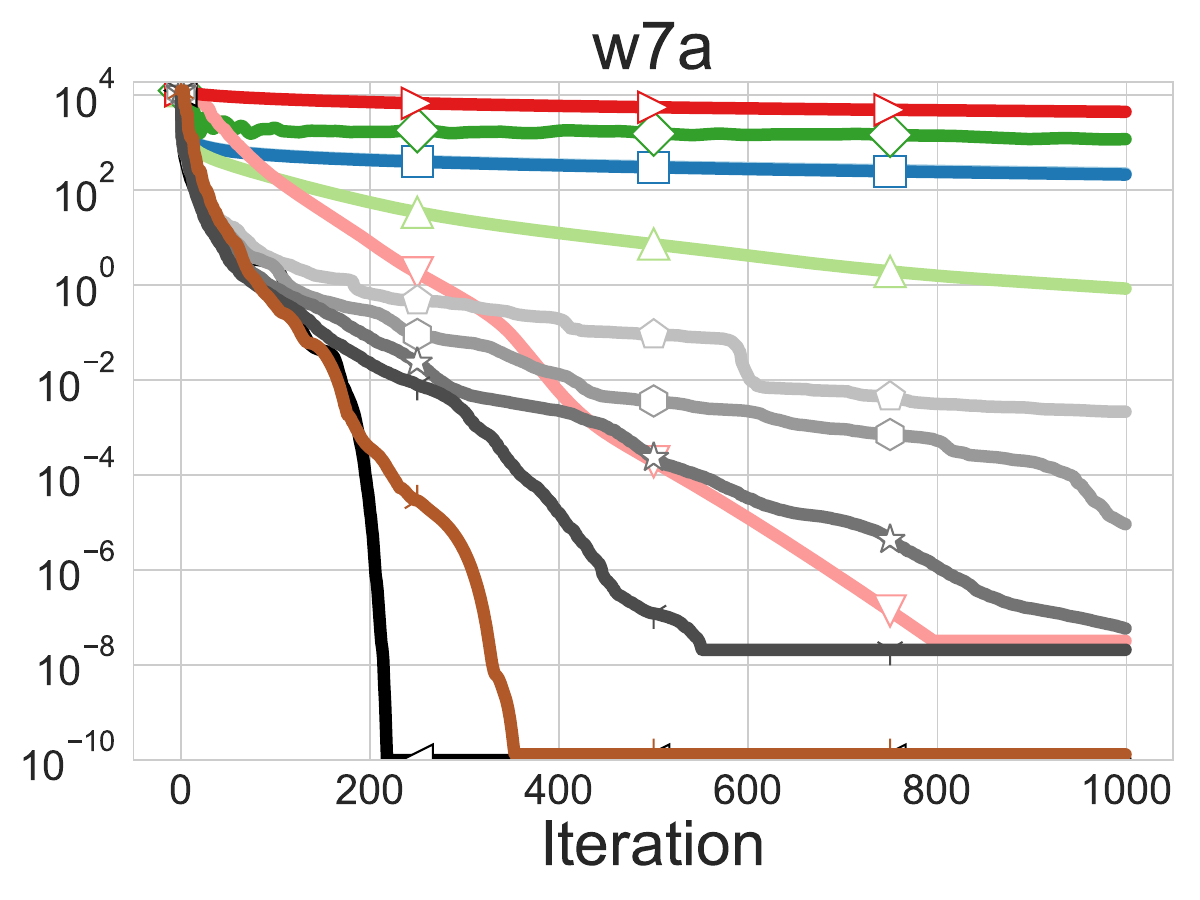}
\includegraphics[scale=0.2]{figs/w8a_objval_svm.pdf}
\\
\includegraphics[scale=0.2]{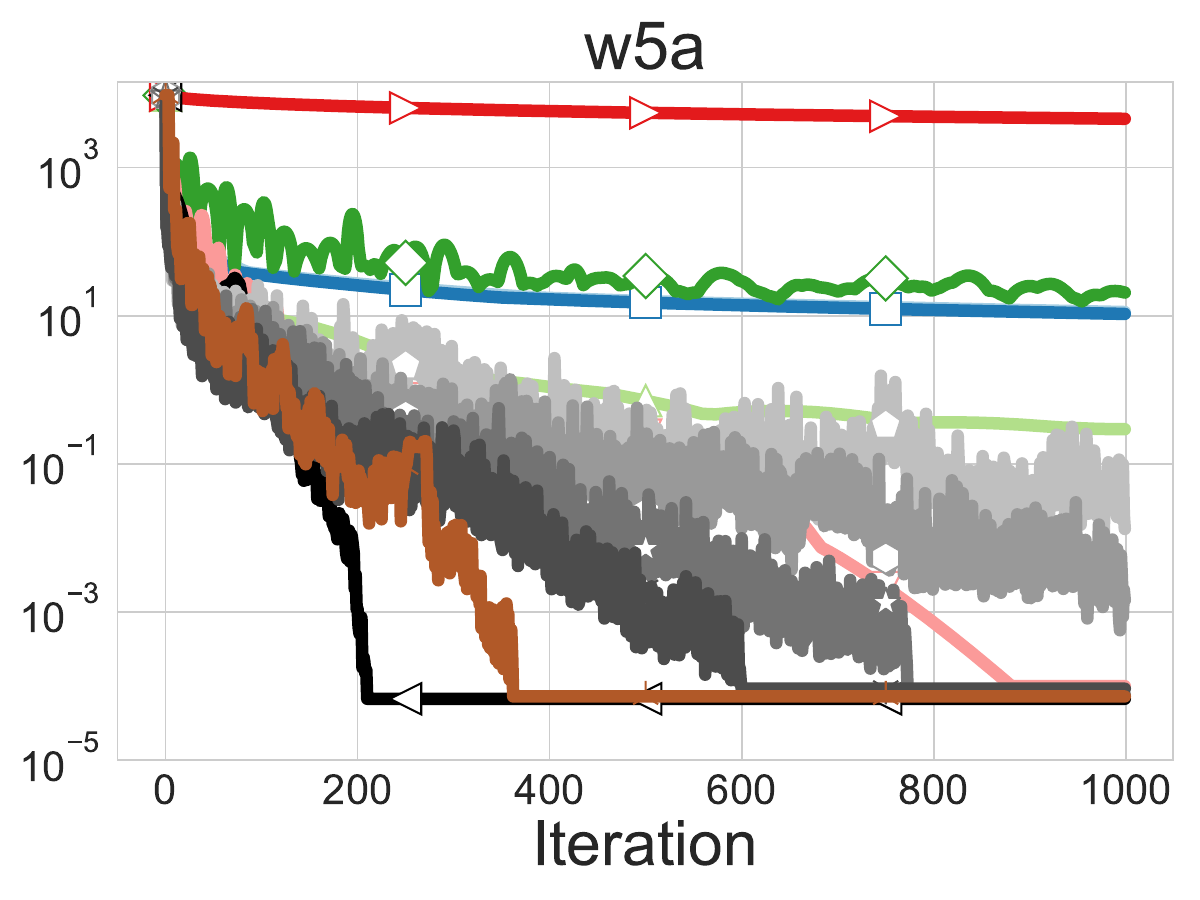}
\includegraphics[scale=0.2]{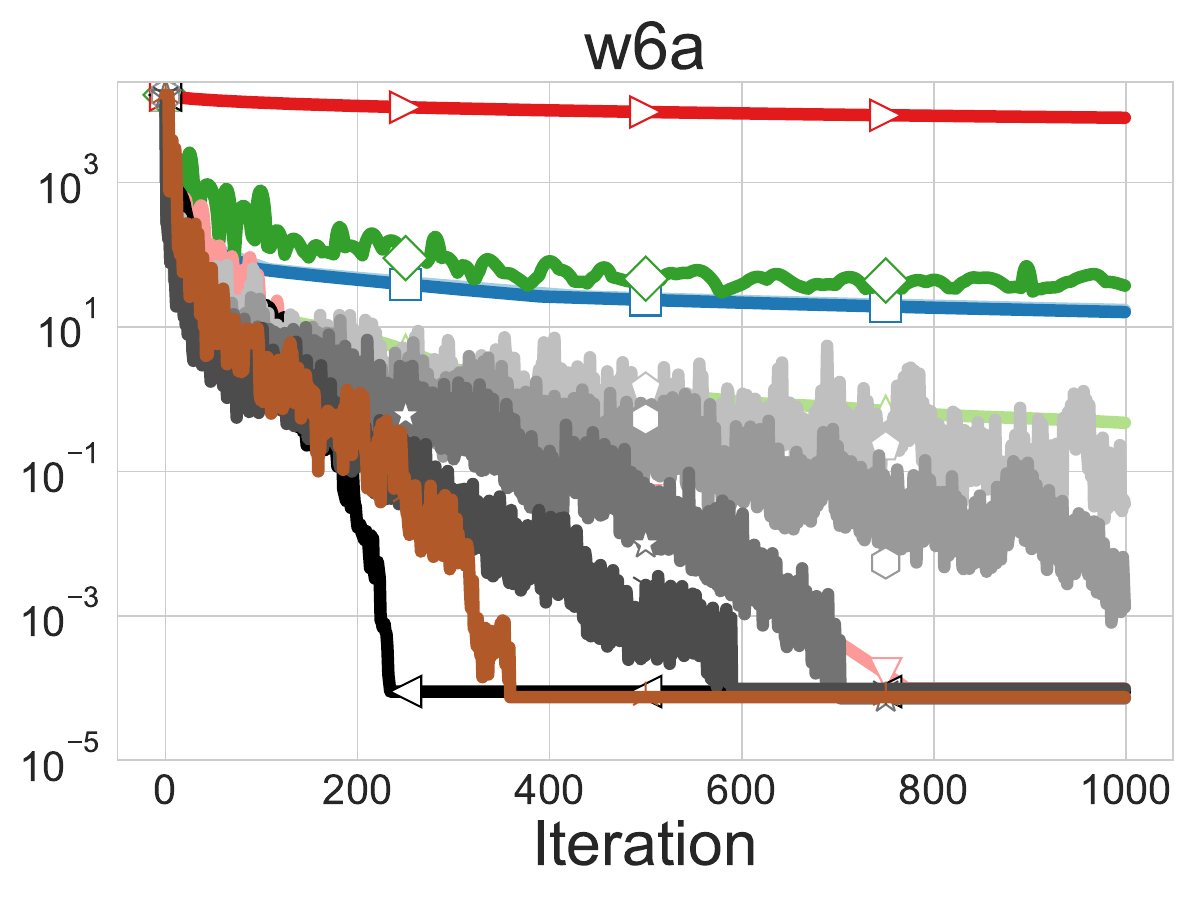}
\includegraphics[scale=0.2]{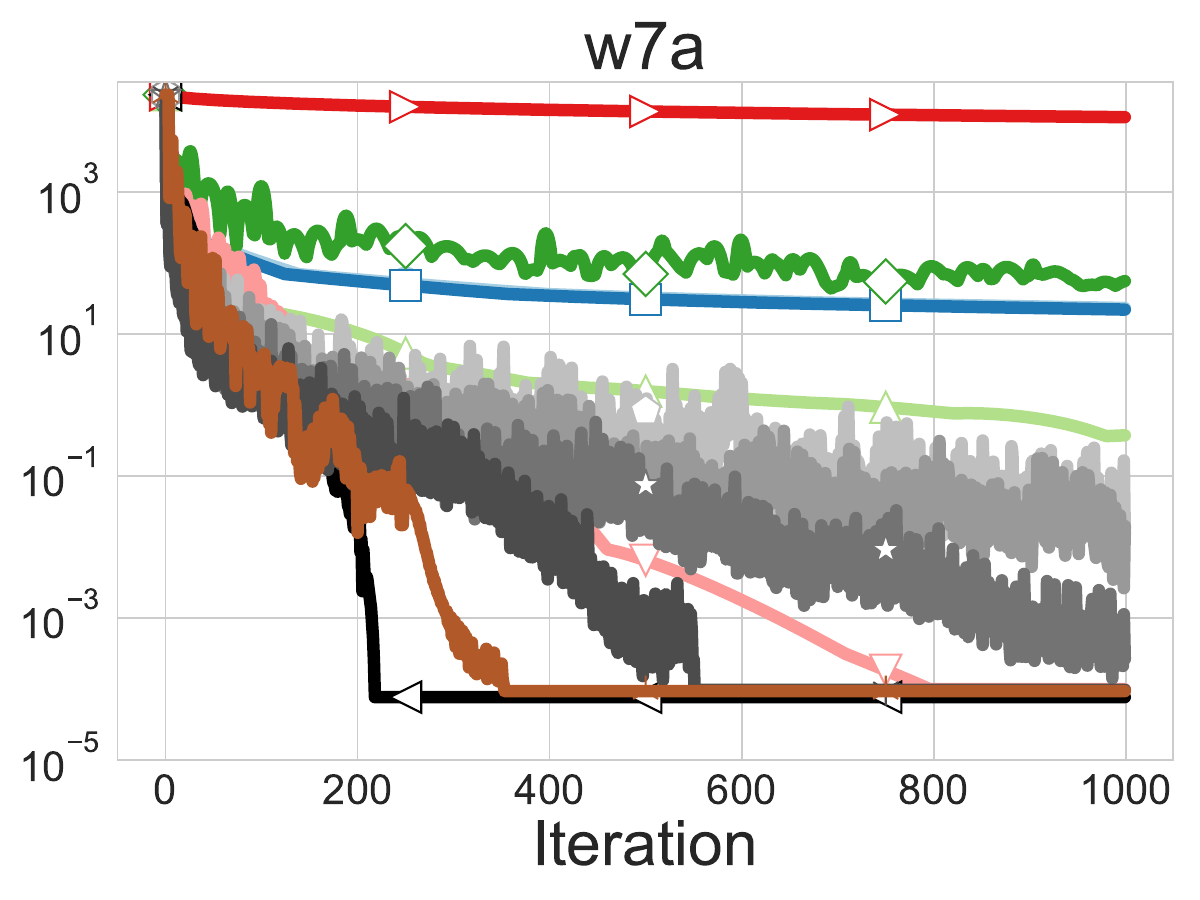}
\includegraphics[scale=0.2]{figs/w8a_gnorm_svm.pdf}
	
\includegraphics[scale=0.38]{figs/legend.pdf}
\caption{More experiments on support vector-machine problem}
\label{fig:svm-add-2}
\end{figure}

\subsection{Additional Experiments on Logistic Regression Problems}
See \Cref{fig:log-add-1} and \Cref{fig:log-add-2}.

\begin{figure}[!h]
\centering
\includegraphics[scale=0.2]{figs/a1a_objval_logistic.pdf}
\includegraphics[scale=0.2]{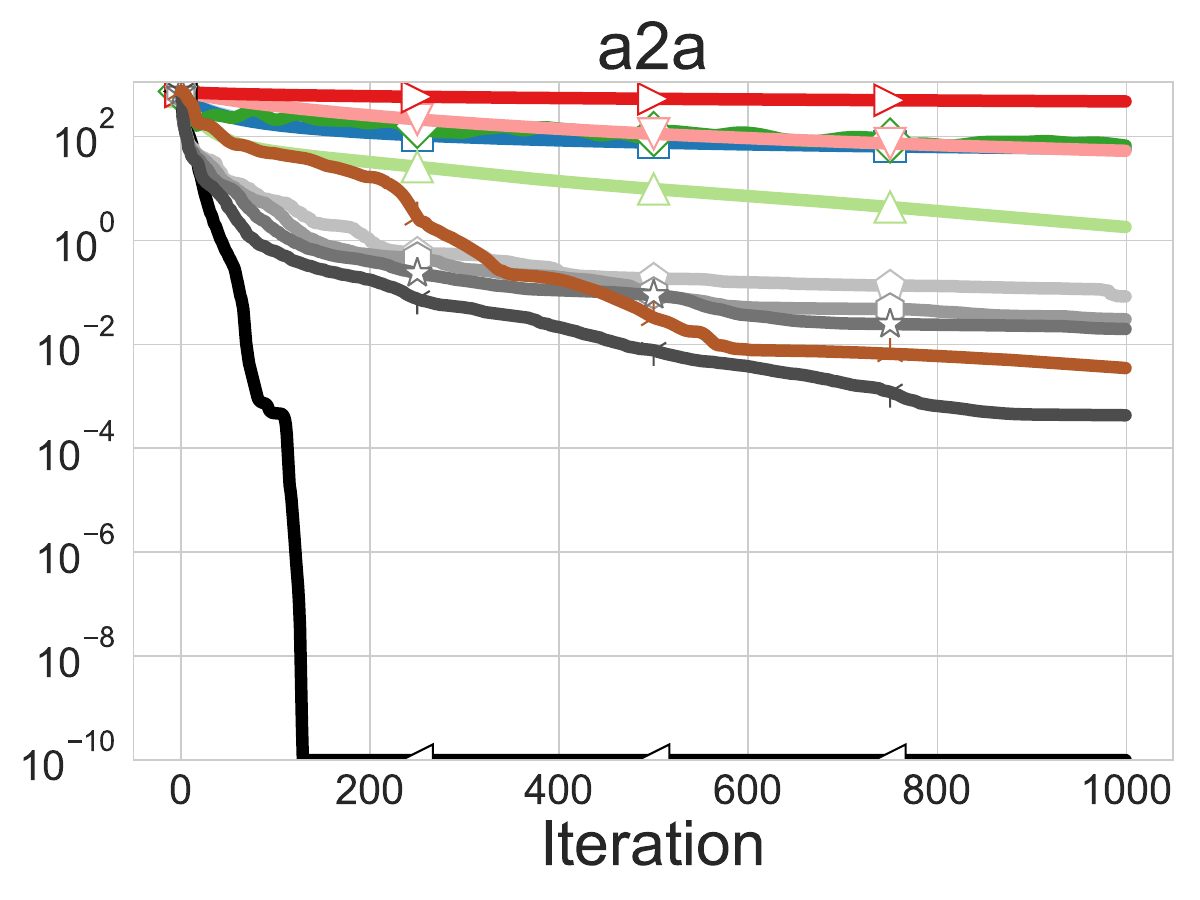}
\includegraphics[scale=0.2]{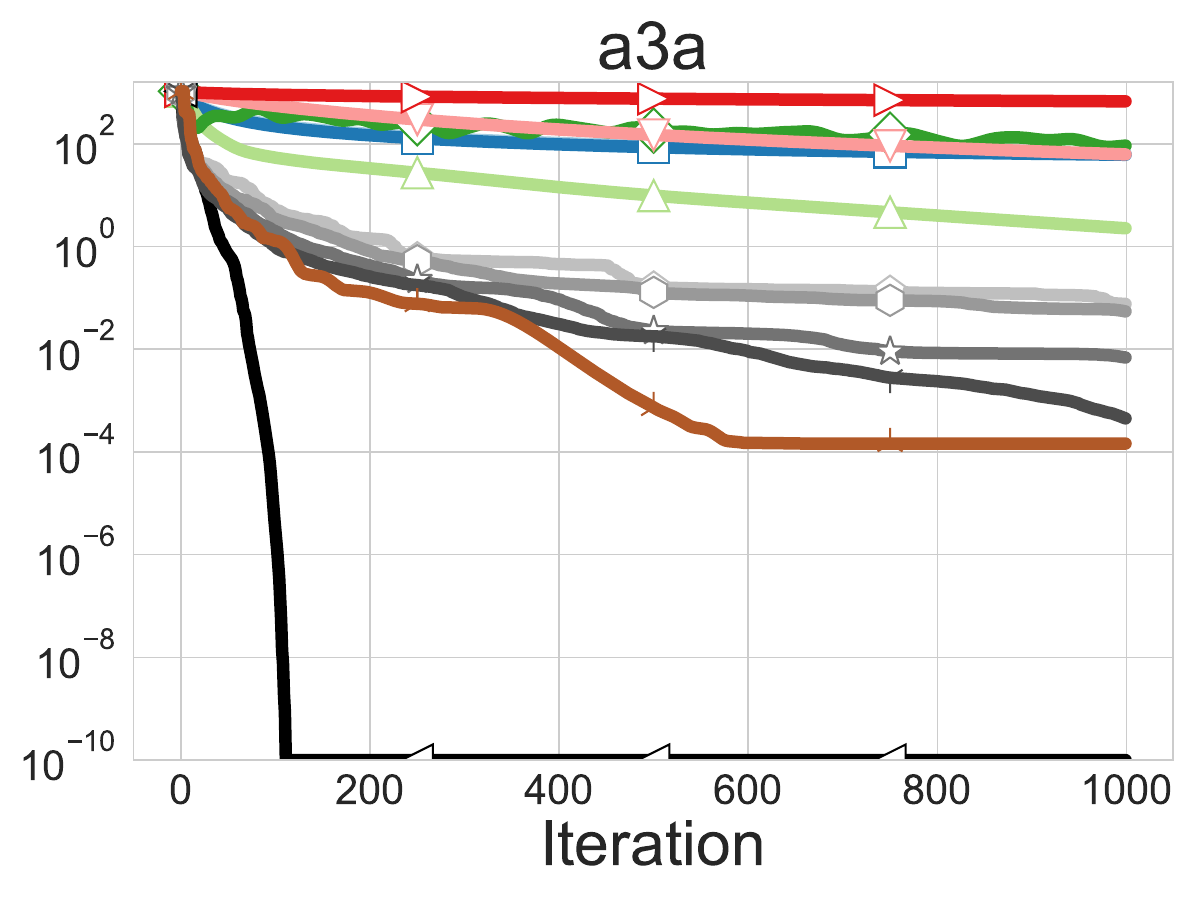}
\includegraphics[scale=0.2]{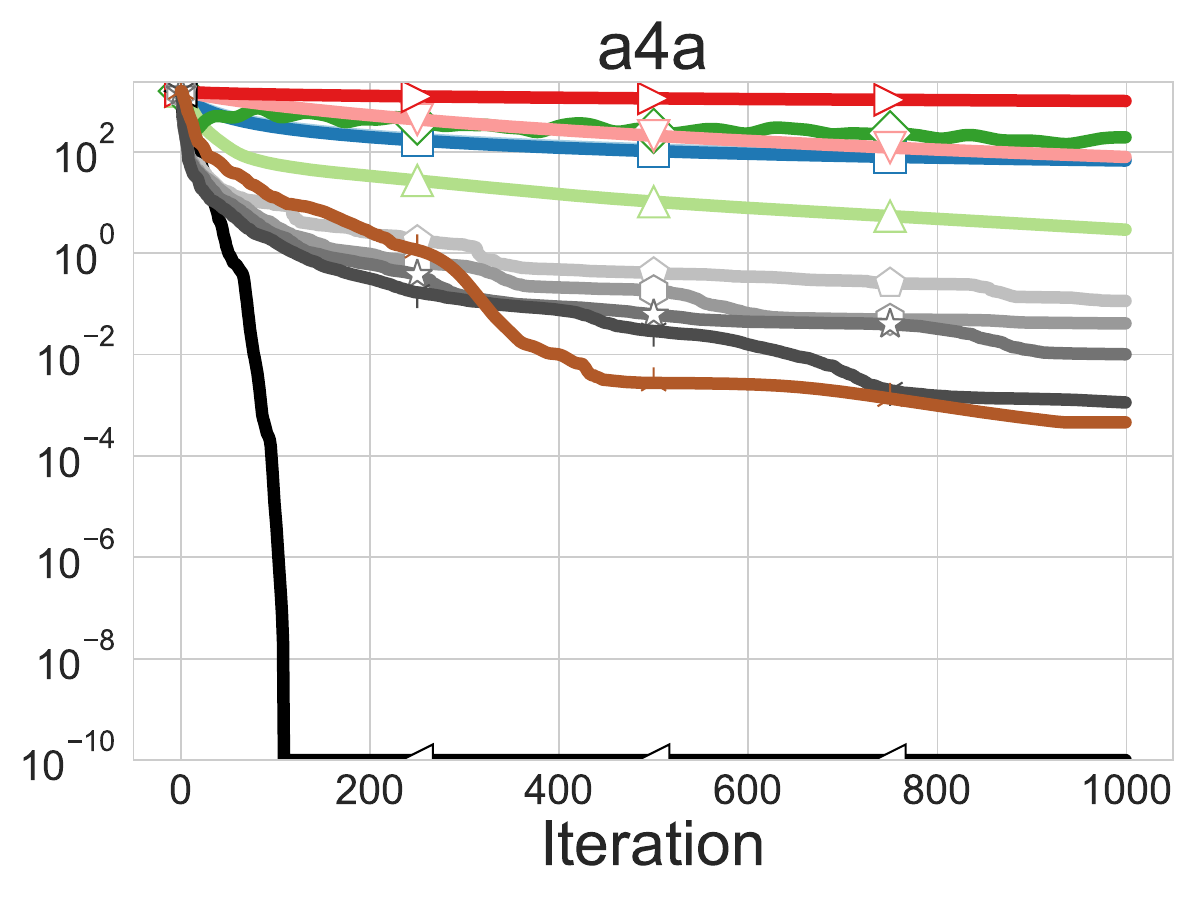}
\\
\includegraphics[scale=0.2]{figs/a1a_gnorm_logistic.pdf}
\includegraphics[scale=0.2]{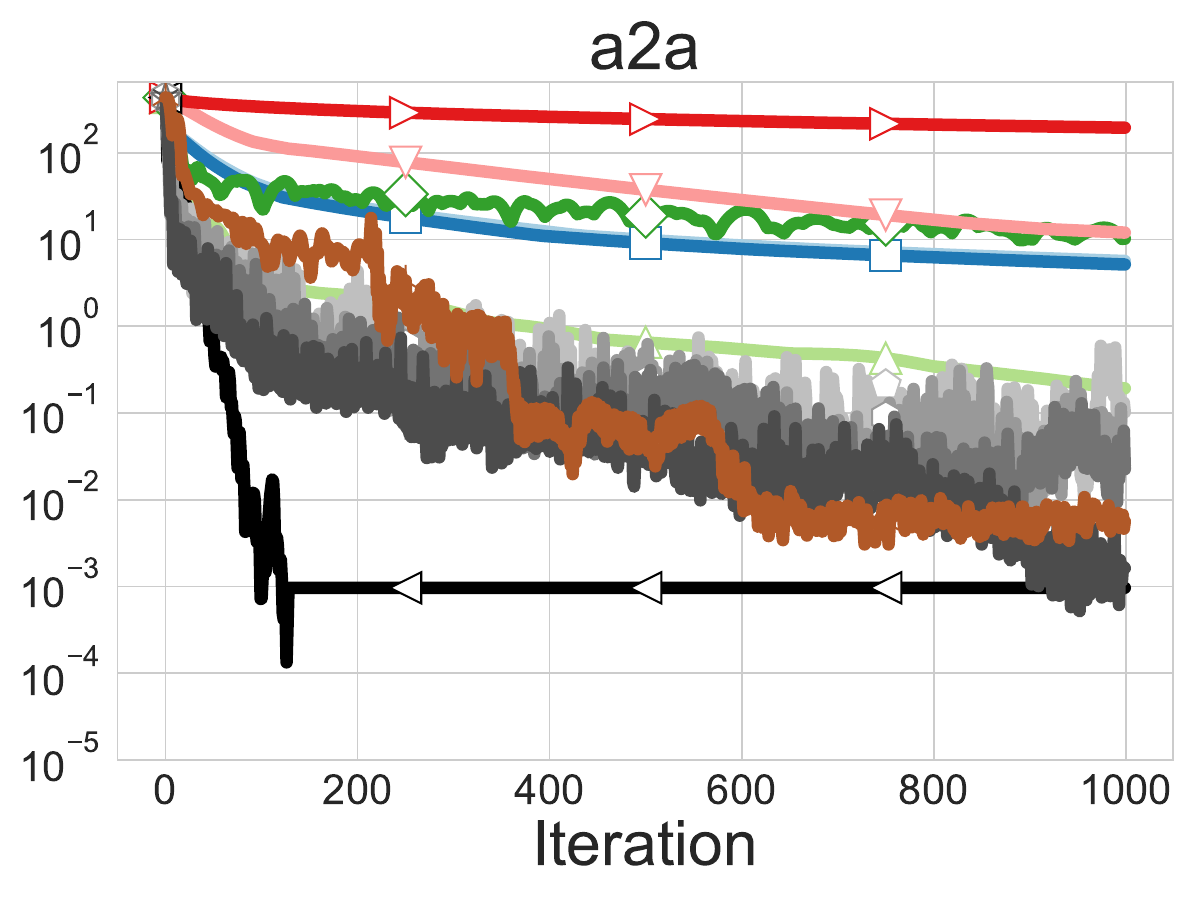}
\includegraphics[scale=0.2]{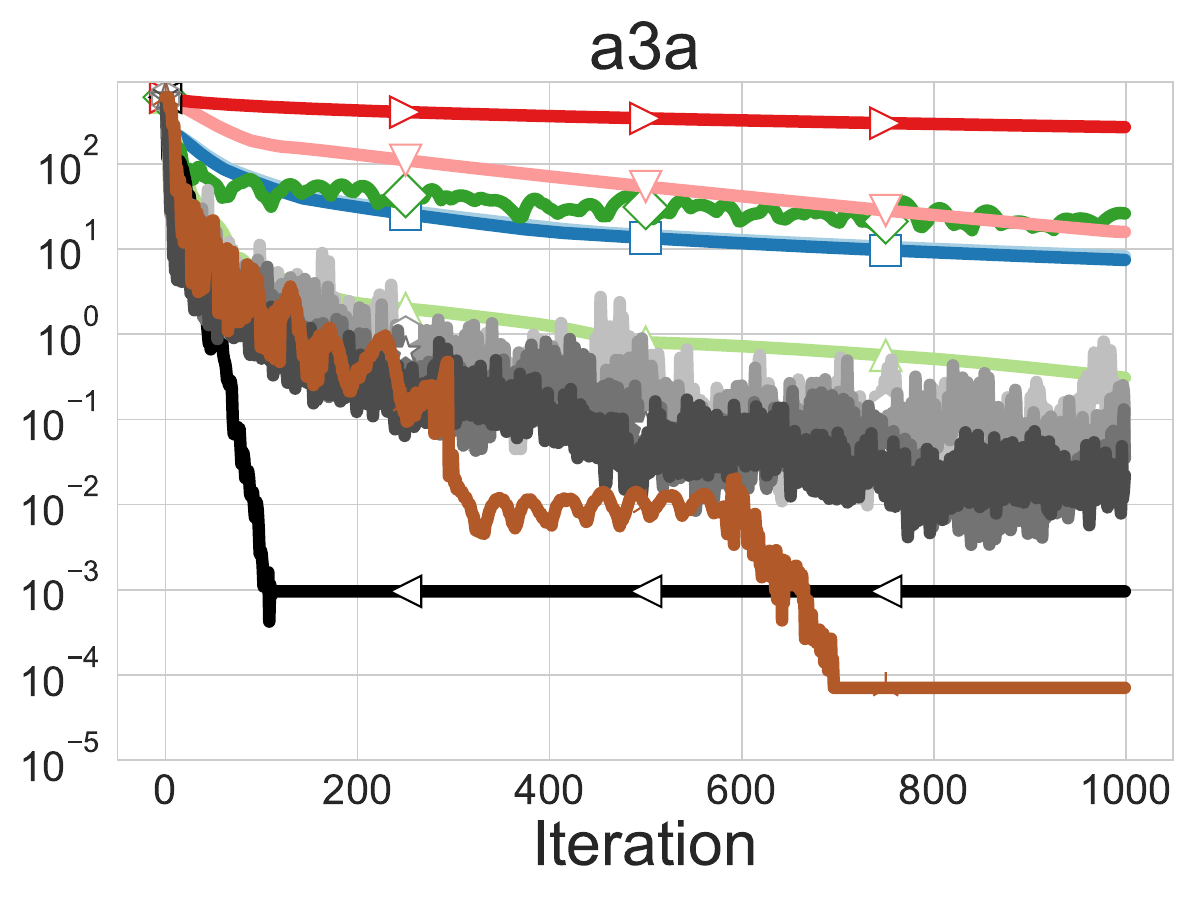}
\includegraphics[scale=0.2]{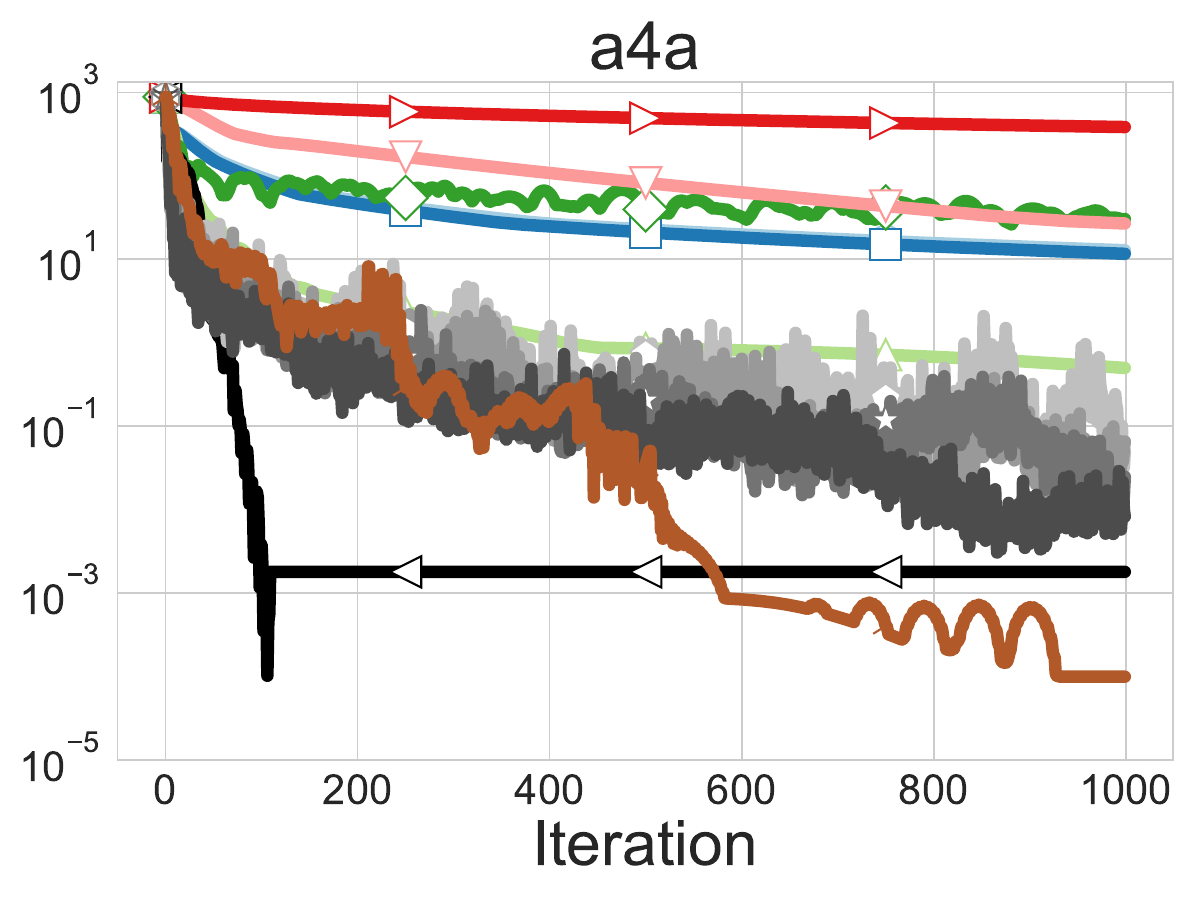}
\\
\includegraphics[scale=0.2]{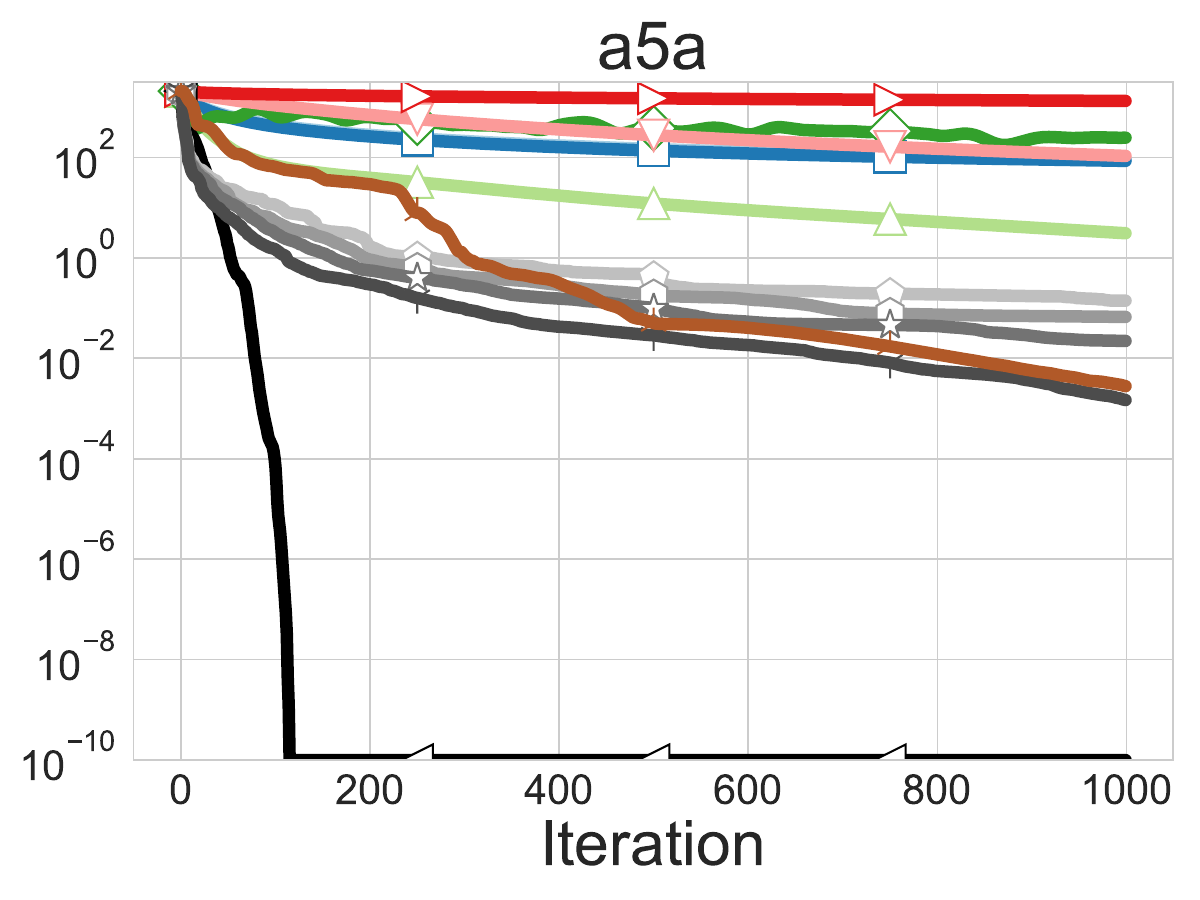}
\includegraphics[scale=0.2]{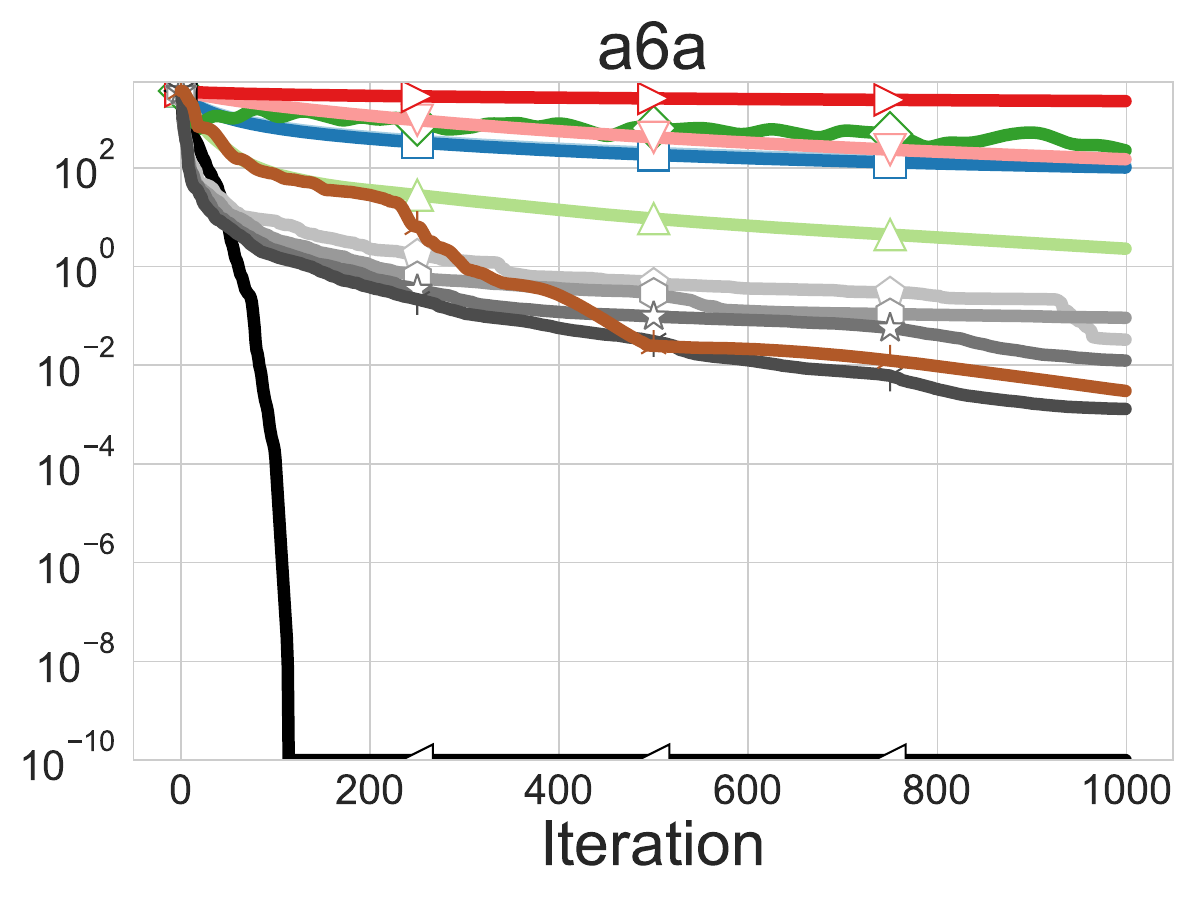}
\includegraphics[scale=0.2]{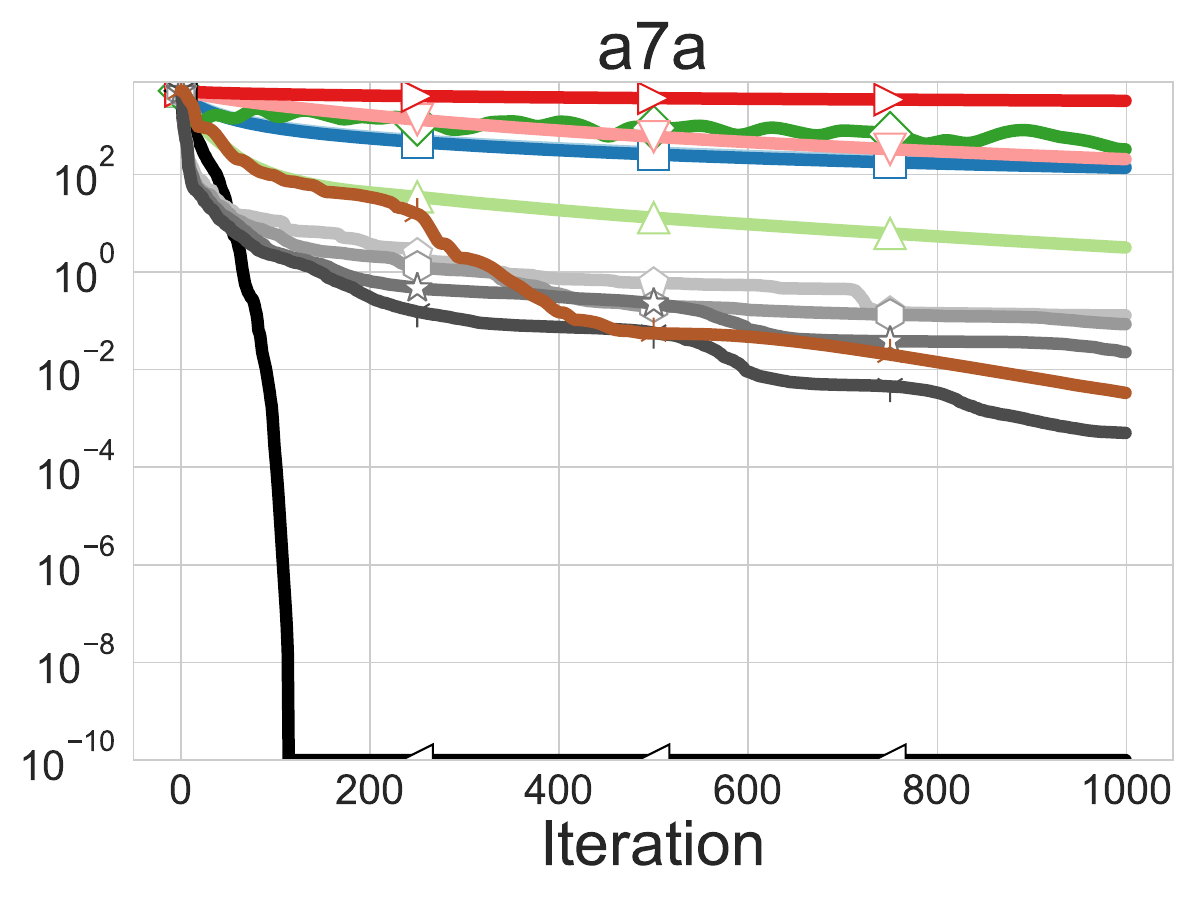}
\includegraphics[scale=0.2]{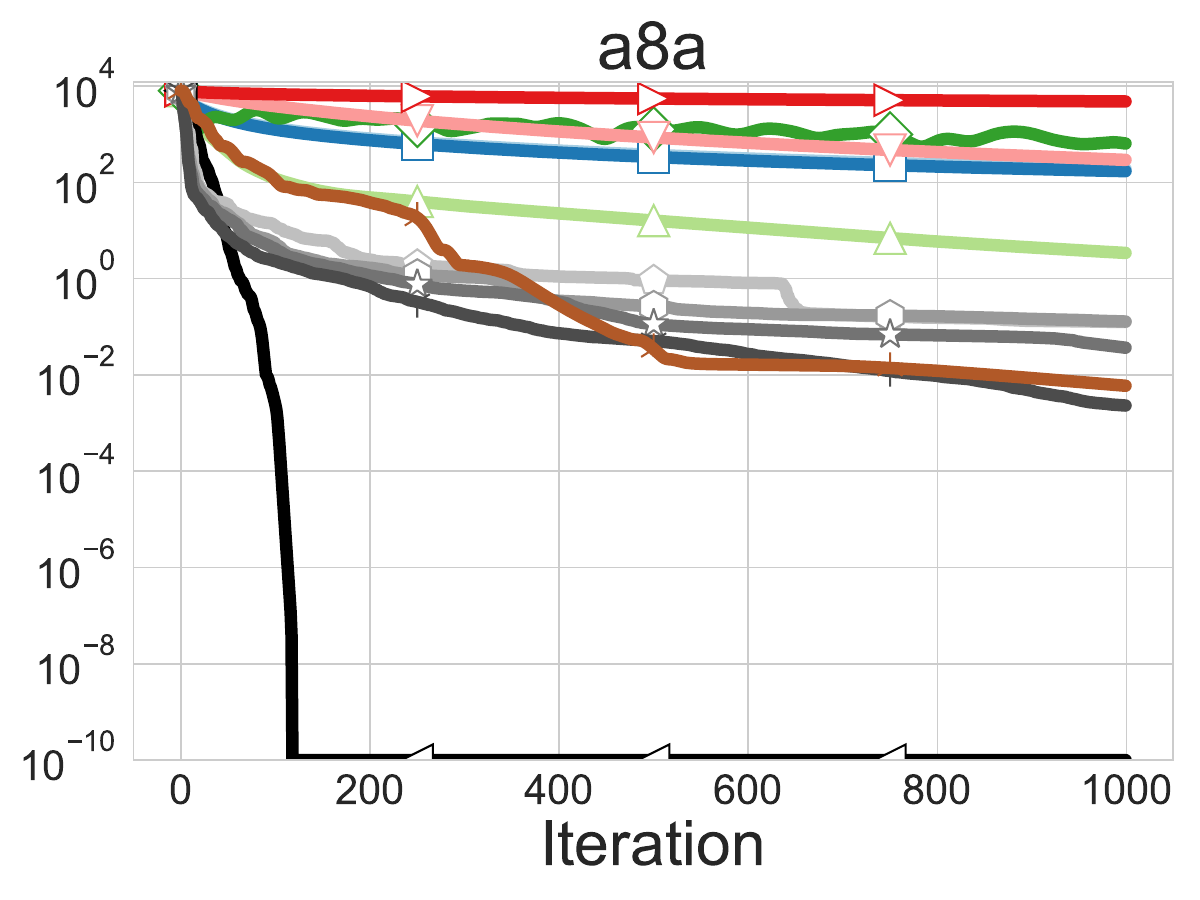}
\\
\includegraphics[scale=0.2]{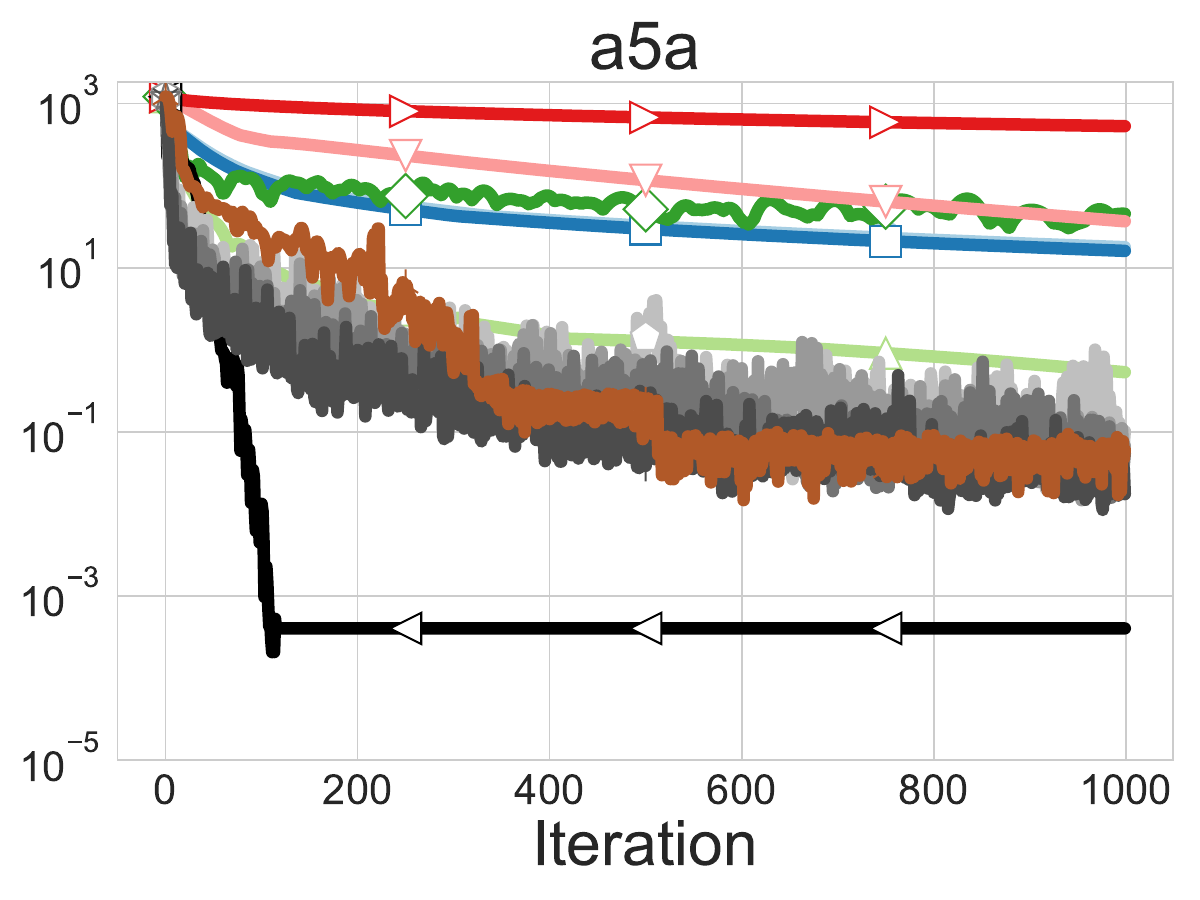}
\includegraphics[scale=0.2]{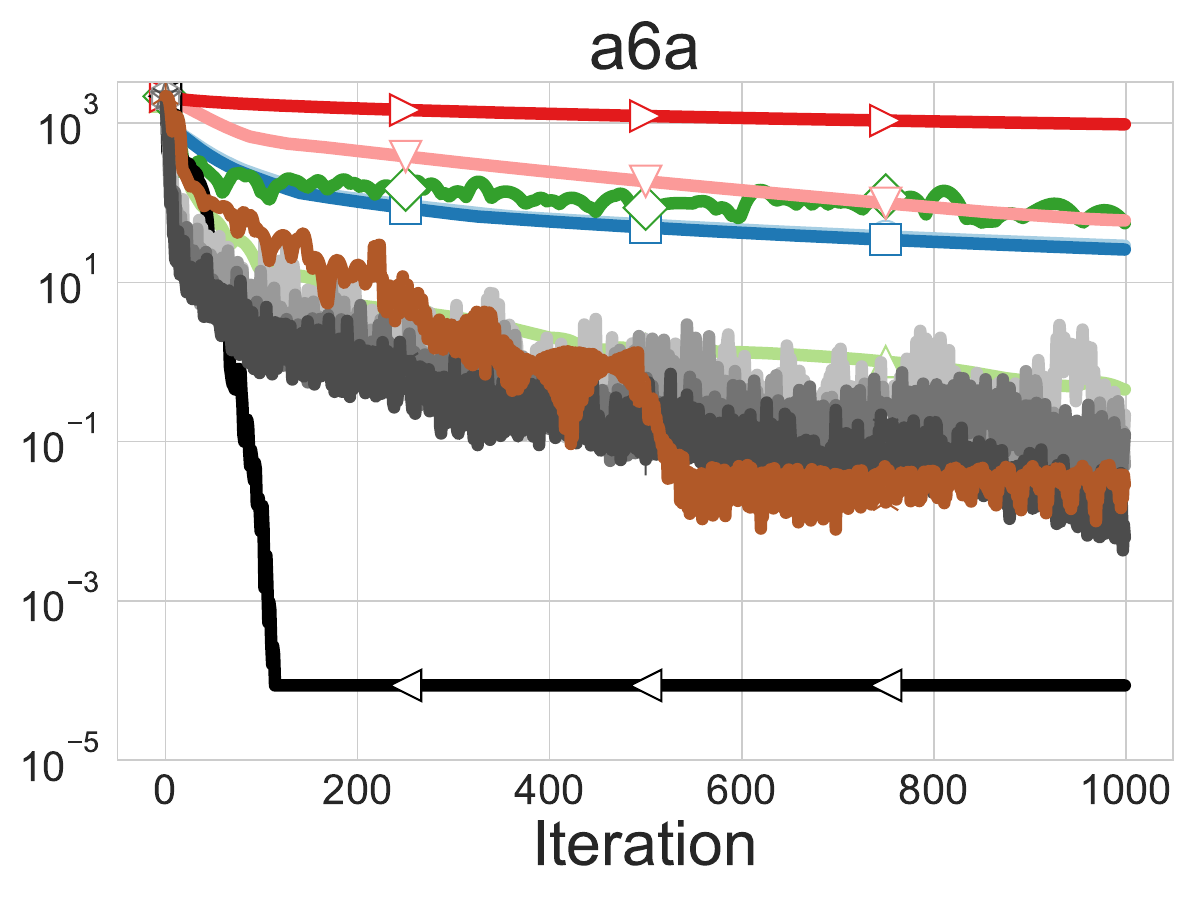}
\includegraphics[scale=0.2]{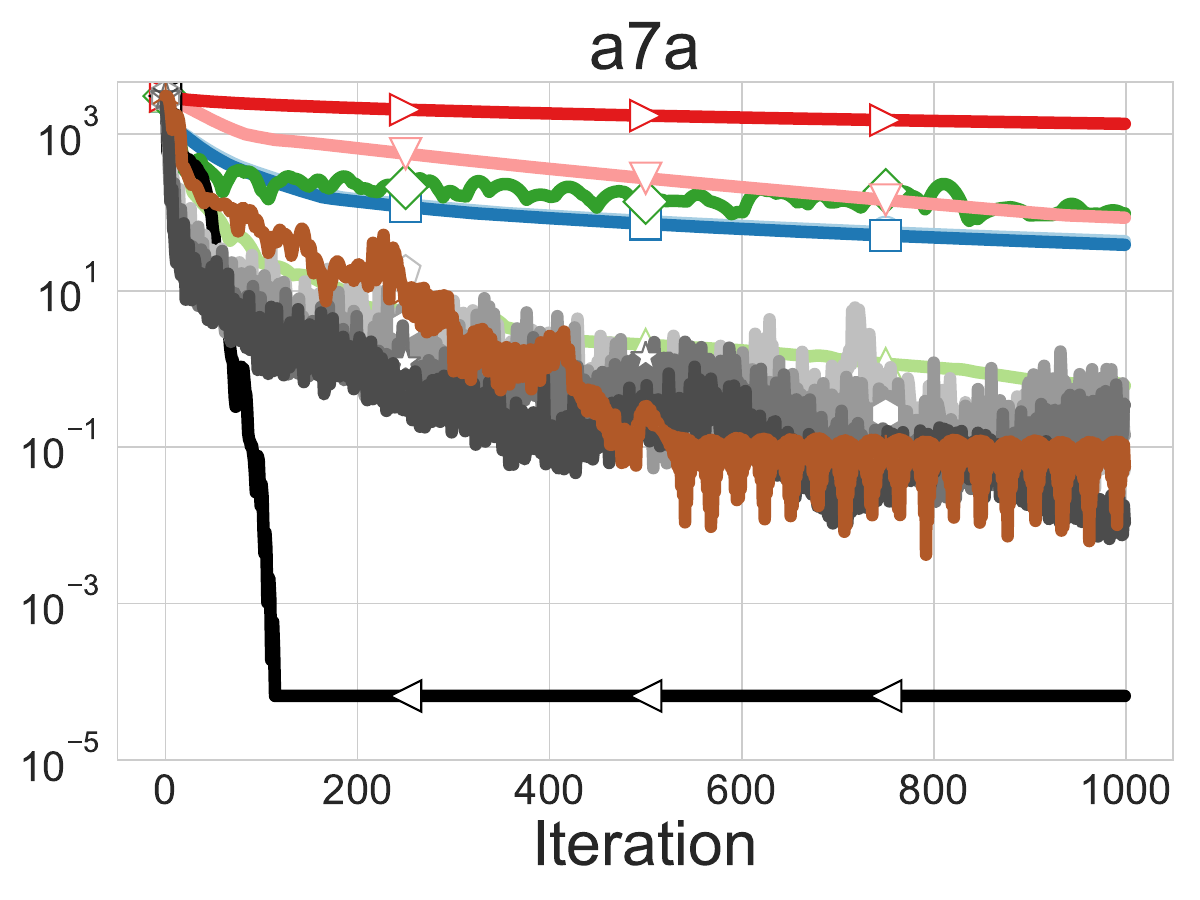}
\includegraphics[scale=0.2]{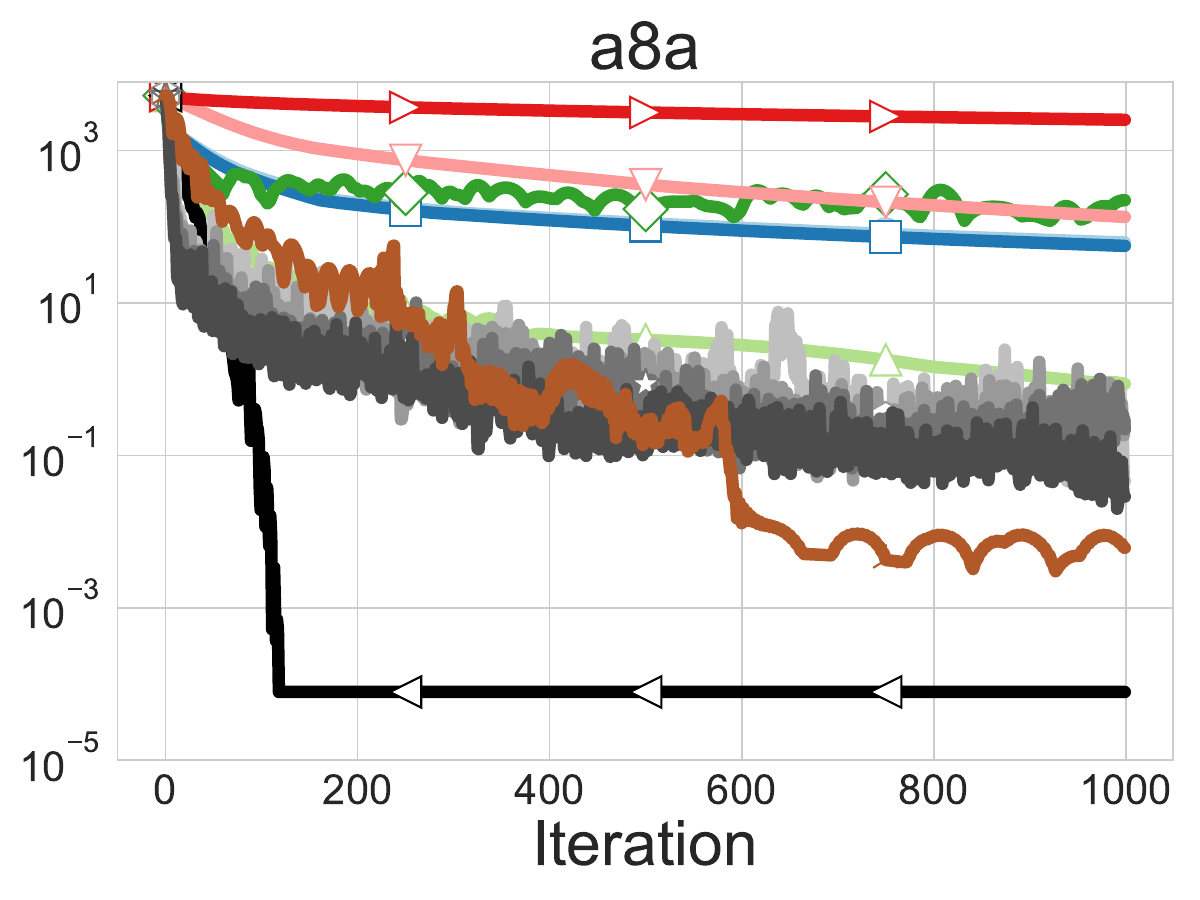}
\\
\includegraphics[scale=0.2]{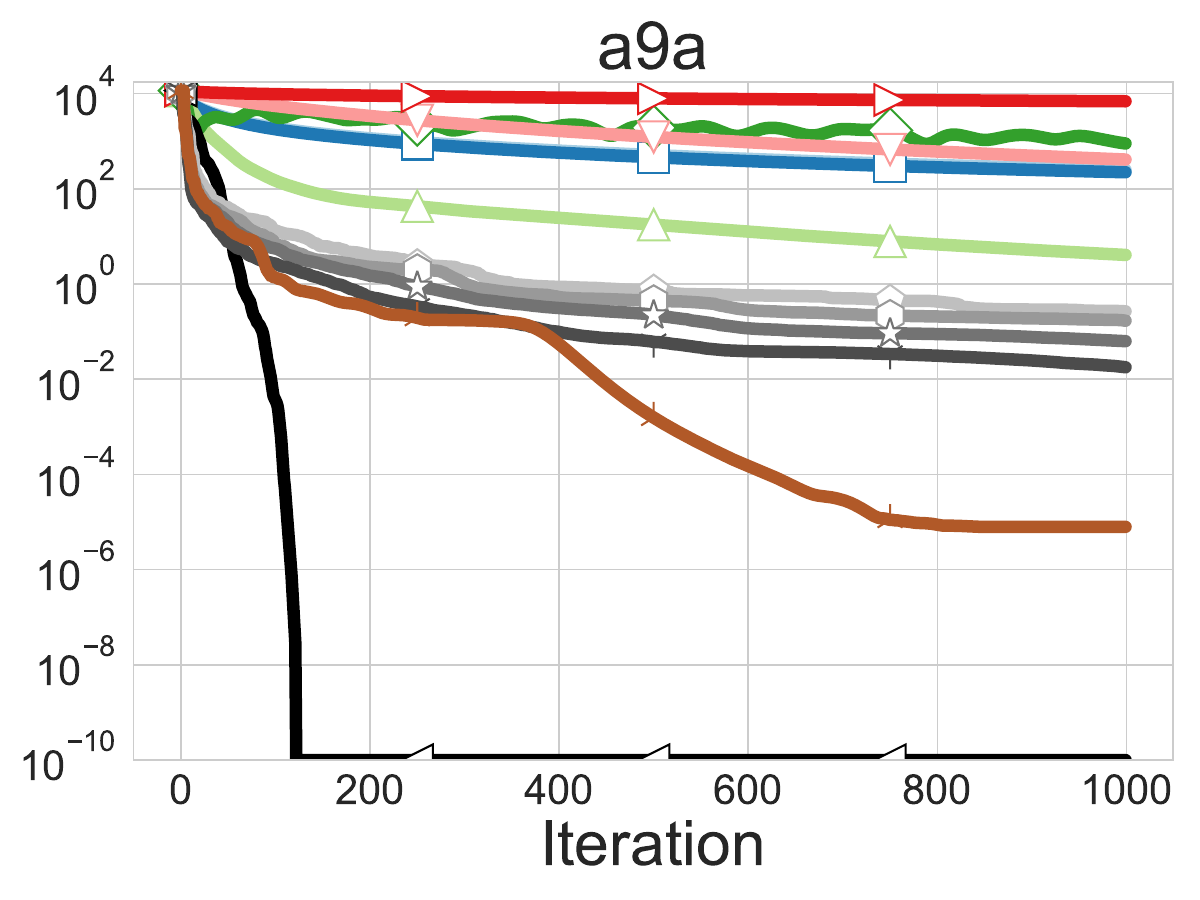}
\includegraphics[scale=0.2]{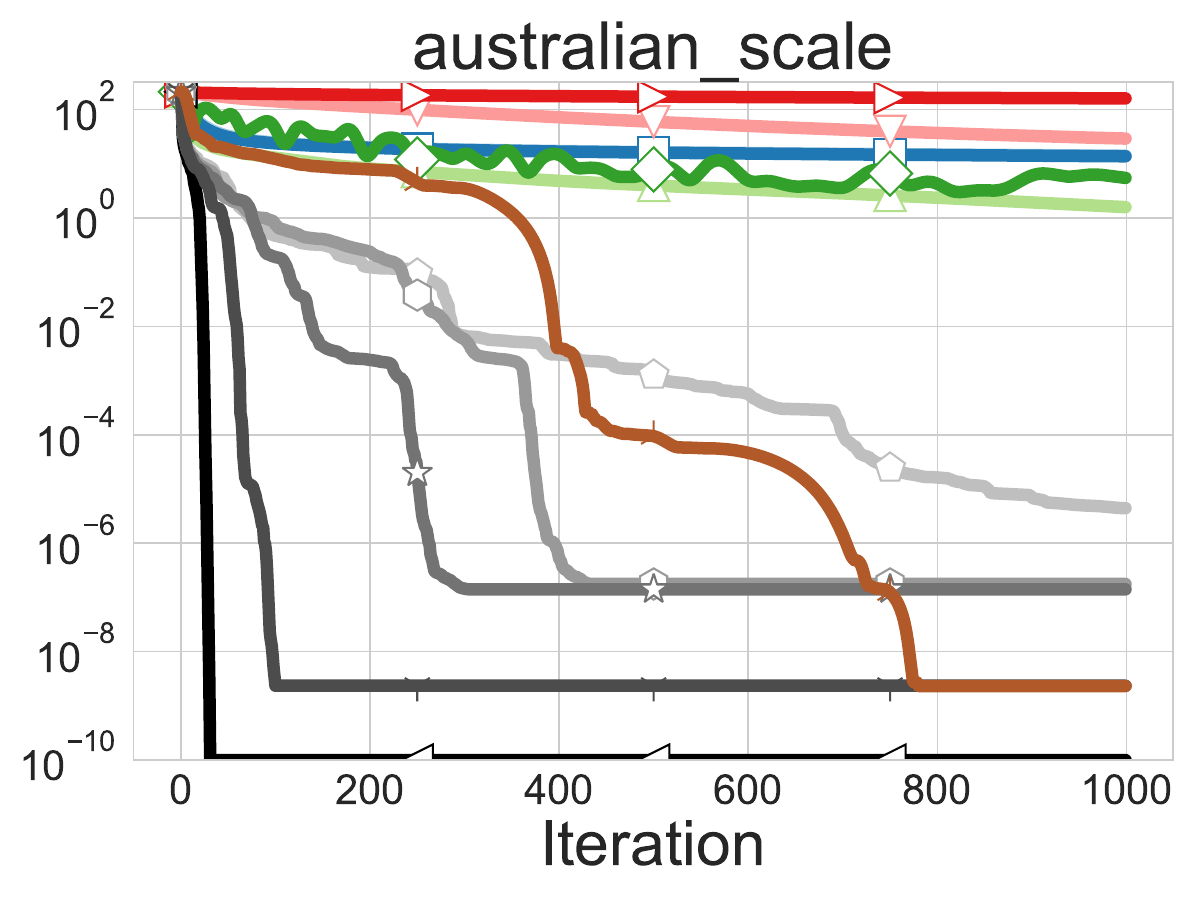}
\includegraphics[scale=0.2]{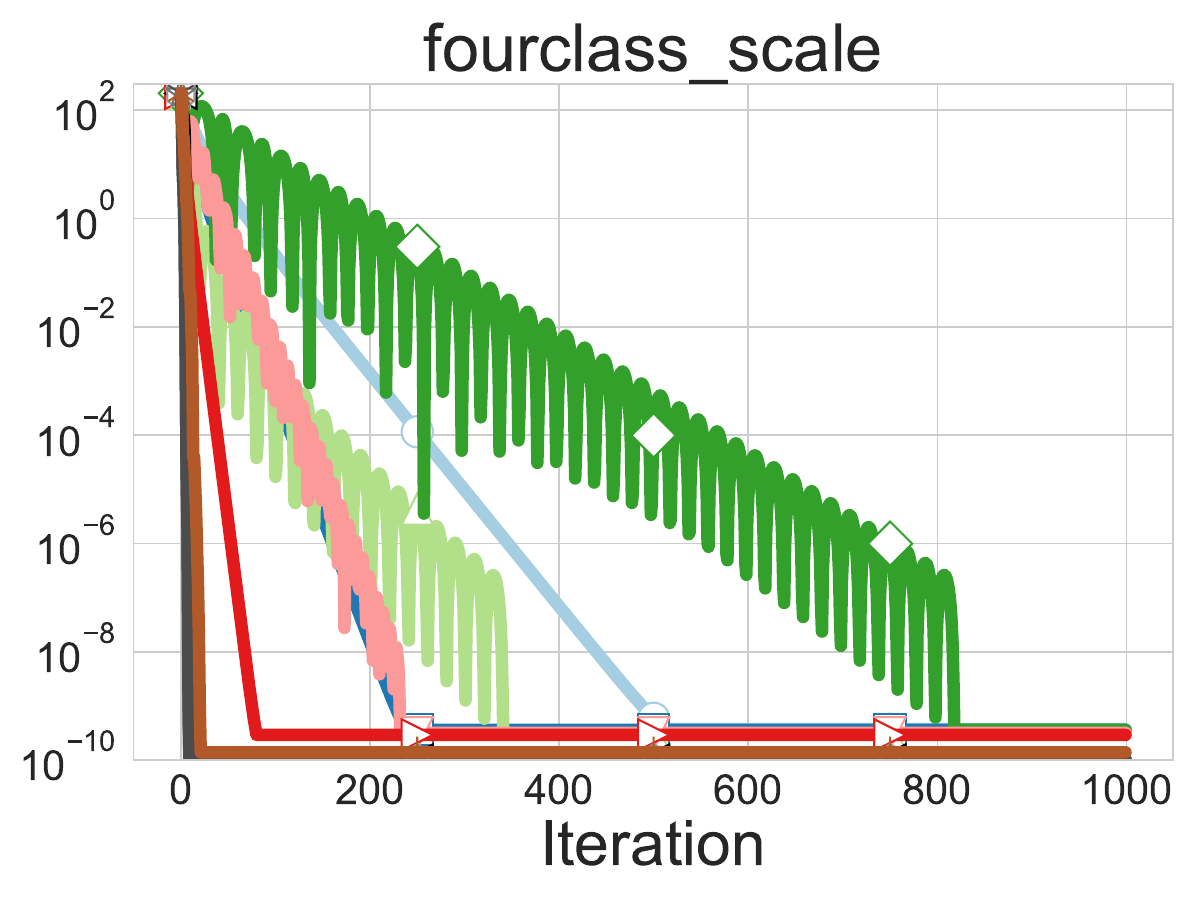}
\includegraphics[scale=0.2]{figs/ijcnn1_objval_logistic.pdf}
\\
\includegraphics[scale=0.2]{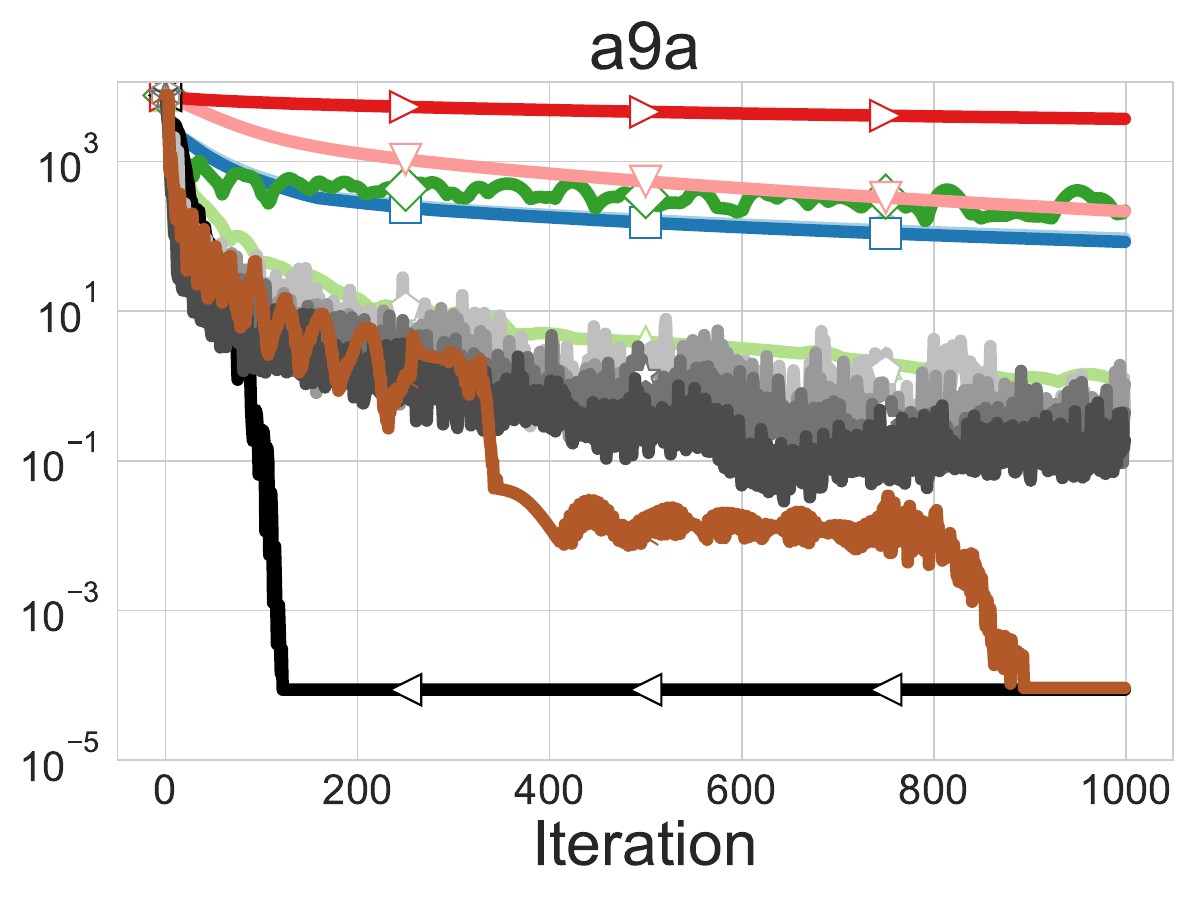}
\includegraphics[scale=0.2]{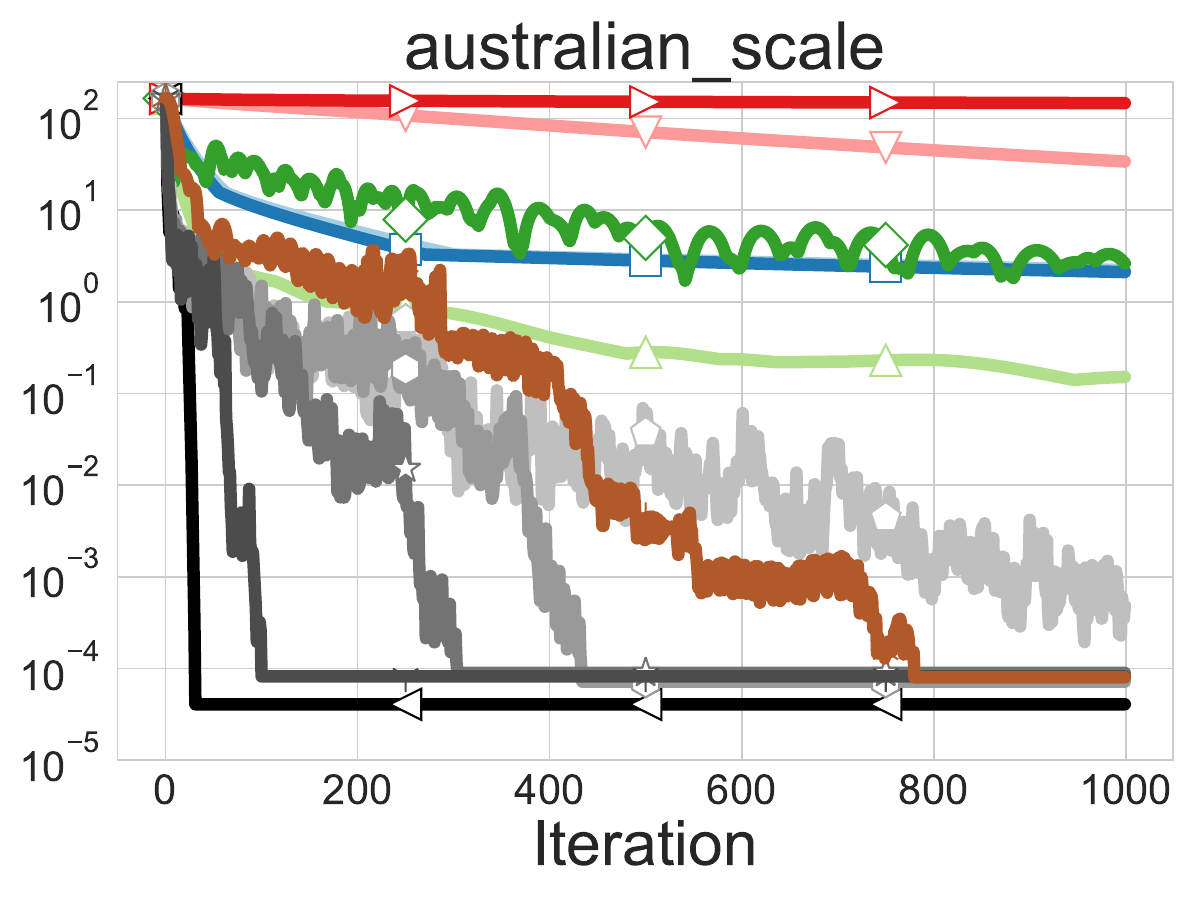}
\includegraphics[scale=0.2]{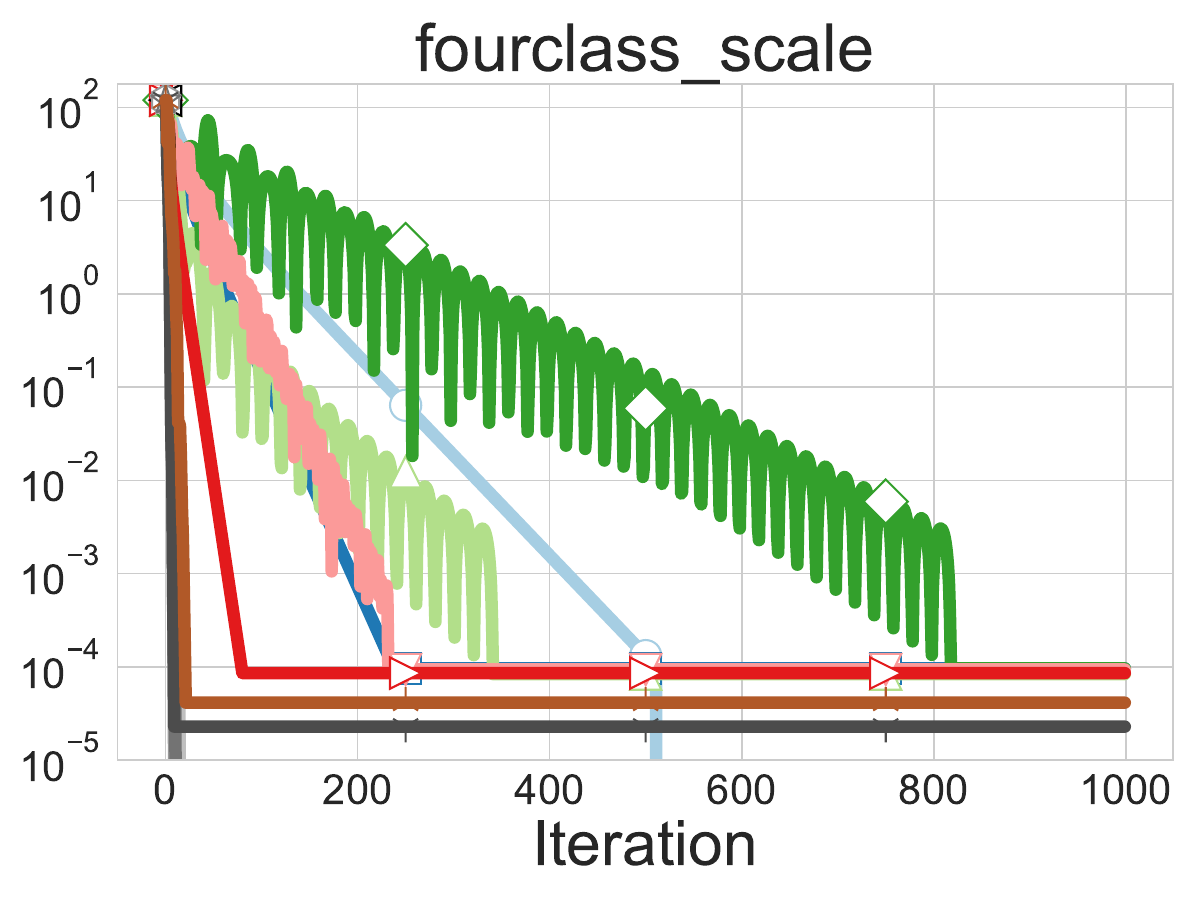}
\includegraphics[scale=0.2]{figs/ijcnn1_gnorm_logistic.pdf}
\\
\includegraphics[scale=0.38]{figs/legend.pdf}
\caption{More experiments on logistic regression problem}
\label{fig:log-add-1}
\end{figure}

\begin{figure}

\centering
\includegraphics[scale=0.2]{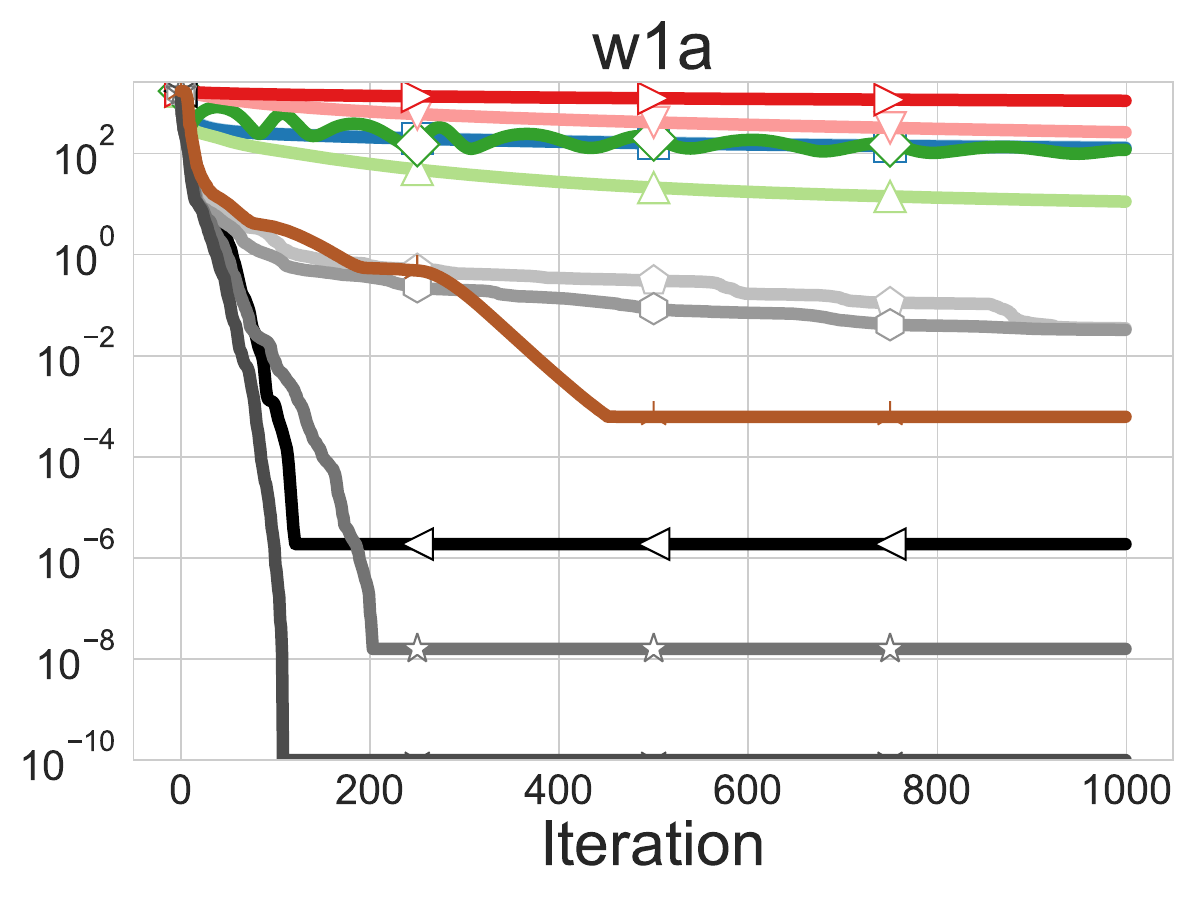}
\includegraphics[scale=0.2]{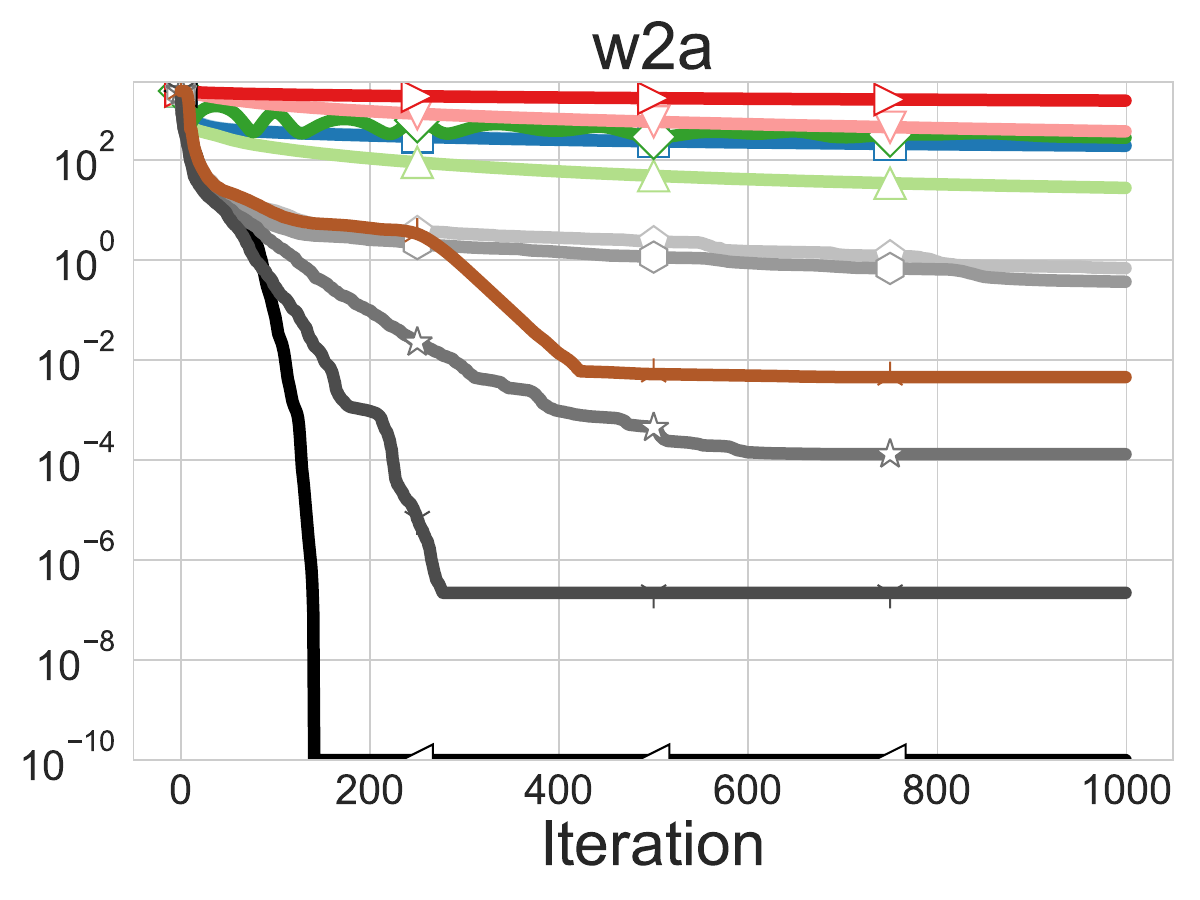}
\includegraphics[scale=0.2]{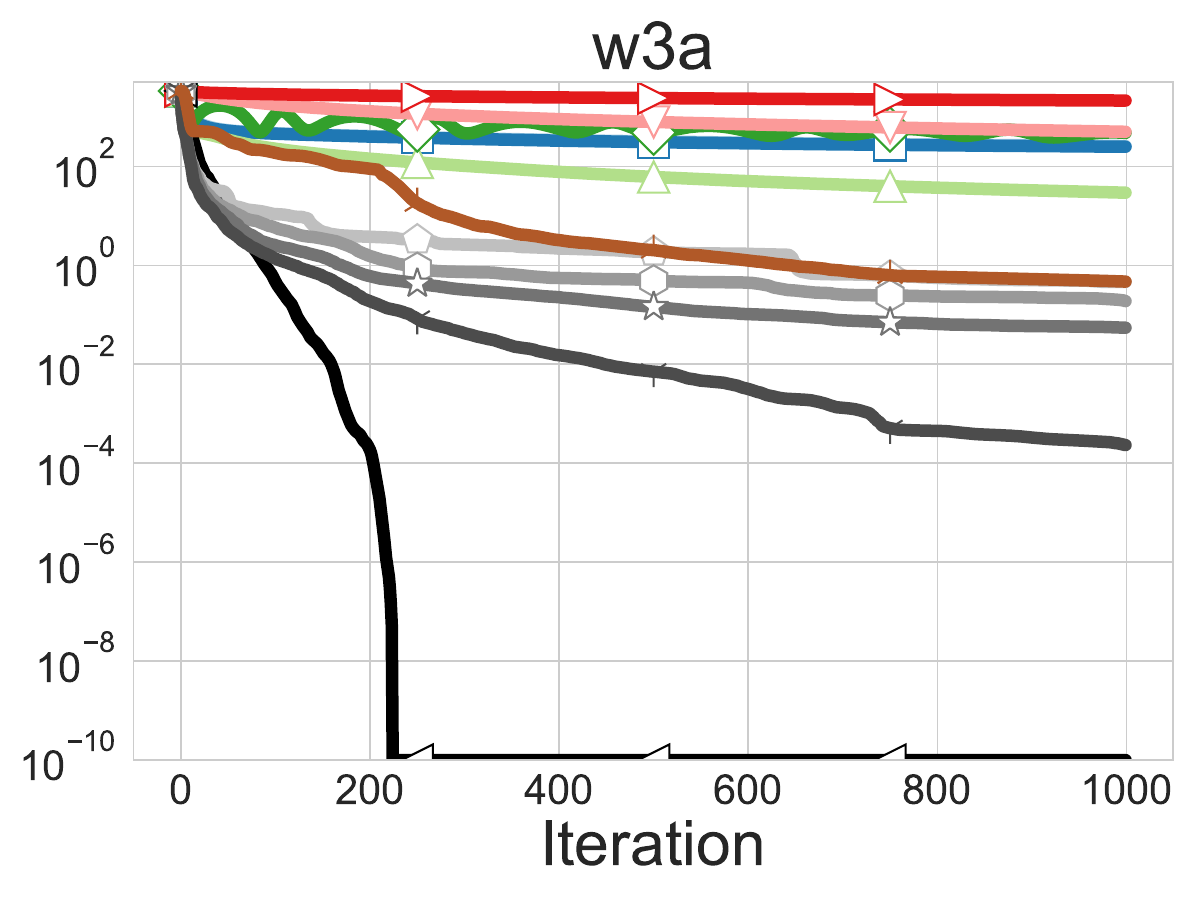}
\includegraphics[scale=0.2]{figs/w4a_objval_logistic.pdf}
\\
\includegraphics[scale=0.2]{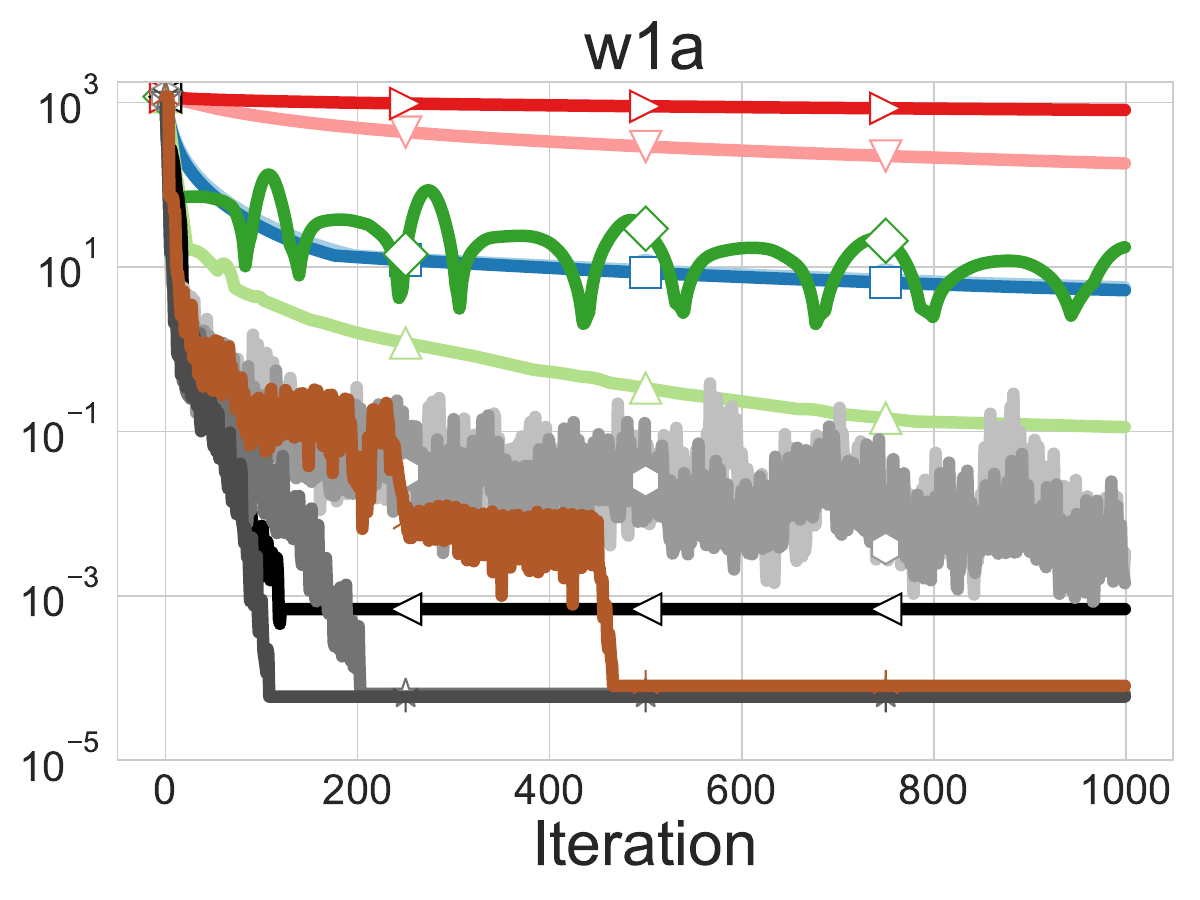}
\includegraphics[scale=0.2]{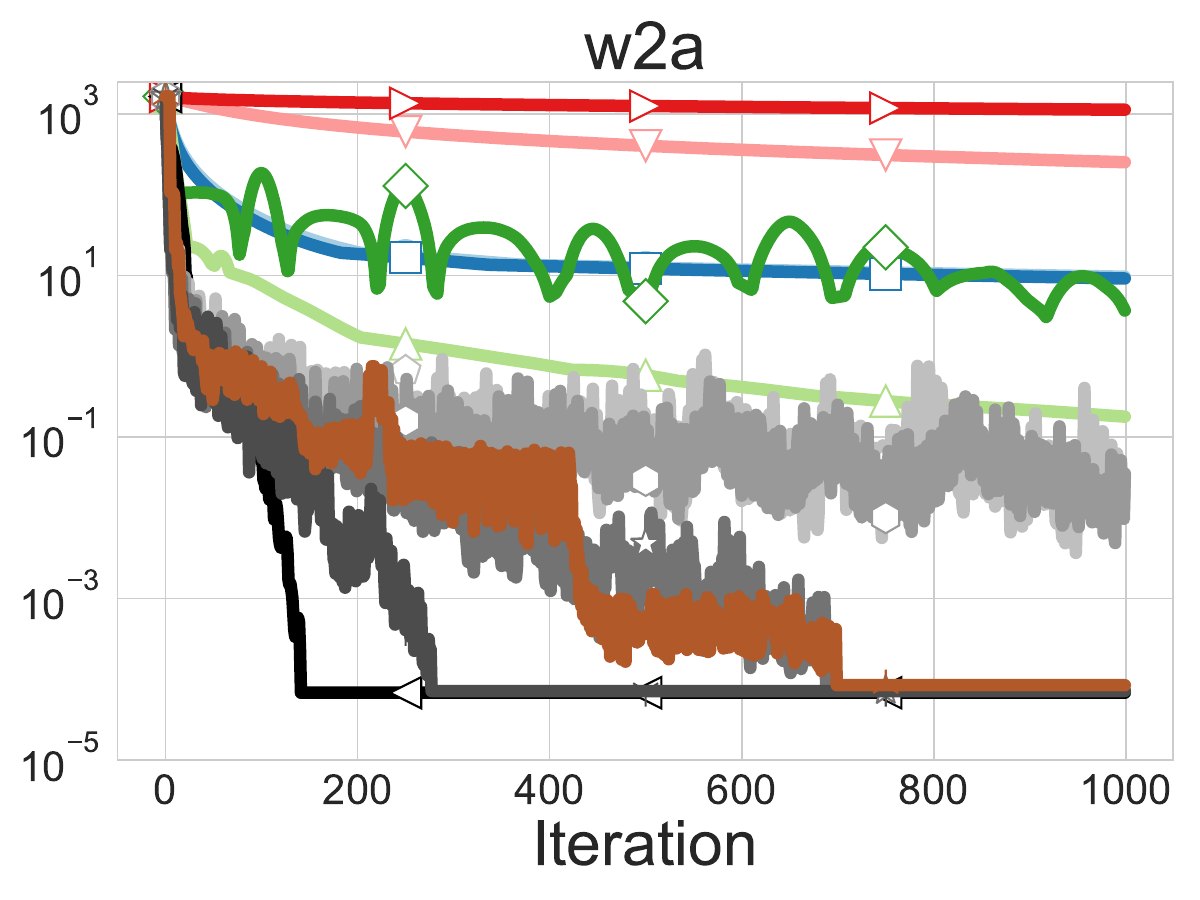}
\includegraphics[scale=0.2]{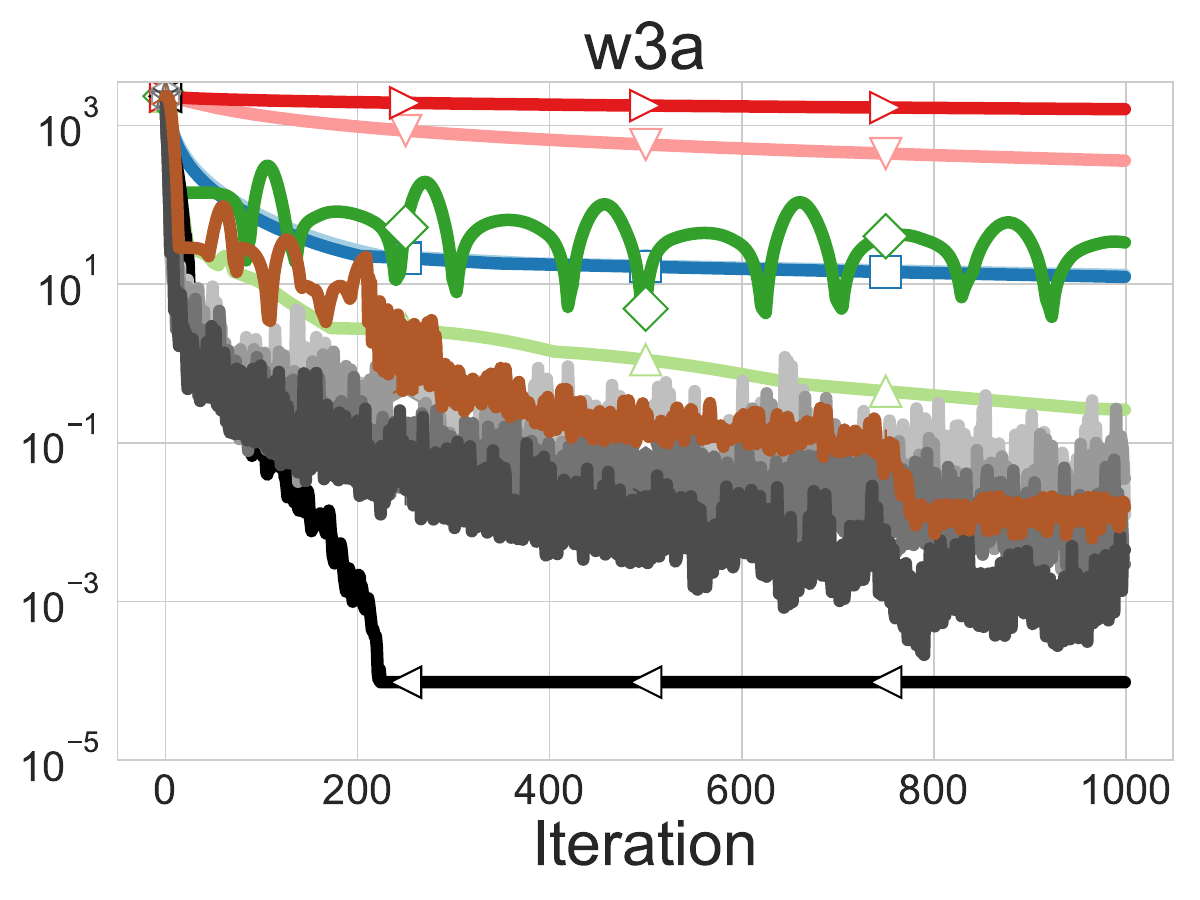}
\includegraphics[scale=0.2]{figs/w4a_gnorm_logistic.pdf}
\\
\includegraphics[scale=0.2]{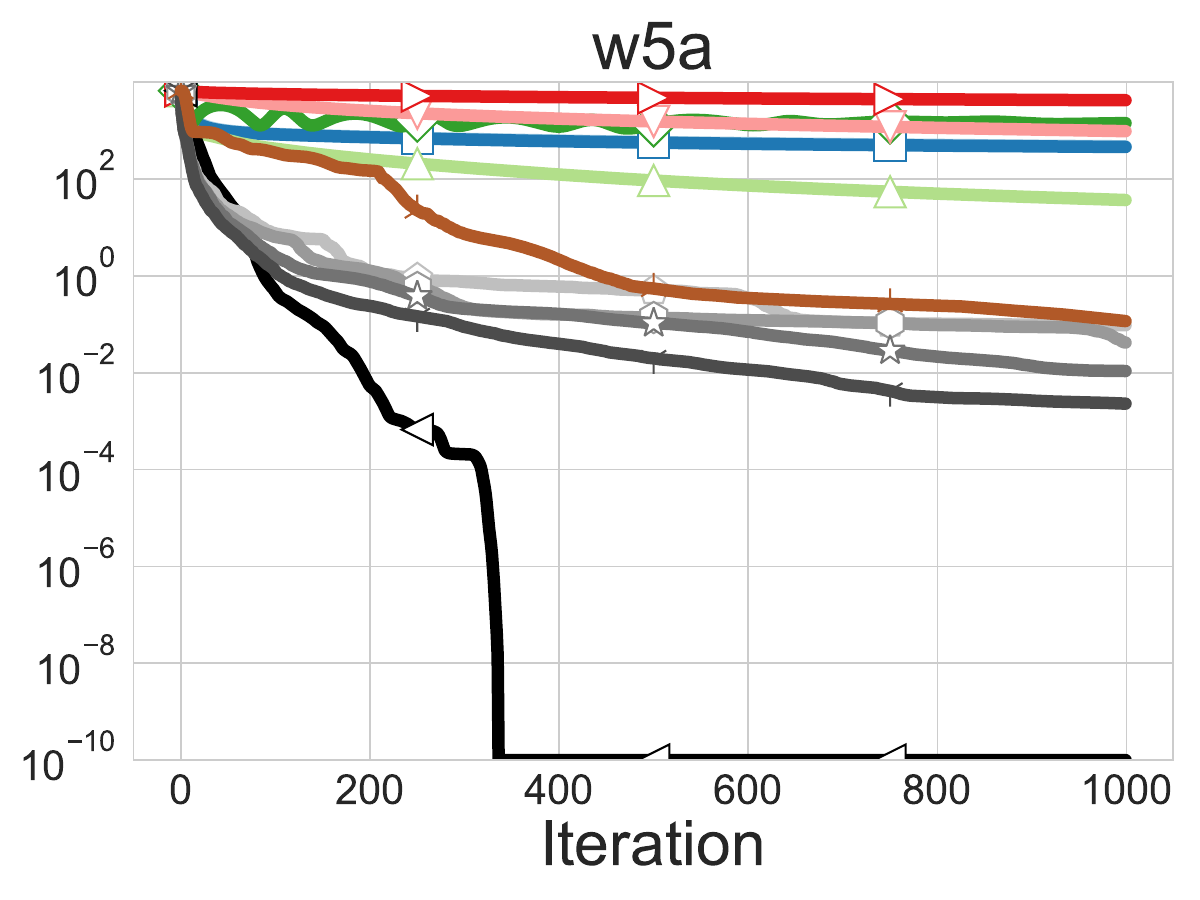}
\includegraphics[scale=0.2]{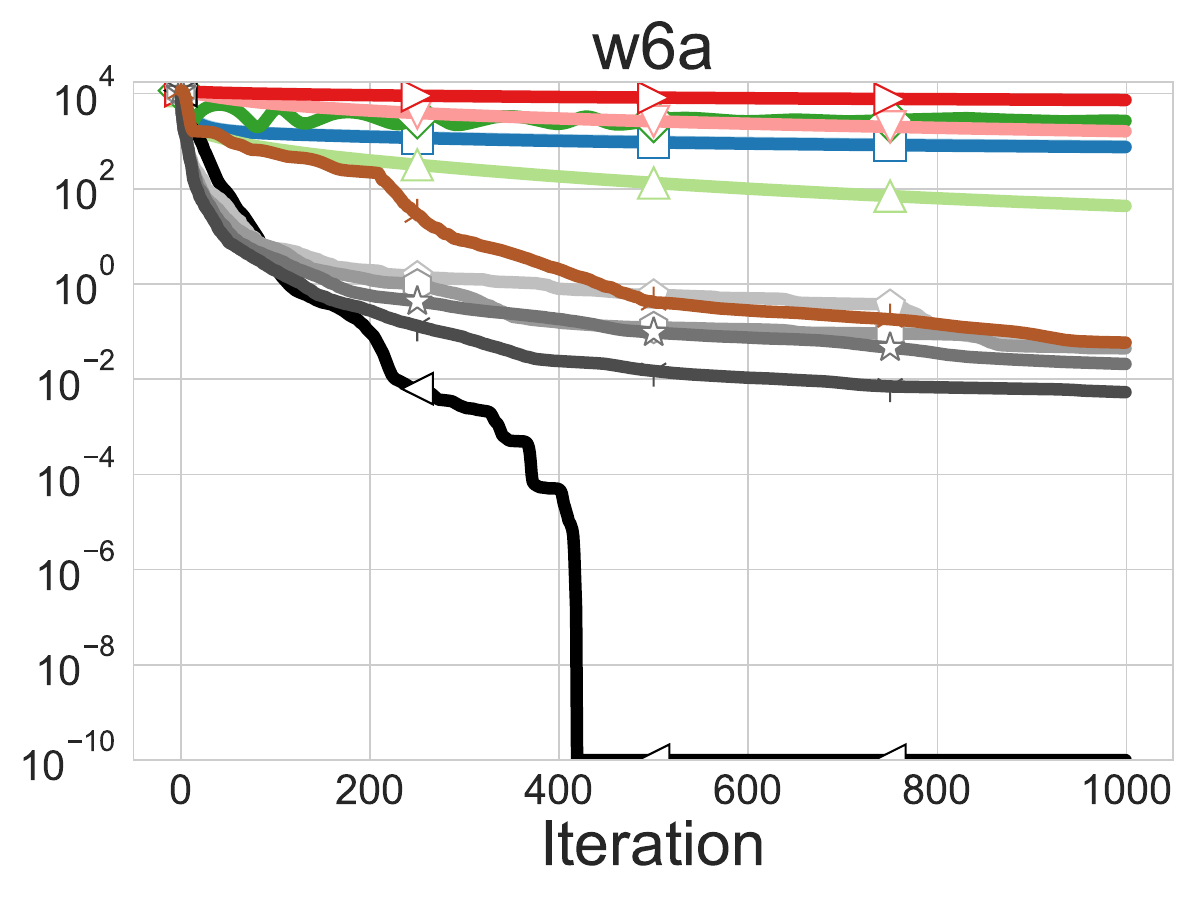}
\includegraphics[scale=0.2]{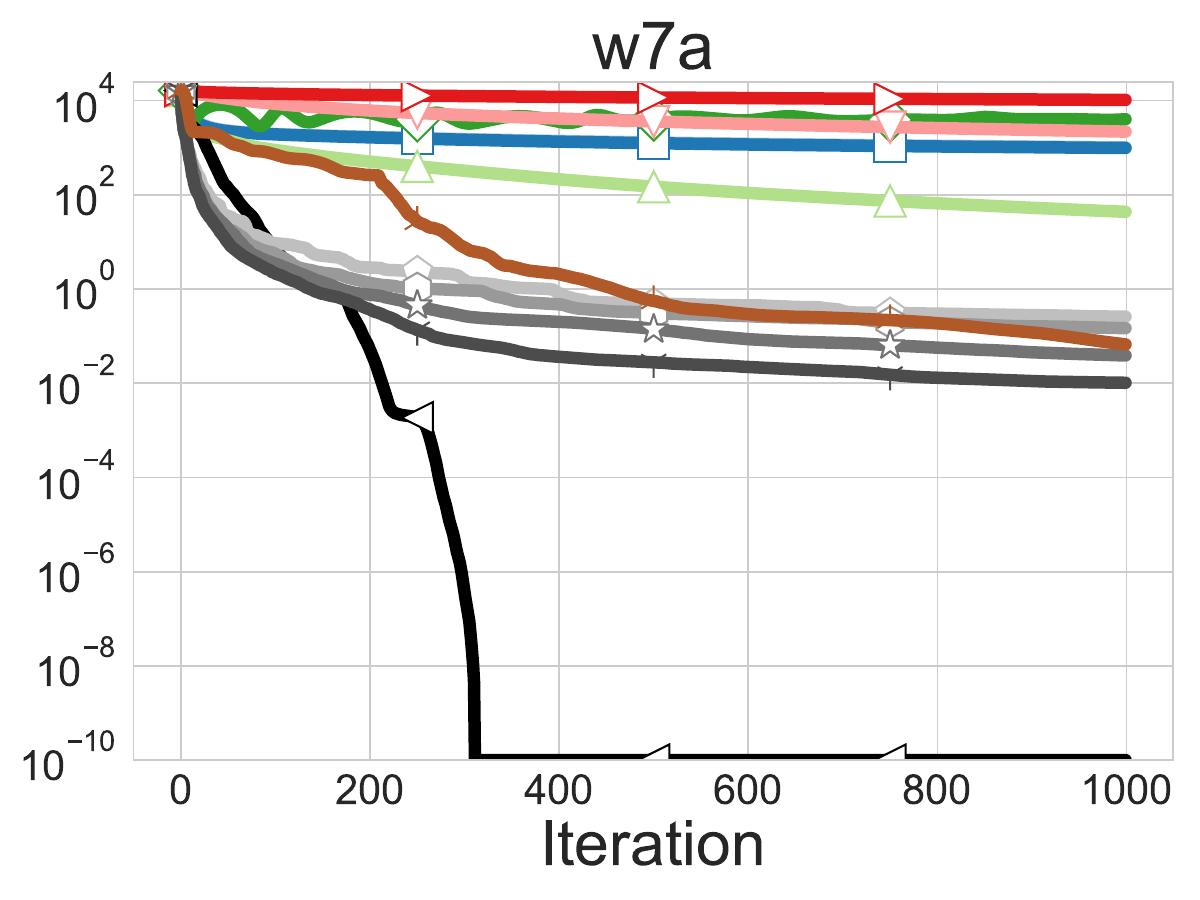}
\includegraphics[scale=0.2]{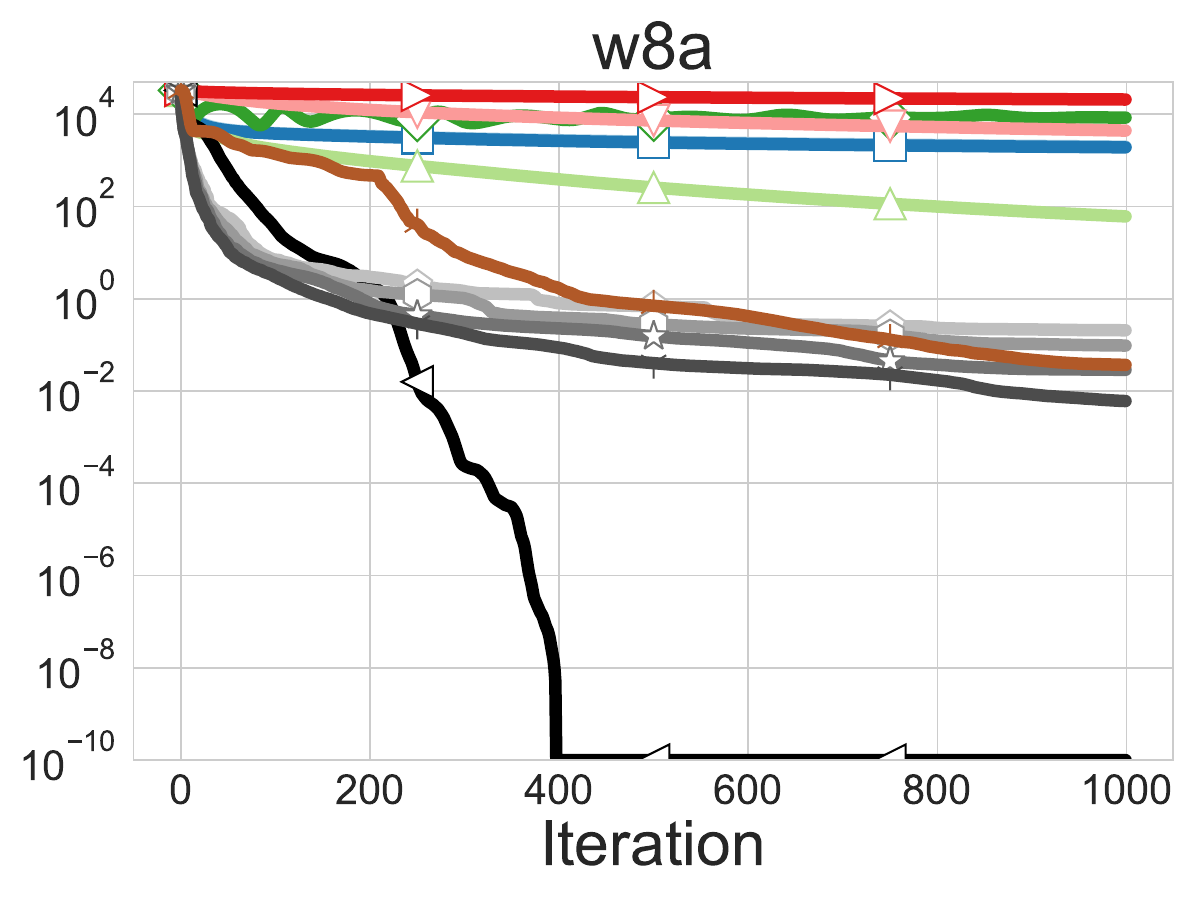}
\\
\includegraphics[scale=0.2]{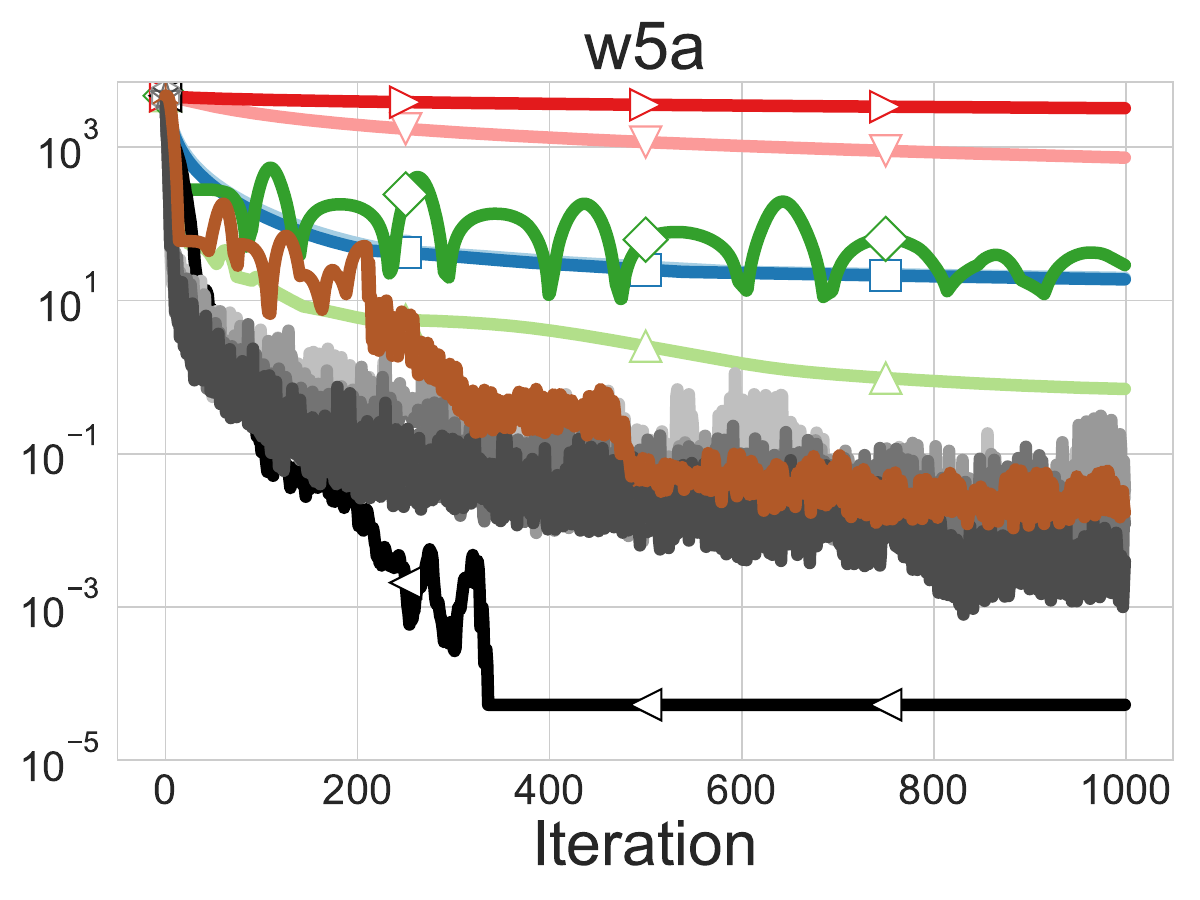}
\includegraphics[scale=0.2]{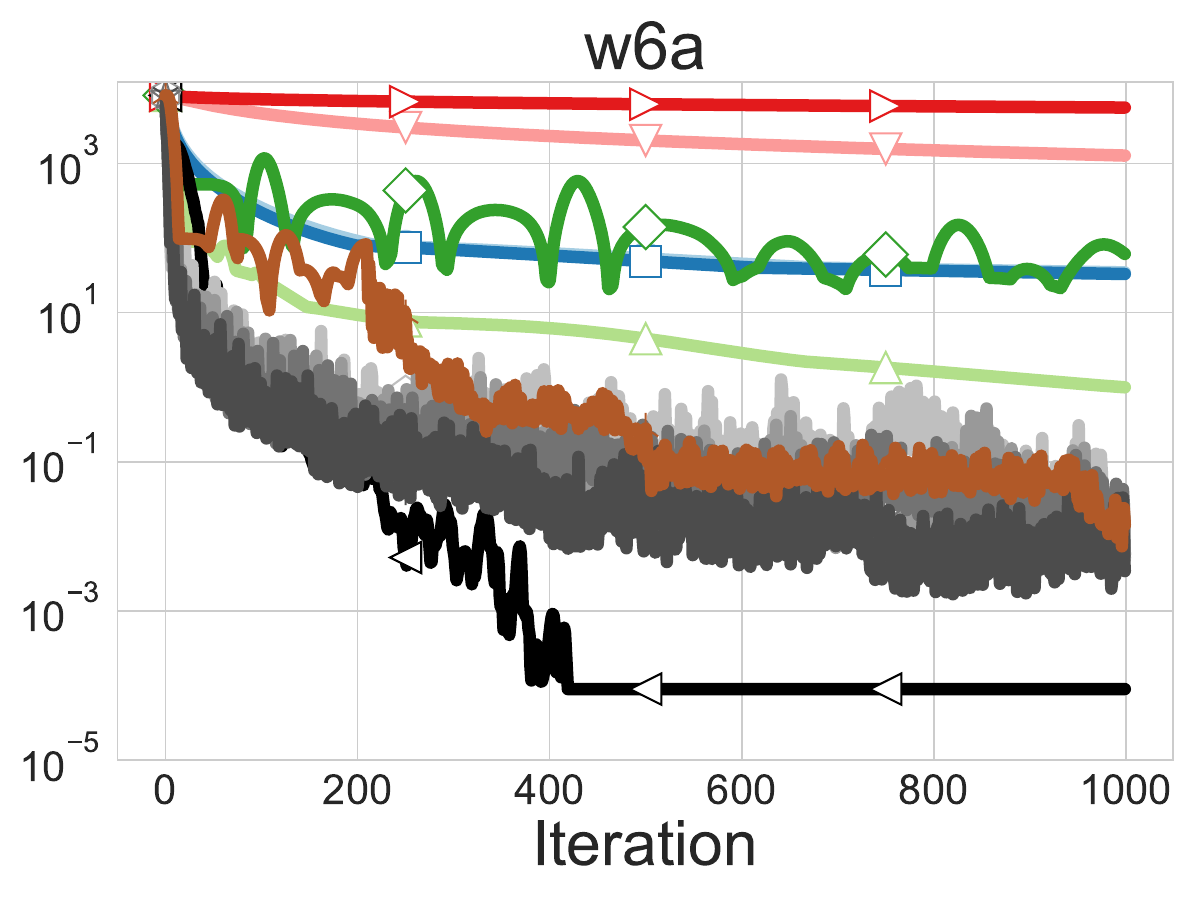}
\includegraphics[scale=0.2]{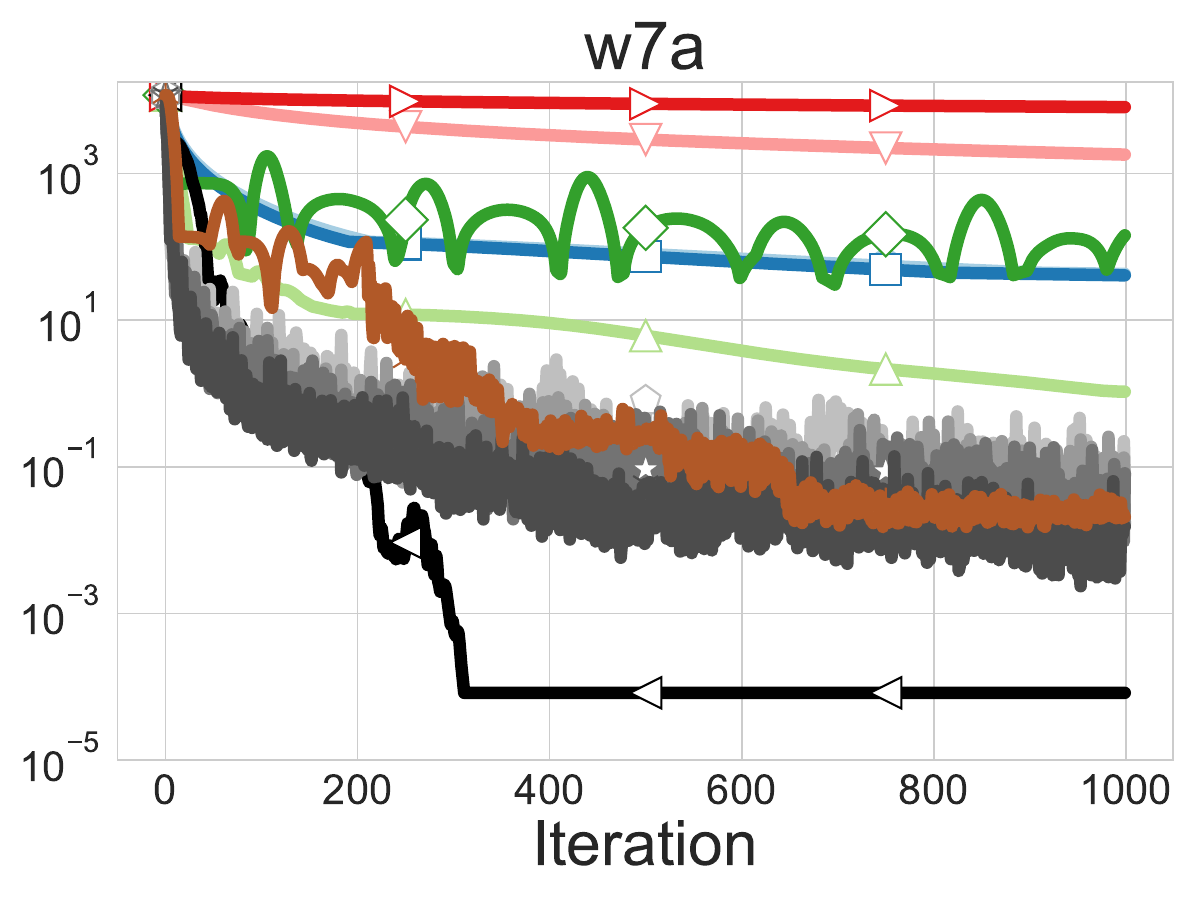}
\includegraphics[scale=0.2]{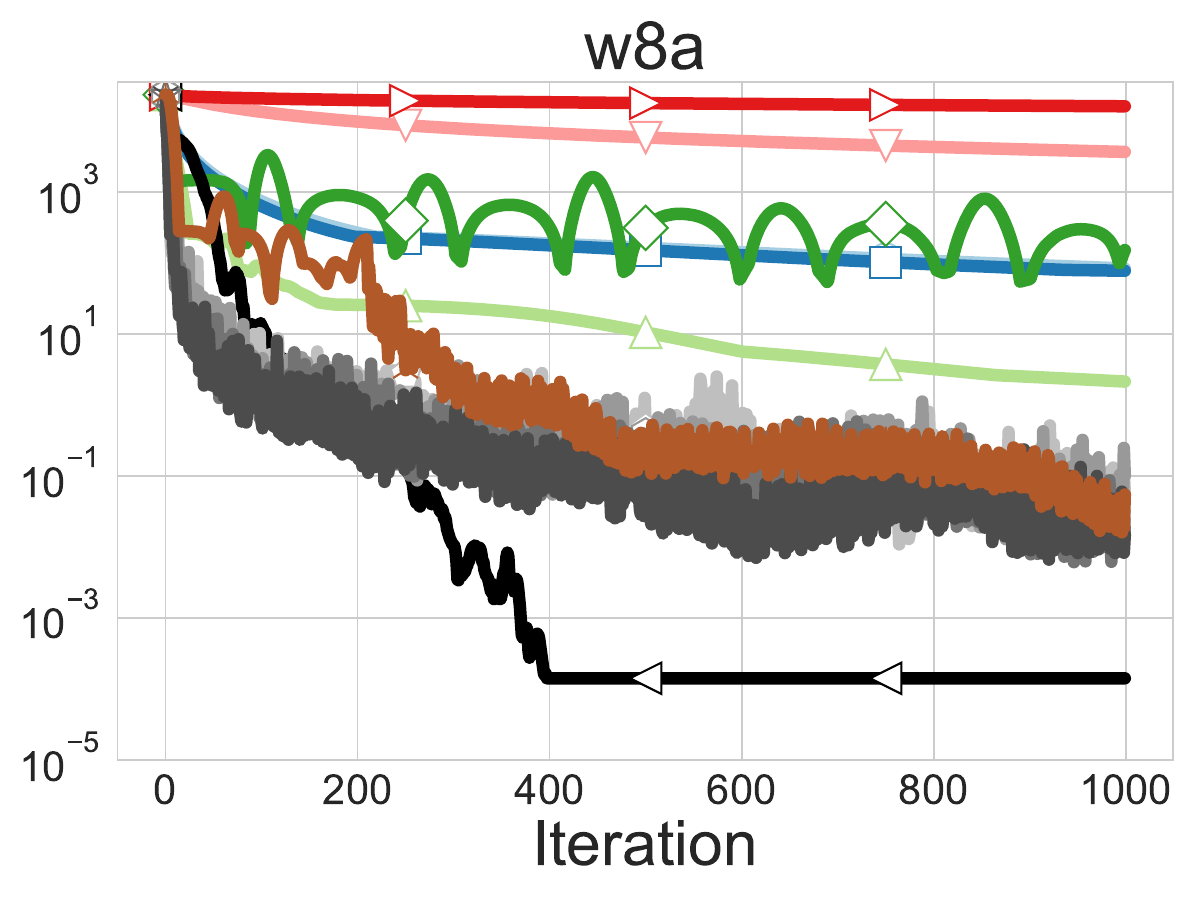}
	
\includegraphics[scale=0.38]{figs/legend.pdf}
\caption{More experiments on logistic regression problem}
\label{fig:log-add-2}
\end{figure}
\end{document}